\documentclass[twoside,11pt]{article}
\usepackage[preprint]{jmlr2e}
\usepackage{xcolor}
\usepackage{pdfpages}

\DeclareMathOperator{\rank}{\mathrm{rank}}
\DeclareMathOperator{\diag}{\mathrm{diag}}

\def\lsb{\left[}
\def\rsb{\right]}
\def\lab{\left|}
\def\rab{\right|} 
\def\la {\left\langle}
\def\ra {\right\rangle} \DeclareMathOperator{\SVD}{SVD}
\DeclareMathOperator{\HOSVD}{HOSVD}
\DeclareMathOperator{\offdiag}{off-diag}

\def\lb{\left(}
\def\rb{\right)}
\def \ttimes { \underset{i=1}{\overset{3}{ \times }}}
\def \jneqi {\underset{j\neq i}{ \times }}

\newcommand{\vecnorm}[2]{\left\| #1\right\|_{#2}}

\newcommand{\matsnorm}[2]{\left\| #1\right\|_{{#2}}}
\newcommand{\whf}[1]{{#1}}

\newcommand{\nucnorm}[1]{\ensuremath{\matsnorm{#1}{\footnotesize{\mbox{$\ast$}}}}}
\newcommand{\fronorm}[1]{\ensuremath{\matsnorm{#1}{\footnotesize{\mathsf{F}}}}}
\newcommand{\opnorm}[1]{\ensuremath{\matsnorm{#1}{}}}
\newcommand{\twoinf}[1]{\ensuremath{\matsnorm{#1}{\footnotesize{\mbox{2,$\infty$}}}}}
\newcommand{\infnorm}[1]{\ensuremath{\matsnorm{#1}{\footnotesize{\mbox{$\infty$} }}}}
\newcommand{\twonorm}[1]{\ensuremath{\matsnorm{#1}{\footnotesize{2}}}}

\newcommand{\bfm}[1]{\bm{#1}}

\newcommand{\E}[2][]{\mathbb{E}_{#1} \left\{ #2 \rule{0mm}{3mm}\right\}}
 \newtheorem{claim}{Claim}[section]
\newcommand{\cjc}[1]{{{#1}}}
\def \ProjTX {\calP_{\tangentX}}
\def \ProjTXL {\calP_{\tangentXL}}
\def \tangentX {T_{\calX^t}}
\def \tangentXL {T_{\calX^{t,\ell}}}
\def\lcb{\left\{}
\def\rcb{\right\}}
\def\va{\bfm a}   \def\mA{\bfm A}  
\def\vb{\bfm b}   \def\mB{\bfm B}  
\def\vc{\bfm c}   \def\mC{\bfm C}  
   \def\mD{\bfm D}  
\def\ve{\bfm e}   \def\mE{\bfm E}  
    
   \def\mG{\bfm G}  
   \def\mH{\bfm H}  
   \def\mI{\bfm I}

   \def\mM{\bfm M}  \def\M{\mathbb{M}}
     
   \def\mO{\bfm O}  
     
   \def\mQ{\bfm Q}  
\def\vr{\bfm r}   \def\mR{\bfm R}  \def\R{\mathbb{R}}
   \def\mS{\bfm S}  
   \def\mT{\bfm T}  
   \def\mU{\bfm U}  
   \def\mV{\bfm V}  
   \def\mW{\bfm W}  
\def\vx{\bfm x}   \def\mX{\bfm X}  
   \def\mY{\bfm Y}  
   \def\mZ{\bfm Z}

\def\calA{{\cal  A}} 
\def\calB{{\cal  B}} 
\def\calC{{\cal  C}} 
 
\def\calE{{\cal  E}} 
 
\def\calG{{\cal  G}} 
\def\calH{{\cal  H}} 
\def\calI{{\cal  I}} 
\def\calJ{{\cal  J}}

\def\calM{{\cal  M}}

\def\calP{{\cal  P}}

\def\calS{{\cal  S}} 
\def\calT{{\cal  T}}

\def\calX{{\cal  X}} 
\def\calY{{\cal  Y}} 
\def\calZ{{\cal  Z}}

\newcommand{\bfsym}[1]{\bm{#1}}

           \def\bDelta {\bfsym {\Delta}}

             \def\bSigma{\bfsym \Sigma}
        \def\bLambda {\bfsym {\Lambda}}

\def \tran {\mathsf{T}}

\DeclareMathOperator{\vect}{vec}

\def \bzero{\bm 0}

\usepackage{bbm}

\def \Rn {\R^{n \times n  \times n}}

\def \Pxtl {\calP_{T_{\calX^{t,\ell}} } }
\def \Pxt {\calP_{T_{\calX^{t}} } }

\firstpageno{1}

\makeatletter
\newcommand\blfootnote[1]{%
	\begingroup
	\renewcommand\thefootnote{}\footnote{#1}%
	\addtocounter{footnote}{-1}%
	\endgroup
}

\makeatother

\title{Implicit Regularization and Entrywise Convergence of Riemannian Optimization  for Low Tucker-Rank Tensor Completion}

\newcommand{\upstairs}[1]{\textsuperscript{#1}}
\newcommand{\affilone}{1}
\newcommand{\affiltwo}{2}

\begin{document}
\maketitle
\begin{center}

\vspace*{-0.8in}

\begin{tabular}{cc}
	Haifeng Wang\upstairs{\affilone,\! $\star$}, 
	Jinchi Chen\upstairs{\affiltwo,\! $\star$}, 
	and
	Ke Wei\upstairs{\affilone} \\[1.5ex]
	\upstairs{\affilone} School of Data Science, Fudan University, Shanghai, China\\
	\upstairs{\affiltwo} School of Mathematics, East China University of Science and Technology, Shanghai, China\\
\end{tabular}
\end{center}
\vspace*{.2in}

\allowdisplaybreaks
\begin{abstract}%
This paper is concerned with the low Tucker-rank tensor completion problem, which is about reconstructing a tensor $\calT\in\R^{n\times n\times n}$ of low multilinear rank from partially observed entries.
\whf{Riemannian optimization algorithms are a class of efficient methods for this problem, but the theoretical convergence  analysis  is still lacking. In this manuscript, we establish the entrywise convergence of the vanilla Riemannian gradient method for low Tucker-rank tensor completion under the nearly optimal sampling complexity $O(n^{3/2})$. Meanwhile, the implicit regularization phenomenon of the algorithm  has also been revealed. 
As far as we know, this is the first work that has shown the  entrywise
convergence and  implicit regularization property of a non-convex method for low
Tucker-rank tensor completion.
The analysis relies on the leave-one-out technique, and some of the technical results developed in the paper might be of broader interest in investigating the properties of other non-convex methods for this problem.}
\end{abstract}

\begin{keywords}
Low  rank tensor completion, Tucker decomposition,  Riemannian gradient,  entrywise convergence, implicit regularization,   leave-one-out
\end{keywords}

\section{Introduction}
Tensors\blfootnote{\upstairs{$\star$}Equal contribution.} are  multidimensional arrays which are ubiquitous in data analysis, including but  not limited to topic modeling \citep{anandkumar2015spectral}, community detection \citep{anandkumar2013tensor}, computer version \citep{ liu2012tensor}, collaborative filtering \citep{ karatzoglou2010multiverse}, and signal processing \citep{cichocki2015tensor}. 
In this paper, we consider the tensor completion problem which is about reconstructing  a tensor from a few observed entries. \whf{Without any additional assumptions, tensor completion is an ill-posed problem which does not even have a unique solution. Therefore computationally efficient solution of this problem is typically based on certain intrinsic low dimensional structures of tensors, a notable example of which is low rank. Compared with matrix, tensor has more complex rank notions up to different tensor decompositions such as CANDECOMP/PARAFAC (CP) decomposition \citep{hitchcock1927expression}, Tucker decomposition \citep{tucker1966some}, tensor train (TT) decomposition \citep{oseledets2011tensor}, and {\color{black}t-SVD decomposition \citep{zhang2016exact}}. In this manuscript, we focus on the Tucker decomposition. Then the low-rank tensor completion problem can be formulated as follows:
\begin{align}\label{eq:optProb}
\min_{\calX\in\R^{n\times n\times n}}~\frac{1}{2p}\fronorm{\calP_{\Omega}\lb\calX\rb-\calP_{\Omega}\lb\calT\rb}^2,\quad\mbox{s.t.}\quad \rank\lb\calX\rb = \vr
\end{align} 
where  $\calT\in\R^{n\times n\times n}$ is the target tensor to be recovered,  $\rank\lb\calX\rb$ is the Tucker rank of $\calX$ which will be specified later, $\Omega$ is a subset of indices for the observed entries, $p = |\Omega|/n^3$ is the sampling rate,  $\calP_{\Omega}$ is the element-wise sampling operator, {\color{black}and $\fronorm{\cdot}$ denotes the Frobenius norm (e.g., $\fronorm{\calX}^2=\sum_{i_1,i_2,i_3}\calX_{i_1,i_2,i_3}^2$)}.}

\subsection{Main contributions}

For the low rank tensor completion problem under Tucker decomposition, many  methods have been developed, see \citep{gandy2011tensor, huang2015provable,han2020optimal, kressner2014low,liu2012tensor, luo2021low, mu2014square,xia2019polynomial,rauhut2017low,tong2021scaling}  and references therein. Among them, Riemannian optimization algorithms are a class of efficient methods. Despite the computational efficiency of Riemannian optimization,  there still lacks theoretical analysis for them. In this manuscript, we fill this gap by providing an entrywise convergence of the   vanilla Riemannian gradient method (RGM) for low rank tensor completion. 

RGM for tensor completion is an extension of the method for matrix completion \citep{vandereycken2013low,wei2016guarantees,wei2020guarantees}. The iteration of RGM is given by
\begin{align*}
\calX^{t+1} = \HOSVD\lb\calX^t-p^{-1}\calP_{T_{\calX^{t}}}\calP_{\Omega}\lb\calX^{t}-\calT\rb\rb,
\end{align*}
where the retraction $\HOSVD$ and the projection $\calP_{T_{\calX^{t}}}$ are specified in Section~\ref{sec: preliminaries}.
Assume the target tensor $\calT$ is incoherent, i.e., there exists a  constant $\mu>0$ such that 
\begin{align*}
\|\mU_i\|_{2,\infty}\leq\sqrt{\frac{\mu r}{n}},\quad i=1,2,3,
\end{align*} 
where $\mU_i\in\R^{n\times r}$ $(i = 1,2,3)$ are the factor matrices of the tensor $\calT$, {\color{black} and $\twoinf{\mU_i} := \max_{j\in [n]}\twonorm{\lsb\mU_i\rsb_{j,:}}$ is the $\ell_{2,\infty}$ norm of $\mU_i$}. Let $\calX^t\in\R^{n\times n\times n}$ with factor matrices $\mX^{t}_i\in\R^{n\times r}$ ($i = 1,2,3$) denote the iterate generated by the RGM in the $t$-th iteration. We have the following main result.
\begin{theorem}\label{thm:informal}
Assume each entry of $\calT$ is observed independently with probability $p$. Let the condition number of $\calT$ be $\kappa$. If $p\geq \frac{O(r^6\kappa^8\log^3 n)}{n^{3/2}}$, the iterates of the Riemannian gradient method {\color{black}(Algorithm~\ref{alg:rgd})} satisfy the following properties with high probability:
\begin{enumerate}
    \item Incoherence of the iterates:
    \begin{align}\label{eq:informal twotoinfinity}
        \twoinf{\mX_i^t} \leq 2\kappa\sqrt{\frac{\mu r}{n}},\quad t\ge 1.
    \end{align}
    \item The iterates $\calX^t$ converges linearly to $\calT$ in terms of the infinity norm: 
    \begin{align}\label{eq: informal infinity convergence}
        \infnorm{\calX^t-\calT} &\leq \lb\frac{1}{2}\rb^t\sigma_{\max}\lb\calT\rb,\quad t\geq 1,
    \end{align}
    {\color{black} where $\sigma_{\max}\lb\calT\rb$ denotes the upper bound on the largest singular values of the matricization of $\calT$ along every mode (see \eqref{eq:sigma_max_ke}), and $
    \infnorm{\calX^t-\calT} := \max_{i_1,i_2,i_3}\lab(\calX^t-\calT)_{i_1,i_2,i_3}\rab
$ denotes the infinity norm of $\calX^t-\calT$.}
        In particular, for the first logarithm number of iterations,  a stronger entrywise convergence  of RGM can be established (see Theorem~\ref{thm: finite iteration}):
    \begin{align*}
        \infnorm{\calX^t-\calT} &\leq \lb\frac{1}{2}\rb^t\frac{1}{n^{3/2}}\sigma_{\max}\lb\calT\rb, \quad \mbox{for }1\leq t\leq t_0,\mbox{ where }t_0=O(\log_2n).
    \end{align*}
\end{enumerate}
\end{theorem}

\begin{remark}
As a byproduct, the convergence of the vanilla RGM for low rank matrix completion follows immediately from Theorem~\ref{thm:informal}. This is also a new result for the matrix case since previous results either require a  stronger initialization procedure\citep{wei2020guarantees} or require an additional  projection in the algorithm \citep{cai2021provable}.
{\color{black} The analyses in \citep{wei2020guarantees,cai2021provable} are all based on the Frobenius norm error metric, which cannot fully exploit the incoherence property of each iterate and thus require sample splitting for initialization (so that the initialization can be unreasonably close to the ground truth)   or explicit projection onto the incoherence region in the algorithm. However, these additional steps are empirically not necessary. In contrast, the analysis in this paper is based on the $\ell_{2,\infty}$ norm and infinity norm which enables us to analyse the incoherence property of each iterate more carefully and establish the convergence of the algorithm that does not have those empirically redundant steps. }
\end{remark}

Inequality~\eqref{eq:informal twotoinfinity} shows that iterates of the RGM remain incoherent even in the absence of explicit regularization which is known as the {\it implicit regularization} phenomenon. To the best of our knowledge, this  is the first work that has shown the stronger entrywise convergence and the implicit regularization phenomenon of a non-convex method for low Tucker-rank tensor completion under the nearly optimal sampling complexity $O(n^{3/2})$.    Table~\ref{theoritical comparison} summarizes the theoretical recovery guarantees of different nonconvex algorithms for both the  Gaussian measurement model and the entrywise measurement model. 
\begin{table}[ht!]
\centering
\caption{Theories of different nonconvex methods for low Tucker-rank tensor completion. 
}
\vspace{0.2cm}
\label{theoritical comparison}
\begin{tabular}{|c|c|c|c|}
\hline Algorithms                    & Sampling complexity                                                                            & Error metric & Sampling scheme  \\\hline
\begin{tabular}[c]{@{}c@{}}Projected GD\\ {\citep{chen2019non}}\end{tabular}                     & $n^{2}r$ & Frobenius &  Gaussian             \\\hline 
\begin{tabular}[c]{@{}c@{}}Regularized GD\\ {\citep{han2020optimal}}\end{tabular}                     & $n^{3/2}r\kappa^4$ & Frobenius &  Gaussian             \\\hline 
\begin{tabular}[c]{@{}c@{}}RGM\\ {\citep{cai2020provable}}\end{tabular}                     & $n^{3/2}r^2\kappa^2$ & Frobenius &  Gaussian             \\\hline 
\begin{tabular}[c]{@{}c@{}}Riemannian Gauss-Newton\\ {\citep{luo2021low}}\end{tabular}                     & $n^{3/2}r^{3/2}\kappa^4$ & Frobenius &  Gaussian             \\\hline  
\begin{tabular}[c]{@{}c@{}}ScaledGD\\ {\citep{tong2021scaling}}\end{tabular}                     & $n^{3/2}r\kappa^2$ & Frobenius &  Gaussian               \\\hline  
\begin{tabular}[c]{@{}c@{}}Grassmannian GD\\ {\citep{xia2019polynomial}}\end{tabular}              & $n^{3/2}r^{7/2}\kappa^4\log^{7/2}n$                       &  Frobenius        & Entrywise  \\\hline
\begin{tabular}[c]{@{}c@{}}ScaledGD\\ {\citep{tong2021scaling}}\end{tabular}                     & $n^{3/2}r^2\kappa\lb\sqrt{r}\vee\kappa^2\rb\log^3n$ & Frobenius &  Entrywise           \\\hline
\begin{tabular}[c]{@{}c@{}}RGM\\ {[}\textbf{this paper}{]}\end{tabular}                     & $n^{3/2}r^6\kappa^8\log^3n$ & \textbf{Infinity} &  Entrywise           \\\hline   
\end{tabular}
\end{table}

The proof of Theorem~\ref{thm:informal} relies on the leave-one-out technique which has been widely used in analyzing various high dimensional data processing methods. Our work is mostly inspired by ~\citep{cai2021nonconvex} for  low rank tensor completion problem under  the CP decomposition and by~\citep{ding2020leave} for low rank matrix completion, but the technical details are substantially different. On the one hand,  CP decomposition and Tucker decomposition are essentially two different decompositions for {\color{black}tensors} and a gradient descent algorithm is analysed in \citep{cai2021nonconvex}. Even though a few  results for the initialization step therein can be used in our proof, the proof details for the iteration procedure are significantly different. Even in the initialization step, we find that the structure of Tucker decomposition can be used to simplify the proof of a related result (Lemma~\ref{spectral norm of delta 1,ell}). On the other hand, tensors are more complex than matrices which means that, compared with the analysis of iterative hard thresholding (IHT) for low rank matrix completion,  the analysis  of RGM for tensor completion are much more complicated. Moreover, differing from IHT, there is one  key different component in RGM, namely the orthogonal projection $\calP_{T_{\calX^t}}$. Indeed,  we need to carefully leverage this projection to obtain the nearly optimal sampling complexity $O(n^{3/2})$. {\color{black}Without this projection step, the near-optimal sampling
complexity can not be achieved for the low rank tensor completion problem.  This is basically because in the tensor case we need to transform the perturbation tensor into an $n\times n^2$ matrix via matricization in order to bound it and exploiting the low dimensional structure is essential for us to reduce the dependence on $n$ from $n^2$ to $n^{3/2}$.  In contrast, for the low rank matrix completion problem,
the convergence result can be established with or without this projection step since it does
not involve the transformation from an tensor to an $n\times n^2$ matrix.}

\subsection{Organization and Notation}
This paper is organized as follows. Section~\ref{algorithm and main result} begins with some tensor preliminaries and then introduces the Riemannian gradient method for low Tucker-rank tensor completion together with the main theoretical results.  The proof strategies of the main results are outlined in Sections~\ref{proof architecture} and~\ref{sec: local}, while the details are provided {\color{black}in the appendix. In Section~\ref{sec:conclusion}, we conclude this paper with a few future directions}.

Throughout this paper, tensors are denoted by capital calligraphic letters, matrices  are denoted by bold capital letters, and vectors are denoted by bold lower case letters. In particular, we will reserve $\calT$ for the ground truth tensor to be recovered, which is an ${n\times n\times n}$ tensor  of multilinear  rank $\vr=(r_1,r_2,r_3)$. For ease of exposition, we assume $r_1=r_2=r_3=r$ for the ground truth tensor. For each integer $d$, define the set $[d] = \left\lbrace 1,2,\cdots, d\right\rbrace$. We denote by $\ve_i$ the $i$-th canonical basis vector,  by $\mI$ the identity matrix with suitable size, and by $\calJ$ the all-one tensor (i.e., all the entries of $\calJ$  are $1$).  For any matrix $\mM$, we use $\lsb\mM\rsb_{i,:}$, $\lsb\mM\rsb_{:,j}$, $\lsb\mM\rsb_{i,j}$ to represent the $i$-th row, the $j$-th column, the $(i,j)$-th element of $\mM$,  respectively. The Frobenius norm, spectral norm and nuclear norm of a matrix $\mM$ are denoted by $\fronorm{\mM}$ , $\opnorm{\mM}$, $\nucnorm{\mM}$, respectively. In addition, the $\ell_{2,\infty}$ norm of a matrix $\mM$ is defined as
 $\twoinf{\mM} := \max_{i\in [n]}\twonorm{\lsb\mM\rsb_{i,:}}$. 
{\color{black}For a tensor $\calX$, its  infinity norm is defined as
$
    \infnorm{\calX} := \max_{i_1,i_2,i_3}\lab\calX_{i_1,i_2,i_3}\rab.
$}
We will use $C, C_0,C_1, \cdots$  to  denote absolute positive constants, whose values may vary from line to line. Lastly, we use the terminology "with high probability" to denote the event happens with probability at least $1-C_1n^{-C_2}$ for some constants $C_1,C_2>0$ and $C_2$ {\color{black}sufficiently} large.
\section{Algorithm and Main Result}
\label{algorithm and main result}
\subsection{Preliminaries}
\label{sec: preliminaries}
We begin this section with some preliminaries of tensors; for a more detailed exposition, see \citet{kolda2009tensor}. For conciseness, we restrict our discussion to $n\times n\times n$ three-way  tensors.

\paragraph{Tensor matricization.} Matricization, also known as unfolding,  transforms a tensor into a matrix along different modes. Given a tensor $\calX\in\R^{n\times n\times n}$, the matricization operators are defined as
\begin{align*}
&\calM_1\lb\calX\rb\in\R^{n\times n^2} : \left[\calM_1\lb\calX\rb\right]_{i_1, i_2+n(i_3-1)} = \calX_{i_1, i_2, i_3},\\
&\calM_2\lb\calX\rb\in\R^{n\times n^2} : \left[\calM_2\lb\calX\rb\right]_{i_2, i_1+n(i_3-1)} = \calX_{i_1, i_2, i_3},\\
&\calM_3\lb\calX\rb\in\R^{n\times n^2} : \left[\calM_3\lb\calX\rb\right]_{i_3, i_1+n(i_2-1)} = \calX_{i_1, i_2, i_3}.
\end{align*}
\paragraph{Mode-$d$ tensor multiplication.} The mode-$1$ product of a tensor $\calX\in\R^{n\times n\times n}$ with a matrix $\mA\in\R^{m\times n}$, denoted  $\calX\times_1\mA$, gives a tensor of size $m\times n\times n$. Elementwise, we have
\begin{align*}
\lb \calX\times_1\mA\rb_{j_1,i_2,i_3} = \sum_{i_1 = 1}^{n}\calX_{i_1,i_2,i_3}\lsb\mA\rsb_{j_1,i_1},
\end{align*}
and $\times_2$ and $\times_3$ are similarly defined.
 A few facts regarding mode-$d$ tensor multiplication are in order:
 \begin{align*}
\calX\times_i\mA\times_j\mB = \calX\times_j\mB \times_i\mA~(i\neq j)\quad\mbox{and}\quad \calX\times_i\mA\times_i\mB = \calX\times_i\lb\mB\mA\rb.
\end{align*}
\paragraph{Tensor norms.} The inner product between two tensors is defined as
 \begin{align*}
     \left\langle\calX,\calZ\right\rangle := \sum_{i_1,i_2,i_3}\calX_{i_1,i_2,i_3}\cdot\calZ_{i_1,i_2,i_3}.
 \end{align*}
The Frobenius norm and the spectral norm of a tensor are defined as
\begin{align*}
    \fronorm{\calX} := \sqrt{\left\langle\calX,\calX\right\rangle}\quad\mbox{and}\quad
\opnorm{\calX} := \sup_{\vx_i\in\R^{n}:\twonorm{\vx_i}=1} \left\langle \calX, \vx_1\circ\vx_2\circ\vx_3\right\rangle,
\end{align*}
where the element of  $\vx_1\circ\vx_2\circ\vx_3\in\R^{n\times n\times n}$ is defined by
\begin{align}\label{rank one tensor}
    \lsb\vx_1\circ\vx_2\circ\vx_3\rsb_{i_1,i_2,i_3} := [\vx_1]_{i_1}\cdot [\vx_2]_{i_2}\cdot [\vx_3]_{i_3}.
\end{align}
The following basic relations regarding the spectral norm, which follow  immediately from the definition, will be very useful: for any $i\in[3]$,
\begin{align*}
\opnorm{\calX}\leq \opnorm{\calM_i\lb\calX\rb}\quad\mbox{and}\quad\opnorm{\calX \times_i \mX} \leq \opnorm{\mX}\cdot \opnorm{\calX}.
\end{align*}
Similar to  the matrix case, the nuclear norm is the dual of spectral norm:
\begin{align*}
    \nucnorm{\calX} := \sup_{\calZ\in\R^{n\times n\times n}, \opnorm{\calZ}\leq 1} \left\langle \calX,\calZ\right\rangle.
\end{align*}
Recall that the condition number for a matrix $\mA$ is given by $\kappa\lb\mA\rb = \sigma_{\max}\lb \mA\rb/\sigma_{\min}\lb \mA\rb$ where $\sigma_{\max}$ and $\sigma_{\min}$ are the largest and smallest nonzero singular values of $\mA$ respectively. This concept can be naturally generalized to  tensors, 
\begin{align*}
    \kappa\lb\calX\rb :=\frac{\sigma_{\max}\lb\calX\rb}{\sigma_{\min}\lb\calX\rb},
\end{align*}
where $\sigma_{\max}\lb\calX\rb$ and $\sigma_{\min}\lb\calX\rb$ are defined as
\begin{align*}
\sigma_{\max}\lb\calX\rb & := \max\left\lbrace \sigma_{\max}\lb\calM_1\lb\calX\rb\rb, \sigma_{\max}\lb\calM_2\lb\calX\rb\rb, \sigma_{\max}\lb\calM_3\lb\calX\rb\rb\right\rbrace\numberthis\label{eq:sigma_max_ke}\\
\sigma_{\min}\lb\calX\rb & := \min\left\lbrace \sigma_{\min}\lb\calM_1\lb\calX\rb\rb, \sigma_{\min}\lb\calM_2\lb\calX\rb\rb, \sigma_{\min}\lb\calM_3\lb\calX\rb\rb\right\rbrace\numberthis\label{eq:sigma_min_ke}.
\end{align*}

\paragraph{Tucker decomposition and HOSVD.} The Tucker decomposition is a higher-order generalization of singular value decomposition (SVD), which has the form
\begin{align}\label{tucker decomposition}
\calX = \calG\times_1\mX_1\times_2\mX_2\times_3\mX_3 = \calG\ttimes\mX_i,
\end{align}
where $\calG\in\R^{r_1\times r_2\times r_3}$ is referred to as the core tensor, and  $\mX_i\in\R^{n\times r_i}$ for $i = 1,2,3$ are the factor matrices. Because $\calG$ is usually unstructured, we can always write the  Tucker  decomposition of $\calX$ into  the form where $\mX_i$ are orthonormal matrices. Given the Tucker decomposition of $\calX$, its matricizations  are given by
\begin{align*}
&\calM_1\lb\calX\rb = \mX_1\calM_1\lb\calG\rb\lb\mX_3\otimes\mX_2\rb^{\tran},\\& \calM_2\lb\calX\rb = \mX_2\calM_2\lb\calG\rb\lb\mX_3\otimes\mX_1\rb^{\tran},\\
&\calM_3\lb\calX\rb = \mX_3\calM_3\lb\calG\rb\lb\mX_2\otimes\mX_1\rb^{\tran},
\end{align*}
where  $\otimes$ denotes the Kronecker product of matrices.

A tensor $\calX\in\R^{n\times n\times n}$ is said to be of multilinear rank $\vr = (r_1, r_2, r_3)$ if $\rank\lb \calM_i(\calX)\rb = r_i$ for $i = 1,2,3$. It  is evident that $1\leq r_i\leq n$ for $i=1,2,3$.
It is well known that SVD can be used to  find the  best low rank approximation of  a matrix. In contrast, computing the best low rank approximation of a  tensor is an NP hard \citep{hillar2013most} problem. That being said, there exists a higher order analogue of SVD, known as Higher Order Singular Value Decomposition (HOSVD) \citep{de2000multilinear}, which is able to return a quasi-optimal approximation; see Algorithm~\ref{alg: hosvd}. HOSVD first estimates the principle factor matrices of each mode by an SVD trunctation of the corresponding matricization, and then formulates the core tensor by multiplying $\calX$  by the transpose of the factor matrix along each mode. Denoting by $\calH_{\vr}(\calX)$  the output of  HOSVD, there holds {\color{black}\citep{de2000multilinear}}
\begin{align*}
\fronorm{\calX-\calH_{\vr}\lb\calX\rb}\leq \sqrt{3}\inf_{\rank\lb\calM_i\lb{\calZ}\rb\rb\leq r_i}\fronorm{\calX-\calZ}.
\end{align*}
Note that when $\calX$  is already  of multilinear rank $\vr$, HOSVD returns the exact Tucker decomposition of $\calX$.
\begin{algorithm}[ht!]
	\caption{HOSVD}
	\label{alg: hosvd}
    \begin{algorithmic}[1]
	\State Input: Tensor $\calX\in\R^{n\times n\times n}$, multilinear rank $\vr = (r_1,r_2,r_3)$.
	\For{$i = 1,2,3$}
	\State $\mX_i = \SVD_{r_i}(\calM_i(\calX))$.
	\EndFor
	\State $\calG = \calX \ttimes\mX_i^\tran $.
	\State Output: $\calG \ttimes\mX_i$.
	\end{algorithmic}
\end{algorithm}

\paragraph{Tensor manifold.} A collection of tensors with multilinear rank $\vr = (r_1,r_2 ,r_3)$  forms a smooth embedded submanifold of $\R^{n\times n\times n}$ \citep{koch2010dynamical}, denoted $\mathbb{M}_{\vr}$, i.e.,
\begin{align*}
\mathbb{M}_{\vr}  = \left\lbrace \calX\in\R^{n\times n\times n}~\mid~\rank(\calX) = \vr\right\rbrace. 
\end{align*} 
Let the Tucker decomposition of $\calX$ be
    $\calX = \calG\ttimes{\mX_i}$,
where $\mX_i^\tran\mX_i = \mI \in\R^{r_i\times r_i}$ and $\calG\in\R^{r_1\times r_2\times r_3}$ has full multilinear rank. 
 The tangent space of $\mathbb{M}_{\vr}$ at $\calX$ is given by \citep{koch2010dynamical}
 \begin{small}
\begin{align*}
	T_{\calX}= \left\{ \calC \ttimes \mX_i + \sum_{i=1}^{3} \calG \times_i \mW_i \jneqi \mX_j ~\mid~ \calC\in\R^{r_1\times r_2\times r_3},\mW_i\in\R^{n\times r_i},   \mW_i^\tran \mX_i = \bzero, i = 1,2,3 \right\}.
\end{align*}
\end{small}{\color{black}Given a tensor $\calZ$, the orthogonal projection of  $\calZ$ onto $T_{\calX}$, denoted $\calP_{T_{\calX}}\lb\calZ\rb$, has the form $\calP_{T_{\calX}}\lb\calZ\rb=\calC \ttimes \mX_i + \sum_{i=1}^{3} \calG \times_i \mW_i \jneqi \mX_j$, where $\calC$ and $\mW_i$ are to be determined. Since the summands are orthogonal to each other, $\calC$ and $\mW_i$ can be obtained independently
by solving the least squares problems, yielding the expression \citep{koch2010dynamical, kressner2014low,cai2020provable}}
\begin{align}
\label{eq: def of projection onto tangent space}
	\calP_{T_{\calX}}\lb\calZ\rb = \calZ\ttimes \mX_i\mX_i^\tran + \sum_{i=1}^3 \calG\times_i \mW_i \jneqi \mX_j,
\end{align}
where $\mW_i$ is defined as
\begin{align*}
	\mW_i = \lb\mI-\mX_i\mX_i^\tran\rb \calM_i\lb \calZ \jneqi \mX_j^{\tran} \rb \calM_i^\dagger\lb\calG\rb.
\end{align*}
Here $\calM_i^\dagger\lb\calG\rb$ denotes the Moore-Penrose  pseudoinverse of  $\calM_i\lb\calG\rb$ \citep{golub1996matrix}, which obeys that $\calM_i\lb\calG\rb\calM_i^\dagger\lb\calG\rb=\mI$ and $\calM_i^\dagger\lb\calG\rb\calM_i\lb\calG\rb$ is an orthogonal projector. 

\subsection{Riemannian Gradient Method}
\begin{algorithm}[ht!]
	\caption{Riemannian Gradient Method (RGM)}
	\label{alg:rgd}
	\begin{algorithmic}[1]
	\State Input: Initialization $\calX^1$ {\color{black} generated via Algorithm~\ref{alg: spectral initial}}, multilinear rank $\vr = (r,r,r)$, parameter $p$.
	\For{$t = 1,\cdots$}
	\For{$i = 1,2,3$}
	\State $\mX_i^{t+1} = \text{SVD}_{r}\lb\calM_i\lb\calX^{t}- p^{-1}\calP_{T_{\calX^{t}}}\calP_{\Omega}\left(\calX^{t}- \calT\right)\rb\rb$,
	\EndFor
	\State $\calG^{t+1} = \lb\calX^{t}- p^{-1}\calP_{T_{\calX^{t}}}\calP_{\Omega}\left(\calX^{t}- \calT\right)\rb\ttimes{\mX_i^{t+1}}^\tran$.
	\State $\calX^{t+1} = \calG^{t+1}\ttimes{\mX_i^{t+1}}$.
	\EndFor
	\end{algorithmic}
\end{algorithm}

The Riemannian gradient method (RGM) for solving \eqref{eq:optProb} is presented in Algorithm~\ref{alg:rgd}. 
Let $\calX^t$ be the current estimator, and $T_{\calX^t}$ be the tangent space of the rank $\vr$ tensor manifold at $\calX^t$. RGM first updates $\calX^t$ along $\calP_{T_{\calX^t}} \calP_{\Omega}\left(\calX^{t}- \calT\right)$, the gradient descent direction projected onto the tangent space $T_{\calX^t}$, using the fixed step size $1$. Then the new estimator $\calX^{t+1}$ is obtained by projecting the  update to  the set of rank $\vr$ tensors via HOSVD. Note that in description of Algorithm~\ref{alg:rgd}, we have included the details of HOSVD.

\subsubsection{Initialization by spectral method with diagonal deletion}

\begin{algorithm}[ht!]
	\caption{Initialization via spectral method with diagonal deletion}
  \label{alg: spectral initial}
    \begin{algorithmic}[1]
    \State Input: $\calP_{\Omega}\lb\calT\rb\in\R^{n\times n\times n}$, multilinear rank $\vr = (r, r, r)$, parameter $p$.
    \For{$i = 1,2,3$}
    \State Let $\mX_i^1 \bSigma^1_i {\mX_i^1}^\tran$ be the top-$r$ eigenvalue decomposition of $\calP_{\offdiag}\lb\widehat{\mT_i}\widehat{\mT_i}^{\tran}\rb$.
    \EndFor
    \State	$\calG^1 =  p^{-1} \calP_{\Omega} (\calT)\ttimes {\mX_i^1}^\tran $.
    \State Output: $\calX^1 = \calG^1\ttimes {\mX_i^1}$.
    \end{algorithmic}
\end{algorithm}

Spectral method is a widely used initialization method in matrix and tensor computation problems \citep{chen2020spectral}, which typically involves the estimation of certain principal subspace from the data matrix or equivalently the Gram matrix corresponding to the data matrix. In some statistical settings, directly using the data matrix may lead to a biased estimator. Though this is usually not a problem when matrices are nearly square (for example in typical matrix recovery problems \citep{jain2013low, keshavan2010matrix, ma2020implicit}), it does  lead to sub-optimal performance for highly unbalanced matrices unless the number of observations is unnecessarily large.
{\color{black} This is the case for the tensor completion problem since  for an $n\times n\times n$ tensor, we need to consider its matricization in the analysis which is an $n\times n^2$ matrix and hence highly unbalanced.}

Specifically for the problem considered in this paper, the sample Gram matrix is given by $\widehat{\mT_i}\widehat{\mT_i}^\tran$ where $\widehat{\mT_i} := p^{-1}\calM_i\lb\calP_{\Omega}\lb\calT\rb\rb \in\R^{n\times n^2}$ is the scaled observation matrix. It is clear that how close the top-$r$ eigenvectors of the Gram matrix are to the target eigenvectors  is determined by $\widehat{\mT_i}\widehat{\mT_i}^\tran-\mT_i\mT^\tran_i$. Such error term can be decomposed into bias part and unbiased parts as follows:
\begin{small}
\begin{align*}
    \widehat{\mT_i}\widehat{\mT_i}^\tran-\mT_i\mT^\tran_i &= \widehat{\mT_i}\widehat{\mT_i}^\tran-\mathbb{E}\lsb \widehat{\mT_i}\widehat{\mT_i}^\tran\rsb +\mathbb{E}\lsb\widehat{\mT_i}\widehat{\mT_i}^\tran\rsb-\mT_i\mT_i^\tran\\
    & = \underbrace{\calP_{\offdiag}\lb\widehat{\mT_i}\widehat{\mT_i}^\tran-\mT_i\mT_i^\tran \rb}_{\text{unbiased}} +\underbrace{\calP_{\diag}\lb \widehat{\mT_i}\widehat{\mT_i}^\tran-p^{-1}\mT_i\mT_i^\tran \rb}_{\text{unbiased}}+\underbrace{\lb p^{-1}-1\rb\calP_{\diag}\lb\mT_i\mT_i^\tran\rb}_{\text{bias}},
\end{align*}
\end{small}where $\calP_{\diag}\lb\mM\rb$ sets the non-diagonal elements of $\mM$ to zeros and $\calP_{\offdiag}\lb\mM\rb = \mM-\calP_{\diag}\lb\mM\rb$. It can be shown that the bias term  and the unbiased diagonal part lead to an unnecessarily large number of samples to ensure a reliable estimator \citep{florescu2016spectral, zhang2018heteroskedastic}. To deal with this difficulty, we adopt the diagonal deletion strategy  proposed in \citep{florescu2016spectral}; that is, performing the spectral method on the matrix $\calP_{\offdiag}\lb\widehat{\mT_i}\widehat{\mT_i}^{\tran}\rb$. In this case, we have
\begin{align*}
    \calP_{\offdiag}\lb\widehat{\mT_i}\widehat{\mT_i}^{\tran}\rb - \mT_i\mT_i^\tran =  \underbrace{\calP_{\offdiag}\lb\widehat{\mT_i}\widehat{\mT_i}^\tran-\mT_i\mT_i^\tran \rb}_{\text{unbiased}} -\underbrace{\calP_{\diag}\lb \mT_i\mT_i^\tran\rb}_{\text{diagonal deletion}}.
\end{align*}
Therefore, if the diagonal elements of $\mT_i\mT_i^{\tran}$ are not too large, the matrix $\calP_{\offdiag}\lb\widehat{\mT}_i\widehat{\mT}_i^{\tran}\rb$ serves as a nearly unbiased estimator of $\mT_i\mT_i^{\tran}$. It implies that the top-$r$ eigenvectors of $\calP_{\offdiag}\lb\widehat{\mT}_i\widehat{\mT}_i^{\tran}\rb$ could form a reliable estimator of the principal subspace of $\mT_i\mT_i^{\tran}$.
The complete description of the initialization procedure is summarized in Algorithm \ref{alg: spectral initial}.

\begin{remark}
The diagonal deletion idea has already been used in various scenarios, including low rank tensor completion \citep{cai2021nonconvex, tong2021scaling}.  We give a  brief introduction here for the paper to be self-contained. In addition to diagonal deletion, one might  also consider
properly reweighting the diagonal entries, see for example \citep{cai2016minimax, cho2017asymptotic, elsener2019sparse, loh2012high, lounici2013sparse, lounici2014high,montanari2018spectral,  zhang2018heteroskedastic, zhu2019high}.
\end{remark}

\subsection{ Implicit Regularization and Entrywise Convergence of RGM}
Since the target tensor $\calT$ is low rank, the application of HOSVD with the true  parameter $\vr=(r,r,r)$ yields the exact Tucker decomposition of $\calT$, denoted
\begin{align*}
\calT = \calS\ttimes\mU_i,
\end{align*}
where $\mU_i\in\R^{n\times r}$ are the top-$r$ left singular vectors of $\calM_i(\calT)$ for $i = 1,2,3$ and $\calS = \calT\ttimes\mU_i^{\tran}\in\R^{r\times r\times r}$ is the core tensor.
As in  matrix completion, the notion of  incoherence is required for us to be able to successfully fill in the  missing entries of a low rank tensor. Specifically,
the incoherence parameter of $\calT$ is defined as
\begin{align*}
\mu := \frac{n}{r}\max\left\lbrace \twoinf{\mU_1}^2, \twoinf{\mU_2}^2, \twoinf{\mU_3}^2 \right\rbrace.
\end{align*}
Clearly, the smallest value for $\mu$ can be $1$, while the largest possible value is $n/r$. A tensor is $\mu$-incoherent with a small $\mu$ implies that the singular vectors of its matricization form are weakly correlated with the canonical basis. Therefore, the energy of the tensor sufficiently spreads out across the measurement basis, and a small random subset  of its entries still contains enough information for successful reconstruction. The following lemma follows directly from the definition   of incoherence, see Section~\ref{subsec:incoherenceT} for  the proof.
\begin{lemma}
	\label{lemma: incoherence for T}
 Letting $\mT_i = \calM_i(\calT)$, we have
	\begin{align*}
		\twoinf{\mT_i} &\leq \sqrt{\frac{\mu r }{n }}\sigma_{\max}(\calT),\quad\twoinf{\mT_i^\tran}\leq \frac{\mu r  }{ n}\sigma_{\max}(\calT),\quad\infnorm{\calT} \leq  \left( \frac{\mu r}{n} \right)^{3/2}\sigma_{\max}(\calT).
	\end{align*}
\end{lemma}

The convergence of the  Riemannian gradient method can be approximately decomposed into two phases. In Phase I, the iterates are not sufficiently close to the target tensor, we need to explicitly show that they remain in the incoherence region based on  induction. In Phase II, the iterates are in a local neighbourhood of the target tensor where the restricted isometry property uniformly holds, and thus implicit regularization and entrywise convergence follows directly from the convergence in terms of the Frobenius norm. 
\begin{theorem}[Phase I convergence]\label{thm: finite iteration}
Suppose that $\calT$ is $\mu$-incoherent and the index set $\Omega$ satisfies the Bernoulli model with parameter $p$. If $n\geq  C_0\kappa^6\mu^3 r^{5}$ and 
\begin{align*}
    p\geq \max\left\lbrace\frac{C_1 \kappa^8\mu^{3.5}r^{6}\log^3 n}{n^{3/2}}, \frac{C_2\kappa^{16}\mu^7r^{12}\log^5 n}{n^2}\right\rbrace
\end{align*}
for some universal constants $C_0$, $C_1$ and $C_2$. Suppose $t_0 = 2\log_2 n+c$. Then with high probability, the iterates of RGM {\color{black}(Algorithm~\ref{alg:rgd})} satisfy 
\begin{align}
    \twoinf{\mX_i^t\mR_i^t-\mU_i}&\leq \lb\frac{1}{2}\rb^{t}\sqrt{\frac{\mu r}{n}},\quad i = 1,2,3,\notag\\
    \label{eq:phaseI infinity}
    \infnorm{\calX^t-\calT}&\leq \lb\frac{1}{2}\rb^{t}\frac{1}{n^{3/2}}\sigma_{\max}\lb\calT\rb,
\end{align}
for $t = 1,\cdots,t_0$, where $\mR_i^t = \arg\min_{\mR^\tran\mR=\mI}\fronorm{\mX_i^t\mR-\mU_i}$. 
\end{theorem}
Theorem~\ref{thm: finite iteration} establishes the performance guarantee of RGM within the first logarithm number of iterations, where $c>0$ in $t_0$ is a constant  specified in Remark~\ref{remark: specify c}. Analyzing the distance between factor matrices $\mX_i^t$ and $\mU_i$ is the key to showing the convergence of the RGM in phase I. As already mentioned, this explicitly shows that RGM automatically forces the iterates to stay incoherent in each iteration. 
\begin{theorem}[Phase II convergence]\label{thm: local convergence}
Suppose that $\calT$ is $\mu$-incoherent and the index set $\Omega$ satisfies the Bernoulli model with parameter $p$.  Suppose  $p\geq\frac{C_2\mu^2r^2\log n}{\varepsilon\cdot n^2}$ and
\begin{align*}
    &\frac{\fronorm{\calX^{t_0}-\calT}}{\sigma_{\min}\lb\calT\rb}\leq \frac{p^{1/2}\varepsilon}{9\lb1+\epsilon\rb},
\end{align*}
where $0<\varepsilon\leq \frac{1}{200}$. Then with high probability, the iterates of RGM {\color{black}(Algorithm~\ref{alg:rgd})} satisfy 
\begin{align}\label{eq: local}
	\fronorm{\calX^t-\calT}\leq \lb\frac{1}{2}\rb^{t-t_0}\fronorm{\calX^{t_0}-\calT},\quad\mbox{for all}~t\geq t_0.
\end{align} 
\end{theorem}

\begin{remark}\label{remark: specify c}
Note that if we set $c=\lceil\log_2(9(1+\varepsilon)/\varepsilon) \rceil$ in Theorem~\ref{thm: finite iteration}, after Phase I, the iterates of RGM will enter the local neighborhood specified in Theorem~\ref{thm: local convergence}. This follows from a simple calculation,
\begin{align*}
	\frac{\fronorm{\calX^{t_0} - \calT} }{\sigma_{\min}(\calT )}&\leq \frac{1}{\sigma_{\min}(\calT)} n^{3/2}\infnorm{\calX^{t_0} - \calT} \stackrel{(a)}{\leq } \frac{1}{\sigma_{\min}(\calT)}  n^{3/2}\cdot {\frac{1}{2^{t_0}}}\frac{1}{n^{3/2}}\sigma_{\max}(\calT)\\
	&\leq  \frac{\kappa }{2^{t_0}}\stackrel{(b)}{\leq } \frac{\varepsilon}{9(1+\varepsilon)}\frac{\kappa}{n}\stackrel{(c)}{\leq } \frac{\sqrt{p}\varepsilon}{9(1+\varepsilon)},
\end{align*}
where $(a)$ is due to \eqref{eq:phaseI infinity},  $(b)$ follows from $t_0 = 2\log_2 n+\lceil\log_2(9(1+\varepsilon)/\varepsilon) \rceil$, and $(c)$ uses the fact $p\geq \kappa^2/n^2$. 
\end{remark}
As an immediate consequence of Theorem~\ref{thm: finite iteration} and Theorem~\ref{thm: local convergence}, one can obtain the proof of Theorem~\ref{thm:informal}.
\begin{proof}[Proof of Theorem~\ref{thm:informal}]
Noting that the results in Theorem~\ref{thm: finite iteration} naturally  show that the inequalities~\eqref{eq:informal twotoinfinity} and~\eqref{eq: informal infinity convergence} hold for $1\leq t\leq t_0$. It only remains to show the above two inequalities hold for $t\geq t_0$. 

A simple calculation yields that
\begin{align*}
    \infnorm{\calX^t-\calT}&\leq \fronorm{\calX^t-\calT}\leq \lb\frac{1}{2}\rb^{t-t_0}\fronorm{\calX^{t_0}-\calT}\leq \lb\frac{1}{2}\rb^{t-t_0} n^{3/2}\infnorm{\calX^{t_0}-\calT}\\
    &\leq \lb\frac{1}{2}\rb^{t-t_0} n^{3/2}\lb\frac{1}{2}\rb^{t_0}\frac{1}{n^{3/2}}\sigma_{\max}\lb\calT\rb = \lb\frac{1}{2}\rb^{t}\sigma_{\max}\lb\calT\rb.
\end{align*}

Only detailed proofs for the  $i = 1$ case are provided for \eqref{eq:informal twotoinfinity}, and the proofs for the other two cases are overall similar.  Applying the Weyl's inequality yields that
\begin{align*}
\sigma_{\min}\lb\calM_1\lb\calX^t\rb\rb&\geq \sigma_{\min}\lb\calM_1\lb\calT\rb\rb-\opnorm{\calM_1\lb\calX^t\rb-\calM_1\lb\calT\rb}\\
&\geq \sigma_{\min}\lb\calT\rb-\fronorm{\calM_1\lb\calX^t\rb-\calM_1\lb\calT\rb}\\
&\geq \sigma_{\min}\lb\calT\rb-\lb\frac{1}{2}\rb^{t-t_0}\fronorm{\calX^{t_0}-\calT}\\
&\geq \sigma_{\min}\lb\calT\rb-\lb\frac{1}{2}\rb^{t-t_0}\frac{\sqrt{p}\varepsilon}{9(1+\varepsilon)}\sigma_{\min}\lb\calT\rb\geq \frac{15}{16}\sigma_{\min}\lb\calT\rb,
\end{align*}
where the last inequality is due to $p\leq 1$ and $\varepsilon\leq\frac{1}{200}$. For any $\ell$ satisfies $1\leq \ell\leq n$, using the fact $\calM_1\lb\calX^t\rb = \mX_1^t\calM_1\lb\calG^t\rb\lb\mX_3^t\otimes\mX_2^t\rb^\tran$,  one can obtain
\begin{align*}
\opnorm{\ve_{\ell}^\tran\mX_1^t} &= \opnorm{\ve_{\ell}^\tran\calM_1\lb\calX^t\rb\lb\mX_3^t\otimes\mX_2^t\rb\calM_1^\dagger\lb\calG^t\rb}\\
& \stackrel{(a)}{\leq} \opnorm{\ve_{\ell}^\tran\calM_1\lb\calX^t\rb}\opnorm{\calM_1^\tran\lb\calG^t\rb\lb\calM_1\lb\calG^t\rb\calM_1^\tran\lb\calG^t\rb\rb^{-1}}\\
&\leq \lb\opnorm{\ve_{\ell}^\tran\calM_1\lb\calX^t-\calT\rb}+\opnorm{\ve_{\ell}^\tran\calM_1\lb\calT\rb}\rb\frac{1}{\sigma_{\min}\lb\calM_1\lb\calG^t\rb\rb}\\
&\leq \lb n\infnorm{\calX^t-\calT}+\twoinf{\calM_1\lb\calT\rb}\rb\frac{1}{\sigma_{\min}\lb\calM_1\lb\calX^t\rb\rb}\\
&\leq \lb n\lb\frac{1}{2}\rb^{t-t_0}\lb\frac{1}{2}\rb^{t_0}\sigma_{\max}\lb\calT\rb+\sqrt{\frac{\mu r}{n}}\sigma_{\max}\lb\calT\rb\rb\frac{16}{15\sigma_{\min}\lb\calT\rb}\\
&\stackrel{(b)}{\leq}\lb n\lb\frac{1}{2}\rb^{t-t_0} \lb\frac{1}{2}\rb^c\frac{1}{n^2}\sigma_{\max}\lb\calT\rb+\sqrt{\frac{\mu r}{n}}\sigma_{\max}\lb\calT\rb\rb\frac{16}{15\sigma_{\min}\lb\calT\rb}\leq 2\kappa\sqrt{\frac{\mu r}{n}},
\end{align*}
where $(a)$ is due to $\sigma_{\min}\lb\calM_1\lb\calG^t\rb\rb>0$ and $(b)$ follows from $t_0 = 2\log_2 n+c$.  
\end{proof}

{\color{black}
\subsection{Numerical Experiments}\label{subsec:empirical}

\begin{figure}[ht!]
    \centering
    \includegraphics[width=0.48\textwidth]{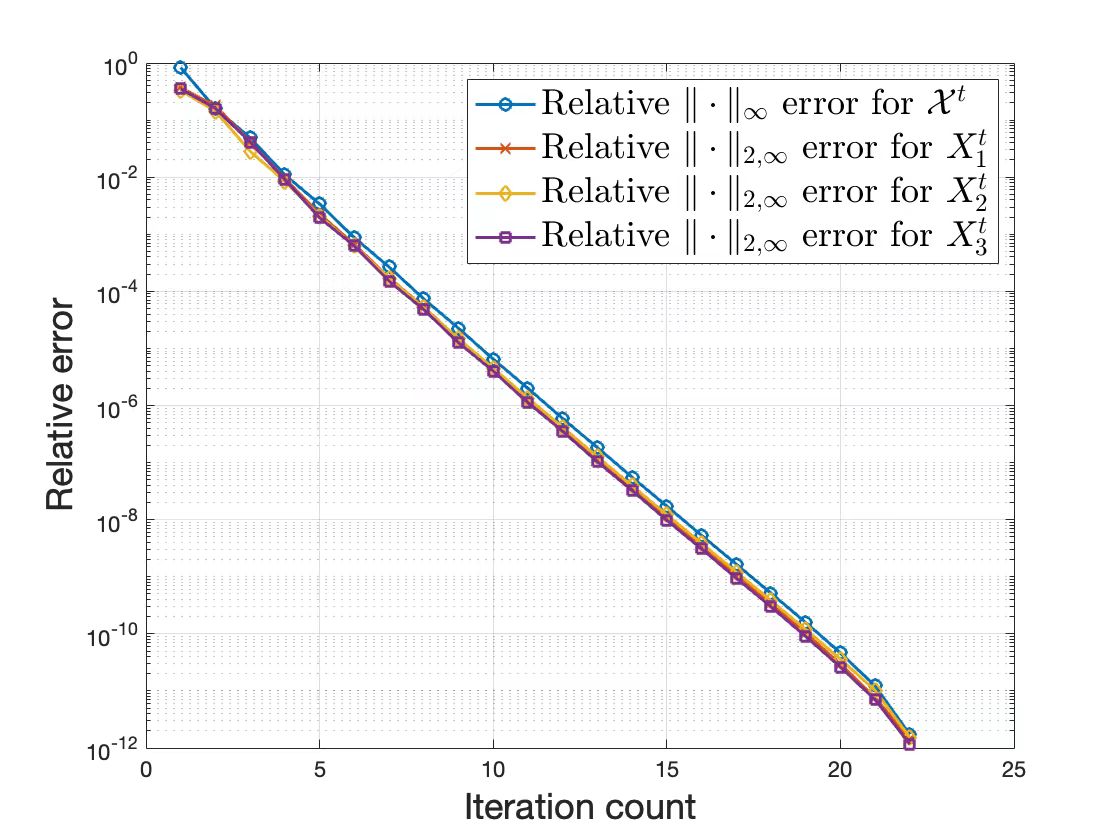}
    \includegraphics[width=0.48\textwidth]{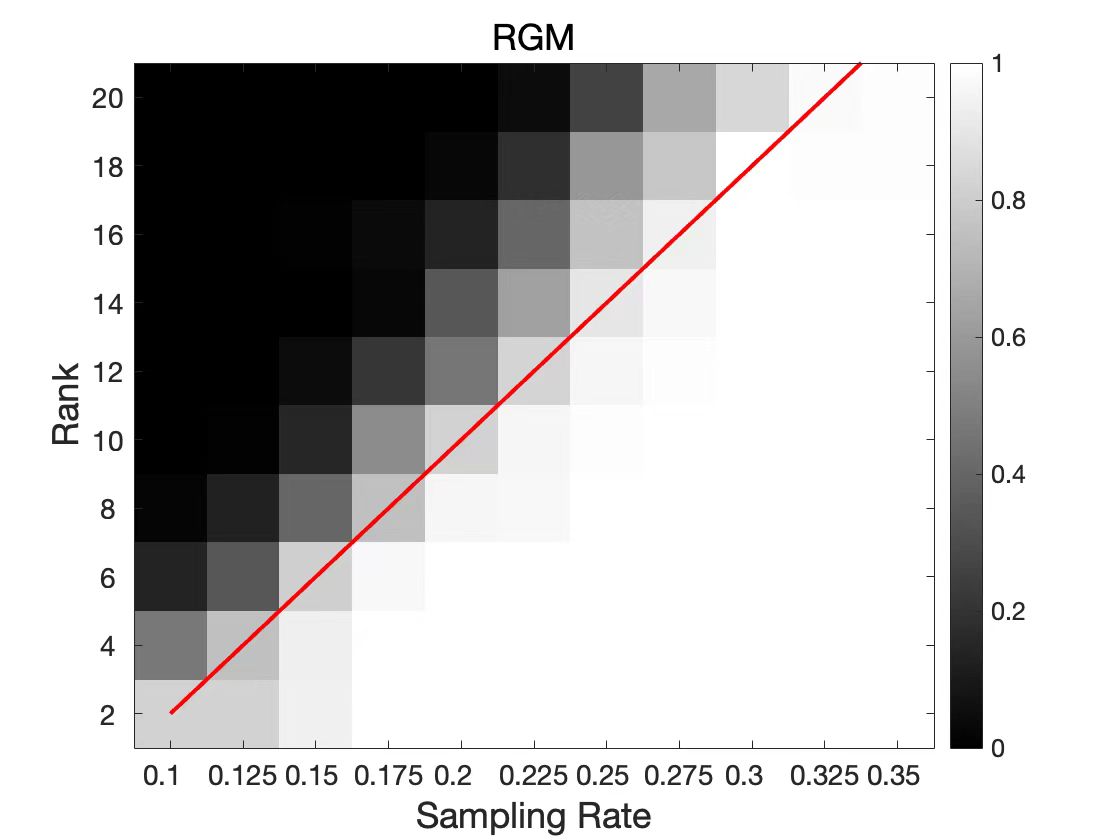}   \caption{(Left) Relative errors of $\calX^t$ and $\mX_i^t$ for $i = 1,2,3$ against iteration count when $n = 100$, $r = 3$, $p = 0.2$; (Right) Phase transition plot for varying $r$ and sampling rate when  $n = 100$ (different colors represent different successful rates out of 100 random tests). }
    \label{fig:my_labely}
\end{figure}
We first test the convergence of RGM for tensor completion problem under two metrics defined below:
\begin{align*}
\mbox{\textbf{Relative $\|\cdot\|_\infty$: }}~\frac{\infnorm{\calX^t-\calT}}{\infnorm{\calT}}\quad\mbox{and}\quad\mbox{\textbf{Relative $\|\cdot\|_{2,\infty}$}:~}~\frac{\twoinf{\mX^t_i\mR^t_i-\mU_i}}{\twoinf{\mU_i}},
\end{align*}
where  $\mR_i^t = \arg\min_{\mR^\tran\mR=\mI}\fronorm{\mX^t_i\mR-\mU_i}$. Tests are conducted with $n  = 100$, $r = 3$ and $p = 0.2$, and the plot of average relative errors over 100 random tests  against iteration count is presented in Figure~\ref{fig:my_labely} (Right). Clearly, a desirable linear convergence can be observed from the plots  for both the iterates and the factor matrices. 


Our results suggests that $O(n^{3/2}r^6)$ are overall sufficient for the successful reconstruction of a low Tucker-rank tensor. For the dependence on $n$, it matches the lower bound for any known polynomial time algorithms. Moreover, it is shown in \citep{barak2016noisy} that, conditioned on some conjecture on computational complexity theory, no polynomial time algorithm can be successful if the sampling complexity is less than $O(n^{3/2})$. For the dependence on $r$, we conduct the phase transition tests for fixed $n=100$, and varying $r$ and $p$, see Figure~\ref{fig:my_labely} (Right). For each pair of $(r,p)$, 100 random random trials are tested and a trial is considered to have successfully recover the target tensor if the output tensor $\calX^t$ satisfies $\frac{\infnorm{\calX^t-\calT}}{\infnorm{\calT}}\leq 10^{-3}$. The phase transition plot indicates that the sampling complexity for successful recovery is about linearly proportional to $r$, which suggests the possibility of reducing this dependency in the future.

}
\section{Proof of Theorem \ref{thm: finite iteration}}
\label{proof architecture}

 The proof of Theorem~\ref{thm: finite iteration} relies on the reconstruction of auxiliary sequences via a \emph{leave-one-out} perturbation argument, and  is much more involved. Thus, this section is devoted to the proof outline of Theorem \ref{thm: finite iteration}, while the proofs of the intermediate results are deferred to later sections.

To obtain the entrywise error bound of the iterates,  the $\ell_{2,\infty}$ norm of the factor matrices needs to be bounded which is quite difficult to control
directly due to the complicated statistical dependency.  To overcome this difficulty, the leave-one-out technique proposes to introduce a collection of leave-one-out versions of $\left\lbrace\calX^t \right\rbrace$, denoted by $\left\lbrace\calX^{t,\ell} \right\rbrace$ for each $1\leq \ell\leq n$. Specifically, for every  $\ell$, define the following auxiliary loss function
\begin{align}\label{eq:new cost}
\frac{1}{2p}\fronorm{\calP_{\Omega_{-\ell}}\lb\calX-\calT\rb}^2+\frac{1}{2}\fronorm{\calP_{\ell}\lb\calX-\calT\rb}^2,
\end{align}
where  $\calP_{\Omega_{-\ell}}\lb\calZ\rb$  and $\calP_{\ell}\lb\calZ\rb$ are defined as follows:
\begin{align*}
\lsb \calP_{\Omega_{-\ell}}(\calZ)\rsb_{i_1, i_2, i_3}  & := \begin{cases}
		\lsb\calZ\rsb_{i_1,i_2,i_3},~&\text{if } i_1\neq \ell,  i_2\neq\ell,  i_3\neq \ell ~\text{and} ~ (i_1,i_2,i_3)\in\Omega,\\
		0, ~&\text{otherwise},
\end{cases}\\
\lsb \calP_{\ell}(\calZ)\rsb_{i_1, i_2, i_3}  & := \begin{cases}
		\lsb\calZ\rsb_{i_1,i_2,i_3},~&\text{if } i_1 = \ell, \text{or}~ i_2 = \ell, \text{or}~ i_3 = \ell,\\
		0, ~&\text{otherwise}.
\end{cases}
\end{align*}
The leave-one-out sequence $\left\lbrace\calX^{t,\ell} \right\rbrace_{t\geq 1}$ is produced by  applying RGM to this new cost  function.
If $\Omega$ satisfies the Bernoulli model,   then we can rewrite \eqref{eq:new cost} as
\begin{align}\label{eq:new cost2}
    \frac{1}{2p}\sum_{i_1,i_2,i_3\neq \ell}\delta_{i_1,i_2,i_3}\lb\calX_{i_1,i_2,i_3}-\calT_{i_1,i_2,i_3}\rb^2+ \frac{1}{2}\sum_{\exists i_j= \ell,~j\in[3]}\lb\calX_{i_1,i_2,i_3}-\calT_{i_1,i_2,i_3}\rb^2,
\end{align}
where $\{\delta_{i_1,i_2,i_3}\}$ are $n^3$ independent Bernoulli random variables. 
Noting that \eqref{eq:new cost2} does not depend on $\{\delta_{i_1,i_2,i_3}:~\exists i_j=\ell,~j\in[3]\}$, the sequence  $\left\lbrace\calX^{t,\ell} \right\rbrace_{t\geq 1}$ is independent of those random variables provided the initial guess $\calX^{1,\ell}$ is independent  of them.  This decoupling of the  statistical dependency turns out to be crucial for us to bound the  $\ell_{2,\infty}$ norm of the factor matrices.
The initial guess $\calX^{1,\ell}$ can be similarly generated by the spectral method with diagonal deletion, but with those entries at locations indexed by  $\{(i_1,i_2,i_3):~\exists i_j=\ell,~j\in[3]\}$ being replaced by the ground truth values. 
The complete procedure to create the leave-one-out sequence $\left\lbrace\calX^{t,\ell} \right\rbrace_{t\geq 1}$ is described in Algorithm~\ref{alg:rgd leave one out}. 
We would like to caution that Algorithm~\ref{alg:rgd leave one out} is by no means a practical algorithm, but only introduced for the sake of analysis. 

\begin{algorithm}[ht!]
	\caption{The $\ell$-th leave-one-out sequence for tensor completion}
 \label{alg:rgd leave one out}
	\begin{algorithmic}[1]
	\State Input: tensors $\calP_{\Omega_{-\ell}}\lb\calT\rb$, $\calP_{\ell}\lb\calT\rb$, multilinear rank $\vr = (r,r,r)$, parameter $p$.
	\For{$i = 1,2,3$}
	\State Let $\mX_i^{1,\ell} \bSigma^{1,\ell}_i {\mX_i^{1,\ell}}^\tran$ be the top-$r$ eigenvalue decomposition of 
$		\calP_{\offdiag}\lb\widehat{\mT}_i^{\ell}{\widehat{\mT}_i^{{\ell}^{^\tran}}}\rb$,
		where $\widehat{\mT}^{\ell}_i =  \calM_i\lb p^{-1}\calP_{\Omega_{-\ell}}\lb\calT\rb+\calP_{\ell}\lb\calT\rb\rb$.
	\EndFor
	\State $\calG^{1,\ell} =  \lb p^{-1} \calP_{\Omega_{-\ell}} (\calT)+\calP_{\ell}\lb\calT\rb\rb\ttimes {\mX_i^{1,\ell}}^\tran$.
	\State $\calX^{1,\ell} = \calG^{1,\ell}\ttimes {\mX_i^{1,\ell}}$.
    \For{$t = 1,\cdots$}
    \For{$i = 1,2,3$}
    \State $\mX_i^{t+1,\ell} = \text{SVD}_{r}\lb\calM_i\lb\calX^{t,\ell}- \calP_{T_{\calX^{t,\ell}}}\lb p^{-1}\calP_{\Omega_{-\ell}}+\calP_{\ell}\rb\left(\calX^{t,\ell}- \calT\right)\rb\rb$.
    	\EndFor
    \State $\calG^{t+1,\ell} = \lb\calX^{t,\ell}- \calP_{T_{\calX^{t,\ell}}}\lb p^{-1}\calP_{\Omega_{-\ell}}+\calP_{\ell}\rb\left(\calX^{t,\ell}- \calT\right)\rb\ttimes{\mX_i^{t+1,\ell}}^\tran$.
    \State $\calX^{t+1,\ell} = \calG^{t+1,\ell}\ttimes{\mX_i^{t+1,\ell}}$.
    \EndFor
	\end{algorithmic}
\end{algorithm}


To facilitate the analysis, it is much more convenient to rewrite the iterates of Algorithm \ref{alg:rgd} into the following perturbation form, 
\begin{align*}
    \calX^{t+1} =\calH_{\vr}\lb \calT+\calE^t\rb,
\end{align*}
where the residual tensor $\calE^t$ is given by
\begin{align}\label{residual tensor et}
    \calE^{t} := \lb\calI-p^{-1}\calP_{T_{\calX^{t}}}\calP_{\Omega}\rb\lb \calX^{t}-\calT\rb,\quad t\geq 1.
\end{align}
For $t = 0$, since $\calX^1$ can be rewritten as 
\begin{align*}
    \calX^1 &= \calH_{\vr}\lb\calT+\calX^1-\calT\rb,
\end{align*}
 the residual tensor in the initialization step is defined as
\begin{align}\label{calE0}
    \calE^0 := \calX^1-\calT=\left( \left( \calI - p^{-1} \calP_{\Omega} \right)(-\calT)  \right) \ttimes \mX_i^1{\mX_i^1}^\tran +\calT\ttimes\mX_i^1{\mX_i^1}^\tran  -\calT.
\end{align}
Similarly, the residual tensors of the $\ell$-th leave one out sequence from Algorithm~\ref{alg:rgd leave one out} are defined by
\begin{align}\label{calE0ell}
        &\calE^{0,\ell} := \lb\lb\calI-p^{-1}\calP_{\Omega_{-\ell}}-\calP_{\ell}\rb\lb-\calT\rb\rb\ttimes {\mX_i^{1,\ell}}{\mX_i^{1,\ell}}^\tran+\calT\ttimes{\mX_i^{1,\ell}}{\mX_i^{1,\ell}}^\tran-\calT,\\
        \label{residual tensor etl}
         &\calE^{t,\ell} := \lb\calI-\calP_{T_{\calX^{t,\ell}}}\lb p^{-1}\calP_{\Omega_{-\ell}}+\calP_{\ell}\rb\rb\lb\calX^{t,\ell}-\calT\rb,\quad t\geq 1,
\end{align} 
which  satisfy
$$ \calX^{1,\ell} =\calH_{\vr}\lb \calT+\calE^{0,\ell}\rb\quad\mbox{and}\quad\calX^{t+1,\ell} =\calH_{\vr}\lb \calT+\calE^{t,\ell}\rb.$$

Let $\mT_i = \calM_i\lb\calT\rb$, $\mE_i^{t-1} = \calM_i\lb\calE^{t-1}\rb$ and $\mE_i^{t-1,\ell} = \calM_i\lb\calE^{t-1,\ell}\rb$ be the mode-$i$ matricizations of the corresponding tensors. It can be seen  from  Algorithms~\ref{alg:rgd} and \ref{alg:rgd leave one out} that
the matrices $\mX_i^t$ and $\mX_i^{t,\ell}$ are the top-$r$ eigenvectors of $\lb\mT_i+\mE_i^{t-1}\rb\lb\mT_i+\mE_i^{t-1}\rb^\tran$ and 
$\lb\mT_i+\mE_i^{t-1,\ell}\rb\lb\mT_i+\mE_i^{t-1,\ell}\rb^\tran$, respectively.
Recall that the  eigenvalue decomposition of $\mT_i\mT_i^\tran$ is $\mT_i\mT_i^\tran = \mU_i\bLambda_i\mU_i^\tran$. If we further define three auxiliary $r\times r$ orthonormal matrices as  follows:
\begin{align*}
	\mR_i^t &= {\arg\min}_{\mR} \fronorm{\mX_i^{t } \mR - \mU_i },\quad \mR_i^{t,\ell} = {\arg\min}_{\mR} \fronorm{\mX_i^{t,\ell} \mR - \mU_i },\\
  \mT_i^{t,\ell} &= {\arg\min}_{\mR} \fronorm{\mX_i^{t,\ell} \mR - \mX^t_i \mR_i^t},
\end{align*}
 the following theorem for the sequences produced by Algorithms~\ref{alg:rgd} and \ref{alg:rgd leave one out}.
 
\begin{theorem}\label{base and induction case theorem}
Under the assumption of Theorem~\ref{thm: finite iteration}, the following inequalities hold with high probability for all $1\leq \ell\leq n$ and $1\leq t\leq t_0$,
\begin{subequations}
	\label{eq induction hypotheis}
	\begin{align}
		\label{spectral norm of error tensor}
		\opnorm{\mE_i^{t-1}} &\leq \frac{1}{2^{20} \kappa^6 \mu^2 r^4} \frac{1}{2^t}  \sigma_{\max}(\calT), \\
		\label{spectral norm of loo error tensor}
		\opnorm{\mE_i^{t-1,\ell} } &\leq\frac{1}{2^{20} \kappa^6\mu^2 r^4} \frac{1}{2^t}  \sigma_{\max}(\calT) ,\\
		\label{2 inf norm of loo}
		\twoinf{\mX_i^{t, \ell} \mR_i^{t,\ell} - \mU_i} &\leq \frac{1}{2^{20} \kappa^2\mu^2 r^4} \frac{1}{2^t} \sqrt{\frac{\mu r}{n}},\\
		\label{F norm  of error tensor}
		\fronorm{ \calE^{t-1} - \calE^{t-1,\ell}} &\leq \frac{1}{2^{20} \kappa^4 \mu^2 r^4} \frac{1}{2^t} \sqrt{\frac{\mu r}{n}} \sigma_{\max}(\calT),  \\
		\label{F norm}
		\fronorm{\mX_i^{t}\mR_i^t - \mX_i^{t,\ell} \mT_i^{t,\ell}} &\leq  \frac{1}{2^{20} \kappa^2\mu^2 r^4} \frac{1}{2^t} \sqrt{\frac{\mu r}{n}} ,\\
		\label{L2 inf norm}
		\twoinf{\mX_i^t\mR_i^t - \mU_i} &\leq \frac{1}{2^{20}\kappa^2\mu^2 r^4} \frac{1}{2^t} \sqrt{\frac{\mu r}{n}}.
	\end{align}
\end{subequations}
\end{theorem}
We also need a  lemma which transfers the convergence result in  terms of  the $\ell_2$ and $\ell_{2,\infty}$ norms to that in terms of the $\ell_\infty$ norm.
\begin{lemma}
	\label{lemma: infinite norm of X-T}
	Let $\calX = \HOSVD_r(\calT + \calE)$ for some perturbation tensor $\calE\in\Rn$. Let the Tucker decomposition of $\calX$  be $\calG\ttimes \mX_i$ with $\mX_i^\tran\mX_i = \mI$.  Define 
 \begin{small}
	\begin{align*}
	\mR_i = \arg\min_{\mR^\tran\mR = \mI}\fronorm{\mX_i \mR - \mU_i}~\mbox{and}~B= \max_{i=1,2,3} \lb  \twoinf{\mX_i\mR_i - \mU_i} +\frac{15\sigma_{\max}(\calT)}{2\sigma_{\min}^2(\calT)} \twoinf{\mU_i} \opnorm{\mE_i} \rb,
	\end{align*}
 \end{small}
	where $\mE_i := \calM_i(\calE)$. Suppose that  $\max_{i=1,2,3}\opnorm{\mE_i} \leq \sigma_{\max}(\calT)/\lb10\kappa^2\rb$. Then one has
	\begin{align*}
		\infnorm{\calX - \calT}&\leq \sigma_{\max}(\calT) \left( B^3 + 3B^2 \lb  \max_{i = 1,2,3} \twoinf{\mU_i} \rb  + 3 B \lb  \max_{i = 1,2,3}  \twoinf{\mU_i} \rb ^2 \right) \\
  &\quad+ \max_{i = 1,2,3}\opnorm{\mE_i }  \prod_{i=1}^{3} \twoinf{\mX_i}.
	\end{align*}
\end{lemma} 
Theorem~\ref{base and induction case theorem} will be proved by induction, with the proof details for the base case and induction step presented in Sections~\ref{sec: base case} and \ref{sec: induction case}, respectively.  The proof of Lemma~\ref{lemma: infinite norm of X-T} can be found in Section~\ref{subsec:proof34}. Equipped with Theorem~\ref{base and induction case theorem} and Lemma~\ref{lemma: infinite norm of X-T},
we are now able to present the proof of Theorem \ref{thm: finite iteration}.
\begin{proof}[Proof of Theorem~\ref{thm: finite iteration}]
The $\ell_{2,\infty}$ convergence for the factor matrices follows directly from \eqref{L2 inf norm}. It only remains to show the $\ell_\infty$ convergence of the iterates.
From \eqref{spectral norm of error tensor}, one can see that
\begin{align*}
    \opnorm{\mE_i^{t-1}}\leq \frac{1}{10\kappa^2}\sigma_{\max}(\calT), \quad i = 1,2,3.
\end{align*}
Applying Lemma~\ref{lemma: infinite norm of X-T} with $\calX := \calX^t$, $\calE := \calE^{t-1}$ yields that
\begin{align}\label{proof infinity finite stage}
    \infnorm{\calX^t-\calT}&\leq \sigma_{\max}(\calT) \left( B^3 + 3B^2 \lb  \max_{i = 1,2,3} \twoinf{\mU_i} \rb + 3 B \lb  \max_{i = 1,2,3} \twoinf{\mU_i} \rb ^2 \right)\notag\\
    &\quad+ \max_{i = 1,2,3}\opnorm{\mE_i^{t-1} }  \prod_{i=1}^{3} \twoinf{\mX^t_i}.
\end{align}
By  \eqref{spectral norm of error tensor} and \eqref{L2 inf norm}, one has
\begin{align}\label{upper bound b}
    B\leq \frac{1}{2^{16}\kappa^2\mu^2r^4}\sqrt{\frac{\mu r}{n}}.
\end{align}
Finally, substituting \eqref{spectral norm of error tensor} and  \eqref{upper bound b} into \eqref{proof infinity finite stage} yields that
\begin{align}\label{eq:kw001}
    \infnorm{\calX^t-\calT}\leq \lb\frac{1}{2}\rb^t\frac{1}{n^{3/2}}\sigma_{\max}\lb\calT\rb.
\end{align}
Noting that  \eqref{eq:kw001} holds with high probability for each $t$ and $t_0=2\log_2 n+c$.     
\end{proof}
 
\section{Proof of Theorem \ref{thm: local convergence}}
\label{sec: local}
In this section, we prove Theorem \ref{thm: local convergence} by induction. It is trivial that  when $t = t_0$ the result holds. Next, we assume \eqref{eq: local} holds for the iterations $t_0, t_0+1,\cdots, t$, and then prove it also holds for $t+1$. The proof relies on several lemmas which reveal the uniform restricted isometry property in a local neighborhood of the ground truth. We first present these lemmas  and postpone the proofs to Section~\ref{sec: proof of key lemmas}.
\begin{lemma}\label{lemma: entry upper bound}
	Suppose $T$ is the tangent space of $\M_r$ at $\calT$. Then 
	\label{lemma: F norm of E projection onto T}
	\begin{align}
		\fronorm{\calP_{T}(\ve_{i_1}\circ \ve_{i_2}\circ\ve_{i_3})}^2 \leq 4\lb \frac{\mu r}{n} \rb^2.
	\end{align}
\end{lemma}

\begin{lemma}
	\label{lemma: local isotropy}
	Suppose $\Omega$ is sampled according to the Bernoulli model and the tensor $\calT$ obeys the incoherence condition with parameter $\mu$. If $p\geq\frac{C_2\mu^2r^2\log n}{\varepsilon\cdot n^2}$, then 
	\begin{align*}
		\opnorm{\calP_{T}\lb p^{-1}\calP_{\Omega}-\calI\rb\calP_{T}}\leq\varepsilon
	\end{align*}
	holds with high probability, where $\varepsilon>0$ is an absolute constant.
\end{lemma}
\begin{remark}
Lemma~\ref{lemma: entry upper bound} and Lemma~\ref{lemma: local isotropy}  highly resemble Lemma 2 in \citep{yuan2016tensor} and Lemma 12 in \citep{tong2021scaling} respectively. However, the definition of the operator $\calP_{T}$ in this paper differs from the one in \citep{tong2021scaling, yuan2016tensor}. Specifically, the operator $\calP_{T}$ in the previous works was defined as 
\begin{align*}
    \calP_{T}\lb\calZ\rb = \calZ\ttimes\mU_i\mU_i^\tran + \sum_{i = 1}^3\calZ\times_i\lb\mI-\mU_i\mU_i^\tran\rb \jneqi \mU_j\mU_j^\tran.
\end{align*}
\end{remark}

\begin{lemma}
	\label{lemma: technique lemmas}
	Let $\calX^t = \calG^t \ttimes \mX_i^t$ and $\calT = \calS\ttimes \mU_i$ be two tensors in $\M_r$ and $T_{\calX^t}, T$ be the tangent spaces of $\M_r$ at $\calX^t$ and $\calT$, respectively. Then
	\begin{align*}
		\opnorm{  {\mX_i^t} {\mX_i^t} ^\tran -  \mU_i\mU_i^\tran  } &\leq \frac{1}{\sigma_{\min}(\calT)}\fronorm{\calX^t - \calT},\quad i=1,2,3,\\
		\opnorm{\calP_{T_{\calX^t}} - \calP_T} &\leq \frac{9 }{\sigma_{\min}(\calT)}\fronorm{\calX^t - \calT}.
	\end{align*}
\end{lemma}

\begin{lemma}
	\label{lemma 2}
	Assume the following inequalities hold for $0<\varepsilon \leq \frac{1}{200}$,	\begin{align}\label{assumption in lemma 2}		
			\opnorm{\calP_T - p^{-1}\calP_{T} \calP_{\Omega} \calP_{T}} \leq \varepsilon \quad\mbox{and}\quad			\frac{\fronorm{\calX^t - \calT}}{\sigma_{\min}(\calT)} \leq \frac{\sqrt{p} \varepsilon}{9(1+\varepsilon)}.
		\end{align}
 Then 
	\begin{align*}
		&\opnorm{\calP_{\Omega} \calP_{T_{\calX^t}}} \leq  2\sqrt{p}(1+\varepsilon) \quad\mbox{and}\quad\opnorm{\calP_{T_{\calX^t}} - p^{-1}\calP_{T_{\calX^t}} \calP_{\Omega} \calP_{T_{\calX^t}} }  \leq5 \varepsilon.
	\end{align*}
\end{lemma}

\begin{proof}[Proof of Theorem~\ref{thm: local convergence}]
First, Lemma~\ref{lemma: local isotropy} implies that 
\begin{align*}
	\opnorm{\calP_T - p^{-1}\calP_{T} \calP_{\Omega} \calP_{T}} \leq \varepsilon
\end{align*}
holds with high probability provided that $p\geq \frac{C_2\mu^2 r^2\log n}{\varepsilon\cdot n^2}$. 

We assume that in the $t$-th iteration $\calX^t$ satisfies 
\begin{align}
	\label{eq induction for local convergence}
	\frac{\fronorm{\calX^{t} - \calT} }{\sigma_{\min}(\calT)} \leq  \frac{\sqrt{p} \varepsilon}{9(1+\varepsilon)}.
\end{align}
Recall that $\calX^{t+1} = \calH_r\lb \calX^t - p^{-1}\Pxt \calP_{\Omega}\lb \calX^t - \calT \rb \rb = \calH_r\lb \calT +\calE^t\rb$. One can see that
\begin{align*}
	\fronorm{\calX^{t+1} - \calT} \leq	\fronorm{\calX^{t+1} - \lb \calT + \calE^t \rb}  + \fronorm{\calE^{t}} \leq (\sqrt{3}+1) \fronorm{\calE^t}\leq 2^2\fronorm{\calE^t},
\end{align*}
where the second inequality follows from the quasi-optimality of $\HOSVD$. It remains to  bound $\fronorm{\calE^t}$. To this end, invoking the triangle inequality gives  that
\begin{align*}
	\fronorm{\calE^t} &= \fronorm{ \calX^t - \calT - p^{-1}\Pxt\calP_{\Omega}\lb \calX^t - \calT \rb } \\
	&\leq \underbrace{ \fronorm{  \lb \Pxt- p^{-1}\Pxt\calP_{\Omega} \Pxt \rb  \lb \calX^t - \calT \rb } }_{ =: \upsilon_1}  + \underbrace{\fronorm{ \lb  \calI - \Pxt\rb \lb \calX^t - \calT \rb } }_{=:\upsilon_2} \\
	&\quad + \underbrace{\fronorm{ p^{-1}\Pxt\calP_{\Omega}\lb \calI -  \Pxt \rb \lb \calX^t - \calT \rb  }}_{=:\upsilon_3}.
\end{align*}
\paragraph{Bounding $\upsilon_1$.} Applying Lemma~\ref{lemma 2} gives that \begin{align*}
		\upsilon_1\leq \opnorm{   \Pxt- p^{-1}\Pxt\calP_{\Omega} \Pxt      }\cdot \fronorm{  \calX^t - \calT  } \leq 5\varepsilon\fronorm{\calX^t - \calT}.
	\end{align*}
\paragraph{Bounding $\upsilon_2$.} The application of~\citep[Lemma 5.2]{cai2020provable}
 gives that
	\begin{align*}
		\upsilon_2 = \fronorm{\lb \calI - \calP_{T_{\calX^t}} \rb(\calT)} \leq \frac{8}{\sigma_{\min}(\calT) }\fronorm{\calX^t - \calT}^2\leq \frac{8\sqrt{p} \varepsilon}{9(1+\varepsilon)}\fronorm{\calX^t - \calT}\leq 2\varepsilon\fronorm{\calX^t - \calT}.
	\end{align*}
\paragraph{Bounding $\upsilon_3$.} A direct computation yields that
	\begin{align*}
		\upsilon_3 &\leq p^{-1} \opnorm{\calP_{\Omega} \calP_{T_{\calX^t}}}\cdot \fronorm{\lb \calI - \calP_{T_{\calX^t}} \rb(\calT)}\\
		&\leq p^{-1}\cdot 2\sqrt{p}(1+\varepsilon)\cdot \frac{8}{\sigma_{\min}(\calT) }\fronorm{\calX^t - \calT}^2\leq 16\varepsilon\fronorm{\calX^t - \calT},
	\end{align*}
	where the second line is due to Lemma~\ref{lemma 2}.
\\

Putting the above bounds  together yields that
\begin{align}
	\label{eq contractive sequence}
	\fronorm{\calX^{t+1} - \calT} &\leq 4\fronorm{\calE^t} \stackrel{(a)}{\leq} 100\varepsilon\cdot \fronorm{\calX^t - \calT}\leq \frac{1}{2}\fronorm{\calX^t - \calT},
\end{align}
where $(a)$ is due to the induction condition and the last step follows from $\varepsilon\leq \frac{1}{200}$.

By the assumption of the theorem, the inequality \eqref{eq induction for local convergence} is valid for $t=t_0$. Since $\fronorm{\calX^t - \calT}$ is a contractive sequence following from \eqref{eq contractive sequence}, the inequality \eqref{eq induction for local convergence} is valid for all $t\geq t_0$ by induction. 
\end{proof}

{\color{black}\section{Conclusion and Discussion}\label{sec:conclusion}
In this paper,  entrywise convergence of the vanilla Riemannian gradient method for low Tucker-rank tensor completion has been established and the implicit regularization property of the method has been revealed. For conciseness of presentation, we focus on  three-way tensors and the noiseless case. Indeed, the results can be extended to general multi-way tensors and the noisy case, with the sampling complexity and error bound  matching these results in \citep{cai2021nonconvex,xia2019polynomial,xia2021statistically}, see the supplement of this paper for details. 
For future work, it is interesting to further optimize the dependency of the sampling complexity on the rank $r$. Additionally, it may also be possible to extend the analysis to the low rank tensor completion problem based on the tensor train decomposition since both the Tucker decomposition and the tensor train decomposition reduce to the same form of decomposition for the matrix case.

}

\vskip 0.2in
\bibliography{refTensor}

\appendix
\section{Proof of Theorem~\ref{base and induction case theorem}}
\subsection{Base Case for Theorem~\ref{base and induction case theorem}}
\label{sec: base case}
We first list two useful lemmas, whose proofs are deferred to Sections~\ref{Proof of Lemma spectral norm of delta 1,ell} and \ref{proof of lemma spectral norm distance of initialization}.
\begin{lemma}
\label{spectral norm of delta 1,ell}
Under the assumption of Theorem \ref{thm: finite iteration}, the following inequality holds with high probability,
\begin{align*}
&\opnorm{\calP_{\offdiag}\left(\widehat{\mT}_i^{\ell}\widehat{\mT}_i^{{\ell}^\tran}\right)-\mT_i\mT_i^\tran}\\
&\quad\leq  C\lb\lb\frac{\mu^{3/2} r^{3/2}}{n^{3/2}p}+\frac{\mu^2 r^2}{n^2p}\rb\log n+ \sqrt{r}  \lb \frac{\mu^{3/2} r^{3/2}\log^3 n}{n^{3/2}p}  + \sqrt{\frac{\mu ^2 r^2 \log^5 n }{n^2p} }  \rb+\frac{\mu r}{n}\rb\sigma^2_{\max}\lb\calT\rb.
\end{align*}
\end{lemma}

\begin{lemma}
	\label{lemma: spectral norm distance of initialization}
	Suppose  $n\geq C_0 \mu^3 r^5 \kappa^8$ and
	\begin{align}
		\label{cond: p}
		p\geq \max\left\{\frac{C_1 \kappa^8\mu^{3.5} r^6 \log^3 n }{n^{3/2}}, \frac{C_2 \kappa^{16}\mu^6 r^{11} \log^5 n }{n^2}\right\}.
	\end{align}
	With high probability, one has 
	\begin{align}
		\label{eq: spectral norm distance of initialization}
		\opnorm{\mX_i^1 \mR_i^1 - \mU_i} \leq  \frac{1}{2^5}\frac{1}{2^{20} \kappa^6 \mu^2 r^4 },\\
		\label{eq: spectral norm distance of initialization ell}
		\opnorm{\mX_i^{1,\ell} \mR_i^{1,\ell} - \mU_i} \leq  \frac{1}{2^5}\frac{1}{2^{20} \kappa^6 \mu^2 r^4 }.
	\end{align}
\end{lemma}

\subsubsection{Bounding \texorpdfstring{$\opnorm{\mE_i^{0} }$}{TEXT} and \texorpdfstring{$\opnorm{\mE_i^{0,\ell} }$}{TEXT}}
In light of the symmetry of tensor matricization  among different modes, we only focus on bounding $\opnorm{\mE_1^{0} }$ and $\opnorm{\mE_1^{0,\ell}}$, and omit the proofs for the other two modes. The spectral norm of $\mE_1^0$ can  be decomposed as follows:
\begin{align*}
	\opnorm{\mE_1^{0}}  &= \opnorm{\calM_1\lb\lb\calI-p^{-1}\calP_{\Omega}\rb\lb-\calT\rb\ttimes \mX_i^1{\mX_i^1}^\tran + \calT \ttimes  \mX_i^1{\mX_i^1}^\tran  - \calT \rb}\\
	&\leq  \underbrace{\opnorm{\calM_1 \lb \lb\calI-p^{-1}\calP_{\Omega}\rb\lb-\calT\rb \ttimes \mX_i^1{\mX_i^1}^\tran \rb} }_{=:\varphi_1}+ \underbrace{ \opnorm{\calM_1 \lb  \calT \ttimes  \mX_i^1{\mX_i^1}^\tran  - \calT \rb}}_{=:\varphi_2}.
\end{align*}
\paragraph{Bounding $\varphi_1$.}
Notice that  $\lb\calI-p^{-1}\calP_{\Omega}\rb\lb-\calT\rb\ttimes \mX_i^1{\mX_i^1}^\tran $ is a tensor of multilinear rank at most $(r,r,r)$. Applying  Lemma~\ref{lemma: Lemma 6 in XYZ} yields that
\begin{align}
	\label{eq I1}
	\varphi_1 &=\opnorm{\calM_1 \lb \lb\calI-p^{-1}\calP_{\Omega}\rb\lb-\calT\rb \ttimes \mX_i^1{\mX_i^1}^\tran \rb}
	\leq\sqrt{r} \opnorm{ \lb\calI-p^{-1}\calP_{\Omega}\rb\lb-\calT\rb\ttimes \mX_i^1{\mX_i^1}^\tran }\notag \\
	&\leq \sqrt{r} \opnorm{\lb p^{-1}\calP_{\Omega} - \calI \rb (\calT)} \stackrel{(a)}{\leq} \sqrt{r} C \lb \frac{\log^3 n}{p} \infnorm{\calT} + \sqrt{\frac{3\log^5 n}{p} \max_{i=1,2,3}\twoinf{  \mT_i^\tran }^2 } \rb \notag  \\
    	&\stackrel{(b)}{\leq} \sqrt{r} C \lb \frac{\mu^{3/2} r^{3/2}\log^3 n}{n^{3/2}p}  + \sqrt{\frac{\mu ^2 r^2 \log^5 n }{n^2p} }  \rb \sigma_{\max}(\calT)  \stackrel{(c)}{\leq}\frac{1}{2} \frac{1}{2^{20} \kappa^6 \mu^2 r^4}\frac{1}{2} \sigma_{\max}(\calT), 
\end{align}
where $(a)$ is due to Lemma~\ref{lemma: spectral norm of E initial}, $(b)$ follows from Lemma~\ref{lemma: incoherence for T}, and $(c)$ is due to the assumption  
	$p\geq \max\left\{ \frac{  C_1  \kappa^6\mu^{3.5}r^{6}  \log^3n  }{n^{3/2}},  \frac{ C_2 \kappa^{12} \mu^6 r^{11}   \log^5 n}{n^2}\right\}$.

\paragraph{Bounding $\varphi_2$.}
Since $\calT = \calT \ttimes \mU_i\mU_i^\tran$, one has
\begin{align*}
	\calT \ttimes  \mX_i^1{\mX_i^1}^\tran  - \calT  
	&= \calT\times_1 \left(  \mX_1^1{\mX_1^1}^\tran - \mU_1\mU_1^\tran \right) \times_2 \mX_2^1{\mX_2^1}^\tran \times_3 \mX_3^1{\mX_3^1}^\tran \\
	&\quad+  \calT\times_1\mU_1\mU_1^\tran \times_2 \left(  \mX_2^1{\mX_2^1}^\tran - \mU_2\mU_2^\tran \right)\times_3 \mX_3^1{\mX_3^1}^\tran  \\
	&\quad + \calT\times_1\mU_1\mU_1^\tran \times_2\mU_2\mU_2^\tran  \times_3 \left(  \mX_3^1{\mX_3^1}^\tran - \mU_3\mU_3^\tran \right).  
\end{align*}
Consequently,
\begin{align}
	\label{eq I2}
	\varphi_2 &=\opnorm{ \calM_1\left( \calT \ttimes  \mX_i^1{\mX_i^1}^\tran  - \calT  \right) }\leq 3\max_{i = 1,2,3}\opnorm{\mX_i^1{\mX_i^1}^\tran-\mU_i\mU_i^\tran}\cdot\max_{i = 1,2,3}\opnorm{\mT_i}\notag\\
 &\leq 6 \max_{i=1,2,3} \opnorm{\mX_i^1 \mR_i^1 - \mU_i}\cdot  \sigma_{\max}(\calT) \stackrel{(a)}{\leq} 6\cdot   \frac{1}{2^{25}\kappa^6 \mu^2 r^4} \cdot \sigma_{\max}(\calT),
\end{align}
where $(a)$ follows from \eqref{eq: spectral norm distance of initialization}.
\\

Putting \eqref{eq I1} and \eqref{eq I2} together shows that
\begin{align*}
	\opnorm{\mE_1^{0} } &\leq  \frac{1}{2^{20}\kappa^6 \mu^2 r^4} \frac{1}{2} \sigma_{\max}(\calT).
\end{align*}
Using the same argument as above, one can obtain
\begin{align*}
	\opnorm{\mE_1^{0,\ell}} &\leq  \frac{1}{2^{20}\kappa^6 \mu^2 r^4} \frac{1}{2}. \sigma_{\max}(\calT).
\end{align*}
\subsubsection{Bounding \texorpdfstring{$\twoinf{\mX_i^{1} \mR_i^{1} - \mU_i}$}{TEXT}}
\label{sec: incoherence of loo}
Theorem 3.1 in \citep{cai2021subspace} shows that with high probability,
\begin{align*}
    \twoinf{\mX_i^{1  } \mR_i^{1} - \mU_i} &\leq C\kappa^2\lb\frac{\kappa^2\mu^{3/2}r^{3/2}\log n}{n^{3/2}p}+\sqrt{\frac{\kappa^4\mu^2r^2\log n}{n^2p}}+\frac{\kappa^2\mu r}{n}\rb\sqrt{\frac{\mu r}{n}}\\
    &\leq \frac{1}{2}\frac{1}{2^{20}\kappa^2\mu^2r^4}\frac{1}{2}\sqrt{\frac{\mu r}{n}},
\end{align*}
where the last inequality follows from the assumption $n\geq C_0 \kappa^6\mu^{3}r^{5}$ and 
\begin{align}\label{twoinf mx1-mu}
    p\geq \max\left\lbrace \frac{C_1\kappa^6\mu^{3.5}r^{5.5}\log n}{n^{3/2}} ,\frac{C_2\kappa^{12}\mu^6r^{10}\log n}{n^2}\right\rbrace.
\end{align}

\subsubsection{Bounding \texorpdfstring{$\fronorm{\mX_i^{1}\mR_i^1 - \mX_i^{1,\ell} \mT_i^{1,\ell}}$}{TEXT}}
 Lemma 4 in \citep{cai2021subspace}, Lemma 8 in \citep{cai2021nonconvex} and Lemma~\ref{rateoptimal lemma 1} demonstrate that with high probability,
\begin{align}
	\label{eq upper bound of Xr-Xt}
	\fronorm{\mX_i^{1}\mR_i^1 - \mX_i^{1,\ell} \mT_i^{1,\ell}} &\leq\fronorm{\mX_i^{1,\ell}{ \mX_i^{1,\ell} }^\tran  - \mX_i^{1} {\mX_i^{1}}^\tran} \notag \\
	&\leq  C \lb \frac{ \kappa^2\mu^{3/2} r^{3/2}\log n}{n^{3/2} p} + \sqrt{\frac{ \kappa^4\mu^2 r^2\log n}{n^2p} }\rb \sqrt{\frac{\mu r}{n}}\notag  \\
	&\leq \frac{1}{2^{25} \kappa^2 \mu^2 r^4}  \frac{1}{2}\sqrt{\frac{\mu r}{n}},
\end{align}
 where the last line is due to the assumption
	$p\geq \max\left\{ \frac {C_1  \kappa^4\mu^{3.5}r^{5.5}\log n}{n^{3/2}} , \frac{C_2 \kappa^8\mu^6r^{10}\log n }{n^2} \right\}$.
The inequality \eqref{eq upper bound of Xr-Xt} directly implies that \eqref{F norm} holds for $t = 1$.

\subsubsection{Bounding \texorpdfstring{$\twoinf{\mX_i^{1,\ell} \mR_i^{1,\ell} - \mU_i} $}{TEXT}}
Invoking the triangle inequality gives
\begin{align}
	\label{eq Xr-U}
	\twoinf{\mX_i^{1,\ell } \mR_i^{1,\ell} - \mU_i}  &\leq \twoinf{ \mX_i^{1, \ell} \mR_i^{1,\ell} - \mX_i^{1}\mR_i^1} + \twoinf{\mX_i^{1} \mR_i^{1} - \mU_i }\notag \\
	&\leq \fronorm{ \mX_i^{1, \ell} \mR_i^{1,\ell} - \mX_i^{1}\mR_i^1 } + \twoinf{\mX_i^{1} \mR_i^{1} - \mU_i }.
\end{align}
By \eqref{eq: spectral norm distance of initialization} and \eqref{eq upper bound of Xr-Xt},  the following inequalities hold with high probability,
\begin{align*}
	&\opnorm{\mX_i^1\mR_i^1-\mU_i}\opnorm{\mU_i} \leq \frac{1}{2} \quad\mbox{and}\quad\opnorm{\mX_i^1\mR_i^1-\mX_i^{1,\ell}\mT_i^{1,\ell}}\opnorm{\mU_i}\leq \frac{1}{4}.
\end{align*}
Applying Lemma~\ref{lemma Ma 2017 lemma 37} with $\mX_1 = \mX_i^1\mR_i^1,~\mX_2 = \mX_i^{1,\ell}\mT_i^{1,\ell}$, one can see that
\begin{align}
	\label{eq second term}
	\fronorm{\mX_i^1\mR_i^1-\mX_i^{1,\ell}\mR_i^{1,\ell}} &\leq 5\fronorm{\mX_i^1\mR_i^1-\mX_i^{1,\ell}\mT_i^{1,\ell}} \leq  \frac{5}{2^{25} \kappa^2 \mu^2 r^4}  \frac{1}{2}\sqrt{\frac{\mu r}{n}},
\end{align}
where the last line is due to \eqref{eq upper bound of Xr-Xt}. Plugging \eqref{twoinf mx1-mu} and \eqref{eq second term} into \eqref{eq Xr-U} gives
\begin{align}\label{eq upper bound of X1ell-u}
	\twoinf{\mX_i^{1,\ell}\mR_i^{1,\ell}-\mU_i} \leq \frac{1}{ 2^{20} \kappa^2 \mu^{2} r^{4}} \frac{1}{2} \sqrt{\frac{\mu r}{n}}.
\end{align}

\subsubsection{Bounding \texorpdfstring{$\fronorm{ \calE^{0} - \calE^{0,\ell}}$}{TEXT}}
By the definitions of $\calE^0$ and $\calE^{0,\ell}$ in \eqref{calE0} and \eqref{calE0ell}, we can decompose $\calE^{0} - \calE^{0,\ell}$  as 
\begin{align*}
	\calE^{0} - \calE^{0,\ell} 
	&=   \lb \lb \calI - p^{-1}\calP_{\Omega} \rb(-\calT) \rb \ttimes{\mX_i^1}{\mX_i^1}^\tran - \lb \lb \calI - p^{-1}\calP_{\Omega} \rb(-\calT) \rb \ttimes{\mX_i^{1,\ell}}{\mX_i^{1,\ell}}^\tran   \\
	&\quad +  \lb \lb p^{-1}\calP_{\Omega} - p^{-1}\calP_{\Omega_{-\ell}} - \calP_{\ell} \rb(\calT) \rb\ttimes{\mX_i^{1,\ell}}{\mX_i^{1,\ell}}^\tran  \\
	&\quad +   \calT \ttimes {\mX_i^1}{\mX_i^1}^\tran  - \calT \ttimes  {\mX_i^{1,\ell}}{\mX_i^{1,\ell}}^\tran.
\end{align*}
It then follows from the triangle inequality that
\begin{align*}
	\fronorm{ \calE^{0} - \calE^{0,\ell}} &\leq \underbrace{ \fronorm{ \lb \lb \calI - p^{-1}\calP_{\Omega} \rb(-\calT) \rb \ttimes{\mX_i^1}{\mX_i^1}^\tran - \lb \lb \calI - p^{-1}\calP_{\Omega} \rb(-\calT) \rb \ttimes{\mX_i^{1,\ell}}{\mX_i^{1,\ell}}^\tran } }_{=:\varphi_3}\\
	&\quad + \underbrace{\fronorm{ \lb \lb p^{-1}\calP_{\Omega} - p^{-1}\calP_{\Omega_{-\ell}} - \calP_{\ell} \rb(\calT) \rb\ttimes{\mX_i^{1,\ell}}{\mX_i^{1,\ell}}^\tran  } }_{=:\varphi_4}\\
	&\quad + \underbrace{ \fronorm{ \calT \ttimes {\mX_i^1}{\mX_i^1}^\tran  - \calT \ttimes  {\mX_i^{1,\ell}}{\mX_i^{1,\ell}}^\tran }}_{=:\varphi_5}.
\end{align*}
\paragraph{Bounding $\varphi_3$.}
Simple calculation reveals that
\begin{align*}
    &\lb \lb \calI - p^{-1}\calP_{\Omega} \rb(-\calT) \rb \ttimes{\mX_i^1}{\mX_i^1}^\tran - \lb \lb \calI - p^{-1}\calP_{\Omega} \rb(-\calT) \rb \ttimes{\mX_i^{1,\ell}}{\mX_i^{1,\ell}}^\tran \\
    &\quad= \lb \lb \calI - p^{-1}\calP_{\Omega} \rb( -\calT) \rb \times_1 \left( {\mX_1^1}{\mX_1^1}^\tran  - {\mX_1^{1,\ell}}{\mX_1^{1,\ell}} \right) \times_2 {\mX_2^{1,\ell}}{\mX_2^{1,\ell}}^\tran  \times_3 {\mX_3^{1,\ell}}{\mX_3^{1,\ell}}^\tran\\
    &\quad\quad +\lb \lb \calI - p^{-1}\calP_{\Omega} \rb( -\calT) \rb \times_1   {\mX_1^1}{\mX_1^1}^\tran   \times_2 \left( {\mX_2^1}{\mX_2^1}^\tran  - {\mX_2^{1,\ell}}{\mX_2^{1,\ell}}^\tran \right)\times_3 {\mX_3^{1,\ell}}{\mX_3^{1,\ell}}^\tran\\
    &\quad \quad  +\lb \lb \calI - p^{-1}\calP_{\Omega} \rb( -\calT) \rb \times_1   {\mX_1^1}{\mX_1^1}^\tran   \times_2 {\mX_2^1}{\mX_2^1}^\tran \times_3\left( {\mX_3^1}{\mX_3^1}^\tran  - {\mX_3^{1,\ell}}{\mX_3^{1,\ell}}^\tran \right).
\end{align*}
Thus we have
\begin{align*}
    \varphi_3 
	&\leq \sum_{i = 1}^3\fronorm{\mX_i^1{\mX_i^1}^\tran-\mX_i^{1,\ell}{\mX_i^{1,\ell}}^\tran}\cdot \sqrt{r} \opnorm{\lb  \calI - p^{-1}\calP_{\Omega} \rb( \calT)  },
\end{align*}
where the inequality is due to Lemma~\ref{lemma: Lemma 6 in XYZ}. Together with \eqref{eq I1} and  \eqref{eq upper bound of Xr-Xt}, one can see that
\begin{align}\label{bounding I6}
	\varphi_3 &\leq 3\max_{i = 1,2,3} \fronorm{ {\mX_i^1}{\mX_i^1}^\tran  - {\mX_i^{1,\ell}}{\mX_i^{1,\ell}}^\tran } \cdot \sqrt{r} \opnorm{\lb  \calI - p^{-1}\calP_{\Omega} \rb( \calT)  }\notag\\
	&\leq 3\cdot \frac{1}{2^{25} \kappa^2 \mu^2 r^4}  \frac{1}{2}\sqrt{\frac{\mu r}{n}}\cdot \sqrt{r}\lb \frac{\mu^{3/2} r^{3/2}\log^3 n}{n^{3/2}p}  + \sqrt{\frac{\mu ^2 r^2 \log^5 n }{n^2p} }  \rb \sigma_{\max}(\calT) \notag\\
	&\leq \frac{1}{4} \frac{1}{2^{20} \kappa^4 \mu^2 r^4} \frac{1}{2} \sqrt{\frac{\mu r}{n}}  \sigma_{\max}(\calT),
\end{align}
where it is assumed that
	$p\geq \max\left\{ \frac {C_1  \kappa^4\mu^{3.5}r^{5.5}\log n}{n^{3/2}} , \frac{C_2 \kappa^8\mu^6r^{10}\log n }{n^2} \right\}$.

\paragraph{Bounding $\varphi_4$.}
Since $\fronorm{\calZ} = \fronorm{\calM_i\lb\calZ\rb}$ for any tensor $\calZ$, one has
\begin{align*}
	\varphi_{4} &= \fronorm{ {\mX_{1}^{1,\ell}}{\mX_{1}^{1,\ell}}^\tran\calM_1\lb \lb p^{-1}\calP_{\Omega} - p^{-1}\calP_{\Omega_{-\ell}} - \calP_{\ell} \rb(\calT) \rb \lb \mX_{3}^{1,\ell} \otimes \mX_{2}^{1,\ell} \rb \lb \mX_{3}^{1,\ell} \otimes \mX_{2}^{1,\ell} \rb^\tran}\\
	&\leq\fronorm{ {\mX_{1}^{1,\ell}}^\tran\calM_1\lb \lb p^{-1}\calP_{\Omega} - p^{-1}\calP_{\Omega_{-\ell}} - \calP_{\ell} \rb(\calT) \rb \lb \mX_{3}^{1,\ell} \otimes \mX_{2}^{1,\ell} \rb }.
\end{align*}
By the definition of $\calP_{\Omega_{-\ell}}$ and $\calP_{\ell}$, it can be seen that the entries of $\calM_1\lb \lb p^{-1}\calP_{\Omega} - p^{-1}\calP_{\Omega_{-\ell}} - \calP_{\ell} \rb(\calT) \rb$ are all zeros except those on the $\ell$-th row or the $j$-th column for any $j\in\Gamma$, where $\Gamma$ is an index set  defined by
\begin{align*}
    \Gamma = \{\ell,n+\ell,\cdots, n(\ell-2)+\ell,n(\ell-1)+1,\cdots,n(\ell-1)+n, n\ell+\ell,\cdots, n(n-1)+\ell\}.
\end{align*}
Using this fact, we have
\begin{align*}
	\varphi_4
	&\leq \fronorm{ \calP_{\ell,:}\lb \calM_1\lb \lb p^{-1}\calP_{\Omega} - p^{-1}\calP_{\Omega_{-\ell}} - \calP_{\ell} \rb(\calT) \rb \rb \lb \mX_{3}^{1,\ell} \otimes \mX_{2}^{1,\ell} \rb}\\
	&\quad + \fronorm{{\mX_{1}^{1,\ell}}^\tran\calP_{-\ell,\Gamma}\lb \calM_1\lb \lb p^{-1}\calP_{\Omega} - p^{-1}\calP_{\Omega_{-\ell}} - \calP_{\ell} \rb(\calT) \rb \rb \lb \mX_{3}^{1,\ell} \otimes \mX_{2}^{1,\ell} \rb}\\
	:&=\varphi_{4,1} + \varphi_{4,2}.
\end{align*}
Here for any matrix $\mM\in\R^{n\times n^2}$,  $\calP_{\ell,:}(\mM)$ and $\calP_{-\ell, \Gamma}(\mM)$ are defined by
\begin{align}
\label{eq submatrix row}
&\lsb \calP_{\ell,:}(\mM) \rsb_{i,j} = \begin{cases}
	\lsb\mM\rsb_{i,j}, &\text{ if } i =\ell,\\
	0, &\text{ otherwise,}
\end{cases}\\
\label{eq submatrix columns}
&\lsb \calP_{-\ell, \Gamma}(\mM) \rsb_{i,j} = \begin{cases}
	\lsb\mM\rsb_{i,j}, &\text{ if }i\neq \ell \text{ and } j\in \Gamma,\\
	0,&\text{otherwise.}
\end{cases}
\end{align}

\begin{itemize}
	\item For $\varphi_{4,1} $, note that
	\begin{align*}
	    &\calP_{\ell,:}\lb \calM_1\lb \lb p^{-1}\calP_{\Omega} - p^{-1}\calP_{\Omega_{-\ell}} - \calP_{\ell} \rb(\calT) \rb \rb \lb \mX_{3}^{1,\ell} \otimes \mX_{2}^{1,\ell} \rb \\
	    &\quad= \sum_{j=1}^{n^2} \lb1 -p^{-1} \delta_{\ell, j} \rb [\calM_1(\calT)]_{\ell, j} \lsb \mX_{3}^{1,\ell} \otimes \mX_{2}^{1,\ell} \rsb_{j,:} := \sum_{j=1}^{n^2} \vx_j^\tran,
	\end{align*}
 where $\mX^{1,\ell}_{2}$ and $\mX^{1,\ell}_{3}$  are independent of $\{\delta_{\ell, j}\}_{j\in[n^2]}$ by construction.
 A simple computation implies that
	\begin{align*}
		\opnorm{\vx_j} &\leq p^{-1} \infnorm{\calT}\cdot \lb  \max_{i = 1,2,3}\twoinf{\mX_i^{1,\ell}}\rb^2,\\
		\opnorm{\sum_{j=1}^{n^2}\E{\vx_j^\tran\vx_j}} &\leq p^{-1}\twoinf{\mT_1}^2\cdot \lb  \max_{i = 1,2,3}\twoinf{\mX_i^{1,\ell}} \rb^4,\\
		\opnorm{\sum_{j=1}^{n^2}\E{\vx_j\vx_j^\tran}} &\leq p^{-1}\twoinf{\mT_1}^2\cdot \lb  \max_{i = 1,2,3}\twoinf{\mX_i^{1,\ell}} \rb^4.
	\end{align*}
	By the matrix Bernstein inequality \citep[Theorem 6.1.1]{tropp2015introduction}, the following inequality holds with high probability,
 \begin{small}
	\begin{align*}
	    	\varphi_{4,1}  &\leq C \lb \frac{\log n}{p}  \infnorm{\calT}\cdot \lb  \max_{i = 1,2,3}\twoinf{\mX_i^{1,\ell}}\rb^2 + \sqrt{ \frac{\log n}{p} \twoinf{\mT_1}^2\cdot \lb  \max_{i = 1,2,3}\twoinf{\mX_i^{1,\ell}} \rb^4  } \rb
	\end{align*}
 \end{small}
   Moreover, from \eqref{eq upper bound of X1ell-u}, we have
	\begin{align}\label{upper bound of x1ell-u}
		\twoinf{\mX_i^{1,\ell}} \leq\twoinf{\mX_i^{1,\ell}\mR_i^{1,\ell} - \mU_i} + \twoinf{\mU_i} \leq 2\sqrt{ \frac{\mu r}{n}}.
	\end{align}
	Therefore, with high probability,
	\begin{align*}
		\varphi_{4,1} 
		&\leq C\lb \frac{4\log n}{p}\lb\frac{\mu r}{n}\rb^{5/2}+\sqrt{\frac{16\log n}{p}\lb\frac{\mu r}{n}\rb^{3}}\rb\sigma_{\max}\lb\calT\rb\\
  &\leq 
  \frac{1}{8} \frac{1}{2^{20} \kappa^4 \mu^2 r^4} \frac{1}{2} \sqrt{\frac{\mu r}{n}}  \sigma_{\max}(\calT),
	\end{align*}
 where the last step is due to the assumption
	$p\geq   \frac{C_2 \kappa^8\mu^6 r^{10}\log n }{n^2}$.

	\item For $\varphi_{4,2} $, it can be rearranged as follows:
	\begin{align*}
		\varphi_{4,2} &\leq \sqrt{r}\opnorm{\calP_{-\ell,\Gamma}\lb \calM_1\lb \lb p^{-1}\calP_{\Omega} - p^{-1}\calP_{\Omega_{-\ell}} - \calP_{\ell} \rb(\calT) \rb \rb \lb \mX_{3}^{1,\ell} \otimes \mX_{2}^{1,\ell} \rb}\\
		&=\sqrt{r}\opnorm{ \sum_{j\in\Gamma} \underbrace{\lsb  \calP_{-\ell,:}\lb \calM_1\lb\lb p^{-1}\calP_{\Omega} - p^{-1}\calP_{\Omega_{-\ell}} - \calP_{\ell} \rb(\calT)  \rb \rb\rsb_{:, j} \lsb\mX_{3}^{1,\ell} \otimes \mX_{2}^{1,\ell} \rsb_{j,:}}_{:=\mZ_j}},
	\end{align*}
	where $\mZ_j$ are independent mean-zero random matrices conditioned on $\mX^{1,\ell}_{2}$ and $\mX^{1,\ell}_{3}$. First, a straightforward computation shows that
	\begin{align*}
		\opnorm{\mZ_j} &\leq \twonorm{ \lsb  \calP_{-\ell,:}\lb \calM_1\left(\lb p^{-1}\calP_{\Omega} - p^{-1}\calP_{\Omega_{-\ell}} - \calP_{\ell} \rb(\calT) \right) \rb \rsb_{:, j} }\cdot \twonorm{\lsb\mX_{3}^{1,\ell} \otimes \mX_{2}^{1,\ell} \rsb_{j,:} }\\
		&\leq \twonorm{ \lsb \calM_1\left(\lb p^{-1}\calP_{\Omega} - p^{-1}\calP_{\Omega_{-\ell}} - \calP_{\ell} \rb(\calT) \right) \rsb_{:, j} }\cdot \twonorm{\lsb\mX_{3}^{1,\ell} \otimes \mX_{2}^{1,\ell} \rsb_{j,:} }\\
		&\leq \frac{1}{p}\twoinf{\mT_1^\tran} \cdot \lb \max_{i = 1,2,3} \twoinf{\mX_i^{1,\ell}} \rb^2.
	\end{align*}
	Moreover, one has 
	\begin{align*}
	\sigma_{\mZ}^2:=&\opnorm{\sum_{j\in\Gamma }\E{\mZ_j\mZ_j^\tran}}\leq \frac{1}{p}\lb \max_{i = 1,2,3} \twoinf{\mX_i^{1,\ell}} \rb^4 \cdot \twoinf{\mT_1}^2\leq \frac{16}{p} \cdot \lb\frac{\mu r }{n}\rb^3 \sigma_{\max}^2 (\calT),\\
	\sigma_{\mZ'}^2:=&\opnorm{\sum_{j\in\Gamma }\E{\mZ_j^\tran\mZ_j}}\leq \frac{2n}{p}\twoinf{\mT_1^\tran}^2\cdot \lb \max_{i = 1,2,3} \twoinf{\mX_i^{1,\ell}} \rb^4\leq \frac{32n}{p}\cdot  \lb\frac{\mu r}{n}\rb^4\sigma_{\max}^2 (\calT).
\end{align*}
Therefore, by the matrix Bernstein inequality, we know that
\begin{align*}
    \varphi_{4,2}&\leq \sqrt{r}\cdot C\lb \frac{\log n}{p}\twoinf{\mT_1^\tran} \cdot \lb \max_{i = 1,2,3} \twoinf{\mX_i^{1,\ell}} \rb^2+\sqrt{\max\left\lbrace \sigma_{\mZ}^2,\sigma_{\mZ'}^2\right\rbrace\log n}\rb\\
    &\leq \sqrt{r}\cdot C\left(  \frac{\log n}{p} \left(\frac{\mu r  }{n}\right)^{3/2} + \sqrt{ \frac{ n\log n}{p}\cdot  \lb \frac{\mu r}{n} \rb^3  }\right) \sqrt{ \frac{\mu r  }{n}}\sigma_{\max}(\calT)  \\
		&\leq\frac{1}{8} \frac{1}{2^{20} \kappa^4 \mu^2 r^4} \frac{1}{2} \sqrt{\frac{\mu r}{n}}  \sigma_{\max}(\calT)
\end{align*}
holds with high probability  under the assumption 
		$p\geq \max\left\{  \frac{C_1 \kappa^4\mu^{3.5} r^{6} \log n }{n^{3/2}}, \frac{ C_2\kappa^8 \mu^7 r^{12}\log n }{n^2} \right\}$.
		
\end{itemize}
Combining the above bounds together, one can find that
\begin{align} \label{bounding I7}
	\varphi_4 &\leq \frac{1}{4} \frac{1}{2^{20} \kappa^4 \mu^2 r^4} \frac{1}{2} \sqrt{\frac{\mu r}{n}}  \sigma_{\max}(\calT).
\end{align}
\paragraph{Bounding $\varphi_5$.}
It follows from the triangle inequality that
\begin{align}\label{bounding I8}
	\varphi_5 &\leq \fronorm{  \calT \times_1 \left( {\mX_1^1}{\mX_1^1}^\tran  - {\mX_1^{1,\ell}}{\mX_1^{1,\ell}}^\tran \right) \times_2 {\mX_2^{1,\ell}}{\mX_2^{1,\ell}}^\tran  \times_3 {\mX_3^{1,\ell}}{\mX_3^{1,\ell}}^\tran } \notag\\
	&\quad + \fronorm{ \calT \times_1   {\mX_1^1}{\mX_1^1}^\tran   \times_2 \left( {\mX_2^1}{\mX_2^1}^\tran  - {\mX_2^{1,\ell}}{\mX_2^{1,\ell}}^\tran \right)\times_3 {\mX_3^{1,\ell}}{\mX_3^{1,\ell}}^\tran }\notag\\
	&\quad + \fronorm{ \calT \times_1   {\mX_1^1}{\mX_1^1}^\tran   \times_2 {\mX_2^1}{\mX_2^1}^\tran \times_3\left( {\mX_3^1}{\mX_3^1}^\tran  - {\mX_3^{1,\ell}}{\mX_3^{1,\ell}} ^\tran \right) }\notag\\
	&\leq 3 \max_{i = 1,2,3} \fronorm{ {\mX_i^1}{\mX_i^1}^\tran  - {\mX_i^{1,\ell}}{\mX_i^{1,\ell}} ^\tran } \cdot \sigma_{\max}(\calT)
	\stackrel{(a)}{\leq}\frac{1}{4} \frac{1}{2^{20} \kappa^4 \mu^2 r^4} \frac{1}{2} \sqrt{\frac{\mu r}{n}}  \sigma_{\max}(\calT),
\end{align}
where  $(a)$ follows from \eqref{eq upper bound of Xr-Xt}. 

\paragraph{Combining $\varphi_3$, $\varphi_4$ and $\varphi_5$.}
Putting \eqref{bounding I6}, \eqref{bounding I7} and \eqref{bounding I8} together yields that
\begin{align*}
	\fronorm{ \calE^{0} - \calE^{0,\ell}}  \leq  \frac{1}{2^{20} \kappa^4 \mu^2 r^4} \frac{1}{2} \sqrt{\frac{\mu r}{n}}  \sigma_{\max}(\calT).
\end{align*}

\subsection{Induction Step for Theorem~\ref{base and induction case theorem}}
\label{sec: induction case}
For the sake of clarity,  a proof roadmap is presented  in Figure \ref{fig:proofsketch}, which shows the dependencies of different quantities during the induction process.
\begin{figure}[ht!]
\begin{center}
\includegraphics[width=0.8\textwidth]{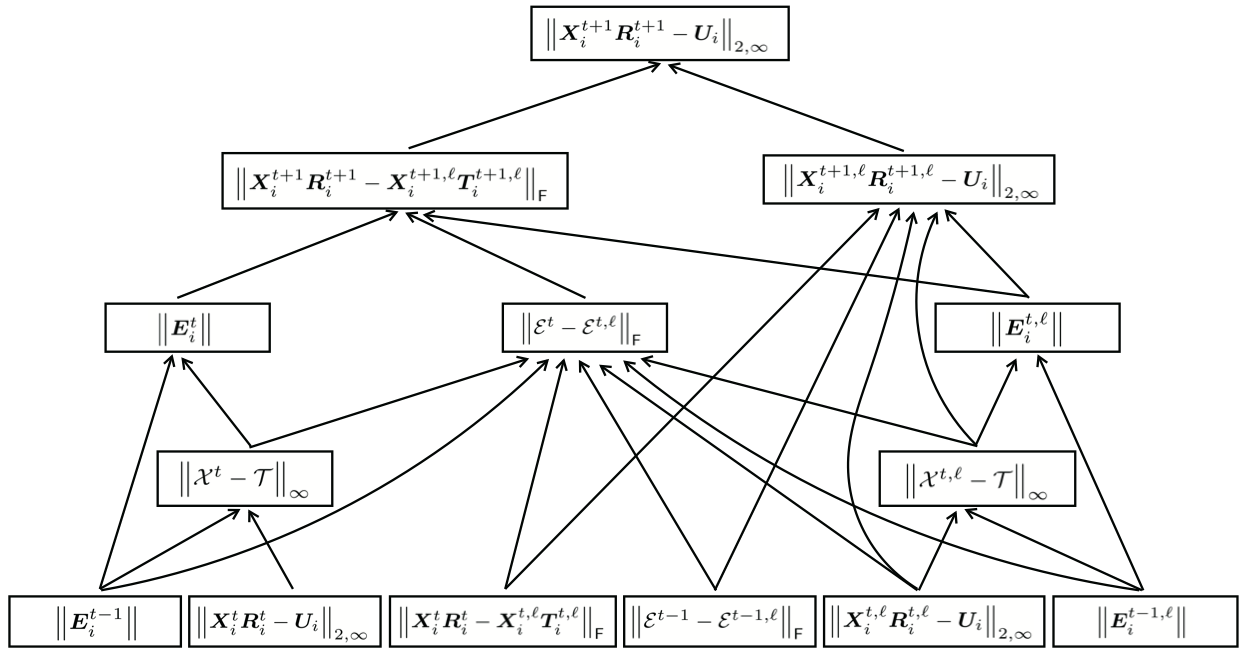}
\end{center}
\caption{Proof roadmap for  the  induction step.}\label{fig:proofsketch}
\end{figure}
We first summarize below several immediate consequences of \eqref{eq induction hypotheis}, which will be useful throughout this section. The proofs of these results are provided in Section \ref{sec: proof of key lemmas}.

\begin{lemma}
	\label{lemma: spectral norm of Xt}
	From the inequalities \eqref{spectral norm of error tensor}, \eqref{spectral norm of loo error tensor} and  \eqref{2 inf norm of loo}, one has
	\begin{align}
		\label{eq spectral norm XR-U}
		\opnorm{\mX_i^t \mR_i^t - \mU_i} &\leq \frac{1}{2^{17} \kappa^4  \mu^2 r^4 }\frac{1}{2^t}, \quad i = 1,2,3,\\
		\label{eq spectral norm XRl-U}
		\opnorm{\mX_i^{t,\ell} \mR_i^{t,\ell} - \mU_i} &\leq \frac{1}{2^{17} \kappa^4  \mu^2 r^4 }\frac{1}{2^t}, \quad i = 1,2,3, \\
		\label{eq spectral norm of Pxt-Pu}
		\opnorm{ {\mX_i^t} {\mX_i^t} ^\tran -  \mU_i\mU_i^\tran } &\leq \frac{1}{2^{17} \kappa^4 \mu^2 r^4} \frac{1}{2^t}, \quad i = 1,2,3,\\
		\label{eq spectral norm of Px-Pu}
		\opnorm{  {\mX_i^{t,\ell}}{\mX_i^{t,\ell}}^\tran - \mU_i\mU_i^\tran } &\leq \frac{1}{2^{17} \kappa^4  \mu^2 r^4 }\frac{1}{2^t}, \quad i = 1,2,3,\\
		\label{eq two inf norm of Px-Pu}
		\twoinf{ {\mX_i^{t,\ell}}{\mX_i^{t,\ell}}^\tran - \mU_i\mU_i^\tran } &\leq  \frac{1}{2^{16} \kappa^2  \mu^2 r^4 }\frac{1}{2^t}\sqrt{\frac{\mu r}{n}}, \quad i = 1,2,3.
	\end{align}
\end{lemma}

\begin{lemma}
	\label{lemma: infty norm at t step}
From the inequalities \eqref{spectral norm of error tensor}, \eqref{spectral norm of loo error tensor}, \eqref{2 inf norm of loo} and \eqref{L2 inf norm}, one has 
	\begin{align}
		\label{infty norm at t step}
		\infnorm{\calX^t - \calT} &\leq \frac{36}{2^{20} \kappa^2 \mu^2 r^{4}}\cdot \frac{1}{2^t} \cdot \left( \frac{\mu r}{n} \right)^{3/2}\cdot \sigma_{\max}(\calT),\\
		\label{infty norm at t step for Xtl}
		\infnorm{\calX^{t,\ell} - \calT} &\leq \frac{36}{2^{20} \kappa^2 \mu^2 r^{4}}\cdot \frac{1}{2^t} \cdot \left( \frac{\mu r}{n} \right)^{3/2}\cdot \sigma_{\max}(\calT).
	\end{align}
\end{lemma}

\begin{lemma}
	\label{lemma: condition number of X}
Assuming \eqref{spectral norm of error tensor} and \eqref{spectral norm of loo error tensor}, the largest and the smallest nonzero singular values of $\calM_i\lb\calX^t\rb$ satisfy
	\begin{align}
		\label{eq condition number of Xt}
		&\left(1+ \frac{1}{2^9  }\right) \sigma_{\max}(\calT) \geq \sigma_{\max}\lb \calM_i\lb\calX^t\rb \rb  \geq \left(1 - \frac{1}{2^9}\right)\sigma_{\min}(\calT)\geq \frac{15}{16} \sigma_{\min}(\calT),\\
		\label{eq condition number of Xtl}
		&\left(1+ \frac{1}{2^9  }\right) \sigma_{\max}(\calT) \geq \sigma_{\max}\lb \calM_i\lb\calX^{t,\ell}\rb \rb   \geq \left(1 - \frac{1}{2^9}\right)\sigma_{\min}(\calT)\geq \frac{15}{16} \sigma_{\min}(\calT).
	\end{align}
	It follows immediately that the condition number of $\calX^t $ obeys 
	\begin{align*}
	    \kappa\lb\calX^t\rb = \frac{\sigma_{\max}\lb\calX^t\rb}{\sigma_{\min}\lb\calX^t\rb}\leq \frac{\lb 1+2^{-9}\rb\sigma_{\max}\lb\calT\rb}{\lb 1-2^{-9}\rb\sigma_{\min}\lb\calT\rb}\leq 2\kappa.
	\end{align*}
\end{lemma}

\subsubsection{Bounding \texorpdfstring{$\opnorm{\mE_i^{t} } $}{TEXT} and \texorpdfstring{$\opnorm{\mE_i^{t,\ell} } $}{TEXT}}
\label{sec: bounding E at t step}
We provide a detailed proof for $\opnorm{\mE_1^{t} }$,  while the proofs for $\opnorm{\mE_i^{t} }$ ($i=2,3$) and $\opnorm{\mE_i^{t,\ell}}$ ($i=1,2,3$) are overall similar. First,  $\opnorm{\mE_1^{t}  } $ can be bounded as follows:
\begin{align*}
	\opnorm{\mE_1^{t}}&\leq \opnorm{ \calM_1\lb \lb \calI - \calP_{T_{\calX^t}}  \rb\lb\calX^t - \calT \rb \rb} + \opnorm{\calM_1\lb \calP_{T_{\calX^t}} \lb \calI - p^{-1} \calP_{\Omega} \rb\lb\calX^t - \calT \rb  \rb} \\
	&\leq \fronorm{\calM_1\lb \lb \calI - \calP_{T_{\calX^t}}  \rb\lb\calX^t - \calT \rb \rb }+ \opnorm{\calM_1\lb \calP_{T_{\calX^t}} \lb \calI - p^{-1} \calP_{\Omega} \rb\lb\calX^t - \calT \rb  \rb} \\
	&=\fronorm{\lb \calI - \calP_{T_{\calX^t}}  \rb\lb\calX^t - \calT \rb }  + \opnorm{\calM_1\lb \calP_{T_{\calX^t}} \lb \calI - p^{-1} \calP_{\Omega} \rb\lb\calX^t - \calT \rb  \rb}\\
	&=\underbrace{ \fronorm{\lb \calI - \calP_{T_{\calX^t}}  \rb\lb \calT \rb } }_{= : \alpha_1} +\underbrace{ \opnorm{\calM_1\lb \calP_{T_{\calX^t}} \lb \calI - p^{-1} \calP_{\Omega} \rb\lb\calX^t - \calT \rb  \rb}}_{=:\alpha_2}.
\end{align*}
To proceed, we need to  decompose $\calP_{T_{\calX^t}}$  into a sum of products of several projectors \citep{cai2020provable}. 
Define the projector $\calP_{\mX^t_i }^{(i)}: \R^{n\times n\times n}\rightarrow \R^{n\times n\times n} $ by 
\begin{align}
\label{def: projection 1}
\calP_{\mX^t_i }^{(i)}(\calZ) = \calZ\times_i \mX_i^t {\mX_i^t }^\tran,~ \forall \calZ\in\R^{n\times n\times n}.
\end{align}
It is evident that
$	\calZ\ttimes  {\mX^t_i}{\mX^t_i}^\tran  = \prod_{i=1}^{3} \calP_{\mX^t_i }^{(i)}(\calZ)$.
The orthogonal complement  of $\calP_{\mX^t_i }^{(i)}$, denoted $\calP_{{\mX_i^{t, \perp}}}^{(i)}: \R^{n\times n\times n}\rightarrow \R^{n\times n\times n}$, is defined as
\begin{align*}
	\calP_{{\mX_i^{t,\perp}}}^{(i)}(\calZ) = \calZ \times_i \lb \mI - \mX^t_i  {\mX^t_i }^\tran \rb.
\end{align*}
Furthermore, we define $\calP_{\calG^t , \left\lbrace\mX^t_j \right\rbrace_{j\neq 1}}^{(j\neq 1)} : \Rn\rightarrow \Rn$ by
\begin{align}
\label{def: projection 2}
\calM_1\lb \calP_{\calG^t, \left\lbrace\mX^t_j \right\rbrace_{j\neq 1}}^{(j\neq 1)}(\calZ)  \rb = \calM_1(\calZ) \lb \mX^t_{3}\otimes \mX^t_{2} \rb \calM_1^\dagger(\calG^t)\calM_1(\calG^t) \lb \mX^t_{3}\otimes \mX^t_{2}\rb^\tran,
\end{align}
and $\calP_{\calG^t , \left\lbrace\mX^t_j \right\rbrace_{j\neq 2}}^{(j\neq 2)}$, $\calP_{\calG^t , \left\lbrace\mX^t_j \right\rbrace_{j\neq 3}}^{(j\neq 3)}$ are similarly defined. All together, it is not  hard to see that $\calP_{T_{\calX^t}}$ can be rewritten as
\begin{align}\label{definition projector}
    \calP_{T_{\calX^t}} = \prod_{i=1}^{3} \calP_{\mX^t_i }^{(i)}+\sum_{i = 1}^3\calP_{{\mX_i^{t, \perp}}}^{(i)}\calP_{\calG^t , \left\lbrace\mX^t_j \right\rbrace_{j\neq i}}^{(j\neq i)}.
\end{align}
Moreover, the following properties hold, whose proofs are provided in Section~\ref{subsec:proj_properties}.
\begin{claim}\label{property of these projection}
    By the definition of  $\calP_{\mX^t_i }^{(i)}$, $\calP_{{\mX_i^{t, \perp}}}^{(i)}$ and $\calP_{\calG^t, \left\lbrace\mX^t_j \right\rbrace_{j\neq i}}^{(j\neq i)}$, one has
    \begin{align*}
    &\calP_{{\mX_i^{t, \perp}}}^{(i)}\calP_{\calG^t , \left\lbrace\mX^t_j \right\rbrace_{j\neq i}}^{(j\neq i)} =  \calP_{\calG^t , \left\lbrace\mX^t_j \right\rbrace_{j\neq i}}^{(j\neq i)}\calP_{{\mX_i^{t, \perp}}}^{(i)},\quad\calP_{{\mX_i^{t}}}^{(i)}\calP_{\calG^t , \left\lbrace\mX^t_j \right\rbrace_{j\neq i}}^{(j\neq i)} =  \calP_{\calG^t , \left\lbrace\mX^t_j \right\rbrace_{j\neq i}}^{(j\neq i)}\calP_{{\mX_i^{t}}}^{(i)},\\
    &\calM_i\lb\calP_{\calG^t , \left\lbrace\mX^t_j \right\rbrace_{j\neq i}}^{(j\neq i)}\lb\calZ\rb\rb = \calM_i\lb\calZ\rb\mY_i^t{\mY_i^t}^\tran, \quad\opnorm{\prod_{j\neq i}\calP_{\mX^t_j }^{(j)}-\calP_{\calG^t , \left\lbrace\mX^t_j \right\rbrace_{j\neq i}}^{(j\neq i)}} = 1,\\
    &\lb \prod_{j\neq i}\calP^{(j)}_{\mX_j^{t}}   - \calP_{\calG^t, \left\lbrace\mX^t_j \right\rbrace_{j\neq i}}^{(j\neq i)}\rb  \lb \calX^t\rb =\bzero,
\end{align*}
where the columns of  $\mY_i^t$ are the top-$r$ right singular vectors of $\calM_i\lb\calX^t\rb$.
\end{claim}
\paragraph{Bounding $\alpha_1$.} 
Based on the decomposition of $\calP_{T_{\calX^t}}$, it has been pointed out in \citep{cai2020provable} that 
\begin{align*}
	\calI - \calP_{T_{\calX^t}} &= \prod_{i=1}^{3}\lb \calP^{(i)}_{\mX_i^{t}} + \calP^{(i)}_{{\mX_{i}^{t, \perp}}}  \rb - \lb \prod_{i=1}^{3}  \calP^{(i)}_{\mX_i^{t}}  + \sum_{i=1}^{3}\calP^{(i)}_{\mX_{i}^{t, \perp}}   \calP_{\calG^t, \left\lbrace\mX^t_j \right\rbrace_{j\neq i}}^{(j\neq i)}\rb\\
	&=\sum_{i=1}^{3} \calP^{(i)}_{\mX_{i}^{t, \perp}}   \lb \prod_{j\neq i}\calP^{(j)}_{\mX_j^{t}}   - \calP_{\calG^t, \left\lbrace\mX^t_j \right\rbrace_{j\neq i}}^{(j\neq i)}\rb  + \sum_{i=1}^3 \calP_{\mX_i^{t}}^{(i)} \prod_{j\neq i} \calP_{\mX_{j}^{t,\perp}}^{(j)}+\prod_{i=1}^{3}\calP^{(i)}_{\mX_{i }^{t,\perp}}.
\end{align*}
Note that, for any $i = 1,2,3$,
\begin{align*}
    \calP^{(i)}_{\mX_{i}^{t, \perp}}\lb\calX^t\rb=\bzero \quad \text{and}\quad
    \lb \prod_{j\neq i}\calP^{(j)}_{\mX_j^{t}}   - \calP_{\calG^t, \left\lbrace\mX^t_j \right\rbrace_{j\neq i}}^{(j\neq i)}\rb  \lb \calX^t\rb =\bzero.
\end{align*}
Therefore,
\begin{small}
\begin{align}\label{eq: key equation 1}
	\lb \calI - \calP_{T_{\calX^t}}  \rb\lb \calT \rb &=\sum_{i=1}^{3} \lb \calP_{\mU_i}^{(i)} - \calP_{\mX_{i}^{t}}^{(i)}  \rb \lb \prod_{j\neq i}\calP_{\mX_j^{t}}^{(j)}   - \calP_{\calG^t, \left\lbrace\mX^t_j \right\rbrace_{j\neq i}}^{(j\neq i)}\rb  \lb \calT \rb \notag\\
	&\quad + \lb \calP_{\mU_1}^{(1)} - \calP_{\mX_{1}^{t}}^{(1)} \rb \calP^{(2)}_{\mX_{2}^{t, \perp}} \calP^{(3)}_{\mX_3^{t}} \lb \calT \rb  + \lb \calP^{(2)}_{\mU_2} - \calP^{(2)}_{\mX_{2}^{t}}  \rb \calP^{(1)}_{\mX_{1}^{t}} \calP^{(3)}_{\mX_3^{t, \perp}}  \lb\calT \rb \notag\\
	&\quad  +\lb \calP^{(3)}_{\mU_3} - \calP^{(3)}_{\mX_{3}^{t}}  \rb \calP^{(1)}_{\mX_{1}^{t, \perp}} \calP^{(2)}_{\mX_2^{t}}  \lb  \calT \rb\notag + \lb \calP^{(1)}_{\mU_1} - \calP^{(1)}_{\mX_{1}^{t}}  \rb\calP^{(2)}_{\mX_{2}^{t, \perp}} \calP^{(3)}_{\mX_{3}^{t, \perp}}  \lb \calT \rb\notag\\
	&=-\sum_{i=1}^{3} \lb \calP_{\mU_i}^{(i)} - \calP_{\mX_{i}^{t}}^{(i)}  \rb \lb \prod_{j\neq i}\calP_{\mX_j^{t}}^{(j)}   - \calP_{\calG^t, \left\lbrace\mX^t_j \right\rbrace_{j\neq i}}^{(j\neq i)}\rb  \lb \calX^t - \calT \rb \notag\\
	&\quad - \lb \calP_{\mU_1}^{(1)} - \calP_{\mX_{1}^{t}}^{(1)} \rb \calP^{(2)}_{\mX_{2}^{t, \perp}} \calP^{(3)}_{\mX_3^{t}} \lb \calX^t - \calT \rb  - \lb \calP^{(2)}_{\mU_2} - \calP^{(2)}_{\mX_{2}^{t}}  \rb \calP^{(1)}_{\mX_{1}^{t}} \calP^{(3)}_{\mX_3^{t, \perp}}  \lb \calX^t - \calT \rb \notag\\
	&\quad  -\lb \calP^{(3)}_{\mU_3} - \calP^{(3)}_{\mX_{3}^{t}}  \rb \calP^{(1)}_{\mX_{1}^{t, \perp}} \calP^{(2)}_{\mX_2^{t}}  \lb \calX^t - \calT \rb-\lb \calP^{(1)}_{\mU_1} - \calP^{(1)}_{\mX_{1}^{t}}  \rb\calP^{(2)}_{\mX_{2}^{t, \perp}} \calP^{(3)}_{\mX_{3}^{t, \perp}}  \lb \calX^t - \calT \rb.
\end{align}
\end{small}
It follows that
\begin{align}
	\label{eq I1 for opnorm of E}
	\alpha_1&\leq \sum_{i=1}^{3} \opnorm{ {\mU_i}\mU_i^\tran- {\mX_i^t}{\mX_i^t}^\tran}  \cdot \opnorm{ \prod_{j\neq i}\calP^{(j)}_{\mX_j^{t}}   - \calP_{\calG^t, \left\lbrace\mX^t_j \right\rbrace_{j\neq i}}^{(j\neq i)}}\cdot \fronorm{\calX^t - \calT} \notag\\
 &\quad+ 4 \max_{i = 1,2,3} \opnorm{{\mU_i}\mU_i^\tran- {\mX_i^t}{\mX_i^t}^\tran}  \cdot  \fronorm{\calX^t - \calT}\notag \\
	&\stackrel{(a)}{\leq } 7\max_{i = 1,2,3} \opnorm{{\mU_i}\mU_i^\tran- {\mX_i^t}{\mX_i^t}^\tran}  \cdot  \fronorm{\calX^t - \calT}\notag \\
	&\leq 7\max_{i = 1,2,3} \opnorm{{\mU_i}\mU_i^\tran- {\mX_i^t}{\mX_i^t}^\tran}  \cdot  n^{3/2}\cdot\infnorm{\calX^t - \calT}\notag \\
	&\stackrel{(b)}{\leq} 7\cdot \frac{1}{2^{17} \kappa^4 \mu^2 r^4 }\frac{1}{2^t} \cdot n^{3/2}\cdot \frac{36}{2^{20}\kappa^2 \mu^2 r^{4}}\cdot \frac{1}{2^t} \cdot \left( \frac{\mu r}{n} \right)^{3/2}\cdot \sigma_{\max}(\calT) \notag \\
	&\leq \frac{1}{2^9}\frac{1}{2^{20} \kappa^6\mu^2 r^4 }\cdot \frac{1}{2^{t+1}}\cdot \sigma_{\max}(\calT),
\end{align}
where  $(a)$ is due to Claim \ref{property of these projection} and   $(b) $ follows from  \eqref{eq spectral norm of Pxt-Pu} and \eqref{infty norm at t step}.

\paragraph{\cjc{Bounding $\alpha_2$.}}
\label{sec: proof of claim alpha2}
By the definition of $\calP_{T_{\calX^t}}$ in \eqref{eq: def of projection onto tangent space} and the triangular inequality, we have
\begin{align*}
	\alpha_2 &\leq \underbrace{\opnorm{\calM_1 \lb \lb \calI - p^{-1} \calP_{\Omega} \rb\lb\calX^t - \calT \rb \ttimes {\mX_i^t}{\mX_i^t}^\tran\rb} }_{= :\alpha_{2,1}} + \underbrace{\sum_{i = 1}^3\opnorm{\calM_1\lb \calG^t \times_i\mW_i^t \jneqi \mX_j^t \rb} }_{= :\alpha_{2,2}},
\end{align*}
where $\mW_i^t$ is given by
\begin{align*}
	\mW_i^t= \lb \mI - \mX_i^t{\mX_i^t}^\tran \rb \calM_i \lb  \lb \calI - p^{-1} \calP_{\Omega} \rb \lb \calX^t - \calT \rb \jneqi {\mX_j^t}^\tran\rb \calM_i^\dagger(\calG^t), \quad i =1,2,3.
\end{align*}
\begin{itemize}
	\item Controlling $\alpha_{2,1}$.  Notice that the tensor $ \lb \calI - p^{-1} \calP_{\Omega} \rb\lb\calX^t - \calT \rb \ttimes {\mX_i^t}{\mX_i^t}^\tran $ has multilinear rank at most $(r , r, r)$. Applying Lemma~\ref{lemma: Lemma 6 in XYZ} yields that
	\begin{align*}
		\alpha_{2,1} &\leq \sqrt{r} \opnorm{\lb \calI - p^{-1} \calP_{\Omega} \rb\lb\calX^t - \calT \rb \ttimes {\mX_i^t}{\mX_i^t}^\tran }\leq \sqrt{r} \opnorm{\lb \calI - p^{-1} \calP_{\Omega} \rb\lb\calX^t - \calT \rb }\\
		&\stackrel{(a)}{\leq}\opnorm{\lb \calI - p^{-1}\calP_{\Omega} \rb (\calJ)}\cdot 2r^{3/2}\infnorm{\calX^t - \calT},
	\end{align*}
	where  $(a)$ follows from Lemma~\ref{lemma: uniform tensor operator norm} and the fact that  $\calX^t-\calT$ has multilinear rank at most $(2r,2r,2r)$. Here $\calJ$ is the all-one tensor.
	\item Controlling $\alpha_{2,2}$. Straightforward computation gives that
	\begin{align*}
		\alpha_{2,2}&\leq  \opnorm{\calM_1(\calG^t)} \cdot\lb \sum_{i=1}^{3} \opnorm{\mW_i^t} \rb\\
		&\leq \opnorm{\calM_1(\calG^t)} \cdot\lb \sum_{i=1}^{3} \opnorm{ \calM_i\lb  \lb \calI - p^{-1} \calP_{\Omega} \rb \lb \calX^t - \calT \rb \jneqi {\mX_j^t}^\tran\rb } \cdot \opnorm{ \calM_i^\dagger(\calG^t)  } \rb\\
		&\stackrel{(a)}{\leq } \opnorm{\calM_1(\calG^t)} \cdot\lb \sum_{i=1}^{3} \sqrt{r}\opnorm{  \lb \calI - p^{-1} \calP_{\Omega} \rb \lb \calX^t - \calT \rb  } \cdot \opnorm{ \calM_i^\dagger(\calG^t)  } \rb\\
		&\leq 3\opnorm{\calM_1(\calG^t)} \cdot \sqrt{r}\opnorm{  \lb \calI - p^{-1} \calP_{\Omega} \rb \lb \calX^t - \calT \rb  } \cdot \max_{i = 1,2,3} \opnorm{ \calM_i^\dagger(\calG^t)  } \\
		&\leq 3\sqrt{r} \cdot \frac{ \max_{i = 1,2,3} \sigma_{\max}\lb \calM_i(\calG^t) \rb }{\min_{i = 1,2,3}\sigma_{\min}\lb \calM_i(\calG^t) \rb }\cdot  \opnorm{  \lb \calI - p^{-1} \calP_{\Omega} \rb \lb \calX^t - \calT \rb  } \\
		&\leq 6r^{3/2}\cdot \kappa\lb\calX^t \rb \cdot \opnorm{\lb \calI -p^{-1}\calP_{\Omega} \rb (\calJ)}\cdot  \infnorm{\calX^t - \calT}\\ 
		&\stackrel{(b)}{\leq } 6r^{3/2}\cdot 2\kappa \cdot \opnorm{\lb \calI - p^{-1}\calP_{\Omega} \rb (\calJ)}\cdot  \infnorm{\calX^t - \calT},
	\end{align*}
	where $(a)$ is due to Lemma~\ref{lemma: Lemma 6 in XYZ} and  $(b)$ follows from Lemma~\ref{lemma: condition number of X}.
\end{itemize} 

Combining the upper bounds of $\alpha_{2,1}$ and $\alpha_{2,2}$ together yields that, with high probability,
	\begin{align*}
		\alpha_2 &\leq \alpha_{2,1}+ \alpha_{2,2}  \leq 14r^{3/2} \kappa \cdot \opnorm{\lb \calI - p^{-1}\calP_{\Omega} \rb (\calJ)}\cdot  \infnorm{\calX^t - \calT}\\
		&\stackrel{(a)}{\leq } 14r^{3/2}\kappa \cdot C\left( \frac{\log^3n}{p} + \sqrt{\frac{n\log^5n}{p}} \right)\cdot \frac{36}{2^{20} \kappa^2 \mu^2 r^{4}} \frac{1}{2^t} \left( \frac{\mu r}{n} \right)^{3/2} \sigma_{\max}(\calT)\\
		&\leq \frac{1}{2^9}\frac{1}{2^{20} \kappa^6\mu^2 r^4 }\cdot \frac{1}{2^{t+1}}\cdot \sigma_{\max}(\calT),
	\end{align*}
	where $(a)$ follows from Lemma~\ref{lemma: spectral norm of E initial} and the assumption
		$p \geq \max \left\{ \frac{C_1\kappa^5\mu^{3/2} r^{3}\log^3 n}{n^{3/2}}, \frac{C_2\kappa^{10}\mu^3r^6\log^5n}{n^2} \right\}$.

Putting  the upper bounds of $\alpha_1$ and $\alpha_2$ together, one has 
\begin{align}
	\label{eq Et}
	\opnorm{\mE_1^{t} } \leq  \frac{1}{2^8}\frac{1}{2^{20} \kappa^6\mu^2 r^4 }\cdot \frac{1}{2^{t+1}}\cdot \sigma_{\max}(\calT)\leq   \frac{1}{2^{20} \kappa^6\mu^2 r^4 }\cdot \frac{1}{2^{t+1}}\cdot \sigma_{\max}(\calT)\leq \frac{1}{4}\sigma_{\max}(\calT).
\end{align} 
Using the same argument as above, one can obtain
\begin{align}
	\label{eq Etl}
	\opnorm{\mE_1^{t, \ell} } \leq  \frac{1}{2^8}\frac{1}{2^{20} \kappa^6\mu^2 r^4 }\cdot \frac{1}{2^{t+1}}\cdot \sigma_{\max}(\calT)\leq   \frac{1}{2^{20} \kappa^6\mu^2 r^4 }\cdot \frac{1}{2^{t+1}}\cdot \sigma_{\max}(\calT)\leq \frac{1}{4}\sigma_{\max}(\calT).
\end{align}

\subsubsection{Bounding \texorpdfstring{$\twoinf{\mX_i^{t+1, \ell} \mR_i^{t+1,\ell} - \mU_i} $}{TEXT}}
We only provide a detailed proof for  the case $i=1$, and the proofs for the other two cases are overall similar.
Since $\mX_1^{t+1, \ell} \bSigma_1^{t+1, \ell} {\mX_1^{t+1, \ell}}^\tran$ is the top-$r$ eigenvalue decomposition of  
\begin{align*}
	\lb \mT_1 + \mE_1^{t, \ell }\rb \lb \mT_1 + \mE_1^{t, \ell}\rb^\tran&=  \mT_1\mT_1 ^\tran + \underbrace{\mT_1 {\mE_1^{t, \ell}}^\tran + \mE_1^{t, \ell} \mT_1^\tran +\mE_1^{t, \ell} {\mE_1^{t, \ell}}^\tran}_{=: \bDelta_1^{t,\ell}},
\end{align*}
we have 
\begin{align*}
	\mX_1^{t+1, \ell} &= \lb\mT_1\mT_1^\tran + \bDelta_1^{t,\ell} \rb \mX_1^{t+1,\ell} {\bSigma_1^{t+1,\ell}}^{-1}.
\end{align*}
Recall that the  eigenvalue decomposition of $\mT_1\mT_1^\tran$ is $\mT_1\mT_1^\tran = \mU_1\bLambda_1\mU_1^\tran$, $\ve_m^\tran \lb \mX_1^{t+1,\ell} \mR_1^{t+1,\ell} - \mU_1 \rb$ can be decomposed as
\begin{align*}
    &\ve_m^\tran \lb \mX_1^{t+1,\ell} \mR_1^{t+1,\ell} - \mU_1 \rb \\
    &= \ve_{m}^\tran \lb\mT_1\mT_1^\tran\mX_1^{t+1,\ell}{\bSigma_{1}^{t+1,\ell}}^{-1}\mR_1^{t+1,\ell}-\mU_1\rb+\ve_m^\tran\bDelta_1^{t,\ell}\mX_1^{t+1,\ell}{\bSigma_1^{t+1,\ell}}^{-1}\mR_1^{t+1,\ell}\\
    &= \ve_{m}^\tran\mU_1\bLambda_1 \lb\mU_1^\tran\mX_1^{t+1,\ell}{\bSigma_{1}^{t+1,\ell}}^{-1}\mR_1^{t+1,\ell}-\bLambda_1^{-1}\rb+\ve_m^\tran\bDelta_1^{t,\ell}\mX_1^{t+1,\ell}{\bSigma_1^{t+1,\ell}}^{-1}\mR_1^{t+1,\ell}\\
    & = \ve_{m}^\tran\mU_1\bLambda_1 \lb\mU_1^\tran\mX_1^{t+1,\ell}\bLambda_1^{-1}\mR_1^{t+1,\ell}-\bLambda_1^{-1}\rb\\
    &\quad + \ve_m^\tran\mU_1\bLambda_1\mU_1^\tran\mX_1^{t+1,\ell}\lb{\bSigma_1^{t+1,\ell}}^{-1}-\bLambda_1^{-1}\rb\mR_1^{t+1,\ell} +\ve_m^\tran\bDelta_1^{t,\ell}\mX_1^{t+1,\ell}{\bSigma_1^{t+1,\ell}}^{-1}\mR_1^{t+1,\ell}\\
    & =\ve_{m}^\tran\mU_1 \lb\bLambda_1\mU_1^\tran\mX_1^{t+1,\ell}-{\mR_1^{t+1,\ell}}^\tran\bLambda_1^{-1}\rb\bLambda_1^{-1}\mR_1^{t+1,\ell}\\
    &\quad + \ve_m^\tran\mU_1\bLambda_1\mU_1^\tran\mX_1^{t+1,\ell}\lb{\bSigma_1^{t+1,\ell}}^{-1}-\bLambda_1^{-1}\rb\mR_1^{t+1,\ell}
    +\ve_m^\tran\bDelta_1^{t,\ell}\mX_1^{t+1,\ell}{\bSigma_1^{t+1,\ell}}^{-1}\mR_1^{t+1,\ell}.
\end{align*}
The triangle inequality then gives that
\begin{align}\label{twoinfxtell}
	\twonorm{\ve_m^\tran \lb \mX_1^{t+1,\ell} \mR_1^{t+1,\ell} - \mU_1 \rb} &\leq \underbrace{ \twonorm{\ve_m^\tran\mU_1\lb\bLambda_1   \mU_1^\tran \mX_1^{t+1,\ell} - {\mR_{1}^{t+1,\ell}}^\tran \bLambda_1 \rb\bLambda_1^{-1}\mR_1^{t+1,\ell}  } }_{=:\beta_1}\notag\\
	&\quad + \underbrace{\twonorm{ \ve_m^\tran\mU_1\bLambda_1  \mU_1^\tran \mX_1^{t+1,\ell} \lb {\bSigma^{t+1,\ell}_1 }^{-1}-\bLambda_1^{-1} \rb\mR_1^{t+1,\ell}  } }_{=:\beta_2}\notag\\
	&\quad +\underbrace{ \twonorm{ \ve_m^\tran \bDelta_1^{t,\ell}  \mX_1^{t+1,\ell} {\bSigma^{t+1,\ell} _1}^{-1}\mR_1^{t+1,\ell} }}_{=:\beta_3}.
\end{align}
\paragraph{Bounding $\beta_1$.}
Notice that, with high probability,
\begin{align}\label{bounding delta t ell}
	\opnorm{ \bDelta_1^{t,\ell} } &= \opnorm{  \mT_1 {\mE_1^{t, \ell}}^\tran + \mE_1^{t, \ell} \mT_1^\tran +\mE_1^{t, \ell} {\mE_1^{t, \ell}}^\tran } \leq  \left(2 \opnorm{\mT_1} + \opnorm{\mE_1^{t ,\ell}} \right) \opnorm{\mE_1^{t ,\ell}}  \notag\\
	&\stackrel{(a)}{\leq }\frac{9}{4}\sigma_{\max}(\calT)\cdot\frac{1}{2^8}\frac{1}{2^{20} \kappa^6\mu^2 r^4 }\cdot \frac{1}{2^{t+1}}\cdot \sigma_{\max}(\calT) \notag\\
	&= \frac{9}{2^{10}}\frac{1}{2^{20}\kappa^4 \mu^2 r^4}\frac{1}{2^{t+1}}\sigma_{\min}^2(\calT)\leq \frac{1}{2}\sigma^2_{\min}(\calT),
\end{align}
where step (a) follows from \eqref{eq Etl}. Given the SVD of  $\mH_1^{t+1,\ell} :={\mX_1^{t+1,\ell}}^\tran\mU_1$, denoted $\mH_1^{t+1,\ell} = \widehat{\mA}_1^{t+1,\ell}\widehat{\bSigma}_1^{t+1,\ell}\widehat{\mB}_1^{{t+1,\ell}^\tran}$, it can be easily shown that $\mR_1^{t+1,\ell} = \widehat{\mA}_1^{t+1,\ell}\widehat{\mB}_1^{{t+1,\ell}^\tran}$. Therefore, the application of Lemma~\ref{lemma: ding lemma 1} yields that
\begin{align}\label{to bound 5.17}
    &\opnorm{\lb\bLambda_1   \mU_1^\tran \mX_1^{t+1,\ell} - {\mR_{1}^{t+1,\ell}}^\tran \bLambda_1 \rb\bLambda_1^{-1}}\notag\\
	&\quad\leq \opnorm{  \bLambda_1   {\mR_{1}^{t+1,\ell}} - {\mH_1^{t+1,\ell}} \bLambda_1  } \opnorm{\bLambda_1^{-1}} \leq \left(  2 + \frac{\sigma_{\max}^2(\calT)}{\sigma_{\min}^2(\calT) - \opnorm{ \bDelta_1^{t,\ell} }}  \right)\opnorm{ \bDelta_1^{t,\ell} }\frac{1}{\sigma_{\min}\lb\bLambda_1\rb} \notag\\
	&\quad \leq \left(  2 + 2\kappa^2 \right)\cdot \frac{9}{2^{10}}\frac{1}{2^{20}\kappa^4 \mu^2 r^4}\frac{1}{2^{t+1}}\leq  \frac{36}{2^{10}}\frac{1}{2^{20}\kappa^2\mu^2 r^4}\frac{1}{2^{t+1}},
\end{align} 
which occurs with high probability. Here the third line is due to the fact $\sigma_{\min}(\bLambda_1) \geq \sigma_{\min}^2(\calT)$ and \eqref{bounding delta t ell}. Thus $\beta_1$ can be bounded with high probability as follows:
\begin{align}\label{eq upper bound of beta1}
    \beta_1\leq \twoinf{\mU_1}\opnorm{\lb\bLambda_1   \mU_1^\tran \mX_1^{t+1,\ell} - {\mR_{1}^{t+1,\ell}}^\tran \bLambda_1 \rb\bLambda_1^{-1}}\leq \frac{36}{2^{10}}\frac{1}{2^{20}\kappa^2\mu^2 r^4}\frac{1}{2^{t+1}} \sqrt{\frac{\mu r}{n}}.
\end{align}

\paragraph{Bounding $\beta_2$.}
Applying the Weyl's inequality yields that $\opnorm{\bLambda_1 - \bSigma^{t+1,\ell} _1}\leq \opnorm{\bDelta_1^{t,\ell}}$. It follows that
\begin{align*}
    \sigma_{r}\left(\bSigma_1^{t+1, \ell} \right) \geq \sigma_{\min}^2(\calT) - \opnorm{\bDelta_1^{t,\ell}}\geq \frac{1}{2}\sigma_{\min}^2(\calT),
\end{align*}
and
\begin{align*}
	\opnorm{ {\bSigma^{t+1,\ell} _1 }^{-1}-\bLambda_1^{-1} }	& = \opnorm{{\bSigma^{t+1,\ell} _1 }^{-1} \left(\bLambda_1 - \bSigma^{t+1,\ell} _1  \right)\bLambda_1^{-1}   } \leq \opnorm{{\bSigma^{t+1,\ell} _1 }^{-1} } \opnorm{\bLambda_1^{-1}  }  \opnorm{ \bLambda_1 - \bSigma^{t+1,\ell} _1 } \\
	&\leq \frac{1}{\sigma_{r}\lb \bSigma_1^{t+1,\ell} \rb}\cdot \frac{1}{\sigma_{\min}^2(\calT)} \cdot \opnorm{\bLambda_1 - \bSigma^{t+1,\ell} _1  } \leq  \frac{2}{\sigma_{\min}^4(\calT)} \cdot \opnorm{\bDelta_1^{t,\ell} }\\
	&\leq  \frac{2}{\sigma_{\min}^4(\calT)} \cdot\frac{9}{2^{10}}\frac{1}{2^{20}\kappa^4 \mu^2 r^4}\frac{1}{2^{t+1}}\sigma_{\min}^2(\calT)=\frac{2}{\sigma^2_{\min}\lb\calT\rb}\frac{9}{2^{10}}\frac{1}{2^{20}\kappa^4 \mu^2 r^4}\frac{1}{2^{t+1}}.
\end{align*}
Therefore, with high probability,
\begin{align}
	\label{eq upper bound of beta2}
	\beta_2 &\leq \twoinf{\mU_1}\cdot \opnorm{\bLambda_1}\cdot \opnorm{ {\bSigma^{t+1,\ell} _1 }^{-1}-\bLambda_1^{-1} } 
	\leq    \frac{18}{2^{10}}\frac{1}{2^{20}\kappa^2 \mu^2 r^4}\frac{1}{2^{t+1}} \sqrt{\frac{\mu r}{n}} .
\end{align}

\paragraph{Bounding $\beta_3$.}
The last term $\beta_3$ can be bounded as follows:
\begin{align}
	\label{eq beta3}
	\beta_3 
	&\leq \twonorm{\ve_m^\tran \bDelta_1^{t,\ell}  }\cdot \opnorm{ {\bSigma^{t+1,\ell} _1 }^{-1}  }= \twoinf{\ve_m^\tran \left( \mT_1 {\mE_1^{t, \ell}}^\tran + \mE_1^{t, \ell} \mT_1^\tran +\mE_1^{t, \ell} {\mE_1^{t, \ell}}^\tran \right)}\cdot \opnorm{ {\bSigma^{t+1,\ell} _1 }^{-1}  }\notag  \\
	&\leq \left( \twoinf{\mT_1}\cdot \opnorm{\mE_1^{t,\ell}} +\twonorm{\ve_m^\tran \mE_1^{t,\ell}}\cdot \left( \opnorm{\mT_1} +  \opnorm{\mE_1^{t,\ell}}\right) \right)\cdot \opnorm{ {\bSigma^{t+1,\ell} _1 }^{-1}  }\notag  \\
	&\leq \sqrt{\frac{\mu r}{n}} \cdot\frac{1}{2^8}\frac{1}{2^{20} \kappa^6\mu^2 r^4 }  \frac{1}{2^{t+1}}  \sigma_{\max}^2(\calT)  \cdot \frac{2}{\sigma_{\min}^2(\calT)} + \twonorm{\ve_m^\tran\mE_1^{t,\ell} }\cdot \frac{3\sigma_{\max}(\calT)}{ \sigma_{\min}^2(\calT) }\notag  \\
	&\leq \frac{1}{2^7}\frac{1}{2^{20} \kappa^2\mu^2 r^4 }  \frac{1}{2^{t+1}}  \sqrt{\frac{\mu r}{n}}  + \twonorm{\ve_m^\tran\mE_1^{t,\ell} }\cdot \frac{3\sigma_{\max}(\calT)}{ \sigma_{\min}^2(\calT) },
\end{align} 
where the third inequality is due to \eqref{eq Etl} and $$\left( \opnorm{\mT_1} +  \opnorm{\mE_1^{t,\ell}} \right)\cdot \opnorm{ {\bSigma^{t+1,\ell} _1 }^{-1}  }\leq \frac{2}{\sigma_{\min}^2(\calT)}\cdot \left( \sigma_{\max}(\calT) + \frac{1}{4}\sigma_{\max}(\calT)) \right)\leq \frac{3\sigma_{\max}(\calT)}{\sigma_{\min}^2(\calT)}.$$
To this end, we focus on bounding  $\twonorm{\ve_m^\tran\mE_1^{t,\ell} }$. Recall that  $\mE_1^{t,\ell}$ is defined by
\begin{align*}
	\mE_1^{t,\ell} = 	\calM_1 \left( \left( \calI -  \calP_{T_{\calX^{t,\ell}}} \left( p^{-1}  \calP_{\Omega_{-\ell}} + \calP_{\ell} \right) \right) \left( \calX^t - \calT \right)\right).
\end{align*}
Thus, invoking the triangle inequality yields that
\begin{align*}
	\twonorm{\ve_m^\tran\mE_1^{t,\ell} } &\leq \twonorm{\ve_m^\tran \calM_1\lb \lb \calI - \calP_{T_{\calX^{t,\ell}}}\rb(\calX^{t,\ell} - \calT) \rb} \\
 &\quad + \twonorm{\ve_m^\tran\calM_1 \lb  \calP_{T_{\calX^{t,\ell}}} \lb \calI -p^{-1}  \calP_{\Omega_{-\ell}} -\calP_{\ell}  \rb(\calX^{t,\ell} - \calT)   \rb } =:\beta_{3,a}+\beta_{3,b}.
\end{align*}
The following two claims provide the upper  bounds for $\beta_{3,a}$ and $\beta_{3,b}$. The proofs of the claims can be found in Section  \ref{proof: claim upper bound of beta3a} and Section  \ref{proof: claim upper bound of beta3b}, respectively. \cjc{
\begin{claim}
	\label{claim: upper bound of beta3a}
Assuming the inequalities in \eqref{eq induction hypotheis} hold, one can obtain
	\begin{align*}
		\beta_{3,a} \leq \frac{1}{2^{6}}\cdot  \frac{1}{2^{20} \kappa^4 \mu^2 r^4 }\frac{1}{2^{t+1}} \sqrt{\frac{\mu r}{n}}\sigma_{\max}(\calT).
	\end{align*}
\end{claim}
\begin{claim}
	\label{claim: upper bound of beta3b}
	Suppose $p \geq\frac{C_1\kappa^8\mu^{3.5}r^6\log^3n}{n^{3/2}}$. Assuming the inequalities in \eqref{eq induction hypotheis} hold, one has
	\begin{align*}
		\beta_{3,b} \leq \frac{1}{2^{6} }\frac{1}{2^{20}\kappa^4\mu^2 r^4} \frac{1}{2^{t+1}} \sqrt{\frac{\mu r}{n}} \sigma_{\max}(\calT).
	\end{align*}
\end{claim}}
Combining Claim \ref{claim: upper bound of beta3a} and Claim \ref{claim: upper bound of beta3b} together reveals that
\begin{align}
	\label{eq 2 infty norm of Etl}
	\twonorm{\ve_m^\tran\mE_1^{t,\ell} } \leq \frac{1}{2^{5} }\frac{1}{2^{20}\kappa^4\mu^2 r^4} \frac{1}{2^{t+1}} \sqrt{\frac{\mu r}{n}} \sigma_{\max}(\calT).
\end{align}
Furthermore, putting  \eqref{eq beta3} and \eqref{eq 2 infty norm of Etl} together yields
\begin{align}
	\label{eq upper bound of beta3}
	\beta_3 \leq \frac{1}{2^3}\frac{1}{2^{20} \kappa^2\mu^2 r^4 }  \frac{1}{2^{t+1}}  \sqrt{\frac{\mu r}{n}}.
\end{align} 

\paragraph{Combining $\beta_1, \beta_2$ and $\beta_3$.}
Plugging \eqref{eq upper bound of beta1}, \eqref{eq upper bound of beta2} and \eqref{eq upper bound of beta3} into \eqref{twoinfxtell} shows that with high probability,
\begin{align}\label{emmxt+1}
	\twonorm{\ve_m^\tran \lb \mX_1^{t+1,\ell} \mR_1^{t+1,\ell} - \mU_1 \rb}
	\leq  \frac{1}{4}\frac{1}{2^{20} \kappa^2\mu^2 r^4 }  \frac{1}{2^{t+1}}  \sqrt{\frac{\mu r}{n}}.
\end{align}
Taking the maximum of \eqref{emmxt+1} over $m$ gives that
\begin{align*}
	\twoinf{ \mX_1^{t+1,\ell} \mR_1^{t+1,\ell} - \mU_1 }  \leq  \frac{1}{4}\frac{1}{2^{20} \kappa^2\mu^2 r^4 }  \frac{1}{2^{t+1}}  \sqrt{\frac{\mu r}{n}}.
\end{align*}
The same bound can be obtained for the  other two modes. Thus, we  have
\begin{align}
	\label{eq: incoherence of loo at t+1}
	\twoinf{ \mX_i^{t+1,\ell} \mR_i^{t+1,\ell} - \mU_i }  \leq  \frac{1}{4}\frac{1}{2^{20} \kappa^2\mu^2 r^4 }  \frac{1}{2^{t+1}}  \sqrt{\frac{\mu r}{n}}, \quad i = 1,2,3.
\end{align}

\subsubsection{Bounding \texorpdfstring{$\fronorm{ \calE^{t} - \calE^{t,\ell}} $}{TEXT}}
By the definition of $\calE^t$ and $\calE^{t,\ell}$ in \eqref{residual tensor et} and \eqref{residual tensor etl}, 
\begin{align*}
    \calE^{t} - \calE^{t,\ell} &= \lb\calI-p^{-1}\calP_{T_{\calX^t}}\calP_{\Omega}\rb\lb\calX^t-\calT\rb  - \lb\calI-\calP_{T_{\calX^{t,\ell}}}\lb p^{-1}\calP_{\Omega_{-\ell}}+\calP_{\ell}\rb\rb\lb\calX^{t,\ell}-\calT\rb\\
    &= \left( \calI -\calP_{T_{\calX^t}} \right) \left( \calX^t - \calT\right) +   \calP_{T_{\calX^t}} \left(\calI - p^{-1}\calP_{\Omega}\right) \left( \calX^t - \calT\right) \\
    &\quad - \left( \calI -\calP_{T_{\calX^{t,\ell}}} \right) \left( \calX^{t,\ell} - \calT\right)+   \calP_{T_{\calX^{t,\ell}}}\left( \calI -  p^{-1}\calP_{\Omega_{-\ell}} - \calP_{\ell} \right)\left( \calX^{t,\ell} - \calT\right)\\
    &=  \left( \calP_{T_{\calX^t}} - \calP_{T_{\calX^{t,\ell}}} \right) \left(  \calT\right) + \left( \calP_{T_{\calX^t}}  - \calP_{T_{\calX^{t,\ell}}} \right) \left( \calI -  p^{-1}\calP_{\Omega} \right) \left( \calX^t - \calT\right) \\
    &\quad -    \calP_{T_{\calX^{t,\ell}}}\lb\calI-p^{-1}\calP_{\Omega}\rb\lb\calX^t-\calT\rb + \calP_{T_{\calX^{t,\ell}}}\lb\calI-p^{-1}\calP_{\Omega_{-\ell}}-\calP_{\ell}\rb\lb\calX^{t,\ell}-\calT\rb.
\end{align*}
Thus we can bound 
$\fronorm{ \calE^{t} - \calE^{t,\ell}}$ as follows:
\begin{align}
	\label{eq xi1-3}
	\fronorm{\calE^t-\calE^{t,\ell}} 
	&\leq \underbrace{\fronorm{\lb\calP_{T_{\calX^{t}}}-\calP_{T_{\calX^{t,\ell}}}\rb\lb\calT\rb}}_{=:\xi_1} +\underbrace{\fronorm{\lb\calP_{T_{\calX^{t}}}-\calP_{T_{\calX^{t,\ell}}}\rb\lb\calI-p^{-1}\calP_{\Omega}\rb\lb\calX^t-\calT\rb}}_{=:\xi_2} \notag \\
	&\quad + \underbrace{\fronorm{ \calP_{T_{\calX^{t,\ell}}}\lb\calI-p^{-1}\calP_{\Omega}\rb\lb\calX^t-\calT\rb - \calP_{T_{\calX^{t,\ell}}}\lb\calI-p^{-1}\calP_{\Omega_{-\ell}}-\calP_{\ell}\rb\lb\calX^{t,\ell}-\calT\rb}}_{=:\xi_3}. 
\end{align}
For the operator $\ProjTX - \ProjTXL$, one can invoke the definition \eqref{definition projector} to deduce that
\begin{align}
	\label{eq: difference of two tangent space}
	&\ProjTX - \ProjTXL\notag\\ 
	&~ = \left(\prod_{i=1}^{3}  \calP_{\mX_i^{t}}^{(i)} -\prod_{i=1}^{3}  \calP_{\mX_i^{t,\ell}}^{(i)}\right)+\left(\sum_{i=1}^{3}  \calP^{(i)}_{\mX_i^{t , \perp}} \calP^{(j\neq i)}_{\calG^{t},\lcb \mX_{j}^{t}\rcb_{j\neq i}}-\sum_{i=1}^{3}  \calP^{(i)}_{\mX_i^{t,\ell, \perp}}\calP^{(j\neq i)}_{\calG^{t,\ell},\lcb \mX_{j}^{t, \ell}\rcb_{j\neq i}}\right)\notag\\
 &~= \prod_{i=1}^3 \left(\calP_{\mX_i^{t}}^{(i)}  -  \calP_{\mX_i^{t,\ell}}^{(i)}   \right) + \sum_{i=1}^3\calP_{\mX_i^{t,\ell}}^{(i)}  \prod_{j\neq i } \left(\calP_{\mX_j^{t}}^{(j)}  -  \calP_{\mX_j^{t,\ell}}^{(j)}   \right) + \sum_{i=1}^{3} \left(\calP_{\mX_i^{t}}^{(i)}  -  \calP_{\mX_i^{t,\ell}}^{(i)}   \right) \prod_{j\neq i} \calP_{\mX_j^{t,\ell}}^{(j)}\notag\\
 &\quad+\sum_{i=1}^{3}  \calP^{(i)}_{\mX_i^{t , \perp}} \left( \calP^{(j\neq i)}_{\calG^{t},\lcb \mX_{j}^{t}\rcb_{j\neq i}}-  \calP^{(j\neq i)}_{\calG^{t,\ell},\lcb \mX_{j}^{t, \ell}\rcb_{j\neq i}} \right) + \sum_{i=1}^{3}  \left( \calP^{(i)}_{\mX_i^{t,\ell}}  - \calP^{(i)}_{\mX_i^{t }}   \right) \calP^{(j\neq i)}_{\calG^{t,\ell},\lcb \mX_{j}^{t, \ell}\rcb_{j\neq i}}\notag\\
 &~=\prod_{i=1}^3 \left(\calP_{\mX_i^{t}}^{(i)}  -  \calP_{\mX_i^{t,\ell}}^{(i)}   \right)  + 	\sum_{i=1}^3\calP_{\mX_i^{t,\ell}}^{(i)}  \prod_{j\neq i } \left(\calP_{\mX_j^{t}}^{(j)}  -  \calP_{\mX_j^{t,\ell}}^{(j)}   \right)\notag\\
	&\quad  + \sum_{i=1}^{3} \left(\calP_{\mX_i^{t}}^{(i)}  -  \calP_{\mX_i^{t,\ell}}^{(i)}   \right) \left( \prod_{j\neq i} \calP_{\mX_j^{t,\ell}}^{(j)} - \calP^{(j\neq i)}_{\calG^{t,\ell},\lcb \mX_{j}^{t, \ell}\rcb_{j\neq i}}  \right) \notag\\
	&\quad + \sum_{i=1}^{3}  \calP^{(i)}_{\mX_i^{t , \perp}} \left( \calP^{(j\neq i)}_{\calG^{t},\lcb \mX_{j}^{t}\rcb_{j\neq i}}-  \calP^{(j\neq i)}_{\calG^{t,\ell},\lcb \mX_{j}^{t, \ell}\rcb_{j\neq i}} \right).
\end{align}
\paragraph{Bounding  $\xi_1$.}
By \eqref{eq: difference of two tangent space},  this term can be bounded as follows:
\begin{align*}
	\xi_1&\leq \underbrace{\fronorm{ \prod_{i=1}^3 \left(\calP_{\mX_i^{t}}^{(i)}  -  \calP_{\mX_i^{t,\ell}}^{(i)}   \right)(\calT) } }_{=:\xi_{1,1}} + \underbrace{\sum_{i=1}^3 \fronorm{\calP_{\mX_i^{t,\ell}}^{(i)} \prod_{j\neq i}\left(\calP_{\mX_j^{t}}^{(j)}  -  \calP_{\mX_j^{t,\ell}}^{(j)}   \right) (\calT)   }}_{=:\xi_{1,2}}\\
	&\quad  + \underbrace{\sum_{i=1}^{3} \fronorm{ \left(\calP_{\mX_i^{t}}^{(i)}  -  \calP_{\mX_i^{t,\ell}}^{(i)}   \right) \left( \prod_{j\neq i} \calP_{\mX_j^{t,\ell}}^{(j)} - \calP^{(j\neq i)}_{\calG^{t,\ell},\lcb \mX_{j}^{t, \ell}\rcb_{j\neq i}}   \right)(\calT) } }_{=:\xi_{1,3}}\\
	&\quad + \underbrace{\sum_{i=1}^{3}\fronorm{  \calP^{(i)}_{\mX_i^{t , \perp}} \left( \calP^{(j\neq i)}_{\calG^{t},\lcb \mX_{j}^{t}\rcb_{j\neq i}}-  \calP^{(j\neq i)}_{\calG^{t,\ell},\lcb \mX_{j}^{t, \ell}\rcb_{j\neq i}} \right)(\calT)  }}_{=:\xi_{1,4}}.
\end{align*}
\begin{itemize}
	\item Bounding $\xi_{1,1},\xi_{1,2}$ and \cjc{$\xi_{1,3}$}. By the definition of $\calP_{\mX_i^{t}}^{(i)}$ and $\calP_{\mX_i^{t,\ell}}^{(i)} $, we have
	\begin{align*}
		\xi_{1,1}
		&\leq \prod_{i=1}^3 \fronorm{\mX_i^t{\mX_i^t}^\tran  -  \mX_i^t{\mX_i^t}^\tran } \cdot \sigma_{\max}(\calT)\\
		&\leq \left(\max_{i=1,2,3} \fronorm{ {\mX_i^{t}}{\mX_i^{t}}^\tran -   {\mX_i^{t,\ell}}{\mX_i^{t,\ell}}^\tran } \right)^3\cdot \sigma_{\max}(\calT),\\
		\xi_{1,2}
		&\leq \sum_{i=1}^3 \fronorm{ \calM_i \left(\prod_{j\neq i}\left(\calP_{\mX_j^{t}}^{(j)}  -  \calP_{\mX_j^{t,\ell}}^{(j)}   \right) (\calT)  \right) } \\
		&\leq \sum_{i=1}^3 \prod_{j\neq i} \fronorm{\mX_j^t{\mX_j^t}^\tran  -  \mX_j^t{\mX_j^t}^\tran } \cdot \sigma_{\max}(\calT)\\
		&\leq 3 \left(\max_{i=1,2,3} \fronorm{ {\mX_i^{t}}{\mX_i^{t}}^\tran -   {\mX_i^{t,\ell}}{\mX_i^{t,\ell}}^\tran}  \right)^2\cdot \sigma_{\max}(\calT),\\
		\xi_{1,3}&\stackrel{(a)}{=} \sum_{i=1}^{3} \fronorm{ \left(\calP_{\mX_i^{t}}^{(i)}  -  \calP_{\mX_i^{t,\ell}}^{(i)}   \right) \left( \prod_{j\neq i} \calP_{\mX_j^{t,\ell}}^{(j)} - \calP^{(j\neq i)}_{\calG^{t,\ell},\lcb \mX_{j}^{t, \ell}\rcb_{j\neq i}}   \right)\lb\calX^{t, \ell} - \calT\rb } \\
		&=\sum_{i=1}^{3} \fronorm{\calM_i \left( \left(\calP_{\mX_i^{t}}^{(i)}  -  \calP_{\mX_i^{t,\ell}}^{(i)}   \right) \left( \prod_{j\neq i} \calP_{\mX_j^{t,\ell}}^{(j)} - \calP^{(j\neq i)}_{\calG^{t,\ell},\lcb \mX_{j}^{t, \ell}\rcb_{j\neq i}}   \right)\lb\calX^{t, \ell} - \calT\rb \right) } \\
		&\leq \sum_{i=1}^{3} \fronorm{ {\mX_i^{t}}{\mX_i^{t}}^\tran -   {\mX_i^{t,\ell}}{\mX_i^{t,\ell}}^\tran   }\cdot \fronorm{   \left( \prod_{j\neq i} \calP_{\mX_j^{t,\ell}}^{(j)} - \calP^{(j\neq i)}_{\calG^{t,\ell},\lcb \mX_{j}^{t, \ell}\rcb_{j\neq i}}   \right)\lb\calX^{t, \ell} - \calT\rb  } \\
		&\stackrel{(b)}{\leq } \sum_{i=1}^{3} \fronorm{ {\mX_i^{t}}{\mX_i^{t}}^\tran -   {\mX_i^{t,\ell}}{\mX_i^{t,\ell}}^\tran   }\cdot \fronorm{   \calX^{t, \ell} - \calT   } \\
		&\leq 3 \max_{i=1,2,3} \fronorm{ {\mX_i^{t}}{\mX_i^{t}}^\tran -   {\mX_i^{t,\ell}}{\mX_i^{t,\ell}}^\tran }  \cdot \fronorm{\calX^{t,\ell} - \calT},
	\end{align*}
	where step (a) and step (b) follow from the Claim~\ref{property of these projection}.
	Combining these three terms together and using \eqref{infty norm at t step for Xtl}, one can see that
	\begin{align}\label{eq xi1 to 3}
		\sum_{i=1}^{3}\xi_{1,i} 
		&\leq \left( \frac{1}{2^{19}\kappa^2 \mu^2 r^4} \frac{1}{2^t}\sqrt{\frac{\mu r}{n}}\right)^3\sigma_{\max}(\calT) + 3\left( \frac{1}{2^{19}\kappa^2 \mu^2 r^4} \frac{1}{2^t}\sqrt{\frac{\mu r}{n}}\right)^2\sigma_{\max}(\calT) \notag \\
		&\quad  + 3\cdot \frac{1}{2^{19}\kappa^2 \mu^2 r^4} \frac{1}{2^t}\sqrt{\frac{\mu r}{n}}\cdot n^{3/2}\frac{36}{2^{20}\kappa^2 \mu^2 r^4} \frac{1}{2^t} \left(\frac{\mu r}{n}\right)^{3/2}\sigma_{\max}(\calT) \notag \\
		&\leq \left( \frac{1}{2^{19}} + \frac{3}{2^{18}} + \frac{3\cdot 36}{2^{19}}\right) \frac{1}{2^{20}\kappa^4 \mu^2 r^4} \frac{1}{2^{t+1}}\sqrt{\frac{\mu r}{n}}\sigma_{\max}(\calT) \notag \\
		&\leq \frac{1}{2^{10}}\frac{1}{2^{20}\kappa^4 \mu^2 r^4} \frac{1}{2^{t+1}}\sqrt{\frac{\mu r}{n}}\sigma_{\max}(\calT),
	\end{align}
	where we have used the following  bound
	\begin{align*}
	    \fronorm{ {\mX_i^{t}}{\mX_i^{t}}^\tran -   {\mX_i^{t,\ell}}{\mX_i^{t,\ell}}^\tran } &= \fronorm{ {\mX_i^{t}\mR_i^{t}} \left({\mX_i^{t}\mR_i^t}\right)^\tran -   {\mX_i^{t,\ell}\mT_i^{t,\ell}} \left({\mX_i^{t,\ell}\mT_i^{t,\ell}} \right)^\tran } \\
	    &\leq 2\fronorm{\mX_i^{t}\mR_i^{t} - \mX_i^{t,\ell}\mT_i^{t,\ell}}\leq 2\cdot \frac{1}{2^{20}\kappa^2 \mu^2 r^4}\frac{1}{2^t}\sqrt{\frac{\mu r}{n}}.
	\end{align*}
	\item For $\xi_{1,4}$, it follows that
	\begin{align*}
		\xi_{1,4}= \sum_{i=1}^{3} \fronorm{  \calP^{(i)}_{\mX_i^{t , \perp}} \left( \calP^{(j\neq i)}_{\calG^{t},\lcb \mX_{j}^{t}\rcb_{j\neq i}}-  \calP^{(j\neq i)}_{\calG^{t,\ell},\lcb \mX_{j}^{t, \ell}\rcb_{j\neq i}} \right)(\calT)  } =:\sum_{i=1}^{3} \xi_{1,4}^i.
	\end{align*}
	The proofs for different $i$ are overall similar, so we only provide  details for $\xi_{1,4}^1$. 
   The proof starts with the following bound
    \begin{align}
		\label{eq xi141 0}
		\xi_{1,4}^1& = \fronorm{\calM_1 \left( \calP^{(1)}_{\mX_1^{t , \perp}} \left( \calP^{(j\neq 1)}_{\calG^{t},\lcb \mX_{j}^{t}\rcb_{j\neq 1}}-  \calP^{(j\neq 1)}_{\calG^{t,\ell},\lcb \mX_{j}^{t, \ell}\rcb_{j\neq 1}} \right)(\calT)\right)}  \notag \\
		&= \fronorm{ \left( {\mU_1}{\mU_1}^\tran -  {\mX_1^{t}}{\mX_1^{t}}^\tran \right) \calM_1 \left( \left( \calP^{(j\neq 1)}_{\calG^{t},\lcb \mX_{j}^{t}\rcb_{j\neq 1}}-  \calP^{(j\neq 1)}_{\calG^{t,\ell},\lcb \mX_{j}^{t, \ell}\rcb_{j\neq 1}} \right)(\calT)  \right)  } \notag \\
		&\leq \opnorm{  {\mU_1}{\mU_1}^\tran -  {\mX_1^{t}}{\mX_1^{t}}^\tran  }\cdot \fronorm{  \calM_1 \left( \left( \calP^{(j\neq 1)}_{\calG^{t},\lcb \mX_{j}^{t}\rcb_{j\neq 1}}-  \calP^{(j\neq 1)}_{\calG^{t,\ell},\lcb \mX_{j}^{t, \ell}\rcb_{j\neq 1}} \right)(\calT)  \right)  } \notag \\
		&\leq 2 \opnorm{ \mX_1^t \mR_1^t - \mU_1 } \cdot \fronorm{  \calM_1 \left( \left( \calP^{(j\neq 1)}_{\calG^{t},\lcb \mX_{j}^{t}\rcb_{j\neq 1}}-  \calP^{(j\neq 1)}_{\calG^{t,\ell},\lcb \mX_{j}^{t, \ell}\rcb_{j\neq 1}} \right)(\calT)  \right)  }.
	\end{align}
	
	Applying Claim~\ref{property of these projection} yields that
	\begin{align}\label{yt-ytl}
		&\fronorm{  \calM_1 \left( \left( \calP^{(j\neq 1)}_{\calG^{t},\lcb \mX_{j}^{t}\rcb_{j\neq 1}}-  \calP^{(j\neq 1)}_{\calG^{t,\ell},\lcb \mX_{j}^{t,\ell}\rcb_{j\neq 1}} \right)(\calT)  \right)  }\notag\\ &~=\fronorm{\calM_1(\calT) \left( \mY_1^t{\mY_1^t}^\tran  -\mY_1^{t,\ell}{\mY_1^{t,\ell}}^\tran \right) } \leq \fronorm{ \mY_1^t{\mY_1^t}^\tran  -\mY_1^{t,\ell}{\mY_1^{t,\ell}}^\tran } \cdot \sigma_{\max}(\calT)\notag\\
		&~\stackrel{(a)}{\leq} \frac{1}{2^{13}}\cdot \frac{1}{2^t}\sqrt{\frac{\mu r}{n}}\sigma_{\max}(\calT),
	\end{align}
	where the columns of $\mY_1^t$ and $\mY_1^{t,\ell}$ are the top-$r$ right singular vectors of $\calM_1\lb\calX^t\rb$ and $\calM_1\lb\calX^{t,\ell}\rb$, respectively,   $(a)$ is due to the following claim, whose proof is presented in Section~\ref{subsec:claimY}.
	\begin{claim}
		\label{claim Y distance}
		Assuming the inequalities in \eqref{eq induction hypotheis} hold,
		one has
		\begin{align*}
			\fronorm{ \mY_i^t{\mY_i^t}^\tran  -\mY_i^{t,\ell}{\mY_i^{t,\ell}}^\tran }\leq  \frac{1}{2^{13}}\cdot \frac{1}{2^t}\sqrt{\frac{\mu r}{n}},\quad i = 1,2,3.
		\end{align*}
	\end{claim}
	Plugging \eqref{yt-ytl} into  \eqref{eq xi141 0} reveals that
	\begin{align*}
		\xi_{1,4}^1 
		&\leq 2\cdot \frac{1}{2^{17} \kappa^4\mu^2 r^4} \frac{1}{2^t}\cdot  \frac{1}{2^{13}}\cdot \frac{1}{2^t}\sqrt{\frac{\mu r}{n}}\sigma_{\max}(\calT)\leq \frac{1}{2^9}\cdot \frac{1}{2^{20}\kappa^4\mu^2 r^4 } \frac{1}{2^{t+1}} \sqrt{\frac{\mu r}{n}}\sigma_{\max}(\calT).
	\end{align*}
The upper bounds of  $\xi_{1,4}^2$ and $\xi_{1,4}^3$ follow immediately via nearly the same argument as above,
\begin{align*}
    \xi_{1,4}^2\leq \frac{1}{2^9}\cdot \frac{1}{2^{20}\kappa^4\mu^2 r^4 } \frac{1}{2^{t+1}} \sqrt{\frac{\mu r}{n}}\sigma_{\max}(\calT)~\mbox{and}~
    \xi_{1,4}^3\leq \frac{1}{2^9}\cdot \frac{1}{2^{20}\kappa^4\mu^2 r^4 } \frac{1}{2^{t+1}} \sqrt{\frac{\mu r}{n}}\sigma_{\max}(\calT).
\end{align*}
Thus, we have
	\begin{align}
		\label{eq xi 14}
		\xi_{1,4} \leq \frac{3}{2^9}\cdot \frac{1}{2^{20}\kappa^4\mu^2 r^4 } \frac{1}{2^{t+1}} \sqrt{\frac{\mu r}{n}}\sigma_{\max}(\calT).
	\end{align}
	 
\end{itemize}
Combining \eqref{eq xi1 to 3} and \eqref{eq xi 14} together shows that
\begin{align}
	\label{eq xi1}
	\xi_{1}&\leq \sum_{i=1}^{4}\xi_{1,i}\leq \frac{4}{2^{9}}\frac{1}{2^{20}\kappa^4\mu^2 r^4 } \frac{1}{2^{t+1}} \sqrt{\frac{\mu r}{n}}\sigma_{\max}(\calT).
\end{align}

\paragraph{\cjc{Bounding  $\xi_2$.}}
For notational convenience, define 
	\begin{align*}
	    \calZ^t = \left(\calI - p^{-1}\calP_{\Omega}\right)(\calX^t - \calT).
	\end{align*}
First, apply \eqref{eq: difference of two tangent space} to deduce that
\begin{align}
	\label{eq xi2a}
	\xi_{2}&\leq \underbrace{ \fronorm{\prod_{i=1}^3 \left(\calP_{\mX_i^{t}}^{(i)}  -  \calP_{\mX_i^{t,\ell}}^{(i)}   \right)\calZ^t } }_{=:\xi_{2,1}}  + \underbrace{\sum_{i=1}^3 \fronorm{  \calP_{\mX_i^{t,\ell}}^{(i)} \prod_{j\neq i} \left(\calP_{\mX_j^{t}}^{(j)}  -  \calP_{\mX_j^{t,\ell}}^{(j)}   \right) \calZ^t} }_{=:\xi_{2,2}} \notag \\
	&\quad  + \underbrace{\sum_{i=1}^{3} \fronorm{ \left(\calP_{\mX_i^{t}}^{(i)}  -  \calP_{\mX_i^{t,\ell}}^{(i)}   \right) \left( \prod_{j\neq i} \calP_{\mX_j^{t,\ell}}^{(j)} - \calP^{(j\neq i)}_{\calG^{t,\ell},\lcb \mX_{j}^{t,\ell}\rcb_{j\neq i}}   \right) \calZ^t  }}_{=:\xi_{2,3}} \notag \\
	&\quad  + \underbrace{ \sum_{i=1}^{3} \fronorm{ \calP^{(i)}_{\mX_i^{t , \perp}} \left( \calP^{(j\neq i)}_{\calG^{t},\lcb \mX_{j}^{t}\rcb_{j\neq i}} -  \calP^{(j\neq i)}_{\calG^{t,\ell},\lcb \mX_{j}^{t,\ell}\rcb_{j\neq i}}  \right) \calZ^t } }_{=:\xi_{2,4}}.
\end{align}
\begin{itemize}
	\item  For $\xi_{2,1}$, it is easily seen that
	\begin{align}
		\label{eq xi21}
		\xi_{2,1}
		&\leq \prod_{i=1}^3 \fronorm{  {\mX_i^{t}}{\mX_i^{t}}^\tran -   {\mX_i^{t,\ell}} {\mX_i^{t,\ell}}^\tran  } \opnorm{  \calM_1 \left( \calZ^t\right) } \stackrel{(a)}{\leq }\prod_{i=1}^3 \fronorm{ {\mX_i^{t}}{\mX_i^{t}}^\tran -   {\mX_i^{t,\ell}} {\mX_i^{t,\ell}}^\tran   }\sqrt{n}\opnorm{  \calZ^t } \notag \\
		&\leq 8 \left( \max_{i = 1,2,3} \fronorm{  \mX_i^{t}\mR_i^t-  \mX_i^{t,\ell} \mT_i^{t,\ell}  }\right)^3 \cdot\sqrt{n}\opnorm{  \calZ^t },
	\end{align}
where  $(a)$ follows from Lemma~\ref{lemma: Lemma 6 in XYZ}. 
	\item  For $\xi_{2,2}$, it can be bounded by the same argument as above,
	\begin{align}
		\label{eq xi22}
		\xi_{2,2} &= \sum_{i=1}^3 \fronorm{\calM_1\left( \calP_{\mX_i^{t,\ell}}^{(i)}  \prod_{j\neq i}\left(\calP_{\mX_j^{t}}^{(j)}  -  \calP_{\mX_j^{t,\ell}}^{(j)}   \right) \calZ^t \right) } \notag \\
		&\leq \sum_{i=1}^3 \prod_{j\neq i} \fronorm{ {\mX_i^{t}}{\mX_i^{t}}^\tran -   {\mX_i^{t,\ell}} {\mX_i^{t,\ell}}^\tran    }\cdot \opnorm{  \calM_1 \left( \calZ^t\right) } \notag \\
		&\leq 3\cdot 4\left( \max_{i = 1,2,3} \fronorm{  \mX_i^{t}\mR_i^t-  \mX_i^{t,\ell} \mT_i^{t,\ell}  }\right)^2 \cdot\sqrt{n}\opnorm{  \calZ^t },
	\end{align}

	\item For $\xi_{2,3}$, 
	{letting
	\begin{align*}
		\calZ_i = \left(\calP_{\mX_i^{t}}^{(i)}  -  \calP_{\mX_i^{t,\ell}}^{(i)}   \right) \left( \prod_{j\neq i} \calP_{\mX_j^{t,\ell}}^{(j)} - \calP^{(j\neq i)}_{\calG^{t,\ell},\lcb \mX_{j}^{t,\ell}\rcb_{j\neq i}}   \right) \calZ^t,
	\end{align*}	
	then we have $\xi_{2,3} = \sum_{i = 1}^3\fronorm{\calZ_i}$. 
 It is not hard to see that $\calZ_i$ is a tensor of multilinear rank at most $(2r,2r,2r)$ for $i = 1,2,3$.
 Thus,  the application  of Lemma~\ref{lemma:XY 19 lemma1} yields that
	\begin{align*}
		\xi_{2,3} &\leq 4r^{3/2} \sum_{i=1}^{3}\opnorm{\left( \calP_{\mX_i^t }^{(i)} -  \calP_{ \mX_i^{t,\ell} }^{(i)}\right)   \left( \prod_{j\neq i} \calP_{\mX_j^{t,\ell}}^{(j)} - \calP^{(j\neq i)}_{\calG^{t,\ell},\lcb \mX_{j}^{t,\ell}\rcb_{j\neq i}}   \right)(\calZ^t) } \\
		&\leq 4r^{3/2} \sum_{i=1}^{3}\opnorm{  \left( \prod_{j\neq i} \calP_{\mX_j^{t,\ell}}^{(j)} - \calP^{(j\neq i)}_{\calG^{t,\ell},\lcb \mX_{j}^{t,\ell}\rcb_{j\neq i}}   \right) (\calZ^t)   }\cdot \opnorm{  {\mX_i^t } {\mX_i^t }^\tran -  { \mX_i^{t,\ell} }{ \mX_i^{t,\ell} }^\tran }.
	\end{align*}
To this end, we focus on bounding 
\begin{align*}
\xi_{2,3}^i := \opnorm{  \left( \prod_{j\neq i} \calP_{\mX_j^{t,\ell}}^{(j)} - \calP^{(j\neq i)}_{\calG^{t,\ell},\lcb \mX_{j}^{t,\ell}\rcb_{j\neq i}}   \right) (\calZ^t)  }
\end{align*}
for $i=1,2,3$. The triangular inequality yields that
\begin{align*}
\xi_{2,3}^1 \leq \opnorm{  \left( \prod_{j\neq 1} \calP_{\mX_j^{t,\ell}}^{(j)} \right) (\calZ^t)  } + \opnorm{    \calP^{(j\neq 1)}_{\calG^{t,\ell},\lcb \mX_{j}^{t,\ell}\rcb_{j\neq 1}}  (\calZ^t)   }:= \xi_{2,3}^{1,a}+\xi_{2,3}^{1,b}.
\end{align*}
\begin{itemize}
    \item Bounding $\xi_{2,3}^{1,a}$. It can be bounded as follows:
    \begin{align}
    \label{eq: xi231a}
    \xi_{2,3}^{1,a} 
    &\leq \opnorm{ \calZ^t  } \cdot \prod_{j\neq 1}\opnorm{  \mX_{j}^{t,\ell}{\mX_{j}^{t,\ell}}^\tran } \leq \opnorm{  \calZ^t   }.
    \end{align}
    
    \item Bounding $\xi_{2,3}^{1,b}$. A simple computation yields that
    \begin{align}
    \label{eq: xi231b}
    \xi_{2,3}^{1,b} 
    &\leq \opnorm{   \calM_1 \left( \calP^{(j\neq 1)}_{\calG^{t,\ell},\lcb \mX_{j}^{t,\ell}\rcb_{j\neq 1}}  ( \calZ^t )   \right)} \notag \\
    &\stackrel{(a)}{=} \opnorm{ \calM_1 \left( \calZ^t   \right) \left( \mX_3^{t,\ell} \otimes \mX_2^{t,\ell}\right) \calM_1^\dagger \left(\calG^{t,\ell} \right) \calM_1\left(\calG^{t,\ell} \right) \left( \mX_3^{t,\ell} \otimes \mX_2^{t,\ell}\right)^\tran } \notag\\
    &\leq \opnorm{ \calM_1 \left( \calZ^t   \right) \left( \mX_3^{t,\ell} \otimes \mX_2^{t,\ell}\right) } \cdot \opnorm{\calM_1^\dagger \left(\calG^{t,\ell} \right) \calM_1\left(\calG^{t,\ell} \right)  }  \cdot \opnorm{  \mX_3^{t,\ell} \otimes \mX_2^{t,\ell}} \notag \\
    &\leq \opnorm{ \calM_1 \left(  \calZ^t    \times_2 {\mX_2^{t,\ell}}^\tran \times_3{\mX_3^{t,\ell}}^\tran \right)  }\notag \\
    &\stackrel{(b)}{\leq } \sqrt{r} \opnorm{ \calZ^t  \times_2 {\mX_2^{t,\ell}}^\tran \times_3{\mX_3^{t,\ell}}^\tran } \leq \sqrt{r} \opnorm{   \calZ^t    },
    \end{align}
    where  $(a)$ is due to the definition of $\calP^{(j\neq 1)}_{\calG^{t,\ell},\lcb \mX_{j}^{t,\ell}\rcb_{j\neq 1}}$ and  $(b)$ follows from Lemma~\ref{lemma: Lemma 6 in XYZ}.
\end{itemize}
Combining \eqref{eq: xi231a} and \eqref{eq: xi231b} together yields that 
\begin{align*}
\xi_{2,3}^1 \leq  \opnorm{   \calZ^t    } + \sqrt{r} \opnorm{   \calZ^t   }\leq 2\sqrt{r} \opnorm{   \calZ^t   }.
\end{align*}
Applying the same argument as above can show that
\begin{align*}
\xi_{2,3}^i \leq 2\sqrt{r} \opnorm{   \calZ^t   }, \quad i = 2,3.
\end{align*}
Thus one has 
\begin{align}
\label{eq xi23}
\xi_{2,3} &\leq 4r^{3/2}\cdot 3\cdot 2\sqrt{r} \opnorm{   \calZ^t  } \cdot \max_{i=1,2,3}\opnorm{{\mX_i^t } {\mX_i^t }^\tran -  { \mX_i^{t,\ell} }{ \mX_i^{t,\ell} }^\tran} \notag \\
&= 24r^2\opnorm{   \calZ^t   } \cdot \max_{i=1,2,3}\fronorm{{\mX_i^t } {\mX_i^t }^\tran -  { \mX_i^{t,\ell} }{ \mX_i^{t,\ell} }^\tran}.
\end{align}

	}
	
	\item  For  $\xi_{2,4}$, it is easy to see that
	\begin{align}
		\label{eq xi24}
		\xi_{2,4}&= \sum_{i=1}^{3} \fronorm{ \calP^{(i)}_{\mX_i^{t , \perp}} \left( \calP^{(j\neq i)}_{\calG^{t},\lcb \mX_{j}^{t}\rcb_{j\neq i}} -  \calP^{(j\neq i)}_{\calG^{t,\ell},\lcb \mX_{j}^{t,\ell}\rcb_{j\neq i}}  \right)  \lb\calZ^t\rb  }  \notag \\
		&\leq \sum_{i=1}^{3}  \fronorm{ \calM_i\left( \left( \calP^{(j\neq i)}_{\calG^{t},\lcb \mX_{j}^{t}\rcb_{j\neq i}} -  \calP^{(j\neq i)}_{\calG^{t,\ell},\lcb \mX_{j}^{t,\ell}\rcb_{j\neq i}}  \right) \lb \calZ^t \rb \right)  } .
	\end{align}
The first term in \eqref{eq xi24} can be bounded as follows:
	\begin{align*}
		&\fronorm{ \calM_1\left( \left( \calP^{(j\neq 1)}_{\calG^{t},\lcb \mX_{j}^{t}\rcb_{j\neq 1}} -  \calP^{(j\neq 1)}_{\calG^{t,\ell},\lcb \mX_{j}^{t,\ell}\rcb_{j\neq 1}}  \right) (\calZ^t) \right)  }\\
		&\quad\stackrel{(a)}{\leq} \fronorm{\calM_1(\calZ^t) \left(  \left( \mX_3^t{\mX_3^t}^\tran \otimes \mX_2^t{\mX_2^t}^\tran  \right) \mY_1^t{\mY_1^t}^\tran -  \left(  {\mX_{3}^{t,\ell}} {\mX_{3}^{t,\ell}}^\tran \otimes  {\mX_{2}^{t,\ell}}{\mX_{2}^{t,\ell}}^\tran  \right) \mY_1^{t,\ell}{\mY_1^{t,\ell}}^\tran \right) }\\
		&\quad \leq \fronorm{\calM_1(\calZ^t) \left(  \mX_3^t{\mX_3^t}^\tran \otimes \mX_2^t{\mX_2^t}^\tran  -  {\mX_{3}^{t,\ell}}{\mX_{3}^{t,\ell}}^\tran \otimes  {\mX_{2}^{t,\ell}}{\mX_{2}^{t,\ell}}^\tran  \right)  \mY_1^t{\mY_1^t}^\tran  }\\
		&\quad\quad + \fronorm{\calM_1(\calZ^t)  \left(  {\mX_{3}^{t,\ell}}{\mX_{3}^{t,\ell}}^\tran \otimes  {\mX_{2}^{t,\ell}}{\mX_{2}^{t,\ell}}^\tran \right) \left( \mY_1^t{\mY_1^t}^\tran  -\mY_1^{t,\ell}{\mY_1^{t,\ell}}^\tran \right)   } =:\xi_{2,4}^a+\xi_{2,4}^{b},
	\end{align*}
	where $(a)$ is due to \eqref{one more projection no change}.
	\begin{itemize}
		\item Controlling $\xi_{2,4}^{a}$. Notice that the matrix 
		\begin{align*}
			\calM_1(\calZ^t) \left(  {\mX_{3}^{t}}{\mX_{3}^{t}}^\tran \otimes {\mX_{2}^{t}}{\mX_{2}^{t}}^\tran  - {\mX_{3}^{t,\ell}}{\mX_{3}^{t,\ell}}^\tran  \otimes  {\mX_{2}^{t,\ell}}{\mX_{2}^{t,\ell}}^\tran  \right)  \mY_1^t{\mY_1^t}^\tran  \in\R^{n\times n^2}
		\end{align*}
		is of rank at most $r$. Letting $$\mD_i^t = {\mX_{i}^{t}}{\mX_{i}^{t}}^\tran  -  {\mX_{i}^{t,\ell}} {\mX_{i}^{t,\ell}}^\tran,$$ then $\xi_{2,4}^{a}$ can be bounded as follows:
		\begin{align*}
			\xi_{2,4}^{a}&\leq \sqrt{r}\opnorm{\calM_1(\calZ^t) \left(   {\mX_{3}^{t}}{\mX_{3}^{t}}^\tran \otimes {\mX_{2}^{t}}{\mX_{2}^{t}}^\tran - {\mX_{3}^{t,\ell}}{\mX_{3}^{t,\ell}}^\tran \otimes  {\mX_{2}^{t,\ell}}{\mX_{2}^{t,\ell}}^\tran  \right)  }\\
			&\leq \sqrt{r}\opnorm{\calM_1(\calZ^t) \left(  \mD_3^t\otimes  {\mX_{2}^{t}} {\mX_{2}^{t}}^\tran \right)    }  + \sqrt{r} \opnorm{\calM_1(\calZ^t) \left(  {\mX_{3}^{t,\ell}} {\mX_{3}^{t,\ell}}^\tran \otimes \mD_2^t \right)   } \\
			&= \sqrt{r}\opnorm{\calM_1\left( \calZ^t \times_{2}   {\mX_{2}^{t}}{\mX_{2}^{t}}^\tran \times_{3} \mD_3^t  \right) } + \sqrt{r}\opnorm{\calM_1 \left(\calZ^t \times_{2} \mD_2^t \times_{3}   {\mX_{3}^{t,\ell}}{\mX_{3}^{t,\ell}}^\tran  \right)}\\
			&\stackrel{(a)}{\leq} r \opnorm{  \calZ^t \times_{2}  {\mX_{2}^{t}}{\mX_{2}^{t}}^\tran \times_{3} \mD_3^t   } + r\opnorm{\calZ^t \times_{2} \mD_2^t \times_{3}  {\mX_{3}^{t,\ell}}{\mX_{3}^{t,\ell}}^\tran   }\\
			&\leq 2r \max_{i=2,3} \fronorm{\mD_i^t } \cdot \opnorm{ \calZ^t } \leq 4r \max_{i=2,3} \fronorm{  \mX_i^{t}\mR_i^t-  \mX_i^{t,\ell} \mT_i^{t,\ell}  } \cdot \opnorm{ \calZ^t },
		\end{align*}
		where  $(a)$ follows from Lemma~\ref{lemma: Lemma 6 in XYZ}. 
		\item Controlling $\xi_{2,4}^{b}$. A simple calculation yields that
		\begin{align*}
			\xi_{2,4}^{b} &\leq 2\sqrt{r}\opnorm{\calM_1(\calZ^t) \left( \lb{\mX_{3}^{t,\ell}}{\mX_3^{t,\ell}}^\tran\rb \otimes \lb {\mX_{2}^{t,\ell}}{\mX_2^{t,\ell}}^\tran\rb \right) \left( \mY_1^t{\mY_1^t}^\tran  -\mY_1^{t,\ell}{\mY_1^{t,\ell}}^\tran \right)   }\\
			&\leq 2\sqrt{r}\opnorm{ \calM_1(\calZ^t) \left( \lb{\mX_{3}^{t,\ell}}{\mX_3^{t,\ell}}^\tran\rb \otimes \lb {\mX_{2}^{t,\ell}}{\mX_2^{t,\ell}}^\tran\rb \right) }\cdot \opnorm{ \mY_1^t{\mY_1^t}^\tran  -\mY_1^{t,\ell}{\mY_1^{t,\ell}}^\tran}\\
			&\leq 2r\opnorm{  \calZ^t \times_2 {\mX_{2}^{t,\ell}}{\mX_2^{t,\ell}}^\tran \times_3 {\mX_{3}^{t,\ell}}{\mX_3^{t,\ell}}^\tran }\cdot \fronorm{ \mY_1^t{\mY_1^t}^\tran  -\mY_1^{t,\ell}{\mY_1^{t,\ell}}^\tran}\\
			&\leq 2r \fronorm{\mY_1^t{\mY_1^t}^\tran  -\mY_1^{t,\ell}{\mY_1^{t,\ell}}^\tran  } \cdot \opnorm{\calZ^t },
		\end{align*}
		where the fourth line is due to Lemma~\ref{lemma: Lemma 6 in XYZ}.
	\end{itemize}
	Combining $\xi_{2,4}^{a}$ and $\xi_{2,4}^{b}$ together gives that
 \begin{small}
	\begin{align}
		\label{eq: upper bound of xi241}
		&\fronorm{ \calM_1\left( \left( \calP^{(j\neq 1)}_{\calG^{t},\lcb \mX_{j}^{t}\rcb_{j\neq 1}} -  \calP^{(j\neq 1)}_{\calG^{t,\ell},\lcb \mX_{j}^{t,\ell}\rcb_{j\neq 1}}  \right) (\calZ^t) \right)  }\notag\\&\quad\leq 4r \max_{i=1,2,3} \fronorm{  \mX_i^{t}\mR_i^t-  \mX_i^{t,\ell} \mT_i^{t,\ell}  } \cdot \opnorm{ \calZ^t }   + 2r  \max_{i = 1,2,3}\fronorm{\mY_i^t{\mY_i^t}^\tran  -\mY_i^{t,\ell}{\mY_i^{t,\ell}}^\tran  } \cdot \opnorm{ \calZ^t } .
	\end{align}
 \end{small}
	Applying the same argument as above shows that  for $i = 2,3$,
 \begin{small}
	\begin{align}
		\label{eq: upper bound of xi242}
		&\fronorm{ \calM_i\left( \left( \calP^{(j\neq i)}_{\calG^{t},\lcb \mX_{j}^{t}\rcb_{j\neq i}} -  \calP^{(j\neq i)}_{\calG^{t,\ell},\lcb \mX_{j}^{t,\ell}\rcb_{j\neq i}}  \right) (\calZ^t) \right)  }\notag\\
		&\quad\leq 4r \max_{i=1,2,3} \fronorm{  \mX_i^{t}\mR_i^t-  \mX_i^{t,\ell} \mT_i^{t,\ell}  } \cdot \opnorm{ \calZ^t }    + 2r  \max_{i = 1,2,3}\fronorm{\mY_i^t{\mY_i^t}^\tran  -\mY_i^{t,\ell}{\mY_i^{t,\ell}}^\tran  } \cdot \opnorm{ \calZ^t }.
	\end{align}
 \end{small}
 	Plugging \eqref{eq: upper bound of xi241} and \eqref{eq: upper bound of xi242} into \eqref{eq xi24}, we can obtain
  \begin{small}
 	\begin{align}
 		\label{eq xi24a}
 		\xi_{2,4} &\leq 12r \max_{i=1,2,3} \fronorm{  \mX_i^{t}\mR_i^t-  \mX_i^{t,\ell} \mT_i^{t,\ell}  } \cdot \opnorm{ \calZ^t }   + 6r  \max_{i = 1,2,3}\fronorm{\mY_i^t{\mY_i^t}^\tran  -\mY_i^{t,\ell}{\mY_i^{t,\ell}}^\tran  } \cdot \opnorm{ \calZ^t } .
 	\end{align}
  \end{small}
 \end{itemize}

Putting \eqref{eq xi21}, \eqref{eq xi22}, \eqref{eq xi23} and \eqref{eq xi24a} together,  one has
	\begin{align*}
		\xi_2 &\leq \lb 8 \left( \max_{i = 1,2,3} \fronorm{  \mX_i^{t}\mR_i^t-  \mX_i^{t,\ell} \mT_i^{t,\ell}  }\right)^3   + 12\left( \max_{i = 1,2,3} \fronorm{  \mX_i^{t}\mR_i^t-  \mX_i^{t,\ell} \mT_i^{t,\ell}  }\right)^2 \rb\cdot\sqrt{n}\opnorm{   \calZ^t  }\\
		&\quad + \lb24r^2\left( \max_{i=1,2,3}\fronorm{{\mX_i^t } {\mX_i^t }^\tran -  { \mX_i^{t,\ell} }{ \mX_i^{t,\ell} }^\tran}\right)   + 12r \max_{i=1,2,3} \fronorm{  \mX_i^{t}\mR_i^t-  \mX_i^{t,\ell} \mT_i^{t,\ell}  }\rb \cdot \opnorm{  \calZ^t }  \notag \\
		&\quad  + 6r  \max_{i = 1,2,3}\fronorm{\mY_i^t{\mY_i^t}^\tran  -\mY_i^{t,\ell}{\mY_i^{t,\ell}}^\tran  } \cdot \opnorm{  \calZ^t  }.
	\end{align*}
	By the induction hypothesis \eqref{F norm}, it can be seen that
	\begin{align*}
		 8\sqrt{n} \left( \max_{i = 1,2,3} \fronorm{  \mX_i^{t}\mR_i^t-  \mX_i^{t,\ell} \mT_i^{t,\ell}  }\right)^3 &\leq 8\sqrt{n} \left( \frac{1}{2^{20} \kappa^2 \mu^2 r^4 } \frac{1}{2^t} \sqrt{\frac{\mu r}{n}} \right)^3  \leq \frac{1}{5\cdot 72} \sqrt{\frac{\mu r}{n}},\\
		 12\sqrt{n}\left( \max_{i = 1,2,3} \fronorm{  \mX_i^{t}\mR_i^t-  \mX_i^{t,\ell} \mT_i^{t,\ell}  }\right)^2  &\leq 12\sqrt{n} \left( \frac{1}{2^{20} \kappa^2 \mu^2 r^4 } \frac{1}{2^t} \sqrt{\frac{\mu r}{n}} \right)^2\leq\frac{1}{5\cdot 72} \sqrt{\frac{\mu r}{n}},\\
		 24r^2 \left(\max_{i=1,2,3}\fronorm{{\mX_i^t } {\mX_i^t }^\tran -  { \mX_i^{t,\ell} }{ \mX_i^{t,\ell} }^\tran} \right) &\leq 24r^2  \frac{1}{2^{20} \kappa^2 \mu^2 r^4 } \frac{1}{2^t} \sqrt{\frac{\mu r}{n}} \leq \frac{1}{5\cdot 72}\sqrt{\frac{\mu r}{n}},\\
		 12r \max_{i=1,2,3} \fronorm{  \mX_i^{t}\mR_i^t-  \mX_i^{t,\ell} \mT_i^{t,\ell}  } &\leq 12r   \frac{1}{2^{20} \kappa^2 \mu^2 r^4 } \frac{1}{2^t} \sqrt{\frac{\mu r}{n}} \leq\frac{1}{5\cdot 72} \sqrt{\frac{\mu r}{n}}.
	\end{align*}
Moreover, Claim \ref{claim Y distance} implies that
\begin{align*}
	6r\max_{i = 1,2,3}\fronorm{\mY_i^t{\mY_i^t}^\tran  -\mY_i^{t,\ell}{\mY_i^{t,\ell}}^\tran  }  &\leq  6r\cdot \frac{1}{2^{13}} \frac{1}{2^t}\sqrt{\frac{\mu r}{n}} \leq\frac{r}{5\cdot 72} \sqrt{\frac{\mu r}{n}}.
\end{align*}
Using these bounds, we can obtain
\begin{align}
	\label{eq xi2}
	\xi_2 &\leq 5\cdot \frac{r}{5\cdot 72} \sqrt{\frac{\mu r}{n}}  \opnorm{ \left(\calI - p^{-1}\calP_{\Omega}\right)(\calX^t - \calT) } \notag \\
	& \leq 5\cdot \frac{r}{5\cdot 72} \sqrt{\frac{\mu r}{n}} \cdot 2r\cdot \opnorm{\lb\calI-p^{-1}\calP_{\Omega}\rb\lb\calJ\rb}\infnorm{\calX^t-\calT}\\
	&\leq 5\cdot \frac{r}{5\cdot 72}\sqrt{\frac{\mu r}{n}}\cdot C\left( \frac{\log^3 n}{p} + \sqrt{\frac{n\log^5 n}{p}}\right) \cdot 2r \cdot\frac{36}{2^{20} \kappa^2 \mu^2 r^{4}}\cdot \frac{1}{2^t} \cdot \left( \frac{\mu r}{n} \right)^{3/2}\cdot \sigma_{\max}(\calT) \notag  \\
	&\leq \frac{1}{2^8}\frac{1}{2^{20}\kappa^4\mu^2 r^4 } \frac{1}{2^{t+1}} \sqrt{\frac{\mu r}{n}}\sigma_{\max}(\calT),
\end{align}
	provided that
		$p\geq \max\left\{ \frac{C_1\kappa^2 \mu^{1.5} r^{3.5}\log^3n}{n^{1.5}},  \frac{C_2\kappa^4 \mu^3 r^7 \log^5n}{n^2} \right\}$.

\paragraph{Bounding  $\xi_3$.}
For notational convenience, define
\begin{align*}
	\calZ_a = \left( \calI - p^{-1} \calP_{\Omega}  \right)  \left( \calX^t-\calX^{t,\ell}\right) \quad \mbox{and}\quad \calZ_b  = \left( p^{-1}\calP_{\Omega_{-\ell}} + \calP_{\ell} - p^{-1} \calP_{\Omega}  \right) \left(\calX^{t,\ell}-\calT\right).
\end{align*}
The triangle inequality gives that
\begin{align*}
	\xi_3 &=\fronorm{ \calP_{T_{\calX^{t,\ell}}}\lb\calI-p^{-1}\calP_{\Omega}\rb\lb\calX^t-\calT\rb - \calP_{T_{\calX^{t,\ell}}}\lb\calI-p^{-1}\calP_{\Omega_{-\ell}}-\calP_{\ell}\rb\lb\calX^{t,\ell}-\calT\rb}\\
	&= \fronorm{ \Pxtl \left(  \left( \calI - p^{-1} \calP_{\Omega}  \right)  \left( \calX^t-\calX^{t,\ell}\right)  +\left( p^{-1}\calP_{\Omega_{-\ell}} + \calP_{\ell} - p^{-1} \calP_{\Omega}  \right) \left(\calX^{t,\ell}-\calT\right) \right) }   \\
	&\leq \fronorm{ \Pxtl  \left(\calZ_a\right) } + \fronorm{\Pxtl \left(\calZ_b\right)  }.
\end{align*}
By the definition of $\Pxtl$, the first term on the right hand side can be bounded as follows:
\begin{align}
	\label{eq upper bound of Za}	
	\fronorm{ \Pxtl  \left(\calZ_a\right) } &\leq \fronorm{\prod_{i=1}^{3}\calP_{\mX_i^{t,\ell} }^{(i)}(\calZ_a) } + \sum_{i=1}^{3} \fronorm{\calP_{\mX_i^{t,\ell,\perp}}^{(i)} \calP^{(j\neq i)}_{\calG^{t,\ell},\lcb \mX_{j}^{t,\ell}\rcb_{j\neq i}} \left(\calZ_a\right)  } \notag \\
	&\leq \fronorm{ \calZ_a\underset{j\neq 1}{ \times } {\mX_j^{t,\ell}}^\tran  }  + \sum_{i=1}^{3} \fronorm{ \calP^{(j\neq i)}_{\calG^{t,\ell},\lcb \mX_{j}^{t,\ell}\rcb_{j\neq i}} (\calZ_a) }.
\end{align}	
Moreover, 
\begin{align}
	\label{eq upper bound of Za 1}
	\fronorm{\calP^{(j\neq 1)}_{\calG^{t,\ell},\lcb \mX_{j}^{t,\ell}\rcb_{j\neq 1}}  (\calZ_a)  } &= \fronorm{ \calM_1\left(\calP^{(j\neq 1)}_{\calG^{t,\ell},\lcb \mX_{j}^{t,\ell}\rcb_{j\neq 1}} (\calZ_a)  \right)  } \notag \\
	&= \fronorm{ \calM_1(\calZ_a)\left(\mX_3^{t,\ell} \otimes \mX_2^{t,\ell}\right) \calM_1^\dagger \left(\calG^{t,\ell}\right)\calM_1\left(\calG^{t,\ell}\right) \left(\mX_3^{t,\ell} \otimes \mX_2^{t,\ell}\right)^\tran } \notag \\
	&\leq \fronorm{\calM_1(\calZ_a)\left(\mX_3^{t,\ell} \otimes \mX_2^{t,\ell}\right) } = \fronorm{\calZ_a\underset{j\neq 1}{ \times } {\mX_j^{t,\ell} }^\tran }.
\end{align}
Similarly, one has
\begin{align}
	\label{eq upper bound of Za 2}
	\fronorm{\calP^{(j\neq i)}_{\calG^{t,\ell},\lcb \mX_{j}^{t,\ell}\rcb_{j\neq i}} (\calZ_a)  } \leq \fronorm{\calZ_a\jneqi {\mX_j^{t,\ell} }^\tran },\quad i=2,3.
\end{align}
Plugging \eqref{eq upper bound of Za 1} and \eqref{eq upper bound of Za 2} into \eqref{eq upper bound of Za} yields that
\begin{align}
	\label{eq Za}
	\fronorm{ \Pxtl  \left(\calZ_a\right) }  &\leq 4\max_{i=1,2,3 }\fronorm{    \calZ_a\jneqi {\mX_{j}^{t,\ell}}^\tran    } .
\end{align}
Using the same argument, one can obtain
\begin{align}
	\label{eq Zb}
	\fronorm{ \Pxtl  \left(\calZ_b\right) } &\leq 4\max_{i=1,2,3}\fronorm{    \calZ_b\jneqi {\mX_{j}^{t,\ell}}^\tran   }.
\end{align}
Combining \eqref{eq Za} and \eqref{eq Zb} together shows that 
\begin{align*} 
	\xi_3&\leq 4\max_{i=1,2,3 }\fronorm{   \calZ_a\jneqi  {\mX_{j}^{t,\ell}}^\tran    }   + 4\max_{i=1,2,3}\fronorm{    \calZ_b\jneqi {\mX_{j}^{t,\ell}}^\tran   }.
\end{align*}
To proceed, we first present the upper bounds of these terms in the following claim, whose proofs are deferred to Section~\ref{first term in 5.5} and Section~\ref{second term in 5.5}, respectively.
\begin{claim}
	\label{claim in R_c}
	Suppose $p\geq \max\left\{\frac{C_1\kappa^2\mu^{1.5}r^{2.5}\log^3n}{n^{1.5}} , \frac{C_2\kappa^4\mu^{3}r^{5}\log^5n}{n^{2}}\right\}$. For any $i = 1,2,3$, the following inequalities hold with high probability
	\begin{align}
		\label{eq: claim rc 1}
		\fronorm{   \calZ_a\jneqi  {\mX_{j}^{t,\ell}}^\tran    }  &\leq \frac{1}{2^{11}}\frac{1}{2^{20}\kappa^4\mu^2 r^4 } \frac{1}{2^{t+1}} \sqrt{\frac{\mu r}{n}}\sigma_{\max}(\calT),\\
		\label{eq: claim rc 2}
		\fronorm{   \calZ_b\jneqi {\mX_{j}^{t,\ell}}^\tran   }&\leq \frac{1}{2^{11}}\frac{1}{2^{20}\kappa^4\mu^2 r^4 } \frac{1}{2^{t+1}} \sqrt{\frac{\mu r}{n}}\sigma_{\max}(\calT).
	\end{align}
\end{claim}

It follows immediately from the claim that
\begin{align}
	\label{eq xi3}
	\xi_3 &\leq  \frac{1}{2^8}\frac{1}{2^{20}\kappa^4\mu^2 r^4 } \frac{1}{2^{t+1}} \sqrt{\frac{\mu r}{n}}\sigma_{\max}(\calT).
\end{align}

\paragraph{Combining $\xi_1, \xi_2$ and $\xi_3$ together.}
Plugging \eqref{eq xi1}, \eqref{eq xi2} and \eqref{eq xi3} into \eqref{eq xi1-3} yields that
\begin{small}
\begin{align}
	\label{eq upper bound of difference E at t+1}
	\fronorm{\calE^t - \calE^{t,\ell}} \leq \frac{1}{2^6}  \frac{1}{2^{20}\kappa^4 \mu^2 r^4}\frac{1}{2^{t+1}} \sqrt{\frac{\mu r}{n}}\sigma_{\max}(\calT).
\end{align}
\end{small}

\subsubsection{Bounding  \texorpdfstring{$\fronorm{\mX_i^{t+1}\mR_i^{t+1} - \mX_i^{t+1,\ell} \mT_i^{t+1,\ell}} $}{TEXT}} 
 Recall that  $\mX_i^{t+1}\bSigma_i^{t+1}{ \mX_i^{t+1}}^\tran $ and $\mX_i^{t+1,\ell}\bSigma_i^{t+1,\ell}{ \mX_i^{t+1,\ell}}^\tran $  are the top-$r$ eigenvalue decomposition of $\lb\mT_i+\mE_i^{t}\rb\lb\mT_i+\mE_i^{t}\rb^\tran$ and $\lb\mT_i+\mE_i^{t,\ell}\rb\lb\mT_i+\mE_i^{t,\ell}\rb^\tran$, respectively.
By the Weyl's inequality, the eigengap $\delta$ between the $r$-th and $r+1$-th eigenvalues of $\lb\mT_i+\mE_i^{t,\ell}\rb\lb\mT_i+\mE_i^{t,\ell}\rb^\tran$ is bounded below as follows:
\begin{align*}
	\delta &\geq \sigma_r(\mT_i\mT_i^\tran) -2\opnorm{\lb\mT_i+\mE_i^{t,\ell}\rb\lb\mT_i+\mE_i^{t,\ell}\rb^\tran-\mT_i\mT_i^\tran}\\
	& \geq \sigma^2_{\min}(\calT) - 2\opnorm{\mT_i{ \mE_i^{t,\ell} }^\tran + \mE_i^{t,\ell} \mT_i^\tran + \mE_i^{t,\ell} {\mE_i^{t,\ell}}^\tran } \geq \left(1 -  \frac{1}{2^{20} } \right) \cdot \sigma_{\min}^2(\calT),
	\end{align*}
where the last inequality is due to \eqref{bounding delta t ell}. Define the perturbation matrix $\mW_i$ by
\begin{align*}
    \mW_i &:=  \lb\mT_i+\mE_i^{t}\rb\lb\mT_i+\mE_i^{t}\rb^\tran- \lb\mT_i+\mE_i^{t,\ell}\rb\lb\mT_i+\mE_i^{t,\ell}\rb^\tran\\
    &= \left(\mT_i + \mE_i^{t}\right)\left( \mE_i^{t} - \mE_i^{t,\ell}\right)^\tran + \left(\mE_i^{t} -  \mE_i^{t, \ell}\right)\left(\mT_i + \mE_i^{t, \ell}\right)^\tran.
\end{align*}
We have 
\begin{align}
	\label{eq f norm of W}
	\opnorm{\mW_i} \leq \fronorm{\mW_i}
	&\leq  \fronorm{\left(\mT_i + \mE_i^{t}\right)\left( \mE_i^{t} - \mE_i^{t,\ell}\right)^\tran   } + \fronorm{ \left(\mE_i^{t} -  \mE_i^{t, \ell}\right)\left(\mT_i + \mE_i^{t, \ell}\right)^\tran } \notag  \\
	&\leq \left(  \opnorm{\mT_i + \mE_i^{t} } + \opnorm{\mT_i + \mE_i^{t, \ell} } \right)   \fronorm{  \calE^{t} - \calE^{t,\ell}} \leq \frac{5}{2}\sigma_{\max}(\calT) \fronorm{ \calE^{t} - \calE^{t,\ell}  }\\
	&\leq \frac{5}{2}\sigma_{\max}(\calT)\cdot \frac{1}{ 2^{20} \kappa^4 \mu^2 r^4} \frac{1}{2^{t+1}} \sqrt{\frac{\mu r}{n}} \cdot \sigma_{\max}(\calT)  \leq \frac{1}{2^{18}}\sigma^2_{\min}(\calT).  \notag 
\end{align}  Applying Lemma~\ref{rateoptimal lemma 1} and Lemma~\ref{Davis-kahan sintheta theorem} yields that 
\begin{align}
	\label{eq F norm at t+1}
	\fronorm{\mX_i^{t+1}\mR_i^{t+1} - \mX_i^{t+1,\ell} \mT_i^{t+1,\ell}}  &\leq \frac{\sqrt{2} \fronorm{\mW_i}}{\delta - \opnorm{\mW_i} } \leq  2\frac{\fronorm{\mW_i}}{ \sigma_{\min}^2(\calT)} \notag \\
	&\stackrel{(a)}{\leq }\frac{2}{ \sigma_{\min}^2(\calT)}\cdot \frac{5}{2}\sigma_{\max}(\calT)\cdot \fronorm{ \calE^{t} - \calE^{t,\ell}  } \notag \\
	&\stackrel{(b)}{\leq} \frac{2}{ \sigma_{\min}^2(\calT)}\cdot \frac{5}{2}\sigma_{\max}(\calT)\cdot  \frac{1}{2^6} \cdot \frac{1}{2^{20}\kappa^4 \mu^2 r^4}\frac{1}{2^{t+1}}\sqrt{\frac{\mu r}{n}} \sigma_{\max}(\calT) \notag \\
	&\leq\frac{1}{2^3}\frac{1}{2^{20}\kappa^2 \mu^2 r^4}\frac{1}{2^{t+1}} \sqrt{\frac{\mu r}{n}},
\end{align}
where  $(a)$ follows from \eqref{eq f norm of W} and  $(b)$ is due to \eqref{eq upper bound of difference E at t+1}.

\subsubsection{Bounding \texorpdfstring{$\twoinf{\mX_i^{t+1}\mR_i^{t+1} - \mU_i}$}{TEXT}}
A direct computation yields that
\begin{align}
	\label{eq XrU}
	\twoinf{\mX_i^{t+1}\mR_i^{t+1} - \mU_i} &\leq \twoinf{\mX_i^{t+1}\mR_i^{t+1} - \mX_i^{t+1,\ell}\mR_i^{t+1, \ell}} + \twoinf{\mX_i^{t+1,\ell}\mR_i^{t+1, \ell} - \mU_i} \notag \\
	&\leq \fronorm{\mX_i^{t+1}\mR_i^{t+1} - \mX_i^{t+1,\ell}\mR_i^{t+1, \ell}} + \twoinf{\mX_i^{t+1,\ell}\mR_i^{t+1, \ell} - \mU_i}.
\end{align}
We will use Lemma~\ref{lemma Ma 2017 lemma 37} with $\mX_1 = \mX_i^{t+1}\mR_i^{t+1}$ and $\mX_2 = \mX_i^{t+1,\ell}\mT_i^{t+1,\ell}$ to bound the first term in~\eqref{eq XrU}.  Following the same argument as in the proof of Lemma~\ref{lemma: spectral norm of Xt}, one can obtain
\begin{align*}
	\opnorm{\mX_i^{t+1} \mR_i^{t+1}- \mU_i}\opnorm{\mU_i} \leq \frac{1}{2^{17} \kappa^4  \mu^2 r^4 }\frac{1}{2^{t+1}} \leq \frac{1}{2}.
\end{align*} 
Furthermore, the inequality \eqref{eq F norm at t+1} shows that
\begin{align*}
	\opnorm{\mX_i^{t+1}\mR_i^{t+1}-\mX_i^{t+1,\ell}\mT_i^{t+1,\ell}}\opnorm{\mU_i}\leq \frac{1}{2^3} \frac{1}{2^{20}\kappa^2 \mu^2 r^4}\frac{1}{2^{t+1}}  \leq  \frac{1}{4}.
\end{align*}
Thus, by Lemma~\ref{lemma Ma 2017 lemma 37}, one has
\begin{align}
	\label{eq first term}
	\fronorm{\mX_i^{t+1}\mR_i^{t+1} - \mX_i^{t+1,\ell}\mR_i^{t+1, \ell}} &\leq 5\fronorm{\mX_i^{t+1}\mR_i^{t+1}-\mX_i^{t+1,\ell}\mT_i^{t+1,\ell}}\notag \\
 &\leq  \frac{5}{2^3} \frac{1}{2^{20}\kappa^2 \mu^2 r^4}\frac{1}{2^{t+1}}  .
\end{align}
Plugging \eqref{eq: incoherence of loo at t+1} and \eqref{eq first term} into \eqref{eq XrU} shows that
\begin{align*}
	\twoinf{\mX_i^{t+1}\mR_i^{t+1}-\mU_i} \leq \frac{1}{2^{20} \kappa^2\mu^2 r^4 }  \frac{1}{2^{t+1}}  \sqrt{\frac{\mu r}{n}}.
\end{align*}


\section{Proofs of Key Lemmas}
\label{sec: proof of key lemmas}

\subsection{Proof of Lemma~\ref{lemma: incoherence for T}}\label{subsec:incoherenceT}
The $\ell_{2,\infty}$ norm of $\mT_1$ can be bounded as follows
\begin{align*}
	\twoinf{\mT_1} &= \twoinf{\mU_1 \calM_1(\calS) (\mU_{3}\otimes \mU_{2} )^\tran} \leq \twoinf{\mU_1}\opnorm{\calM_1(\calS)}  \leq \sqrt{\frac{\mu r}{n}}\sigma_{\max}(\calT).
\end{align*}
For $\twoinf{\mT_1^\tran}$, a direct calculation yields that
\begin{align*}
	\twoinf{\mT_1^\tran} &= \twoinf{(\mU_{3}\otimes \mU_{2} ) \calM_1^\tran (\calS)\mU_1^\tran  } \leq \twoinf{\mU_{2}}\cdot \twoinf{\mU_{3}}\cdot \sigma_{\max}(\calT) \leq \frac{\mu r}{n}\sigma_{\max}(\calT),
\end{align*}
where the second line following from the fact $\twoinf{\mU_3\otimes \mU_2} \leq \twoinf{\mU_2}\cdot \twoinf{\mU_3}$. 
Using the same argument, one can obtain
\begin{align*}
    \twoinf{\mT_i} \leq \sqrt{\frac{\mu r}{n}}\sigma_{\max}(\calT) \quad\mbox{and}\quad \twoinf{\mT_i^\tran} \leq \frac{\mu r}{n}\sigma_{\max}(\calT),\quad i = 2,3.
\end{align*}
Lastly, by the definition of tensor infinity norm, one has
\begin{align*}
	\infnorm{\calT} &=\infnorm{\calM_1(\calT)} =\infnorm{\mU_1\calM_1(\calS)\lb \mU_{3} \otimes \mU_{2} \rb^\tran} \\
	&\leq \prod_{j=1}^{3}\twoinf{\mU_j}\cdot \opnorm{\calM_1(\calS)} \leq \lb \frac{\mu r}{n}\rb^{3/2}\sigma_{\max}(\calT).
\end{align*}

\subsection{Proof of Lemma~\ref{lemma: infinite norm of X-T}}\label{subsec:proof34}
Notice that $\mX_i$ is the top-$r$ singular vectors of 
\begin{align*}
	\lb\mT_i+\mE_i\rb \lb\mT_i+\mE_i\rb^\tran= \mT_i\mT_i^\tran+ \underbrace{ \mT_i\mE_i^\tran + \mE_i\mT_i^\tran + \mE_i\mE_i^\tran}_{=:\bDelta_i},
\end{align*}
where $\mT_i = \calM_i\lb\calT\rb$ and $\mE_i = \calM_i\lb\calE\rb$. Let  $\mT_i\mT_i^\tran=\mU_i\bLambda_i\mU_i^\tran$ be the eigenvalue decomposition. 
We will apply Lemma~\ref{lemma: Ma 2017 lemma 45} to bound $\opnorm{\mX_i\mR_i - \mU_i}$ where $\mR_i = \arg\min_{\mR^\tran\mR = \mI}\fronorm{\mX_i\mR-\mU_i}$. Invoking the triangle inequality shows that
\begin{align*}
	\opnorm{\bDelta_i} &\leq \left( 2\opnorm{\mT_i} +\opnorm{\mE_i} \right)  \opnorm{\mE_i}\leq \frac{5}{2}\sigma_{\max} \left(\calT\right)  \opnorm{\mE_i}\leq \frac{5}{2}\sigma_{\max}\lb\calT\rb\cdot  \frac{1}{10\kappa^2}\sigma_{\max}\lb\calT\rb\leq \frac{1}{4}\sigma_{\min}^2\lb\calT\rb,
\end{align*}
where the second and third lines follow from the assumption $\max_{i = 1,2,3}\opnorm{\mE_i}\leq \frac{\sigma_{\max}\lb\calT\rb}{10\kappa^2}\leq \frac{\sigma_{\max}(\calT)}{2}$.
Thus, the requirement in Lemma~\ref{lemma: Ma 2017 lemma 45} is valid, and we have
\begin{align*}
	\opnorm{\mX_i\mR_i - \mU_i} \leq \frac{3}{\sigma_{\min}^2(\calT)} \opnorm{\bDelta_i} \leq \frac{15\sigma_{\max}(\calT)}{2\sigma_{\min}^2(\calT)}  \opnorm{\mE_i}.
\end{align*}
Furthermore,  $\|\mX_i\mX_i^\tran - \mU_i \mU_i^\tran\|_{2,\infty}$ can be bounded as follows:
	\begin{align}\label{eq estimation of Px-Pu}
		\twoinf{\mX_i\mX_i^\tran - \mU_i \mU_i^\tran} 
		&\leq \twoinf{(\mX_i\mR_i - \mU_i ) (\mX_i\mR_i)^\tran } + \twoinf{ \mU_i (\mX_i\mR_i - \mU_i )^\tran}\notag \\
		&\leq \twoinf{\mX_i\mR_i - \mU_i } + \twoinf{\mU_i}\cdot \opnorm{ \mX_i\mR_i - \mU_i} \notag\\
		&\leq\twoinf{\mX_i\mR_i - \mU_i } + \frac{15\sigma_{\max}(\calT)}{2\sigma_{\min}^2(\calT)}  \opnorm{\mE_i} \cdot \twoinf{\mU_i} \notag\\
		&\leq \max_{i=1,2,3} \left( \twoinf{\mX_i\mR_i - \mU_i } + \frac{15\sigma_{\max}(\calT)}{2\sigma_{\min}^2(\calT)}  \opnorm{\mE_i}  \twoinf{\mU_i}\right)=:B.
	\end{align}
Next, we turn to control $\infnorm{\calX - \calT}$,
\begin{align*}
	\infnorm{\calX - \calT} &= \infnorm{ \lb \calT + \calE\rb \ttimes {\mX_i}{\mX_i}^\tran - \calT} \\
	&= \infnorm{ \calT  \ttimes {\mX_i}{\mX_i}^\tran - \calT \ttimes {\mU_i}{\mU_i}^\tran + \calE  \ttimes {\mX_i}{\mX_i}^\tran } \\
	&\leq \infnorm{\calT  \ttimes {\mX_i}{\mX_i}^\tran - \calT \ttimes {\mU_i}{\mU_i}^\tran}+\infnorm{ \calE  \ttimes {\mX_i}{\mX_i}^\tran}\\
	&\leq \underbrace{ \infnorm{ \calT \ttimes \lb {\mX_i}{\mX_i}^\tran - {\mU_i}{\mU_i}^\tran\rb } }_{=:\omega_1}  + \underbrace{ \sum_{i=1}^{3} \infnorm{\calT \times_i{\mU_i}{\mU_i}^\tran \jneqi \lb {\mX_j}{\mX_j}^\tran - {\mU_j}{\mU_j}^\tran\rb} }_{=:\omega_2} \\
	&\quad +\underbrace{  \sum_{i=1}^{3} \infnorm{\calT \times_i\lb {\mX_i}{\mX_i}^\tran - {\mU_i}{\mU_i}^\tran\rb  \jneqi {\mU_j}{\mU_j}^\tran } }_{=:\omega_3}  + \underbrace{ \infnorm{ \calE  \ttimes {\mX_i}{\mX_i}^\tran} }_{=:\omega_4}.
\end{align*}
\paragraph{Bounding $\omega_1$.} It is easy to see that
	\begin{align*}
		\omega_1 &=\infnorm{\calM_1\lb \calT \ttimes \lb {\mX_i}{\mX_i}^\tran - {\mU_i}{\mU_i}^\tran\rb \rb} \\
		&\leq \twoinf{{\mX_1}{\mX_1}^\tran - {\mU_1}{\mU_1}^\tran} \cdot \opnorm{\mT_1}\cdot \twoinf{\lb{\mX_3}{\mX_3}^\tran - {\mU_3}{\mU_3}^\tran\rb\otimes\lb{\mX_2}{\mX_2}^\tran - {\mU_2}{\mU_2}^\tran\rb} \\
		&\leq  \sigma_{\max}\lb\calT\rb \cdot \prod_{i=1}^3 \twoinf{ {\mX_i}{\mX_i}^\tran -  {\mU_i}\mU_i^\tran}\leq  B^3\cdot \sigma_{\max}(\calT),
	\end{align*}
	where the last inequality is due to \eqref{eq estimation of Px-Pu}.
\paragraph{Bounding $\omega_2$.} A  directly calculation gives
\begin{align*}
		\omega_2&=\sum_{i=1}^{3} \infnorm{\calM_i\lb\calT \times_i{\mU_i}{\mU_i}^\tran \jneqi \lb {\mX_j}{\mX_j}^\tran - {\mU_j}{\mU_j}^\tran\rb\rb} \\
		&\leq \sum_{i=1}^3 \twoinf{\mU_i} \opnorm{\mT_i} \prod_{j\neq i}\twoinf{{\mX_j}{\mX_j}^\tran - {\mU_j}{\mU_j}^\tran}\leq  3 \max_{i=1,2,3}\twoinf{\mU_i} \cdot \sigma_{\max}(\calT)\cdot B^2.
	\end{align*}
	\paragraph{Bounding $\omega_3$.} A straightforward computation shows that
	\begin{align*}
		\omega_3 & = \sum_{i=1}^{3} \infnorm{\calT \times_i\lb {\mX_i}{\mX_i}^\tran - {\mU_i}{\mU_i}^\tran\rb  \jneqi {\mU_j}{\mU_j}^\tran } \leq 3 \left( \max_{i=1,2,3}\twoinf{\mU_i} \right)^2 \cdot \sigma_{\max}(\calT)\cdot B.
	\end{align*}
	\paragraph{Bounding $\omega_4$.} By the definition of the infinity norm, one has
	\begin{align*}
		\omega_4 &=\sup_{\ve_{i_1}, \ve_{i_2}, \ve_{i_3} } \lab\calE  \times_1 \ve_{i_1}^\tran{\mX_1}{\mX_1}^\tran  \times_2\ve_{i_2}^\tran{\mX_2}{\mX_2}^\tran \times_3\ve_{i_3}^\tran{\mX_3}{\mX_3}^\tran\rab \\
		&\leq \opnorm{\calE}\cdot \left( \max_{i=1,2,3}\twoinf{\mX_i} \right)^3\stackrel{(a)}{\leq} \opnorm{\calM_1(\calE)}\cdot  \left( \max_{i=1,2,3}\twoinf{\mX_i} \right)^3.
	\end{align*}
	where step (a) follows from the fact $\|\calX\|\leq  \|\calM_i(\calX)\|$ for any tensor $\calX$.
\\

Combining these four upper bounds together, we complete the proof.

\subsection{Proof of Lemma~\ref{lemma: entry upper bound}}
By the definition of operator $\calP_{T}$, one can obtain
\begin{align*}
	\fronorm{\calP_{T} \left( \ve_{i_1}\circ \ve_{i_2}\circ\ve_{i_3}\right)}^2 = \fronorm{  \left( \ve_{i_1}\circ \ve_{i_2}\circ\ve_{i_3} \right)  \ttimes {\mU_i}{\mU_i}^\tran}^2 + \sum_{i=1}^3\fronorm{\calS \times_i \mW_i\jneqi \mU_j}^2 =: \psi_1+\psi_2.
\end{align*}
\paragraph{Bounding $\psi_1$.} Recall the definition of $\ve_{i_1}\circ \ve_{i_2}\circ\ve_{i_3}$ in \eqref{rank one tensor}. One has
	\begin{align*}
		\psi_1  &=\fronorm{\mU_1\mU_1^\tran \calM_1\lb\ve_{i_1}\circ \ve_{i_2}\circ\ve_{i_3}\rb\lb \mU_3\mU_3^\tran \otimes \mU_2\mU_2^\tran \rb^\tran}^2\\
		&=\fronorm{\mU_1\mU_1^\tran \ve_{i_1}\lb\ve_{i_3}\otimes\ve_{i_2}\rb^\tran \lb \mU_3\mU_3^\tran \otimes \mU_2\mU_2^\tran \rb^\tran}^2\\
		&=\fronorm{\mU_1\mU_1^\tran\ve_{i_1}\lb \mU_3\mU_3^\tran \ve_{i_3}\otimes \mU_2\mU_2^\tran\ve_{i_2} \rb^\tran }^2\\
		&\leq \twonorm{\mU_1^\tran\ve_{i_1}}^2\cdot \twonorm{\mU_2^\tran\ve_{i_2}}^2\cdot \twonorm{\mU_3^\tran\ve_{i_3}}^2\leq \lb \frac{\mu r}{n}\rb^{3}.
	\end{align*}
\paragraph{Bounding $\psi_2$.} We only provide details for bounding $\fronorm{\calS\times_1\mW_1\times_2\mU_2\times_3\mU_3 }^2$,
	\begin{align*}
		&\fronorm{\calS\times_1\mW_1\times_2\mU_2\times_3\mU_3 }^2 \\
  &\quad= \fronorm{\mW_1\calM_1(\calS)(\mU_3\otimes\mU_2)^\tran}^2\\
		&\quad=\fronorm{\lb\mI - \mU_1\mU_1^\tran \rb\calM_1\lb\ve_{i_1}\circ \ve_{i_2}\circ\ve_{i_3}\rb\lb\mU_3\otimes\mU_2\rb \calM_1(\calS)^{\dagger}\calM_1(\calS)(\mU_3\otimes\mU_2)^\tran}^2\\
		&\quad\leq \twonorm{(\ve_{i_3}\otimes\ve_{i_2})^\tran(\mU_3\otimes\mU_2) }^2 \cdot \opnorm{\calM_1(\calS)^{\dagger}\calM_1(\calS)}^2\leq \twonorm{\mU_2^\tran \ve_{i_2}}^2\cdot \twonorm{\mU_3^\tran\ve_{i_3}}^2\leq  \lb\frac{\mu r}{n}\rb^2 .
	\end{align*}
	Similarly, one can obtain
	\begin{align*}
		\fronorm{\calS\times_1\mU_1\times_2\mW_2\times_3\mU_3 }^2 \leq   \lb\frac{\mu r}{n}\rb^2  \text{ and } \fronorm{\calS\times_1\mU_1\times_2\mU_2\times_3\mW_3 }^2 \leq  \lb\frac{\mu r}{n}\rb^2.
	\end{align*}
	Therefore, $\psi_2\leq 3\lb \frac{\mu r}{n} \rb^2.$
	\\
	
	Combining the above bounds together shows that
	\begin{align*}
		\fronorm{\calP_{T} \left( \ve_{i_1}\circ \ve_{i_2}\circ\ve_{i_3} \right)}^2 \leq 4\lb \frac{\mu r}{n} \rb^2.
	\end{align*}

\subsection{Proof of Lemma~\ref{lemma: local isotropy}}
Let $\calE_{i_1,i_2,i_3}$ be the tensor with  the $(i_1,i_2,i_3)$-th entry being  1 and all the other entries being 0. The operator norm $\opnorm{\calP_{T}\lb p^{-1} \calP_{\Omega}-\calI\rb\calP_{T}}$ can be expressed as follows:
\begin{align*}
	&\opnorm{\calP_{T}\lb p^{-1} \calP_{\Omega}-\calI\rb\calP_{T}}\\
 &~= \sup_{\fronorm{\calZ} = 1,\calZ\in\R^{n\times n\times n}}\fronorm{\calP_{T}\lb p^{-1}\calP_{\Omega}-\calI\rb\calP_{T}\lb\calZ\rb}\\
	&~ = \sup_{\fronorm{\calZ} = 1,\calZ\in\R^{n\times n\times n}} \fronorm{\sum_{i_1,i_2,i_3}\lb p^{-1}\delta_{i_1,i_2,i_3}-1\rb\left\langle\calP_{T}\lb\calZ\rb,\calE_{i_1,i_2,i_3}\right\rangle\calP_{T}\lb\calE_{i_1,i_2,i_3}\rb}\\
	&~ = \sup_{\fronorm{\calZ} = 1,\calZ\in\R^{n\times n\times n}} \opnorm{\sum_{i_1,i_2,i_3}\lb p^{-1}\delta_{i_1,i_2,i_3}-1\rb\left\langle\calZ,\calP_{T}\lb\calE_{i_1,i_2,i_3}\rb\right\rangle\vect\lb\calP_{T}\lb\calE_{i_1,i_2,i_3}\rb\rb}\\
	&~ = \sup_{\fronorm{\calZ} = 1,\calZ\in\R^{n\times n\times n}}\opnorm{\sum_{i_1,i_2,i_3}\lb p^{-1}\delta_{i_1,i_2,i_3}-1\rb\left\langle\vect(\calZ),\vect(\calP_{T}(\calE_{i_1,i_2,i_3}))\right\rangle\vect\lb\calP_{T}\lb\calE_{i_1,i_2,i_3}\rb\rb}\\
	&~ =  \sup_{\fronorm{\calZ} = 1,\calZ\in\R^{n\times n\times n}}\opnorm{\sum_{i_1,i_2,i_3}\lb p^{-1}\delta_{i_1,i_2,i_3}-1\rb\vect\lb\calP_{T}\lb\calE_{i_1,i_2,i_3}\rb\rb\vect\lb\calP_{T}\lb\calE_{i_1,i_2,i_3}\rb\rb^{\tran}\vect\lb\calZ\rb}\\
	&~ = \opnorm{\sum_{i_1,i_2,i_3}\lb p^{-1}\delta_{i_1,i_2,i_3}-1\rb\vect(\calP_{T}(\calE_{i_1,i_2,i_3})\vect(\calP_{T}(\calE_{i_1,i_2,i_3}))^{\tran}},
\end{align*}
where $\vect(\cdot)$ is used to denote the vectorization of a tensor. Define 
\begin{align*}
    \mS_{i_1,i_2,i_3} :=\lb p^{-1}\delta_{i_1,i_2,i_3}-1\rb\vect(\calP_{T}(\calE_{i_1,i_2,i_3}))\vect(\calP_{T}(\calE_{i_1,i_2,i_3}))^{\tran}\in\R^{n^3\times n^3},
\end{align*}
which are independent symmetric matrices with mean zero. Thus we can apply the matrix Bernstein inequality to bound the above term. Lemma~\ref{lemma: F norm of E projection onto T} gives that
\begin{align*}
	&\opnorm{\mS_{i_1,i_2,i_3}}\leq p^{-1}\fronorm{\calP_{T}\lb\calE_{i_1,i_2,i_3}\rb}^2\leq \frac{4}{p}\lb\frac{\mu r}{n}\rb^2.
\end{align*}
Moreover, the operator norm of $\sum_{i_1,i_2,i_3}\mathbb{E}\left[ \mS_{i_1,i_2,i_3}\mS_{i_1,i_2,i_3}^{\tran}\right]$ can be bounded as follows:
\begin{align*}
	&\opnorm{\sum_{i_1,i_2,i_3}\mathbb{E}\left[ \mS_{i_1,i_2,i_3}\mS_{i_1,i_2,i_3}^{\tran}\right]}\\
 &\leq p^{-1}\opnorm{\sum_{i_1,i_2,i_3}\vect(\calP_{T}(\calE_{i_1,i_2,i_3}))\vect(\calP_{T}(\calE_{i_1,i_2,i_3}))^{\tran}\vect(\calP_{T}(\calE_{i_1,i_2,i_3}))\vect(\calP_{T}(\calE_{i_1,i_2,i_3}))^{\tran}}\\
	&\leq p^{-1}\max_{i_1,i_2,i_3}\fronorm{\calP_{T}(\calE_{i_1,i_2,i_3})}^2\opnorm{\sum_{i_1,i_2,i_3}\vect(\calP_{T}(\calE_{i_1,i_2,i_3}))\vect(\calP_{T}(\calE_{i_1,i_2,i_3}))^{\tran}}\leq \frac{4}{p}\lb\frac{\mu r}{n}\rb^2,
\end{align*}
where the least inequality follows from  Lemma~\ref{lemma: F norm of E projection onto T} and the fact 
\begin{align*}
	&\opnorm{\sum_{i_1,i_2,i_3}\vect(\calP_{T}(\calE_{i_1,i_2,i_3}))\vect(\calP_{T}(\calE_{i_1,i_2,i_3}))^{\tran}} \\&~ = \sup_{\fronorm{\calZ}  = 1,\calZ\in\R^{n\times n\times n}}\twonorm{\sum_{i_1,i_2,i_3}\vect(\calP_{T}(\calE_{i_1,i_2,i_3}))\vect(\calP_{T}(\calE_{i_1,i_2,i_3}))^{\tran}\vect(\calZ)}\\
	&~ = \sup_{\fronorm{\calZ}  = 1,\calZ\in\R^{n\times n\times n}} \fronorm{\sum_{i_1,i_2,i_3}\langle\calP_{T}(\calE_{i_1,i_2,i_3}),\calZ\rangle\calP_{T}(\calE_{i_1,i_2,i_3})}\\
	&~ = \sup_{\fronorm{\calZ}  = 1,\calZ\in\R^{n\times n\times n}} \fronorm{\calP_{T}\lb\sum_{i_1,i_2,i_3}\langle\calP_{T}(\calE_{i_1,i_2,i_3}),\calZ\rangle\calE_{i_1,i_2,i_3}\rb}\\
	&~\leq \sup_{\fronorm{\calZ}  = 1,\calZ\in\R^{n\times n\times n}} \fronorm{\sum_{i_1,i_2,i_3}\langle\calP_{T}(\calE_{i_1,i_2,i_3}),\calZ\rangle\calE_{i_1,i_2,i_3}}\\
	& ~= \sup_{\fronorm{\calZ}  = 1,\calZ\in\R^{n\times n\times n}}\sqrt{\sum_{i_1,i_2,i_3}\langle\calP_{T}(\calE_{i_1,i_2,i_3}),\calZ\rangle^2} \\
 &~= \sup_{\fronorm{\calZ}  = 1,\calZ\in\R^{n\times n\times n}}\sqrt{\sum_{i_1,i_2,i_3}\langle\calE_{i_1,i_2,i_3},\calP_{T}(\calZ)\rangle^2}\leq 1.
\end{align*}
The application of Bernstein inequality yields that with high probability,
\begin{align*}
	&\opnorm{\sum_{i_1,i_2,i_3}\lb p^{-1}\delta_{i_1,i_2,i_3}-1\rb\vect(\calP_{T}(\calE_{i_1,i_2,i_3})\vect(\calP_{T}(\calE_{i_1,i_2,i_3}))^{\tran}} \\&\quad\leq C\lb\frac{\log n}{p}\lb\frac{\mu r}{n}\rb^2+\sqrt{\frac{\log n}{p}\lb\frac{\mu r}{n}\rb^2}\rb\leq \varepsilon,
\end{align*}
provided $p\geq\frac{C_2\mu^2r^2\log n}{\varepsilon\cdot  n^2}$, where $\varepsilon $ is a small absolute constant.

\subsection{Proof of Lemma~\ref{lemma: technique lemmas}}
Applying Lemma~\ref{lemma: lemma 4.1 in Wei} gives that for $i = 1,2,3$,
\begin{align*}
    \opnorm{\mX_i^t{\mX_i^t}^\tran-\mU_i\mU_i^\tran} \leq \frac{\fronorm{\calM_i\lb\calX^t\rb-\calM_i\lb\calT\rb}}{\sigma_{\min}\lb\calM_i\lb\calT\rb\rb}\leq \frac{\fronorm{\calX^t-\calT}}{\sigma_{\min}\lb\calT\rb}.
\end{align*}
To prove  the second inequality of Lemma~\ref{lemma: technique lemmas}, first note that for any tensor $\calZ\in\R^{n\times n\times n}$,
\begin{align*}
\left(\calP_{T_{\calX^t}} - 	\calP_{T}\right)(\calZ)  = \prod_{i = 1}^3\calP_{\mX_i^t}^{(i)}(\calZ) -  \prod_{i = 1}^3\calP_{\mU_i}^{(i)}(\calZ)  + \sum_{i = 1}^3\left( \calP^{(j\neq i)}_{\calG^{t},\lcb \mX_{j}^{t}\rcb_{j\neq i}}\calP_{\mX_i^{t, \perp}}^{(i)} -  \calP^{(j\neq i)}_{\calS,\lcb \mU_{j}\rcb_{j\neq i}}\calP_{{\mU_i}^\bot}^{(i)} \right)(\calZ).
\end{align*}
Taking the Frobenius norm on both sides and applying the triangle inequality give
\begin{small}
\begin{align*}
\fronorm{ \left(\calP_{T_{\calX^t}} - 	\calP_{T}\right)(\calZ) } 
\leq\sum_{i=1}^{3} \underbrace{ \fronorm{ \left( \calP^{(j\neq i)}_{\calG^{t},\lcb \mX_{j}^{t}\rcb_{j\neq i}}\calP_{\mX_i^{t, \perp}}^{(i)} -  \calP^{(j\neq i)}_{\calS,\lcb \mU_{j}\rcb_{j\neq i}}\calP_{{\mU_i}^\bot}^{(i)} \right)(\calZ) }}_{=:\varsigma_i} +  \underbrace{ \fronorm{\prod_{i = 1}^3\calP_{\mX_i^t}^{(i)}(\calZ)  - \prod_{i = 1}^3\calP_{\mU_i}^{(i)}(\calZ)}  }_{=: \varsigma_4 }.
\end{align*}
\end{small}
\paragraph{Bounding $\varsigma_1$.} It is easy to see that
	\begin{align}
		\label{eq varsigma1}
			\varsigma_1 &= \fronorm{\left( \calP^{(j\neq 1)}_{\calG^{t},\lcb \mX_{j}^{t}\rcb_{j\neq 1}}\calP_{\mX_1^{t, \perp}}^{(1)} -  \calP^{(j\neq 1)}_{\calS,\lcb \mU_{j}\rcb_{j\neq 1}}\calP_{{\mU_1}^\perp}^{(1)} \right)(\calZ)}\notag \\
			&\leq \fronorm{\left( \lb\calP^{(j\neq 1)}_{\calG^{t},\lcb \mX_{j}^{t}\rcb_{j\neq 1}}-\calP^{(j\neq 1)}_{\calS,\lcb \mU_{j}\rcb_{j\neq 1}}\rb\calP_{\mX_1^{t, \perp}}^{(1)}\right)(\calZ)}+\fronorm{\lb \calP^{(j\neq 1)}_{\calS,\lcb \mU_{j}\rcb_{j\neq 1}}\lb\calP_{\mX_1^{t, \perp}}^{(1)}-\calP_{{\mU_1}^\perp}^{(1)} \rb\rb\lb\calZ\rb}\notag\\
			&\leq \fronorm{\calM_1\left( \lb\calP^{(j\neq 1)}_{\calG^{t},\lcb \mX_{j}^{t}\rcb_{j\neq 1}}-\calP^{(j\neq 1)}_{\calS,\lcb \mU_{j}\rcb_{j\neq 1}}\rb(\calZ) \right)}+\opnorm{{\mX_1^t}{\mX_1^t}^\tran-{\mU_1}{\mU_1}^\tran}\cdot \fronorm{\calZ} \notag \\
			&\leq \fronorm{\calM_1\left( \lb\calP^{(j\neq 1)}_{\calG^{t},\lcb \mX_{j}^{t}\rcb_{j\neq 1}}-\calP^{(j\neq 1)}_{\calS,\lcb \mU_{j}\rcb_{j\neq 1}}\rb(\calZ) \right)}+\frac{\fronorm{\calX^t - \calT}}{\sigma_{\min}(\calT)}\fronorm{\calZ},
	\end{align}
where we have used the first inequality in this lemma. Denote by $\mY_1^t$ and $\mV_1$ the top-$r$ right singular vectors of $\calM_1\lb\calX^t\rb$ and $\calM_1\lb\calT\rb$, respectively. Claim~\ref{property of these projection} gives that
\begin{align}\label{spectral norm yt-v}
	&\fronorm{\calM_1\left( \lb\calP^{(j\neq 1)}_{\calG^{t},\lcb \mX_{j}^{t}\rcb_{j\neq 1}}-\calP^{(j\neq 1)}_{\calS,\lcb \mU_{j}\rcb_{j\neq 1}}\rb(\calZ) \right)}\notag \\
	&~= \fronorm{\calM_1(\calZ)\left( \mY_1^t {\mY_1^t}^\tran -\mV_1\mV_1^\tran \right)} \leq \opnorm{\mY_1^t {\mY_1^t}^\tran -\mV_1\mV_1^\tran} \fronorm{\calZ}\leq \frac{\fronorm{\calX^t - \calT}}{\sigma_{\min}(\calT)}\fronorm{\calZ},
\end{align}
where the last inequality is due to Lemma~\ref{lemma: lemma 4.1 in Wei}. Plugging \eqref{spectral norm yt-v} into  \eqref{eq varsigma1} gives  
\begin{align*}
	\varsigma_1 &\leq \frac{2\fronorm{\calX^t - \calT}}{\sigma_{\min}(\calT)}\fronorm{\calZ}.
\end{align*}

\paragraph{Bounding $\varsigma_2$ and $\varsigma_3$.}
Following a similar argument as above, we can obtain
\begin{align*}
	\varsigma_2 \leq \frac{2\fronorm{\calX^t - \calT}}{\sigma_{\min}(\calT)}\fronorm{\calZ}\quad\mbox{and}\quad
	\varsigma_3 \leq \frac{2\fronorm{\calX^t - \calT}}{\sigma_{\min}(\calT)}\fronorm{\calZ}.
\end{align*}
\paragraph{Bounding $\varsigma_4$.} A simple calculation yields that
\begin{align*} 
	\varsigma_4 &\leq \fronorm{\calZ\times_1\lb {\mX_1^t}{\mX_1^t}^\tran -  {\mU_1}\mU_1^\tran\rb\times_2 {\mX_2^t}{\mX_2^t}^\tran \times_3 {\mX_3^t}{\mX_3^t}^\tran }\\
	&\quad+ \fronorm{  \calZ\times_1 {\mU_1}\mU_1^\tran \times_2\lb {\mX_2^t}{\mX_2^t}^\tran -  {\mU_2}\mU_2^\tran \rb\times_3 {\mX_3^t} {\mX_3^t}^\tran } \\
	&\quad + \fronorm{  \calZ\times_1 {\mU_1}\mU_1^\tran \times_2 {\mU_2}\mU_2^\tran \times_3\lb {\mX_3^t}{\mX_3^t}^\tran -  {\mU_3}{\mU_3}^\tran\rb}\\
	&\leq 3\max_{i=1,2,3}\opnorm{ {\mX_i^t}{\mX_i^t}^\tran - {\mU_i}{\mU_i}^\tran  }\cdot \fronorm{\calZ}\\
	&\stackrel{(a)}{\leq}3\max_{i=1,2,3} \frac{\fronorm{\calM_i\lb\calX^t - \calT\rb}}{\sigma_{\min}\lb\calM_i\lb\calT\rb\rb}\fronorm{\calZ}\leq 3\frac{\fronorm{\calX^t-\calT}}{\sigma_{\min}\lb\calT\rb}\fronorm{\calZ},
\end{align*}
where $(a)$ is due to Lemma~\ref{lemma: lemma 4.1 in Wei}.
\\

Putting together all of the preceding bounds on $\varsigma_1$, $\varsigma_2$, $\varsigma_3$ and $\varsigma_4$ immediately establishes the lemma.

\subsection{Proof of Lemma~\ref{lemma 2}} 

For any $\calZ\in\R^{n\times n\times n}$, we have
\begin{align*}
	\fronorm{\calP_{\Omega}\calP_T(\calZ) }^2 &= \la \calP_{\Omega}\calP_T(\calZ) , \calP_{\Omega}\calP_T(\calZ) \ra\\
	&=\la \calP_T(\calZ) , \calP_{\Omega}\calP_T(\calZ) \ra - p\la \calP_{T}(\calZ),  \calP_{T}(\calZ)\ra + p\la \calP_{T}(\calZ),  \calP_{T}(\calZ)\ra\\
	&\leq p\lb \opnorm{ \calP_T - p^{-1}\calP_{T} \calP_{\Omega} \calP_{T} } + 1\rb \cdot \fronorm{\calZ}^2\leq p(1+\varepsilon)\fronorm{\calZ}^2,
\end{align*}
which implies that
	$\opnorm{\calP_{\Omega}\calP_T} \leq \sqrt{p(1+\varepsilon)}.$
Consequently,
\begin{align*}
	\opnorm{ \calP_{\Omega} \calP_{T_{\calX^t}} } & \leq \opnorm{\calP_{\Omega} \lb \calP_T - \calP_{T_{\calX^t}}\rb} + \opnorm{\calP_{\Omega} \calP_{T}}\leq  \opnorm{  \calP_T - \calP_{T_{\calX^t}} } + \opnorm{\calP_{\Omega} \calP_{T}}\\
	&\stackrel{(a)}{\leq } \frac{9}{\sigma_{\min}(\calT)}\fronorm{\calX^t - \calT} + \sqrt{(1+\epsilon)p}\stackrel{(b)}{\leq } \sqrt{p}\lb \frac{\varepsilon}{1+ \varepsilon} + \sqrt{1+\varepsilon} \rb\leq 2\sqrt{p}(1+\varepsilon),
\end{align*}
where  $(a)$ follows from Lemma~\ref{lemma: technique lemmas} and  $(b)$ is due to the assumption \eqref{assumption in lemma 2}. 

The spectral norm of $\calP_{T_{\calX^t}} - p^{-1}\calP_{T_{\calX^t}} \calP_{\Omega} \calP_{T_{\calX^t}}$ can be bounded as follows:
\begin{align*}
	\opnorm{\calP_{T_{\calX^t}} - p^{-1}\calP_{T_{\calX^t}} \calP_{\Omega} \calP_{T_{\calX^t}} } 
	&\leq \opnorm{  \calP_{T_{\calX^t}} - \calP_T  } +p^{-1} \opnorm{  \lb \calP_{T_{\calX^t}} - \calP_T \rb\calP_{\Omega} \calP_{T_{\calX^t}}     } \\
	&\quad + p^{-1} \opnorm{  \calP_T \calP_{\Omega} \lb \calP_{T_{\calX^t}} - \calP_T  \rb  }+\opnorm{ \calP_{T } - p^{-1}\calP_{T} \calP_{\Omega} \calP_{T}   }\\
	&\leq \opnorm{\calP_{T_{\calX^t}} - \calP_T} + p^{-1} \opnorm{\calP_{T_{\calX^t}} - \calP_T}\cdot \opnorm{\calP_{\Omega} \calP_{T_{\calX^t}} } \\
	&\quad + p^{-1} \opnorm{  \calP_T \calP_{\Omega}} \cdot \opnorm{ \calP_{T_{\calX^t}} - \calP_T  }+\opnorm{ \calP_{T } - p^{-1}\calP_{T} \calP_{\Omega} \calP_{T}   }\\
	&\leq \opnorm{\calP_{T_{\calX^t}} - \calP_T}  \lb 1+ p^{-1}\opnorm{\calP_{\Omega} \calP_{T_{\calX^t}} } + p^{-1} \opnorm{  \calP_T \calP_{\Omega}}    \rb + \varepsilon\\
	&\leq \frac{9}{\sigma_{\min}(\calT)}\fronorm{\calX^t - \calT}  \lb 1+ p^{-1} \frac{\sqrt{p}\varepsilon}{1+\varepsilon}  + \frac{2}{p} \opnorm{  \calP_T \calP_{\Omega}}    \rb + \varepsilon\\
	&\leq \sqrt{p}\frac{\varepsilon}{1+\varepsilon} \lb 1+ p^{-1} \frac{\sqrt{p}\varepsilon}{1+\varepsilon} + \frac{2}{p}\cdot \sqrt{p}\cdot \sqrt{1+\varepsilon}\rb + \varepsilon\\
	&\leq \frac{\varepsilon}{1+\varepsilon} \lb 1+  \frac{ \varepsilon}{1+\varepsilon} +2 \sqrt{1+\varepsilon}\rb + \varepsilon\\
	&\leq \frac{\varepsilon}{1+\varepsilon} \lb \frac{1+2\varepsilon+ 2(1+\varepsilon)^2 }{1+\varepsilon} \rb + \varepsilon\leq 5\varepsilon,
\end{align*}
which completes the proof.

\subsection{Proof of Lemma~\ref{spectral norm of delta 1,ell}}
\label{Proof of Lemma spectral norm of delta 1,ell}
Lemma~\ref{spectral norm of delta 1,ell} is similar to Lemma $6$ in \citep{cai2021subspace}, but with three slices being  left  out instead  of one slice. As being noted  in \citep{cai2021nonconvex}, the proof can be easily adapted from the one slice case to the three slices case.  Here we only point  out  that the structure of Tucker decomposition can be used to simplify the proof in our problem. First the triangle inequality gives that
\begin{align}
\label{eq: tt-tt}
	&\opnorm{\calP_{\offdiag}\left(\widehat{\mT}_i^{\ell}\widehat{\mT}_i^{{\ell}^{\tran}} \right)-\mT_i\mT_i^\tran}\notag\\
	&\quad\leq \opnorm{ \calP_{\offdiag} \lb\mE_i^{-1,\ell}{\mE_i^{-1,\ell}}^\tran \rb}+\opnorm{ \calP_{\offdiag} \lb \mT_i{\mE_i^{-1,\ell}}^\tran + \mE_i^{-1,\ell}{\mT_i}^\tran \rb } +  \opnorm{ \calP_{\diag}\lb \mT_i\mT_i^\tran \rb},
\end{align}
where $\mE_i^{-1,\ell}$ is defined as $\mE_i^{-1,\ell} = \calM_i\lb\lb p^{-1}\calP_{\Omega_{-\ell}}+\calP_{\ell}-\calI\rb\lb\calT\rb\rb$.
\paragraph{Bounding $\opnorm{\calP_{\offdiag} \lb\mE_i^{-1,\ell}{\mE_i^{-1,\ell}}^\tran\rb}$.} 
By adapting the analysis for $\calP_{\offdiag}\lb\mE_i^{-1}\mE_i^{-1}\rb$ in \citep[Lemma 1]{cai2021subspace} where $\mE_i^{-1} = \calM_i\lb p^{-1}\calP_{\Omega}-\calI\rb\lb\calT\rb$, we have
\begin{align}\label{step1}
	\opnorm{\calP_{\offdiag} \lb\mE_i^{-1,\ell}{\mE_i^{-1,\ell}}^\tran\rb}\leq C \lb\frac{\mu^{3/2} r^{3/2}}{n^{3/2}p}+\frac{\mu^2 r^2}{n^2p}\rb\log n\cdot\sigma^2_{\max}\lb\calT\rb.
\end{align}

\paragraph{Bounding $\opnorm{\calP_{\offdiag} \lb\mT_i{\mE_i^{-1,\ell}}^\tran  + \mE_i^{-1,\ell}{\mT_i}^\tran \rb}$.}
For this term, we can utilize  the Tucker structure to immediately obtain the upper bound. Invoking the triangular inequality gives that
\begin{align*}
	\opnorm{\calP_{\offdiag} \lb\mT_i{\mE_i^{-1,\ell}}^\tran  + \mE_i^{-1,\ell}{\mT_i}^\tran \rb }\leq \opnorm{\calP_{\offdiag} \lb\mT_i{\mE_i^{-1,\ell}}^\tran\rb}  + \opnorm{\calP_{\offdiag} \lb \mE_i^{-1,\ell}{\mT_i}^\tran \rb}.
\end{align*}	
It suffices to control the spectral norm of $\calP_{\offdiag} \lb\mT_1{\mE_1^{-1,\ell}}^\tran\rb$.  
To this end, we have
\begin{align*}
	\opnorm{\calP_{\offdiag} \lb\mT_1{\mE_1^{-1,\ell}}^\tran\rb}  &\leq  2\opnorm{\mT_1{\mE_1^{-1,\ell}}^\tran} =2\opnorm{\calM_1(\calT)\calM_1^\tran(\calE^{-1,\ell})} \\
	&=2\opnorm{\mU_1\calM_1(\calS)(\mU_3\otimes\mU_2)^\tran\calM_1^\tran(\calE^{-1,\ell}) }\\
	&\leq 2\sigma_{\max}\lb\calT\rb\opnorm{\calM_1\lb \calE^{-1,\ell} \times_2 \mU_2^\tran \times_3 \mU_3^\tran \rb}\\
	&\stackrel{(a)}{\leq} 2\sigma_{\max}\lb\calT\rb\sqrt{r}\opnorm{\calE^{-1,\ell} \times_2 \mU_2^\tran \times_3 \mU_3^\tran}\\
	&\leq 2\sigma_{\max}\lb\calT\rb \sqrt{r}\opnorm{\lb\calI-p^{-1}\calP_{\Omega}\rb\lb\calT\rb}\\
	&\stackrel{(b)}{\leq} \sigma_{\max}(\calT)\cdot\sqrt{r} C\lb \frac{\log^3 n}{p} \infnorm{\calT} + \sqrt{\frac{3\log^5 n}{p} \twoinf{  \mT_1^\tran }^2 } \rb,
\end{align*}
where $(a)$ follows from Lemma~\ref{lemma: Lemma 6 in XYZ} and $(b)$ is due to Lemma~\ref{lemma: spectral norm of E initial}.
Therefore, the following inequality holds with high probability,
\begin{align}\label{step2}
\opnorm{\calP_{\offdiag} \lb\mT_i{\mE_i^{-1,\ell}}^\tran  + \mE_i^{-1,\ell}{\mT_i}^\tran \rb } \leq C \sigma^2_{\max}(\calT)\cdot\sqrt{r}  \lb \frac{\mu^{3/2} r^{3/2}\log^3 n}{n^{3/2}p}  + \sqrt{\frac{\mu ^2 r^2 \log^5 n }{n^2p} }  \rb.
\end{align}

\paragraph{Bounding $\opnorm{\calP_{\diag}\lb \mT_i\mT_i^\tran \rb}$.}
A straightforward computation implies that
\begin{align}\label{step3}
	\opnorm{ \calP_{\diag}\lb \mT_i\mT_i^\tran \rb } &=\max_{j\in [n]} \left| \lb \mT_i\mT_i^\tran\rb_{j,j} \right|=\twoinf{\mT_i}^2\leq \frac{\mu r}{n}\sigma^2_{\max}\lb\calT\rb.
\end{align}
\\

Plugging \eqref{step1}, \eqref{step2} and \eqref{step3} into \eqref{eq: tt-tt} yields that, with high probability,
\begin{align*}
	&\opnorm{\calP_{\offdiag}\left(\widehat{\mT}_i^{\ell}{\widehat{\mT}_i^{{\ell}^{\tran}}} \right)-\mT_i\mT_i^\tran} \\
	&\quad  \leq C \lb\lb\frac{\mu^{3/2} r^{3/2}}{n^{3/2}p}+\frac{\mu^2 r^2}{n^2p}\rb\log n+ \sqrt{r}  \lb \frac{\mu^{3/2} r^{3/2}\log^3 n}{n^{3/2}p}  + \sqrt{\frac{\mu ^2 r^2 \log^5 n }{n^2p} }  \rb+\frac{\mu r}{n}\rb\sigma^2_{\max}\lb\calT\rb.
\end{align*}

\subsection{Proof of Lemma~\ref{lemma: spectral norm distance of initialization}}
\label{proof of lemma spectral norm distance of initialization}
	Recall that $\mX_i^{1} \bSigma^{1} _i {\mX_i^{1} }^\tran$ is the top-$r$ eigenvalue decomposition of 
	\begin{align}
	\label{eq: G}
		\mG:=\calP_{\offdiag}\lb\widehat{\mT_i}\widehat{\mT_i}^{\tran}\rb.
	\end{align}
	We denote by $\mG^\natural$ the Gram matrix $ \calM_i(\calT)\calM_i^\tran(\calT)$.  Applying Lemma 1 in \citep{cai2021subspace} yields that, with high probability,
	\begin{align}
		\label{eq spectral norm of G}
		\opnorm{\mG - \mG^\natural} &\leq C \lb \frac{ \mu^{3/2} r^{3/2}\sigma_{\max}^2( \calT ) \log n}{n^{3/2}p}  + \sqrt{ \frac{ \mu^2 r^2 \sigma_{\max}^4( \calT ) \log n}{n^2 p}} + \twoinf{\calM_i(\calT)}^2\rb \notag \\
		&\stackrel{(a)}{\leq} \frac{1}{2^7}\frac{1}{2^{20} \kappa^6 \mu^2 r^4 }\sigma_{\min}^2(\calT),
	\end{align}
	 where  $(a)$ follows from  \eqref{cond: p} and $n\geq C_0 \mu^3 r^5 \kappa^8$. 	Since $\calT$ is a tensor with multilinear rank $\vr=(r,r,r)$, we have $\sigma_{r}\lb\mG^\natural\rb = \sigma^2_{\min}\lb\mT_i\rb>0$, $\sigma_{r+1}\lb\mG^\natural\rb=0$. From \eqref{eq spectral norm of G}, one can see that $\opnorm{\mG - \mG^\natural}\leq \frac{1}{4}\sigma_{r}\lb\mG^\natural\rb$. Therefore, Lemma~\ref{lemma: Ma 2017 lemma 45} is applicable, which gives that 
	\begin{align*}
		\opnorm{\mX_i^1 \mR_i^1 - \mU_i} & \leq \frac{3}{ \sigma_{r}(\mG^\natural) } \opnorm{ \mG - \mG^\natural}\leq \frac{3}{ \sigma_{\min}^2(\calT) } \opnorm{ \mG - \mG^\natural}\\
		&\leq \frac{3}{ \sigma_{\min}^2(\calT) } \frac{1}{2^7}\frac{1}{2^{20} \kappa^6 \mu^2 r^4 }\sigma_{\min}^2(\calT)\leq \frac{1}{2^5}\frac{1}{2^{20} \kappa^6 \mu^2 r^4 }.
	\end{align*} 
Next, we turn to bound $\opnorm{\mX_i^{1,\ell}\mR_i^{1,\ell} - \mU_i}$. Notice that $\mX_i^{1,\ell}\bSigma_{i}^{1,\ell}{\mX_i^{1,\ell}}^\tran$ is the top-$r$ eigenvalue decomposition of 
$\mG^\ell:= \calP_{\offdiag}\left( \widehat{\mT}_i^{\ell}\widehat{\mT}_i^{{\ell}^\tran} \right).$
By Lemma~\ref{spectral norm of delta 1,ell}, one can obtain that, with high probability,
\begin{align*}
&\opnorm{\mG^\ell - \mG^\natural} \\
&\leq  C\lb\lb\frac{\mu^{3/2} r^{3/2}}{n^{3/2}p}+\frac{\mu^2 r^2}{n^2p}\rb\log n+ \sqrt{r}  \lb \frac{\mu^{3/2} r^{3/2}\log^3 n}{n^{3/2}p}  + \sqrt{\frac{\mu ^2 r^2 \log^5 n }{n^2p} }  \rb+\frac{\mu r}{n}\rb\sigma^2_{\max}\lb\calT\rb\\
&\stackrel{(a)}{\leq} \frac{1}{2^7}\frac{1}{2^{20} \kappa^6 \mu^2 r^4 }\sigma_{\min}^2(\calT)\leq \frac{1}{4}\sigma_{\min}^2(\calT),
\end{align*}
 where  $(a)$ follows from  \eqref{cond: p} and $n\geq C_0 \mu^3 r^5 \kappa^8$. Applying Lemma~\ref{lemma: Ma 2017 lemma 45} immediately yields that \begin{align*}
\opnorm{\mX_i^{1,\ell}\mR^{1,\ell} - \mU_i} &\leq \frac{3}{ \sigma_{\min}^2(\calT) } \opnorm{ \mG^\ell - \mG^\natural}\leq \frac{3}{ \sigma_{\min}^2(\calT) } \frac{1}{2^7}\frac{1}{2^{20} \kappa^6 \mu^2 r^4 }\sigma_{\min}^2(\calT)\leq \frac{1}{2^5}\frac{1}{2^{20} \kappa^6 \mu^2 r^4 }.
\end{align*}

\subsection{Proof of Lemma \ref{lemma: spectral norm of Xt}} 
Noting that $\calE^0$ is indeed associated with $\calX^1$ (see \eqref{calE0}), we need to show Lemma~\ref{lemma: spectral norm of Xt} separately  for the case $t=1$.
 According to Lemma~\ref{lemma: spectral norm distance of initialization}, one can see that
\begin{align*}
    \opnorm{\mX_i^{1}\mR_i^1 - \mU_i}  \leq \frac{1}{2^{17} \kappa^4\mu^2 r^4} \frac{1}{2}\quad\mbox{and}\quad\opnorm{\mX_i^{1,\ell}\mR_i^{1,\ell} - \mU_i} \leq \frac{1}{2^{25} \kappa^6\mu^2 r^4} \leq \frac{1}{2^{17} \kappa^4\mu^2 r^4} \frac{1}{2}.
\end{align*}
Noticing the definition of $\mG$ (see \eqref{eq: G}) and $\mG^\natural$ ($\mG^\natural=\calM_i(\calT)\calM_i^\tran(\calT)$),  the spectral norm of $\mX_i^1{\mX_i^1}^\tran - \mU_i\mU_i^\tran$ can be bounded as follows:
\begin{align*}
\opnorm{\mX_i^1{\mX_i^1}^\tran - \mU_i\mU_i^\tran} &\leq \frac{2\opnorm{\mG - \mG^\natural}}{\sigma_{\min}^2(\calT) -  \opnorm{\mG - \mG^\natural}}\\
&\leq \frac{2}{\left(1-\frac{1}{8} \right)\sigma^2_{\min}(\calT)}\cdot \frac{1}{2^7} \frac{1}{2^{20}\kappa^6 \mu^2 r^4}\sigma_{\min}^2(\calT)\leq \frac{1}{2^{17} \kappa^4 \mu^2 r^4}\frac{1}{2},
\end{align*}
where first line is due to Lemma~\ref{rateoptimal lemma 1} and Lemma~\ref{Davis-kahan sintheta theorem}, and the second line follows from \eqref{eq spectral norm of G}.
Using the same argument yields that
\begin{align*}
\opnorm{\mX_i^{1,\ell}{\mX_i^{1,\ell}}^\tran - \mU_i\mU_i^\tran} \leq \frac{1}{2^{17} \kappa^4 \mu^2 r^4}\frac{1}{2}.
\end{align*} 
By the triangle inequality,
\begin{align*}
	\twoinf{ {\mX_i^{1,\ell}}{\mX_i^{1,\ell}}^\tran - \mU_i\mU_i^\tran } &\leq \twoinf{\left( \mX_i^{1,\ell}\mR_i^{1,\ell} -\mU_i \right) \left(\mX_i^{1,\ell}\mR_i^{1,\ell} \right)^\tran } + \twoinf{ \mU_i \left( \mX_i^{1,\ell}\mR_i^{1,\ell} -\mU_i \right)^\tran}\\
	&\leq \twoinf{  \mX_i^{1,\ell}\mR_i^{1,\ell} -\mU_i } + \twoinf{\mU_i}\cdot \opnorm{ \mX_i^{1,\ell}\mR_i^{1,\ell} -\mU_i }\\
	&\leq \frac{1}{2^{20}\kappa^2\mu^2 r^4} \frac{1}{2} \sqrt{\frac{\mu r}{n}}  + \sqrt{\frac{\mu r}{n}}\cdot \frac{1}{2^{17} \kappa^4  \mu^2 r^4 }\frac{1}{2}\leq \frac{1}{2^{16}\kappa^2 \mu^2 r^4}\frac{1}{2}\sqrt{\frac{\mu r}{n}}.
\end{align*}

For the case $t\geq 2$, first recall that the columns of $\mX_i^t$ are the top-r eigenvectors of 
\begin{align*}
	(\mT_i + \mE_i^{t-1} )(\mT_i + \mE_i^{t-1} )^\tran &= \mT_i \mT_i^\tran + \underbrace{ \mT_i{\mE_i^{t-1}}^\tran + \mE_i^{t-1} \mT_i^\tran + \mE_i^{t-1} {\mE_i^{t-1}}^\tran}_{=:\bDelta_i^{t-1}}.
\end{align*}
Thus a simple computation yields that
\begin{align}
	\label{eq: upper bound of delta}
	\opnorm{\bDelta_i^{t-1}} &\leq \left(2 \opnorm{\mT_i} + \opnorm{\mE_i^{t-1}}\right) \cdot \opnorm{\mE_i^{t-1}} \notag \\
	&\stackrel{(a)}{\leq} \left(2 \sigma_{\max}(\calT) + \frac{1}{8}\sigma_{\max}(\calT)\right)\cdot \frac{1}{2^{20} \kappa^6\mu^2 r^4} \frac{1}{2^t}  \sigma_{\max}(\calT) \notag  \\
	&= \frac{17}{2^3}\cdot \frac{1}{2^{20} \kappa^6 \mu^2 r^4} \frac{1}{2^t}  \sigma_{\max}^2(\calT) \leq \frac{1}{4}  \sigma_{\min}^2(\calT), 
\end{align}
where  $(a)$ follows from the induction inequality \eqref{spectral norm of error tensor}.
Applying Lemma~\ref{lemma: Ma 2017 lemma 45} gives that
\begin{align*}
	\opnorm{\mX_i^t \mR_i^t - \mU_i} &\leq \frac{3}{\sigma_{r}(\mT_i\mT_i^\tran)} \opnorm{\bDelta_i^{t-1}}\stackrel{(a)}{\leq }\frac{3}{\sigma_{\min}^2(\calT)} \cdot\frac{17}{2^{23} \kappa^6\mu^2 r^4} \frac{1}{2^t}  \sigma_{\max}^2(\calT)
 \leq \frac{1}{2^{17} \kappa^4 \mu^2 r^4 }\frac{1}{2^t},
\end{align*}
where  $(a)$ follows from \eqref{eq: upper bound of delta}. For \eqref{eq spectral norm of Pxt-Pu}, a direct calculation yields that
\begin{align*}
	\opnorm{ {\mX_i^t}{\mX_i^t}^\tran  -  {\mU_i}{\mU_i}^\tran } &\stackrel{(a)}{\leq} 2\cdot  \frac{8}{7\sigma_{\min}^2(\calT)} \opnorm{\bDelta_i^{t-1}} \\
 &\leq  \frac{16}{7\sigma_{\min}^2(\calT)} \cdot  \frac{17}{2^{23} \kappa^6 \mu^2 r^4} \frac{1}{2^t}  \sigma_{\max}^2(\calT)\leq \frac{1}{2^{17} \kappa^4\mu^2 r^4} \frac{1}{2^t},
\end{align*}
where  $(a)$ is due to Lemma~\ref{rateoptimal lemma 1} and Lemma~\ref{Davis-kahan sintheta theorem}.  Following a  similar argument, one also has 
\begin{align*}
    \opnorm{ {\mX_i^{t,\ell}}{\mX_i^{t,\ell}}^\tran  -  {\mU_i}{\mU_i}^\tran }\leq \frac{1}{2^{17} \kappa^4\mu^2 r^4} \frac{1}{2^t}~\mbox{and}~ \twoinf{ {\mX_i^{t,\ell}}{\mX_i^{t,\ell}}^\tran - \mU_i\mU_i^\tran }\leq \frac{1}{2^{16} \kappa^2  \mu^2 r^4 }\frac{1}{2^t}\sqrt{\frac{\mu r}{n}} .
\end{align*}

\subsection{Proof of Lemma~\ref{lemma: infty norm at t step}}
Recognizing that  $\opnorm{\mE_i^{t-1}} \leq \frac{1}{2^{20}\kappa^6 \mu^2 r^4}\frac{1}{2^t}\sigma_{\max}(\calT)\leq \sigma_{\max}\lb\calT\rb/\lb 10\kappa^2\rb$ (\eqref{spectral norm of error tensor}), we can apply Lemma~\ref{lemma: infinite norm of X-T} to prove this lemma.
Combining \eqref{spectral norm of error tensor} and \eqref{L2 inf norm} together gives that
\begin{align*}
    B &= \max_{i = 1,2,3} \left(\twoinf{\mX_i^t\mR_i^t - \mU_i}+\frac{15\sigma_{\max}\lb\calT\rb}{2\sigma^2_{\min}\lb\calT\rb} \twoinf{\mU_i}\cdot \opnorm{\mE_i^{t-1}} \right)\\
    &\leq \frac{1}{2^{20}\kappa^2 \mu^2 r^4 } \frac{1}{2^t}\sqrt{\frac{\mu r}{n}} + \frac{15\sigma_{\max}\lb\calT\rb}{2\sigma^2_{\min}\lb\calT\rb}\cdot \sqrt{\frac{\mu r}{n}}\cdot \frac{1}{2^{20}\kappa^6 \mu^2 r^4}\frac{1}{2^t}\sigma_{\max}(\calT)\leq\frac{9}{2^{20} \kappa^2 \mu^2 r^4 }\cdot \frac{1}{2^t}\sqrt{\frac{\mu r}{n}}.
\end{align*}
 Using the same argument as proving Theorem~\ref{thm: finite iteration} yields that

\begin{align*}
    \infnorm{\calX^t - \calT} 
	&\leq \frac{36}{2^{20} \kappa^2 \mu^2 r^{4}}\cdot \frac{1}{2^t} \cdot \left( \frac{\mu r}{n} \right)^{3/2}\cdot \sigma_{\max}(\calT).
\end{align*}
Similarly, we have
\begin{align*}
    \infnorm{\calX^{t,\ell} - \calT}\leq \frac{36}{2^{20} \kappa^2 \mu^2 r^{4}}\cdot \frac{1}{2^t} \cdot \left( \frac{\mu r}{n} \right)^{3/2}\cdot \sigma_{\max}(\calT).
\end{align*}

\subsection{Proof of Lemma~\ref{lemma: condition number of X}}
We first prove the inequality \eqref{eq condition number of Xt}. Recall that the tensor $\calX^t$ can be rewritten as
\begin{align*}
	\calX^t = \lb \calT + \calE^{t-1} \rb \ttimes {\mX_i^t}{\mX_i^t}^\tran = \calT + \lb \calT + \calE^{t-1} \rb \ttimes {\mX_i^t}{\mX_i^t}^\tran - \calT.
\end{align*}
Applying the Weyl's inequality reveals that
\begin{align*}
	\left| \sigma_k\lb \calM_i\lb\calX^t\rb \rb - \sigma_{k}\lb \calM_i\lb \calT \rb\rb \right| &\leq \opnorm{\calM_i \lb  \lb \calT + \calE^{t-1} \rb \ttimes {\mX_i^t}{\mX_i^t}^\tran - \calT \rb}\\
	&\leq \underbrace{\opnorm{\calM_i \lb  \calT  \ttimes {\mX_i^t}{\mX_i^t}^\tran - \calT \rb}}_{=:\vartheta_1} + \underbrace{\opnorm{\calM_i\lb  \calE^{t-1}   \ttimes {\mX_i^t}{\mX_i^t}^\tran  \rb }}_{=:\vartheta_2}.
\end{align*}
\paragraph{Bounding $\vartheta_1$.} 
Using the same argument as controlling~\eqref{eq I2}, one can obtain
	\begin{align}
		\label{eq condition I1}
		\vartheta_1 
		&\leq 3\sigma_{\max}(\calT)\cdot  \frac{1}{2^{17}\kappa^4 \mu^2 r^4 } \leq \frac{1}{2^{10}} \sigma_{\min}(\calT).
	\end{align}
\paragraph{Bounding $\vartheta_2$.} It is easily seen that
	\begin{align}
		\label{eq: condition I2}
		\vartheta_2 \leq \max_{i = 1,2,3}\opnorm{\calM_i\lb  \calE^{t-1}  \rb } 
		&\stackrel{(a)}{\leq } \frac{1}{2^{20} \kappa^6 \mu^2 r^4} \frac{1}{2^t} \cdot \sigma_{\max}(\calT) \leq \frac{1}{2^{10} }\sigma_{\min}(\calT),
	\end{align}
	where $(a)$ is due to \eqref{spectral norm of error tensor} and the fact $\kappa \geq 1$. 
\\

Combining \eqref{eq condition I1} and \eqref{eq: condition I2} together shows that
\begin{align*}
	\left| \sigma_k\lb \calM_i\lb\calX^t\rb \rb - \sigma_{k}\lb \calM_i\lb \calT \rb\rb \right|  \leq \frac{1}{2^{9}  }\sigma_{\min}(\calT),
\end{align*}
which implies that
\begin{align*}
	\left(1+ \frac{1}{2^9  }\right) \sigma_{\max}(\calT) \geq \sigma_{\max}\lb \calM_i\lb\calX^t\rb \rb  \geq  \sigma_{\min}\lb \calM_i\lb\calX^t\rb \rb \geq \left(1 - \frac{1}{2^9}\right)\sigma_{\min}(\calT),
\end{align*}
as well as
	$\kappa\lb\calX^t\rb \leq 2\kappa.$
The inequality  \eqref{eq condition number of Xtl} can be proved in a similar way and the details are omitted.

\section{Proofs of Claims}

\subsection{Proof of Claim \ref{property of these projection}}\label{subsec:proj_properties}
By the definition of projectors, it is easy to verify that
\begin{align*}
    \calM_i\lb\lb\calP_{{\mX_i^{t, \perp}}}^{(i)}\calP_{\calG^t , \left\lbrace\mX^t_j \right\rbrace_{j\neq i}}^{(j\neq i)}\rb\lb\calZ\rb\rb = \calM_i\lb\lb\calP_{\calG^t , \left\lbrace\mX^t_j \right\rbrace_{j\neq i}}^{(j\neq i)}\calP_{{\mX_i^{t, \perp}}}^{(i)}\rb\lb\calZ\rb\rb,
\end{align*}
which implies that the above two projectors are commutable. Using the same strategy, one can obtain the second equality.

For the third equality, by symmetry, it suffices to prove $\calM_1\lb\calP_{\calG^t , \left\lbrace\mX^t_j \right\rbrace_{j\neq 1}}^{(j\neq 1)}\lb\calZ\rb\rb = \calM_1\lb\calZ\rb\mY_1^t{\mY_1^t}^\tran$. 
Let ${\mQ^t_1}\bSigma_1^t{\mY_1^t}^\tran$ be the singular vector decomposition of $\calM_1\lb\calG^t\rb\lb\mX_3^t\otimes\mX_2^t\rb^\tran$, where $\mQ_1^t\in\R^{r\times r}$, $\bSigma_1^t\in\R^{r\times r}$ and $\mY_1^t\in\R^{n^2\times r}$. We find that
\begin{align*}
    \calM_1\lb\calX^t\rb &= \mX_1^t\calM_1\lb\calG^t\rb\lb\mX_3^t\otimes\mX_2^t\rb^\tran
 = \mX_1^t{\mQ^t_1}\bSigma_1^t{\mY_1^t}^\tran,
\end{align*}
which implies that the columns of  $\mY_1^t$ are the top-$r$ right singular vectors of $\calM_1\lb\calX^t\rb$. A direct calculation gives that
\begin{align}\label{as a projection}
    \calM_1\lb\calP_{\calG^t , \left\lbrace\mX^t_j \right\rbrace_{j\neq 1}}^{(j\neq 1)}\lb\calZ\rb\rb &= \calM_1\lb\calZ\rb \lb\mX^t_3\otimes\mX^t_2\rb\calM_1^{\dagger}\lb\calG^t\rb\calM_1\lb\calG^t\rb\lb\mX^t_3\otimes\mX^t_2\rb^\tran\notag\\
    & \stackrel{(a)}{=}  \calM_1\lb\calZ\rb\lb\calM_1\lb\calG^t\rb\lb\mX^t_3\otimes\mX^t_2\rb^\tran\rb^\dagger\calM_1\lb\calG^t\rb\lb\mX^t_3\otimes\mX^t_2\rb^\tran\notag\\
    &= \calM_1\lb\calZ\rb\lb{\mQ^t_1}\bSigma_1^t{\mY_1^t}^\tran\rb^\dagger{\mQ^t_1}\bSigma_1^t{\mY_1^t}^\tran \notag \\
    &=\calM_1\lb\calZ\rb{\mY_1^t} {\bSigma_1^t}^{-1}{\mQ^t_1}^\tran{\mQ^t_1}\bSigma_1^t{\mY_1^t}^\tran= \calM_1\lb\calZ\rb\mY_1^t{\mY_1^t}^\tran,
\end{align}
where $(a)$ follows from the fact that $\mX_3^t\otimes\mX_2^t$ has orthonormal columns. 

In addition, one has
\begin{align}\label{one more projection no change}
    &\calM_1\lb\calP_{\calG^t , \left\lbrace\mX^t_j \right\rbrace_{j\neq 1}}^{(j\neq 1)}\lb\calZ\rb\rb \notag\\&~=\calM_1\lb\calZ\rb\lb\mX^t_3\otimes\mX^t_2\rb\lb\mX^t_3\otimes\mX^t_2\rb^\tran \lb\mX^t_3\otimes\mX^t_2\rb\calM_1^{\dagger}\lb\calG^t\rb\calM_1\lb\calG^t\rb\lb\mX^t_3\otimes\mX^t_2\rb^\tran \notag\\
    &~= \calM_1\lb\calZ\rb \lb\mX^t_3\otimes\mX^t_2\rb\lb\mX^t_3\otimes\mX^t_2\rb^\tran\mY_1^t{\mY_1^t}^\tran.
\end{align}

The last two inequalities in Claim \ref{property of these projection} had already been used in \citep{cai2020provable}and can be easily verified by definition. Thus we omit the proof here.

\subsection{Proof of Claim \ref{claim: upper bound of beta3a}}
\label{proof: claim upper bound of beta3a}
By the equation \eqref{eq: key equation 1}, we have
\begin{align*}
	\beta_{3,a} &= \twonorm{\ve_m^\tran \calM_1\lb \lb \calI - \calP_{T_{\calX^{t,\ell}}}\rb\lb\calX^{t,\ell} - \calT\rb \rb} \\ 
	&\leq  \sum_{i=1}^{3} \underbrace{ \twonorm{ \ve_m^\tran \calM_1\lb  \lb \calP^{(i)}_{\mU_i} - \calP^{(i)}_{\mX_{i}^{t,\ell}}  \rb \lb \prod_{j\neq i}\calP^{(j)}_{\mX_j^{t, \ell }}   - \calP^{(j\neq i)}_{\calG^{t,\ell},\lcb \mX_{j}^{t,\ell}\rcb_{j\neq i}}\rb  \lb \calX^{t,\ell}- \calT \rb \rb }}_{=:\beta_{3,a}^i}\\
	&\quad + \underbrace{ \twonorm{\ve_m^\tran \calM_1\lb \lb \calP^{(1)}_{\mU_1} - \calP^{(1)}_{\mX_1^{t, \ell }}  \rb \calP^{(2)}_{\mX_{2}^{t,\ell, \perp}} \calP^{(3)}_{\mX_3^{t,\ell}} \lb \calX^{t,\ell} - \calT \rb \rb}  }_{=: \beta_{3,a}^4}\\
	&\quad + \underbrace{ \twonorm{\ve_m^\tran \calM_1 \lb  \lb \calP^{(3)}_{\mU_3} - \calP^{(3)}_{\mX_3^{t, \ell }}  \rb \calP^{(1)}_{\mX_{1}^{t,\ell,  \perp}} \calP^{(2)}_{\mX_2^{t, \ell}}  \lb \calX^{t,\ell} - \calT \rb \rb}}_{=: \beta_{3,a}^5}\\
	&\quad  +\underbrace{ \twonorm{\ve_m^\tran \calM_1 \lb \lb \calP^{(2)}_{\mU_2} - \calP^{(2)}_{\mX_{2}^{t,\ell}}  \rb \calP^{(3)}_{\mX_{3}^{t,\ell, \perp}} \calP^{(1)}_{\mX_1^{t,\ell}}  \lb \calX^{t,\ell} - \calT \rb\rb} }_{=:\beta_{3,a}^6}\\
	&\quad + \underbrace{ \twonorm{\ve_m^\tran \calM_1 \lb \lb \calP^{(1)}_{\mU_1} - \calP^{(1)}_{\mX_{1}^{t,\ell}}  \rb\calP^{(2)}_{\mX_{2}^{t,\ell,\perp}} \calP^{(3)}_{\mX_{3}^{t,\ell, \perp}}  \lb \calX^{t,\ell}- \calT \rb  \rb}}_{=:\beta_{3,a}^7}.
\end{align*}
\paragraph{Bounding $\beta_{3,a}^1$.} A straightforward computation yields that
\begin{align*}
		\beta_{3,a}^1&\leq  \twonorm{\ve_m^\tran \lb  {\mU_1} \mU_1^\tran -  {\mX_1^{t,\ell}}{\mX_1^{t,\ell}}^\tran \rb}\cdot \opnorm{ \prod_{j\neq 1}\calP^{(j)}_{\mX_j^{t, \ell }}   - \calP^{(j\neq 1)}_{\calG^{t,\ell},\lcb \mX_{j}^{t,\ell}\rcb_{j\neq 1}} }\cdot \fronorm{\calX^{t,\ell} - \calT} \\
		&= \twonorm{\ve_m^\tran \lb {\mU_1} \mU_1^\tran -  {\mX_1^{t,\ell}}{\mX_1^{t,\ell}}^\tran \rb} \fronorm{\calX^{t,\ell} - \calT},
	\end{align*}
	where the last line follows from Claim~\ref{property of these projection}.
\paragraph{Bounding $\beta_{3,a}^2$ and $\beta_{3,a}^3$.} It follows from the triangle inequality that
\begin{align*}
		\beta_{3,a}^2& =  \twonorm{ \ve_m^\tran \calM_1\lb  \lb \calP^{(2)}_{\mU_2} - \calP^{(2)}_{\mX_{2}^{t,\ell}}  \rb \lb \prod_{j\neq 2}\calP^{(j)}_{\mX_j^{t, \ell }}   - \calP^{(j\neq 2)}_{\calG^{t,\ell},\lcb \mX_{j}^{t,\ell}\rcb_{j\neq 2}}\rb  \lb \calX^{t,\ell}- \calT \rb \rb }\\
		&\leq  \twonorm{ \ve_m^\tran \calM_1\lb  \lb \calP^{(2)}_{\mU_2} - \calP^{(2)}_{\mX_{2}^{t, \ell}}  \rb \lb \prod_{j\neq 2}\calP^{(j)}_{\mX_j^{t, \ell }} \rb  \lb \calX^{t,\ell}- \calT \rb \rb } \\
		&\quad  +  \twonorm{ \ve_m^\tran \calM_1\lb  \lb \calP^{(2)}_{\mU_2} - \calP^{(2)}_{\mX_{2}^{t, \ell}}  \rb \lb \calP^{(j\neq 2)}_{\calG^{t,\ell},\lcb \mX_{j}^{t,\ell}\rcb_{j\neq 2}}\rb  \lb \calX^{t,\ell}- \calT \rb \rb }.
	\end{align*}
The first term on the right hand side can be bounded by
\begin{align*}
	&\twonorm{ \ve_m^\tran \calM_1\lb  \lb \calP^{(2)}_{\mU_2} - \calP^{(2)}_{\mX_{2}^{t, \ell}}  \rb \lb \prod_{j\neq 2}\calP^{(j)}_{\mX_j^{t, \ell }} \rb  \lb \calX^{t,\ell}- \calT \rb \rb }\\
	&\quad = \twonorm{ \ve_m^\tran \mX_1^{t,\ell} {\mX_1^{t,\ell}}^\tran \calM_1\lb     \calX^{t,\ell}- \calT \rb \left( \left(\mX_3^{t,\ell} {\mX_3^{t,\ell}}^\tran \right) \otimes \left({\mU_2} \mU_2^\tran -  {\mX_2^{t,\ell}}{\mX_2^{t,\ell}}^\tran  \right)\right) }\\
	&\quad \leq \twonorm{ \ve_m^\tran \mX_1^{t,\ell} } \opnorm{\calM_1\lb     \calX^{t,\ell}- \calT \rb }\cdot \opnorm{  \left(\mX_3^{t,\ell} {\mX_3^{t,\ell}}^\tran \right) \otimes \left({\mU_2} \mU_2^\tran -  {\mX_2^{t,\ell}}{\mX_2^{t,\ell}}^\tran \right)  }\\
	&\quad \leq  \twoinf{\mX_1^{t,\ell}} \cdot \opnorm{ {\mU_2} \mU_2^\tran -  {\mX_2^{t,\ell}}{\mX_2^{t,\ell}}^\tran   }\cdot \fronorm{\calX^{t,\ell} - \calT}.
\end{align*}
By the definition of $\calP^{(j\neq 2)}_{\calG^{t,\ell},\lcb \mX_{j}^{t,\ell}\rcb_{j\neq 2}}$ in \eqref{def: projection 2}, we have
\begin{align*}
&\calM_2\lb \calP_{\calG^{t,\ell}, \left\lbrace\mX^{t,\ell}_j \right\rbrace_{j\neq 2}}^{(j\neq 2)}\lb\calX^{t,\ell} - \calT\rb  \rb\\ &\quad= \calM_2\lb\calX^{t,\ell} - \calT\rb \lb \mX^{t,\ell}_{3}\otimes \mX^{t,\ell}_{1} \rb \calM_2^\dagger\lb\calG^{t,\ell}\rb\calM_2\lb\calG^{t,\ell}\rb \lb \mX^{t,\ell}_{3}\otimes \mX^{t,\ell}_{1}\rb^\tran\\
&\quad=\mW_2^{t,\ell}\calM_2\lb\calG^{t,\ell}\rb \lb \mX^{t,\ell}_{3}\otimes \mX^{t,\ell}_{1}\rb^\tran=\calM_2\left( \calG^{t,\ell} \times_1 \mX^{t,\ell}_{1} \times_2 \mW_2^{t,\ell} \times_3 \mX^{t,\ell}_{3}\right),
\end{align*}
where $\mW_2^{t,\ell} = \calM_2(\calX^{t,\ell} - \calT) \lb \mX^{t,\ell}_{3}\otimes \mX^{t,\ell}_{1} \rb \calM_2^\dagger\lb\calG^{t,\ell}\rb$.  It follows that
\begin{align*}
 \calP_{\calG^{t,\ell}, \left\lbrace\mX^{t,\ell}_j \right\rbrace_{j\neq 2}}^{(j\neq 2)}\lb\calX^{t,\ell} - \calT\rb  =  \calG^{t,\ell} \times_1 \mX^{t,\ell}_{1} \times_2 \mW_2^{t,\ell} \times_3 \mX^{t,\ell}_{3}.
\end{align*}
Thus one has 
\begin{align}\label{second term on right hand side}
	 &\twonorm{ \ve_m^\tran \calM_1\lb  \lb \calP^{(2)}_{\mU_2} - \calP^{(2)}_{\mX_{2}^{t, \ell}}  \rb \lb \calP^{(j\neq 2)}_{\calG^{t,\ell},\lcb \mX_{j}^{t,\ell}\rcb_{j\neq 2}}\rb  \lb \calX^{t,\ell}- \calT \rb \rb }\notag\\
	 &\quad \leq  \twonorm{ \ve_m^\tran \calM_1\lb \lb \calP^{(j\neq 2)}_{\calG^{t,\ell},\lcb \mX_{j}^{t,\ell}\rcb_{j\neq 2}}\rb  \lb \calX^{t,\ell}- \calT \rb \rb }\cdot \opnorm{ {\mU_2} \mU_2^\tran -  {\mX_2^{t,\ell}}{\mX_2^{t,\ell}}^\tran  } \notag\\
	&\quad = \twonorm{ \ve_m^\tran \calM_1\left( \calG^{t,\ell} \times_1\mX_1^{t,\ell}\times_2 \mW_2^{t,\ell} \times_3\mX_3^{t,\ell} \right)}\cdot \opnorm{{\mU_2} \mU_2^\tran -  {\mX_2^{t,\ell}}{\mX_2^{t,\ell}}^\tran   } \notag\\
	&\quad \leq  \twoinf{\mX_1^{t,\ell}}\cdot \sigma_{1}\left(\calM_1\left(\calX^{t,\ell}\right)\right)\cdot  \opnorm{\mW_2^{t,\ell}}\cdot \opnorm{ {\mU_2} \mU_2^\tran -  {\mX_2^{t,\ell}}{\mX_2^{t,\ell}}^\tran  }\notag\\
 &\quad\leq \twoinf{\mX_1^{t,\ell}}\cdot \kappa\lb\calX^{t,\ell}\rb\cdot  \fronorm{\calX^{t,\ell} - \calT}\cdot \opnorm{ {\mU_2} \mU_2^\tran -  {\mX_2^{t,\ell}}{\mX_2^{t,\ell}}^\tran   }
\end{align}
where the last line follows from
\begin{align*}
	\opnorm{\mW_2^{t,\ell}} &\leq \opnorm{\calM_2\lb\calX^{t,\ell} - \calT\rb}\cdot \opnorm{ \calM_2^\dagger\left(\calG^{t,\ell}\right)} \leq \fronorm{\calX^{t,\ell} - \calT} \cdot \frac{1}{\sigma_{r}\lb\calM_2\lb\calX^{t,\ell}\rb\rb  }.
\end{align*}
Consequently,
\begin{align*}
	\beta_{3,a}^2 	&\leq \twoinf{\mX_1^{t,\ell}} \cdot \opnorm{{\mU_2} \mU_2^\tran -  {\mX_2^{t,\ell}}{\mX_2^{t,\ell}}^\tran    }\cdot \fronorm{\calX^{t,\ell} - \calT} \cdot \lb 1+ \kappa\lb \calX^{t,\ell}\rb\rb \\
	&\leq \twoinf{\mX_1^{t,\ell}} \cdot \opnorm{ {\mU_2} \mU_2^\tran -  {\mX_2^{t,\ell}}{\mX_2^{t,\ell}}^\tran   }\cdot \fronorm{\calX^{t,\ell} - \calT} \cdot (1+ 2\kappa ).
\end{align*}
Using the same argument, one can obtain
	\begin{align*}
		\beta_{3,a}^3	&\leq \twoinf{\mX_1^{t,\ell}} \cdot \opnorm{ {\mU_3} \mU_3^\tran -  {\mX_3^{t,\ell}}{\mX_3^{t,\ell}}^\tran     }\cdot \fronorm{\calX^{t,\ell} - \calT}  \cdot (1+ 2\kappa ).
	\end{align*}
\paragraph{Bounding $\beta_{3,a}^j$ for $j = 4,5,6,7$.} A straightforward computation yields that
	\begin{align*}
		\beta_{3,a}^4 & \leq \twonorm{\ve_m^\tran \lb  {\mU_1}\mU_1^\tran -  {\mX_1^{t,\ell}}{\mX_1^{t,\ell}}^\tran  \rb}\cdot \fronorm{\calX^{t,\ell} - \calT}, \\
		\beta_{3,a}^5 & \leq  \twoinf{\mX_1^{t,\ell}} \cdot \opnorm{ {\mU_3}\mU_3^\tran -  {\mX_3^{t,\ell}}{\mX_3^{t,\ell}}^\tran } \cdot \fronorm{\calX^{t,\ell} - \calT}, \\
		\beta_{3,a}^6 & \leq  \twoinf{\mX_1^{t,\ell}} \cdot \opnorm{ {\mU_2}\mU_2^\tran -  {\mX_2^{t,\ell}}{\mX_2^{t,\ell}}^\tran } \cdot \fronorm{\calX^{t,\ell} - \calT}, \\
		\beta_{3,a}^7 & \leq \twonorm{\ve_m^\tran \lb {\mU_1}\mU_1^\tran -  {\mX_1^{t,\ell}}{\mX_1^{t,\ell}}^\tran \rb}\cdot \fronorm{\calX^{t,\ell} - \calT}.
	\end{align*}
\\

Combining all the bounds together shows that
\begin{small}
\begin{align*}
	\beta_{3,a} &\leq\lb 3\twonorm{\ve_m^\tran \lb {\mU_1}\mU_1^\tran -  {\mX_1^{t,\ell}}{\mX_1^{t,\ell}}^\tran \rb} + 8\kappa\twoinf{\mX_1^{t,\ell}}   \max_{i = 2,3}\opnorm{ {\mU_i}\mU_i^\tran -  {\mX_i^{t,\ell}}{\mX_i^{t,\ell}}^\tran}  \rb\fronorm{\calX^{t,\ell} - \calT} \\
	& \stackrel{(a)}{\leq } \left(   \frac{3}{2^{16} \kappa^2  \mu^2 r^4 }\frac{1}{2^t}\sqrt{\frac{\mu r}{n}} +   \frac{16\kappa}{2^{16} \kappa^4 \mu^2 r^4 }\frac{1}{2^t}\sqrt{\frac{\mu r}{n}}  \right) \cdot n^{3/2}\cdot \frac{36}{2^{20} \kappa^2 \mu^2 r^{4}}\cdot \frac{1}{2^t} \cdot \left( \frac{\mu r}{n} \right)^{3/2}\cdot \sigma_{\max}(\calT) \\
	 &\leq \frac{19\cdot 36}{2^{16}}\cdot  \frac{1}{2^{20} \kappa^4 \mu^2 r^4 }\frac{1}{2^{t+1}} \sqrt{\frac{\mu r}{n}}\sigma_{\max}(\calT)\leq \frac{1}{2^{6}}\cdot  \frac{1}{2^{20} \kappa^4 \mu^2 r^4 }\frac{1}{2^{t+1}} \sqrt{\frac{\mu r}{n}}\sigma_{\max}(\calT),
\end{align*}
\end{small}
where  $(a)$ follows from Lemma~\ref{lemma: spectral norm of Xt}.

\subsection{Proof of Claim~\ref{claim: upper bound of beta3b}}
\label{proof: claim upper bound of beta3b}
Let $\calZ^{t,\ell} := \lb \calI -p^{-1}  \calP_{\Omega_{-\ell}} -\calP_{\ell}  \rb\lb\calX^{t,\ell} - \calT\rb $. By the definition of $\calP_{T_{\calX^{t,\ell}}}$ in \eqref{eq: def of projection onto tangent space}, the term $\beta_{3,b}$ can be bounded as follows:
\begin{align*}
	\beta_{3,b} &= \twonorm{\ve_m^\tran \calM_1 \lb \calZ^{t,\ell }\ttimes {\mX_i^{t,\ell}}{\mX_i^{t,\ell}}^\tran + \sum_{i=1}^3 \calG^{t,\ell}\times_i \mW_i^{t,\ell}\jneqi \mX_j^{t,\ell} \rb}\\
	&\leq \sum_{i=1}^3 \underbrace{ \twonorm{\ve_m^\tran \calM_1 \lb \calG^{t,\ell}\times_i \mW_i^{t,\ell}\jneqi \mX_j^{t,\ell}\rb}}_{=:\beta_{3,b}^i} + \underbrace{ \twonorm{\ve_m^\tran \calM_1 \lb \calZ^{t,\ell }\ttimes {\mX_i^{t,\ell}}{\mX_i^{t,\ell}}^\tran  \rb}}_{=:\beta_{3,b}^{4} },
\end{align*}
where $\mW_i^{t,\ell}$ is given by
	$\mW_i^{t,\ell} = \left( \mI - \mX_i^{t,\ell}{\mX_i^{t,\ell}}^\tran\right) \calM_i\left(\calZ^{t,\ell} \jneqi { \mX_j^{t,\ell}}^\tran \right) \calM_i^\dagger \lb\calG^{t,\ell}\rb.$

\paragraph{Bounding $\beta_{3,b}^1$.} 
It can be bounded as follows:
\begin{small}
\begin{align*}
		\beta_{3,b}^1 &=  \twonorm{\ve_m^\tran \calM_1 \lb \calG^{t,\ell}\times_1 \mW_1^{t,\ell}\times_2 \mX_2^{t,\ell} \times_3 \mX_3^{t,\ell}\rb}= \twonorm{\ve_m^\tran\mW_1^{t,\ell}  \calM_1 \left( \calG^{t,\ell} \right)\left( \mX_3^{t,\ell} \otimes  \mX_2^{t,\ell}  \right)^\tran }\\
		&= \twonorm{\ve_m^\tran\left( \mI- \mX_1^{t,\ell} {\mX_1^{t,\ell}}^\tran  \right) \calM_1\lb\calZ^{t,\ell}\rb  \left( \mX_3^{t,\ell} \otimes  \mX_2^{t,\ell}  \right)  \calM_1^\dagger \lb\calG^{t,\ell}\rb \calM_1 \left( \calG^{t,\ell} \right)\left( \mX_3^{t,\ell} \otimes  \mX_2^{t,\ell}  \right)^\tran}\\
		&\leq \twonorm{\ve_m^\tran \calM_1\lb\calZ^{t,\ell}\rb \left( \mX_3^{t,\ell} \otimes  \mX_2^{t,\ell}  \right)  \calM_1^\dagger \lb\calG^{t,\ell}\rb \calM_1 \left( \calG^{t,\ell} \right)\left( \mX_3^{t,\ell} \otimes  \mX_2^{t,\ell}  \right)^\tran} \\
		&\quad +\twonorm{\ve_m^\tran  \mX_1^{t,\ell} {\mX_1^{t,\ell}}^\tran  \calM_1\lb\calZ^{t,\ell}\rb  \left( \mX_3^{t,\ell} \otimes  \mX_2^{t,\ell}  \right)  \calM_1^\dagger \lb\calG^{t,\ell}\rb \calM_1 \left( \calG^{t,\ell} \right)\left( \mX_3^{t,\ell} \otimes  \mX_2^{t,\ell}  \right)^\tran}\\
		&\stackrel{(a)}{\leq} \twonorm{\ve_m^\tran \calM_1\lb \calZ^{t,\ell}\rb \left(\mX_3^{t,\ell} \otimes \mX_2^{t,\ell}\right)}  + \twonorm{\ve_m^\tran  \mX_1^{t,\ell}}\cdot \opnorm{\calM_1\lb\calZ^{t,\ell}\rb  \left( \mX_3^{t,\ell} \otimes  \mX_2^{t,\ell}  \right) }\\
		&=\twonorm{\ve_m^\tran \calM_1\lb \calZ^{t,\ell}\rb \left(\mX_3^{t,\ell} \otimes \mX_2^{t,\ell}\right)}  + \twonorm{\ve_m^\tran  \mX_1^{t,\ell}}\cdot \opnorm{\calM_1 \left(\calZ^{t,\ell}\times_2 {\mX_2^{t,\ell}}^\tran \times_3{\mX_3^{t,\ell}}^\tran \right) }\\
		&\stackrel{(b)}{\leq}  \twonorm{\ve_m^\tran \calM_1\lb \calZ^{t,\ell}\rb \left(\mX_3^{t,\ell} \otimes \mX_2^{t,\ell}\right)}  + \sqrt{r} \twoinf{  \mX_1^{t,\ell}} \opnorm{\calZ^{t,\ell}\times_2 {\mX_2^{t,\ell}}^\tran \times_3{\mX_3^{t,\ell}}^\tran}\\
		&\leq \twonorm{\ve_m^\tran \calM_1\lb \calZ^{t,\ell}\rb \left(\mX_3^{t,\ell} \otimes \mX_2^{t,\ell}\right)}  + \sqrt{r} \twoinf{  \mX_1^{t,\ell}} \opnorm{\calZ^{t,\ell}},
	\end{align*}
 \end{small}
where step (a) is due to the fact $\opnorm{\calM_1^\dagger \left(\calG^{t,\ell}\right) \calM_1 \left( \calG^{t,\ell}\right)}\leq 1$ and step (b) has used Lemma~\ref{lemma: Lemma 6 in XYZ}.
\paragraph{Bounding $\beta_{3,b}^2$ and $\beta_{3,b}^3$.}
A direct calculation yields that
	\begin{align*}
		\beta_{3,b}^2 &=\twonorm{\ve_m^\tran \calM_1 \lb \calG^{t,\ell}\times_1 \mX_1^{t,\ell} \times_2 \mW_2^{t,\ell} \times_3 \mX_3^{t,\ell} \rb } \\
		&= \twonorm{\ve_m^\tran \mX_1^{t,\ell}   \calM_1 \lb \calG^{t,\ell} \rb \left(  \mX_3^{t,\ell } \otimes \mW_2^{t,\ell} \right)^\tran  } \leq \twoinf{\mX_1^{t,\ell}} \cdot \opnorm{  \calM_1 \lb \calG^{t,\ell} \rb }\cdot \opnorm{\mW_2^{t,\ell}}\\
		&\leq \twoinf{\mX_1^{t,\ell}} \cdot \opnorm{  \calM_1 \lb \calG^{t,\ell} \rb }\cdot  \opnorm{\calM_2\left(\calZ^{t,\ell} \underset{j\neq 2}{ \times } {\mX_j^{t,\ell}}^\tran \right) } \cdot \opnorm{\calM_2^\dagger\left(\calG^{t,\ell}\right)}\\
		&\leq  \twoinf{\mX_1^{t,\ell}} \cdot \sqrt{r} \opnorm{ \calZ^{t,\ell}} \cdot \kappa\left(\calX^{t,\ell}\right)\leq 2\kappa\sqrt{r} \twoinf{\mX_1^{t,\ell}}\cdot \opnorm{ \calZ^{t,\ell}},
	\end{align*}
	where the last line is due to $\kappa\lb\calX^{t,\ell}\rb\leq 2\kappa$, see Lemma~\ref{lemma: condition number of X}.
	Using the same argument, one can obtain
	\begin{align*}
	\beta_{3,b}^3 &=  \twonorm{\ve_m^\tran \calM_1 \lb \calG^{t,\ell} \times_1 \mX_1^{t,\ell} \times_2 \mX_2^{t,\ell}  \times_3 \mW_3^{t,\ell} \rb} \leq 2\kappa\sqrt{r} \twoinf{\mX_1^{t,\ell}}\cdot \opnorm{ \calZ^{t,\ell}}.
\end{align*}
\paragraph{Bounding $\beta_{3,b}^4$.} It can be bounded as follows:
\begin{align*}
	\beta_{3,b}^4 &=  \twonorm{\ve_m^\tran \calM_1 \lb \calZ^{t,\ell }\ttimes {\mX_i^{t,\ell}}{\mX_i^{t,\ell}}^\tran  \rb}=\twonorm{\ve_m^\tran \calM_1 \lb \calZ^{t,\ell }\ttimes {\mX_i^{t,\ell}}^\tran \ttimes \mX_i^{t,\ell} \rb}\\ 
	&= \twonorm{\ve_m^\tran \mX_1^{t,\ell}\calM_1 \lb \calZ^{t,\ell }\ttimes {\mX_i^{t,\ell}}^\tran\rb \left( \mX_3^{t,\ell} \otimes \mX_2^{t,\ell} \right) ^\tran}\leq \twoinf{ \mX_1^{t,\ell}}  \opnorm{ \calM_1 \lb \calZ^{t,\ell }\ttimes {\mX_i^{t,\ell}}^\tran\rb } \\
	&\stackrel{(a)}{\leq} \sqrt{r} \twoinf{\mX_1^{t,\ell}}\cdot \opnorm{ \calZ^{t,\ell }\ttimes {\mX_i^{t,\ell}}^\tran}\leq \sqrt{r} \twoinf{\mX_1^{t,\ell}}\cdot \opnorm{ \calZ^{t,\ell}},
\end{align*}
where step (a) is due to Lemma~\ref{lemma: Lemma 6 in XYZ}.
\\

Combining these four terms together reveals that
\begin{align}
	\label{eq beta3b 0}
	\beta_{3,b} &\leq   \underbrace{ \twonorm{\ve_m^\tran \calM_1\lb \calZ^{t,\ell}\rb(\mX_3^{t,\ell} \otimes \mX_2^{t,\ell})} }_{=:\gamma} +6\kappa  \twoinf{  \mX_1^{t,\ell}}\cdot \sqrt{r}\opnorm{\calZ^{t,\ell}}.
\end{align}
For the second term on the right hand side, we apply Lemma~\ref{lemma: uniform tensor operator norm} to show that
\begin{align}
	\label{eq 2th term in beta3b}
	6\kappa  \twoinf{  \mX_1^{t,\ell}}\cdot \sqrt{r} \opnorm{\calZ^{t,\ell}} &=6\kappa  \twoinf{  \mX_1^{t,\ell}}\cdot \sqrt{r}  \opnorm{   \lb \calI -p^{-1}  \calP_{\Omega_{-\ell}} -\calP_{\ell}  \rb(\calX^{t,\ell} - \calT)  } \notag \\
	&\leq  12\kappa  \sqrt{\frac{\mu r}{n}} \sqrt{r} \opnorm{   \lb \calI -p^{-1}  \calP_{\Omega }    \rb(\calX^{t,\ell} - \calT)  } \notag \\
	&\leq 12\kappa  \sqrt{\frac{\mu r}{n}}  \sqrt{r} \opnorm{   \lb \calI -p^{-1}  \calP_{\Omega }    \rb(\calJ)  } \cdot 2r \infnorm{\calX^{t,\ell} - \calT}\notag \\
	&\leq \frac{1}{2^{9} }\frac{1}{2^{20}\kappa^4\mu^2 r^4} \frac{1}{2^{t+1}} \sqrt{\frac{\mu r}{n}} \sigma_{\max}(\calT),
\end{align}
where the second line is due to $\twoinf{\mX_1^{t,\ell}}\leq \twoinf{\mX_1^{t,\ell}\mR_1^{t,\ell} - \mU_1} + \twoinf{\mU_1}\leq 2\sqrt{\frac{\mu r}{n}}$ and the last step has used the assumption
	$p \geq \max\left\{ \frac{ C_1 \kappa^3\mu^{1.5}r^{3}\log^3n }{n^{3/2}}, \frac{C_2 \kappa^6\mu^3 r^6\log^5n }{n^2} \right\}$.

Next, we turn to control the term $\gamma$. 
Recall that $\calZ^{t,\ell} = \lb \calI - p^{-1}\calP_{\Omega_{-\ell}} - \calP_{\ell}\rb(\calX^{t,\ell} - \calT)$. Letting 
\begin{align*}
	\calZ^{t,m} :&= \lb \calI - p^{-1}\calP_{\Omega_{-\ell}} - \calP_{\ell}\rb(\calX^{t,m} - \calT),\\
	\calZ^{t,m,\ell}:&= \lb \calI - p^{-1}\calP_{\Omega_{-\ell}} - \calP_{\ell}\rb(\calX^{t,\ell} -\calX^{t,m}),
\end{align*}
we have $\calZ^{t,\ell} = \calZ^{t,m} + \calZ^{t,m,\ell}$. If $m = \ell$, it is not hard to see that $\gamma = 0$. Thus,  without loss of generality, we assume  $m\neq \ell$. Denoting by $\mO_i = \mR_i^{t,m} {\mR_i^{t,\ell}}^\tran$ the special orthogonal matrix,  the triangle inequality gives that
\begin{small}
\begin{align*}
	\gamma &= \twonorm{\ve_m^\tran \calM_1\lb \calZ^{t,\ell}\rb\lb\mX_3^{t,\ell} \otimes \mX_2^{t,\ell}\rb} \\
	&\leq \twonorm{\ve_m^\tran \calM_1\lb \calZ^{t,\ell}\rb\lb \mX_3^{t,\ell} \otimes \mX_2^{t,\ell}  - \mX_3^{t,m}\mO_3  \otimes \mX_2^{t,m}\mO_2\rb} +  \twonorm{\ve_m^\tran \calM_1\lb \calZ^{t,\ell}\rb\lb \mX_3^{t,m}\mO_3  \otimes \mX_2^{t,m}\mO_2\rb}\\
	&\leq  \twonorm{\ve_m^\tran \calM_1\lb \calZ^{t,\ell}\rb\lb \mX_3^{t,\ell} \otimes \mX_2^{t,\ell}  - \mX_3^{t,m}\mO_3  \otimes \mX_2^{t,m}\mO_2\rb} +  \twonorm{\ve_m^\tran \calM_1\lb \calZ^{t,m}\rb\lb \mX_3^{t,m}\mO_3  \otimes \mX_2^{t,m}\mO_2\rb}\\
	&\quad +  \twonorm{\ve_m^\tran \calM_1\lb \calZ^{t,m, \ell}\rb\lb \mX_3^{t,m}\mO_3  \otimes \mX_2^{t,m}\mO_2\rb}\\
	& =\twonorm{\ve_m^\tran \calM_1\lb \calZ^{t,\ell}\rb\lb \mX_3^{t,\ell} \otimes \mX_2^{t,\ell}  - \mX_3^{t,m}\mO_3  \otimes \mX_2^{t,m}\mO_2\rb}  +  \twonorm{\ve_m^\tran \calM_1\lb \calZ^{t,m}\rb\lb \mX_3^{t,m} \otimes \mX_2^{t,m} \rb} \\
 &\quad+  \twonorm{\ve_m^\tran \calM_1\lb \calZ^{t,m, \ell}\rb\lb \mX_3^{t,m} \otimes \mX_2^{t,m} \rb}\\
	:&=\gamma_1 + \gamma_2 + \gamma_3.
\end{align*}
\end{small}
For the sake of clarity, we first give the upper bounds for the above three terms whose proofs can be found in Sections~\ref{proof of gamma 1}, \ref{proof of gamma 2} and \ref{proof of gamma 3}:
\begin{align}
\label{eq upper bound of gamma1}
\gamma_1 &\leq  \frac{1}{2^{11} }\frac{1}{2^{20}\kappa^4\mu^2 r^4} \frac{1}{2^{t+1}} \sqrt{\frac{\mu r}{n}} \sigma_{\max}(\calT),\\
\label{eq upper bound of gamma2}
\gamma_2 &\leq  \frac{1}{2^{11} }\frac{1}{2^{20}\kappa^4\mu^2 r^4} \frac{1}{2^{t+1}} \sqrt{\frac{\mu r}{n}} \sigma_{\max}(\calT),\\
\label{eq upper bound of gamma3}
\gamma_3&\leq  \frac{1}{2^{11} }\frac{1}{2^{20}\kappa^4\mu^2 r^4} \frac{1}{2^{t+1}} \sqrt{\frac{\mu r}{n}} \sigma_{\max}(\calT).
\end{align}
Thus we have
\begin{align}
	\label{eq 1th term in beta3b}
	\gamma &\leq \gamma_1 + \gamma_2 + \gamma_3 \leq  \frac{3}{2^{11} }\frac{1}{2^{20}\kappa^4\mu^2 r^4} \frac{1}{2^{t+1}} \sqrt{\frac{\mu r}{n}} \sigma_{\max}(\calT).
\end{align}\\

Plugging  \eqref{eq 1th term in beta3b} and \eqref{eq 2th term in beta3b} into \eqref{eq beta3b 0} yields that
\begin{align*}
	\beta_{3,b} &\leq \frac{1}{2^{6} }\frac{1}{2^{20}\kappa^4\mu^2 r^4} \frac{1}{2^{t+1}} \sqrt{\frac{\mu r}{n}} \sigma_{\max}(\calT).
\end{align*}

\subsubsection{Proof of \texorpdfstring{\eqref{eq upper bound of gamma1}}{TEXT}}
\label{proof of gamma 1}
A simple calculation gives that
\begin{align*}
	\gamma_1 &=\twonorm{\ve_m^\tran \calM_1\lb \calZ^{t,\ell}\rb\lb \mX_3^{t,\ell} \otimes \mX_2^{t,\ell}  - \mX_3^{t,m}\mR_3^{t,m}{\mR_3^{t,\ell}}^\tran  \otimes \mX_2^{t,m}\mR_2^{t,m}{\mR_2^{t,\ell}}^\tran \rb} \\
	&=\twonorm{\ve_m^\tran \calM_1\lb \calZ^{t,\ell}\rb\lb \mX_3^{t,\ell} \otimes \mX_2^{t,\ell}  - \mX_3^{t,m}\mR_3^{t,m}{\mR_3^{t,\ell}}^\tran  \otimes \mX_2^{t,m}\mR_2^{t,m}{\mR_2^{t,\ell}}^\tran \rb \left( \mR_3^{t,\ell} \otimes \mR_2^{t,\ell} \right)} \\
	&=\twonorm{\ve_m^\tran \calM_1\lb \calZ^{t,\ell}\rb\lb \mX_3^{t,\ell} \mR_3^{t,\ell} \otimes \mX_2^{t,\ell} \mR_2^{t,\ell} - \mX_3^{t,m}\mR_3^{t,m}  \otimes \mX_2^{t,m}\mR_2^{t,m}\rb} \\ 
	&\leq \underbrace{ \twonorm{\ve_m^\tran \calM_1\lb \calZ^{t,\ell}\rb\lb \mX_3^{t,\ell} \mR_3^{t,\ell} \otimes \mX_2^{t,\ell} \mR_2^{t,\ell} - \mU_3\otimes \mU_2\rb}  }_{=:\gamma_{1,a}}\\
	&\quad + \underbrace{\twonorm{\ve_m^\tran \calM_1\lb \calZ^{t,\ell}\rb\lb \mX_3^{t,m} \mR_3^{t,m} \otimes \mX_2^{t,m} \mR_2^{t,m} - \mU_3\otimes \mU_2\rb} }_{=:\gamma_{1,b}}.
\end{align*}
For the first term  $\gamma_{1,a}$, it can be bounded as follows:
\begin{small}
	\begin{align}
	\label{eq: sum of delta}
		\gamma_{1,a}&=\sqrt{\sum_{s=1}^{r^2} \lb \sum_{j=1}^{n^2} \lsb \calM_1\lb \calZ^{t,\ell}\rb \rsb_{m,j} \lsb\mX_3^{t,\ell} \mR_3^{t,\ell} \otimes \mX_2^{t,\ell} \mR_2^{t,\ell} - \mU_3\otimes \mU_2 \rsb_{j,s} \rb^2} \notag \\
		&=\sqrt{\sum_{s=1}^{r^2} \lb \sum_{j\notin\Gamma} \lb p^{-1}\delta_{m,j}-1 \rb \lsb \calM_1\lb \calX^{t,\ell} - \calT\rb \rsb_{m,j} \lsb\mX_3^{t,\ell} \mR_3^{t,\ell} \otimes \mX_2^{t,\ell} \mR_2^{t,\ell} - \mU_3\otimes \mU_2 \rsb_{j,s} \rb^2} \notag \\
		&\leq r \max_{s\in [r^2]}  \left| \sum_{j\notin \Gamma} \lb p^{-1}\delta_{m,j}-1 \rb \lsb \calM_1\lb \calX^{t,\ell} - \calT\rb \rsb_{m,j} \lsb\mX_3^{t,\ell} \mR_3^{t,\ell} \otimes \mX_2^{t,\ell} \mR_2^{t,\ell} - \mU_3\otimes \mU_2 \rsb_{j,s} \right|  \notag \\
		&\leq r \max_{s\in [r^2]}  \sum_{j\notin \Gamma} \left|  p^{-1}\delta_{m,j}-1   \right| \cdot  \left|\lsb \calM_1\lb \calX^{t,\ell} - \calT\rb \rsb_{m,j} \right| \cdot \left| \lsb\mX_3^{t,\ell} \mR_3^{t,\ell} \otimes \mX_2^{t,\ell} \mR_2^{t,\ell} - \mU_3\otimes \mU_2 \rsb_{j,s} \right| \notag  \\
		&\leq r \sum_{ j\notin \Gamma}  \lb p^{-1}\delta_{m,j}+1 \rb \cdot \infnorm{\calX^{t,\ell} - \calT} \cdot  \infnorm{\mX_3^{t,\ell} \mR_3^{t,\ell} \otimes \mX_2^{t,\ell} \mR_2^{t,\ell} - \mU_3\otimes \mU_2} \notag \\
		&\stackrel{(a)}{\leq} r\cdot 3n^2\cdot \infnorm{\calX^{t,\ell} - \calT} \cdot \twoinf{ \mX_3^{t,\ell} \mR_3^{t,\ell} \otimes \mX_2^{t,\ell} \mR_2^{t,\ell} - \mU_3\otimes \mU_2 }\\
		&\leq r\cdot 3n^2\cdot \infnorm{\calX^{t,\ell} - \calT}  \lb\twoinf{ \mX_3^{t,\ell} \mR_3^{t,\ell} - \mU_3}  \twoinf{\mX_2^{t,\ell}} +\twoinf{ \mX_2^{t,\ell} \mR_2^{t,\ell} - \mU_2}  \twoinf{\mU_3}\rb \notag \notag \\
		&\stackrel{(b)}{\leq} 3n^2r \cdot \frac{36}{2^{20} \kappa^2 \mu^2 r^{4}} \frac{1}{2^t}  \left( \frac{\mu r}{n} \right)^{3/2}\sigma_{\max}(\calT)\cdot  \frac{1}{2^{20}\kappa^2\mu^2 r^4} \frac{1}{2^t} \sqrt{\frac{\mu r}{n}} \cdot \frac{7}{3}\sqrt{\frac{\mu r}{n}} \notag \\
		&\leq \frac{1}{2^{12} }\frac{1}{2^{20}\kappa^4\mu^2 r^4} \frac{1}{2^{t+1}} \sqrt{\frac{\mu r}{n}} \sigma_{\max}(\calT). \notag
	\end{align}
 \end{small}
	Here in $(a)$, we use the fact that whenever $p\geq \frac{C_1\log n}{n^{3/2}},$ with high probability, all rows have at most $2n^2p$ observed entries \citep[Theorem 1]{arratia1989tutorial}. Step (b) has used the fact that
	\begin{align}\label{eq: xitl}
	&\twoinf{\mX_i^{t,\ell} \mR_i^{t,\ell} - \mU_i} \leq \frac{1}{2^{20}\kappa^2\mu^2 r^4} \frac{1}{2^t} \sqrt{\frac{\mu r}{n}},\notag\\
	   &\twoinf{\mX_i^{t,\ell} }\leq \twoinf{\mX_i^{t,\ell} \mR_i^{t,\ell} - \mU_i} + \twoinf{\mU_i}\leq\frac{9}{8}\sqrt{\frac{\mu r}{n}}\leq \frac{4}{3}\sqrt{\frac{\mu r}{n}}.
	\end{align}
	Similarly, one can obtain
	\begin{align*}
		\gamma_{1,b} \leq \frac{1}{2^{12} }\frac{1}{2^{20}\kappa^4\mu^2 r^4} \frac{1}{2^{t+1}} \sqrt{\frac{\mu r}{n}} \sigma_{\max}(\calT).
	\end{align*}
Combining the above bounds together gives that
\begin{align}
	\label{eq gamm1}
	\gamma_1 &\leq \frac{1}{2^{11} }\frac{1}{2^{20}\kappa^4\mu^2 r^4} \frac{1}{2^{t+1}} \sqrt{\frac{\mu r}{n}} \sigma_{\max}(\calT).
\end{align}

\subsubsection{Proof of \texorpdfstring{\eqref{eq upper bound of gamma2}}{TEXT}}
\label{proof of gamma 2}
Since $\calX^{t,m}$ and $\mX_i^{t,m}$ are independent of $\{\delta_{m,j}\}$ by construction, for any fixed $s\in [r^2]$,
\begin{align*}
    &\sum_{j=1}^{n^2}  \lsb \calM_1(\calZ^{t,m})\rsb_{m,j} \lsb \mX_3^{t,m}  \otimes \mX_{2}^{t,m}  \rsb_{j,s} \\
    &\quad= \sum_{j\notin \Gamma_m}  \lb p^{-1}\delta_{m,j} - 1\rb \lsb \calM_1(\calX^{t,m} - \calT)\rsb_{m,j} \lsb \mX_3^{t,m}  \otimes \mX_{2}^{t,m}  \rsb_{j,s}:= \sum_{j\notin \Gamma_m} X_{j},
\end{align*}
where $\Gamma_m$ is an index set defined by
\begin{small}
\begin{align*}
    \Gamma_m = \left\{ m, n+m, \cdots, n(m-2)+m, n(m-1)+1,\cdots, n(m-1)+n, nm+m,\cdots, n(n-1)+m\right\}.
\end{align*}
\end{small}
 Thus,
\begin{align*}
	 \gamma_2 &= \sqrt{\sum_{s=1}^{r^2} \lb \sum_{j=1}^{n^2}  \lsb \calM_1(\calZ^{t,m})\rsb_{m,j} \lsb \mX_3^{t,m}   \otimes \mX_{2}^{t,m}  \rsb_{j,s} \rb^2 }\\
	&\leq r\max_{1\leq s\leq r^2} \left| \sum_{j=1}^{n^2}  \lsb \calM_1(\calZ^{t,m})\rsb_{m,j} \lsb \mX_3^{t,m}   \otimes \mX_{2}^{t,m}  \rsb_{j,s}  \right| = r \max_{1\leq s\leq r^2 } \left| \sum_{j\notin \Gamma_m}  X_j\right|.
\end{align*}
A direct calculation gives that
\begin{align*}
	|X_j| &\leq p^{-1}\cdot \infnorm{\calX^{t,m} - \calT}\cdot \lb \max_{i = 1,2,3}\twoinf{\mX_i^{t,m}} \rb^2, \\
	\left| \sum_{ j\notin \Gamma_m} \E{X_j^2} \right| &\leq p^{-1}\sum_{ j\notin \Gamma_m}\left| \lsb \calM_1(\calX^{t,m} - \calT)\rsb_{m,j} \right|^2 \cdot \left|  \lsb \mX_3^{t,m}   \otimes \mX_{2}^{t,m} \rsb_{j,s}  \right|^2\\
	&\leq p^{-1} \lb \max_{i = 1,2,3}\twoinf{\mX_i^{t,m}} \rb^4 \cdot n^2 \infnorm{  \calX^{t,m} - \calT }^2.
\end{align*}
Applying the matrix Bernstein inequality and taking a union bound over $s$ shows that, with high probability,
\begin{align}
	\label{eq gamm2}
	\gamma_2&\leq C\cdot r\left( \frac{\log n}{p}  + \sqrt{\frac{n^2 \log n}{p}  } \right) \infnorm{\calX^{t,m} - \calT} \lb \max_{i = 1,2,3}\twoinf{\mX_i^{t,m}} \rb^2 \notag  \\
	&\leq C \cdot r\left( \frac{\log n}{p}  + \sqrt{\frac{n^2 \log n}{p}  } \right)\cdot  \frac{36}{2^{20} \kappa^2 \mu^2 r^{4}}  \frac{1}{2^t}  \left( \frac{\mu r}{n} \right)^{3/2} \sigma_{\max}(\calT)\cdot 4\frac{\mu r}{n} \notag \\
	&\leq \frac{1}{2^{11} }\frac{1}{2^{20}\kappa^4\mu^2 r^4} \frac{1}{2^{t+1}} \sqrt{\frac{\mu r}{n}} \sigma_{\max}(\calT),
\end{align}
provided that 
	$p\geq  \frac{C_2 \kappa^4\,\mu^4 r^6\log n}{n^2}.$

\subsubsection{Proof of \texorpdfstring{\eqref{eq upper bound of gamma3}}{TEXT}}
\label{proof of gamma 3}
Now we  bound $\gamma_3 = \twonorm{\ve_m^\tran \calM_1\lb \calZ^{t,m, \ell}\rb\lb \mX_3^{t,m} \otimes \mX_2^{t,m} \rb}$. To simplify notation, we denote by $\calH_{\Omega_{\ell}}$ the operator $\calI - p^{-1}\calP_{\Omega_{-\ell}} - \calP_{\ell}$ and define 
\begin{align*}
   \mC_i^{t,m,\ell} = {\mX_i^{t,\ell}}{\mX_i^{t,\ell}}^\tran-{\mX_i^{t,m}}{\mX_i^{t,m}}^\tran\quad \mbox{and}\quad
   \mD_i^{t,m,\ell} =  \mX_i^{t,m} \mT_i^{t,m}- \mX_i^{t,\ell} \mT_i^{t,\ell}.
\end{align*}
Then ${\mX_i^{t,m}}{\mX_i^{t,m}}^\tran-{\mX_i^{t,\ell}}{\mX_i^{t,\ell}}^\tran$ can be rewritten as 
\begin{align}
	\label{eq: split projection}
	{\mX_i^{t,m}}{\mX_i^{t,m}}^\tran-{\mX_i^{t,\ell}}{\mX_i^{t,\ell}}^\tran 
	&= \mD_i^{t,m,\ell}  \left( \mX_i^{t,m}  \mT_i^{t,m}\right)^\tran + \lb\mX_i^{t,\ell}  \mT_i^{t,\ell}  \rb \left( \mD_i^{t,m,\ell}  \right)^\tran.
\end{align}
It is not hard to see that
\begin{align}
\label{eq: upper bound of Dlm}
\fronorm{\mD_i^{t,m,\ell} } \leq  \fronorm{\mX_i^{t,m}  \mT_i^{t,m} - \mX_i^t \mR_i^t} + \fronorm{ \mX_i^{t,\ell}  \mT_i^{t,\ell} - \mX_i^t \mR_i^t }\leq 2\cdot \frac{1}{2^{20}\kappa^2\mu^2 r^4}\frac{1}{2^t}\sqrt{\frac{\mu r}{n}},
\end{align}
where the last inequality is due to \eqref{F norm}.
Recall that $\calZ^{t,m,\ell}= \calH_{\Omega_{\ell}}(\calX^{t,\ell} -\calX^{t,m})$. We have \cjc{
\begin{align*}
	\calX^{t,\ell}-\calX^{t,m} &= \lb\calT+\calE^{t-1,\ell}\rb\ttimes  {\mX_i^{t,\ell}}{\mX_i^{t,\ell}}^\tran - \lb\calT+\calE^{t-1,m}\rb\ttimes {\mX_i^{t,m}}{\mX_i^{t,m}}^\tran \notag \\
	&= \lb\calT+\calE^{t-1,\ell}\rb\ttimes {\mX_i^{t,\ell}}{\mX_i^{t,\ell}}^\tran  - \lb\calT+\calE^{t-1,\ell}\rb\ttimes {\mX_i^{t,m}}{\mX_i^{t,m}}^\tran  \\
 &\quad+  \lb  \calE^{t-1,\ell} - \calE^{t-1,m}\rb \ttimes {\mX_i^{t,m}}{\mX_i^{t,m}}^\tran \notag \\ 
 &=\lb\calT+\calE^{t-1,\ell}\rb\times_1\mC_1^{t,m,\ell}\times_2{\mX_2^{t,m}}{\mX_2^{t,m}}^\tran \times_3{\mX_3^{t,m}}{\mX_3^{t,m}}^\tran \notag \\
	&\quad + \lb\calT+\calE^{t-1,\ell}\rb\times_1{\mX_1^{t,m}}{\mX_1^{t,m}}^\tran \times_2\mC_2^{t,m,\ell}\times_3{\mX_3^{t,m}}{\mX_3^{t,m}}^\tran \notag \\
	&\quad +\lb\calT+\calE^{t-1,\ell}\rb\times_1{\mX_1^{t,m}}{\mX_1^{t,m}}^\tran \times_2{\mX_2^{t,m}}{\mX_2^{t,m}}^\tran \times_3\mC_3^{t,m,\ell} \notag\\
	&\quad +\lb\calT+\calE^{t-1,\ell}\rb\times_1\mC_1^{t,m,\ell}\times_2\mC_2^{t,m,\ell}\times_3{\mX_3^{t,m}}{\mX_3^{t,m}}^\tran \notag \\
	&\quad +\lb\calT+\calE^{t-1,\ell}\rb\times_1\mC_1^{t,m,\ell}\times_2{\mX_2^{t,m}}{\mX_2^{t,m}}^\tran \times_3\mC_3^{t,m,\ell} \notag \\
	&\quad +\lb\calT+\calE^{t-1,\ell}\rb\times_1{\mX_1^{t,m}}{\mX_1^{t,m}}^\tran \times_2\mC_2^{t,m,\ell}\times_3\mC_3^{t,m,\ell} +\lb\calT+\calE^{t-1,\ell}\rb\ttimes\mC_i^{t,m,\ell}\\
 &\quad + \lb\calE^{t-1,\ell}-\calE^{t-1}\rb\ttimes{\mX_i^{t,m}}{\mX_i^{t,m}}^\tran  + \lb\calE^{t-1}-\calE^{t-1,m}\rb\ttimes{\mX_i^{t,m}}{\mX_i^{t,m}}^\tran \\
	:&= \sum_{i=1}^{9}\calZ_i.
\end{align*}
}
It follows that
\begin{align*}
	\gamma_3&= \twonorm{\ve_m^\tran \calM_1\lb \calH_{\Omega_{\ell}}\lb\calX^{t,\ell} -\calX^{t,m}\rb \rb \lb \mX_3^{t,m}  \otimes \mX_2^{t,m} \rb}\\
	&\leq \sum_{i=1}^{9} \twonorm{\ve_m^\tran \calM_1\lb \calH_{\Omega_{\ell}}\lb\calZ_i\rb \rb \lb \mX_3^{t,m}  \otimes \mX_2^{t,m} \rb} :=\sum_{i=1}^{9} \gamma_{3,i}.
\end{align*}

\paragraph{Bounding  $\gamma_{3,1}$.} From \eqref{eq: split projection}, $\gamma_{3,1}$ can be decomposed as follows
\begin{small}
\begin{align*}
	\gamma_{3,1}
	&\leq  \twonorm{\ve_{m}^{\tran}\lb\calM_1\lb\calH_{\Omega_{\ell}}\lb\lb\calT+\calE^{t-1,\ell}\rb\times_1\lb \mX_1^{t,\ell} \mT_1^{t,\ell} \rb{\mD_1^{t,m,\ell} }^\tran\underset{i\neq 1}{ \times }{\mX_i^{t,m}}{\mX_i^{t,m}}^\tran  \rb\rb\rb\lb\mX_3^{t,m}\otimes\mX_2^{t,m}\rb}\\
	&\quad + \twonorm{\ve_{m}^{\tran}\lb\calM_1\lb\calH_{\Omega_{\ell}}\lb\lb\calT+\calE^{t-1,\ell}\rb\times_1\mD_1^{t,m,\ell}  \lb \mX_1^{t,m} \mT_1^{t,m}\rb^\tran \underset{i\neq 1}{ \times }{\mX_i^{t,m}}{\mX_i^{t,m}}^\tran \rb\rb\rb\lb\mX_3^{t,m}\otimes\mX_2^{t,m}\rb}\\
	& =: \gamma_{3,1}^a+\gamma_{3,1}^b.
\end{align*}
\end{small}
\begin{itemize}
    \item Bounding $\gamma_{3,1}^a$. For simplification, we define 
    \begin{align*}
        \calC_1 := \left( \calT + \calE^{t-1, \ell} \right) \times_1 \left( \mX_1^{t,\ell} \mT_1^{t,\ell}\right){\mD_1^{t,m,\ell}}^\tran \times_2 {\mX_2^{t,m}}^\tran \times_3{\mX_3^{t,m}}^\tran.
    \end{align*}
It can be seen that
\begin{align*}
\twoinf{\calM_1\left(\calC_1 \right)} &\leq \twoinf{\left( \mX_1^{t,\ell} \mT_1^{t,\ell}\right){\mD_1^{t,m,\ell}}^\tran}\cdot \opnorm{\calM_1\left(\calT + \calE^{t-1,\ell}\right)}\\
&\leq \twoinf{ \mX_1^{t,\ell}}\cdot \fronorm{{\mD_1^{t,m,\ell}}}\cdot \opnorm{\calM_1\left(\calT + \calE^{t-1,\ell}\right)}\\
&\leq 2\sqrt{\frac{\mu r}{n}}\cdot   \frac{2}{2^{20}\kappa^2\mu^2 r^4}\frac{1}{2^t}\sqrt{\frac{\mu r}{n}}\cdot 2\sigma_{\max}(\calT),
\end{align*}
where the last line is due to $\twoinf{\mX_i^{t,\ell}}\leq 2\sqrt{\frac{\mu r}{n}}$, $\opnorm{\calM_1\lb\calE^{t,\ell}\rb}\leq \sigma_{\max}(\calT)$, and \eqref{eq: upper bound of Dlm}.
Then the term $\gamma_{3,1}^a$ can be bounded as follows:
\begin{align}
\label{eq gamma31a}
\gamma_{3,1}^a &\leq \fronorm{ \calM_1\left( \calH_{\Omega_{\ell}} \left( \calC_1 \times_2 \mX_2^{t,m} \times_3 \mX_3^{t,m} \right) \right) \left( \mX_3^{t,m} \otimes \mX_2^{t,m}\right)} \notag \\
&=\fronorm{  \calH_{\Omega_{\ell}} \left( \calC_1 \times_2 \mX_2^{t,m} \times_3 \mX_3^{t,m} \right)  \times_2 {\mX_2^{t,m}}^\tran \times_3 {\mX_3^{t,m}}^\tran } \notag \\
&=\sup_{\calZ\in\R^{n\times r\times r}: \fronorm{\calZ}=1} \la\calH_{\Omega_{\ell}} \left( \calC_1 \times_2 \mX_2^{t,m} \times_3 \mX_3^{t,m} \right)  \times_2 {\mX_2^{t,m}}^\tran \times_3 {\mX_3^{t,m}}^\tran, \calZ \ra \notag \\
&=\sup_{\calZ\in\R^{n\times r\times r}: \fronorm{\calZ}=1} \la \calH_{\Omega_{\ell}}\left( \calC_1 \times_2 \mX_2^{t,m} \times_3 \mX_3^{t,m} \right) , \calZ \times_2 {\mX_2^{t,m}}  \times_3 {\mX_3^{t,m}} \ra \notag \\
&\stackrel{(a)}{\leq} C\left( p^{-1}
\log ^3 n + \sqrt{p^{-1} n \log^5n}\right)\cdot \twoinf{\calM_1\left(\calC_1\right)} \prod_{i = 2}^3\lb\fronorm{\mX_i^{t,m}} \cdot\twoinf{\mX_i^{t,m}}\rb  \notag \\
&\leq C\left( p^{-1}
\log ^3 n + \sqrt{p^{-1} n \log^5n}\right)\cdot \sqrt{\frac{\mu r}{n}} \frac{8}{2^{20}\kappa^2\mu^2 r^4}\frac{1}{2^t}\sqrt{\frac{\mu r}{n}}\sigma_{\max}(\calT)\cdot r\cdot 4\frac{\mu r}{n} \notag \\
&\leq  \frac{1}{2^{16}}\cdot \frac{1}{2^{20}\kappa^4 \mu^2 r^4} \frac{1}{2^{t+1}} \sqrt{\frac{\mu r}{n}}\sigma_{\max}(\calT),
\end{align}
where step (a) follows from Lemma~\ref{lemma: loo tong lemma 14} and the last step has used the assumption 
\begin{align*}
p\geq \max\left\lbrace \frac{C_1\kappa^2 \mu^{1.5}r^{2.5}\log^3n}{n^{3/2}}, \frac{C_2\kappa^4\mu^3r^5\log^5n}{n^2}\right\rbrace
\end{align*}

\item Bounding $\gamma_{3,1}^b$. Let $\mC_2, \mC_3$ be matrices defined by
\begin{align*}
    &\mC_2:= \mD_1^{t,m,\ell}  \lb \mX_1^{t,m} \mT_1^{t,m} \rb^\tran \lb\calM_1\lb\calT+\calE^{t-1,m}\rb\rb\lb\mX_3^{t,m}\otimes\mX_2^{t,m}\rb \in\R^{n\times r^2},\\
	&\mC_3:= \lb\mX_3^{t,m}\otimes\mX_2^{t,m}\rb^{\tran}\in\R^{r^2\times n^2}.
\end{align*}
Then we have
\begin{align*}
	\gamma_{3,1}^b &= \sqrt{\sum_{s= 1}^{r^2}\lb\sum_{j = 1,j\notin\Gamma}^{n^2}\sum_{k = 1}^{r^2}\lsb\mC_2\rsb_{m,k}\lsb\mC_3\rsb_{k,j}\lb 1-p^{-1}\delta_{m,j}\rb\lsb\mX_3^{t,m}\otimes\mX_2^{t,m}\rsb_{j,s}\rb^2} \notag \\
	&=\sqrt{\sum_{s= 1}^{r^2}\lb\sum_{k = 1}^{r^2}\lsb\mC_2\rsb_{m,k} \sum_{j = 1,j\notin\Gamma}^{n^2} \lsb\mC_3\rsb_{k,j}\lb 1-p^{-1}\delta_{m,j}\rb\lsb\mX_3^{t,m}\otimes\mX_2^{t,m}\rsb_{j,s}\rb^2} \notag \\
	&\leq \sqrt{\sum_{s= 1}^{r^2} \sum_{k = 1}^{r^2}\lsb\mC_2\rsb_{m,k}^2\cdot \sum_{k=1}^{r^2}  \lb\sum_{j = 1,j\notin\Gamma}^{n^2} \lsb\mC_3\rsb_{k,j}\lb 1-p^{-1}\delta_{m,j}\rb \lsb\mX_3^{t,m}\otimes\mX_2^{t,m}\rsb_{j,s}\rb^2} \notag \\
	&\leq \twoinf{\mC_2}\cdot r^2 \cdot  \max_{1\leq s,k\leq r^2} \lab \sum_{j = 1,j\notin\Gamma}^{n^2}\lsb \mC_3\rsb_{k,j}\lb 1-p^{-1}\delta_{m,j}\rb\lsb\mX_3^{t,m}\otimes\mX_2^{t,m}\rsb_{j,s}\rab \notag \\
	:&=\twoinf{\mC_2}\cdot r^2 \cdot  \max_{1\leq s,k\leq r^2} \lab \sum_{j = 1,j\notin\Gamma}^{n^2}X_{j}^{s,k}\rab. \notag
\end{align*}

Since  $\{\delta_{m,j}\}_{j\in[n^2]}$ and $\mX_i^{t,m}$ are independent by construction,   $X_j^{s,k}$ are independent mean zero random variables with 
\begin{align*}
	\lab X_j^{s,k}\rab &\leq p^{-1}\infnorm{\mC_3}  \infnorm{\mX_3^{t,m} \otimes \mX_2^{t,m}}
\leq 
p^{-1}\twoinf{\mX_3^{t,m}\otimes\mX_2^{t,m} }^2,\\
	\lab \sum_{ j = 1,j\notin \Gamma}^{n^2} \E{ {X_j^{s,k}}^2 } \rab &\leq p^{-1} \sum_{j = 1}^{n^2}\lsb \mC_3\rsb_{k,j}^2 \twoinf{\mX_3^{t,m}\otimes\mX_2^{t,m} }^2 \leq p^{-1}  \twoinf{\mX_3^{t,m}\otimes\mX_2^{t,m} }^2.
\end{align*}
We apply the Bernstein inequality to obtain that,  with high probability,
\begin{align*}
	\lab \sum_{j = 1,j\notin \Gamma}^{n^2}X_{j}^{s,k}\rab &\leq C\left( \frac{\log n}{p}  \twoinf{\mX_3^{t,m}\otimes\mX_2^{t,m} }^2+ \sqrt{\frac{\log n}{p}} \twoinf{\mX_3^{t,m}\otimes\mX_2^{t,m} }\right)\\
 &\leq C\left( \frac{\log n}{p} \cdot 2^4\left( \frac{\mu r}{n}\right)^2+ \sqrt{\frac{\log n}{p}}\cdot 2^2 \frac{\mu r}{n}\right).
\end{align*}
 Furthermore, a simple computation yields that
\begin{align*}
	\twoinf{\mC_2} &\leq \twoinf{ \mD_i^{t,m,\ell} }\cdot \opnorm{\calM_1\lb\calT+\calE^{t-1,m}\rb} \\
	&\leq 2\max_{\ell\in [n]}\fronorm{\mX_i^{t,\ell} \mT_i^{t,\ell} - \mX_i^{t} \mR_i^{t} }\cdot \opnorm{\calM_1\lb\calT+\calE^{t-1,m}\rb} \\
	&\leq 4\sigma_{\max}(\calT)\cdot \max_{\ell\in [n]}\fronorm{\mX_i^{t,\ell} \mT_i^{t,\ell} - \mX_i^{t} \mR_i^{t} }.
\end{align*}
Thus we have
\begin{align}
\label{eq: gamma31b}
	\gamma_{3,1}^b &\leq    \twoinf{\mC_2} \cdot r^2 \cdot  \max_{1\leq s,k\leq r^2} \lab \sum_{j = 1,j\notin\Gamma}^{n^2}X_{j}^{s,k}\rab \notag\\
	&\leq 4\sigma_{\max}(\calT)\max_{\ell\in [n]}\fronorm{\mX_i^{t,\ell} \mT_i^{t,\ell} - \mX_i^{t} \mR_i^{t} } \cdot r^2 \cdot  \max_{1\leq s,k\leq r^2} \lab \sum_{j = 1,j\notin\Gamma}^{n^2}X_{j}^{s,k}\rab \notag\\
	&\leq  4\sigma_{\max}(\calT)\frac{1}{2^{20} \kappa^2\mu^2 r^4} \frac{1}{2^t} \sqrt{\frac{\mu r}{n}} \cdot r^2 \cdot \left( \frac{\log n}{p} \cdot 2^4\left( \frac{\mu r}{n}\right)^2+ \sqrt{\frac{\log n}{p}}\cdot 2^2 \frac{\mu r}{n}\right) \notag \\
	&\leq \frac{1}{2^{16}}\frac{1}{2^{20} \kappa^4\mu^2 r^4} \frac{1}{2^{t+1}} \sqrt{\frac{\mu r}{n}}\sigma_{\max}(\calT),
\end{align}
provided that
	$p\geq \frac{C_2 \kappa^2\mu^2r^6\log n}{n^2}.$
\end{itemize}

Combining \eqref{eq gamma31a} and \eqref{eq: gamma31b} together yields that
\begin{align*}
\gamma_{3,1} \leq \frac{1}{2^{15}}\frac{1}{2^{20} \kappa^4\mu^2 r^4} \frac{1}{2^{t+1}} \sqrt{\frac{\mu r}{n}}\sigma_{\max}(\calT).
\end{align*}

\paragraph{Bounding of $\gamma_{3,i}$ for $i=2,\cdots,9$.}
Following the same argument of bounding $\gamma_{3,1}^a$, one can obtain
\begin{align*}
	\gamma_{3,i} \leq \frac{1}{2^{15}}\frac{1}{2^{20} \kappa^4\mu^2 r^4} \frac{1}{2^{t+1}} \sqrt{\frac{\mu r}{n}}\sigma_{\max}(\calT),\quad i=2,\cdots, 9.
\end{align*}
\\

Putting together all of the bounds on $\gamma_{3,i}$ for $i = 1,\cdots,9$ yields that
\begin{align*}
	\gamma_{3}&\leq \frac{1}{2^{11}}\frac{1}{2^{20} \kappa^4\mu^2 r^4} \frac{1}{2^{t+1}} \sqrt{\frac{\mu r}{n}}\sigma_{\max}(\calT).
\end{align*}

\subsection{Proof of Claim \ref{claim Y distance}}\label{subsec:claimY}
We only show the details  for bounding $\fronorm{\mY^t_1\mY^t_1-\mY_{1}^{t,\ell}{\mY_1^{t,\ell}}^\tran}$, and the proofs for the other two cases are similar.
 Recall that $\calX^t = \left( \calT  + \calE^{t-1}\right)\ttimes  {\mX_i^t}{\mX_i^t}^\tran$ and $\calX^{t,\ell} = \left( \calT  + \calE^{t-1, \ell}\right)\ttimes {\mX_i^{t, \ell}}{\mX_i^{t, \ell}}^\tran$. In addition, $\mY_1^t$ and $\mY_1^{t,\ell}$ are the top-$r$ eigenvectors of $\calM_1^\tran\lb\calX^t\rb\calM_1\lb\calX^t\rb$ and $\calM_1^\tran \lb\calX^{t,\ell}\rb\calM_1\lb\calX^{t,\ell}\rb$, respectively. Let $\mW_1$,  $\mD_1$ and $\mD_{3,2}$ be the matrices defined as
\begin{align*}
	&\mW_1 :=  \calM_1^\tran\lb\calX^t\rb\calM_1\lb\calX^t\rb - \calM_1^\tran \lb\calX^{t,\ell}\rb\calM_1\lb\calX^{t,\ell}\rb,\quad\mD_1:={\mX_1^{t}}{\mX_1^{t}}^{\tran} -{\mX_1^{t,\ell}}{\mX_1^{t,\ell}}^{\tran},\\
	&\mD_{3,2} := \mX_3^t{\mX_3^t}^\tran\otimes\mX_2^t{\mX_2^t}^\tran -{\mX_3^{t,\ell}}{\mX_3^{t,\ell}}^\tran \otimes{\mX_2^{t,\ell}}{\mX_2^{t,\ell}}^\tran.
\end{align*}
By the definition of $\calX^{t}$ and $\calX^{t,\ell}$, we have
\begin{small}
\begin{align*}
	\fronorm{\mW_1} &\leq \fronorm{\mD_{3,2}\calM_1^{\tran}\lb \calT+\calE^{t-1} \rb {\mX_1^{t}}{\mX_1^{t}}^{\tran} \calM_1\lb \lb \calT+\calE^{t-1} \rb\underset{i\neq 1}{ \times }{\mX_i^{t}}{\mX_i^{t}}^\tran\rb } \\
	&\quad +\fronorm{ \calM_1^{\tran}\left( \lb\calE^{t-1} - \calE^{t-1,\ell}\rb\underset{i\neq 1}{ \times }{\mX_i^{t,\ell}}{\mX_i^{t,\ell}}^\tran \right) {\mX_1^{t}}{\mX_1^{t}}^{\tran} \calM_1\lb \lb\calT+\calE^{t-1} \rb\underset{i\neq 1}{ \times }{\mX_i^{t}}{\mX_i^{t}}^\tran\rb  }\\
	&\quad + \fronorm{\calM_1^{\tran}\left( \lb\calT+\calE^{t-1,\ell}\rb\underset{i\neq 1}{ \times }{\mX_i^{t,\ell}}{\mX_i^{t,\ell}}^\tran \right) \mD_1\calM_1\lb\lb \calT+\calE^{t-1}\rb\underset{i\neq 1}{ \times }{\mX_i^{t}}{\mX_i^{t}}^\tran \rb  } \\
	&\quad + \fronorm{ \calM_1^{\tran}\left( \lb\calT+\calE^{t-1,\ell}\rb\underset{i\neq 1}{ \times }{\mX_i^{t,\ell}}{\mX_i^{t,\ell}}^\tran \right)  \mX_1^{t,\ell}{\mX_1^{t,\ell}}^{\tran}\calM_1\left(\lb\calE^{t-1} - \calE^{t-1,\ell}\rb\underset{i\neq 1}{ \times }{\mX_i^{t}}{\mX_i^{t}}^\tran\right) } \\
	&\quad + \fronorm{ \left( {\mX_3^{t,\ell}}{\mX_3^{t,\ell}}^\tran \otimes{\mX_2^{t,\ell}}{\mX_2^{t,\ell}}^\tran  \right)\calM_1^{\tran}\left( \calT+\calE^{t-1,\ell} \right)  \mX_1^{t,\ell}{\mX_1^{t,\ell}}^{\tran}\calM_1\left(\calT + \calE^{t-1,\ell}\right)\mD_{3,2}}\\
	&\leq \fronorm{ \mD_{3,2}} \cdot \left( \opnorm{ \mT_1+\mE_1^{t-1} }^2+ \opnorm{\mT_1+\mE_1^{t-1,\ell} }^2 \right) + \fronorm{\mD_1}\cdot \opnorm{\mT_1+\mE_1^{t-1}}\cdot  \opnorm{\mT_1+\mE_1^{t-1,\ell}}  \\
	&\quad  + \fronorm{\calE^{t-1} - \calE^{t-1,\ell}}\cdot \left( \opnorm{\mT_1+\mE_1^{t-1}}+ \opnorm{\mT_1+\mE_1^{t-1,\ell}} \right)\\ 
	&\leq 8\sigma_{\max}(\calT)^2 \fronorm{ \mD_{3,2} } + 4\sigma_{\max}(\calT) \fronorm{\calE^{t-1} - \calE^{t-1,\ell}}+ 4\sigma_{\max}(\calT)^2\fronorm{\mX_1^t {\mX_1^t}^\tran - \mX_1^{t,\ell}{\mX_1^{t,\ell}}^\tran},
\end{align*}
\end{small}
where the last line is due to \eqref{spectral norm of error tensor} and $\eqref{spectral norm of loo error tensor}$. Furthermore, a straightforward computation yields that
\begin{align*}
	\fronorm{ \mD_{32} } &\leq\fronorm{ \mX_3^t{\mX_3^t}^\tran - {\mX_3^{t,\ell}}{\mX_3^{t,\ell}}^\tran }\cdot \fronorm{ \mX_2^t{\mX_2^t}^\tran} + \fronorm{ {\mX_3^{t,\ell}}{\mX_3^{t,\ell}}^\tran  }\cdot \fronorm{ \mX_2^t{\mX_2^t}^\tran - {\mX_2^{t,\ell}}{\mX_2^{t,\ell}}^\tran  }\\
	&\leq 2r\cdot \max_{i=2,3}\fronorm{ \mX_i^t{\mX_i^t}^\tran - {\mX_i^{t,\ell}}{\mX_i^{t,\ell}}^\tran}\leq 4r\cdot\max_{i = 2,3}\fronorm{\mX_i^t\mR_i^t-\mX_i^{t,\ell}\mT_i^{t,\ell}}\\
 &\leq  \frac{1}{2^{18} \kappa^2\mu^2 r^3} \frac{1}{2^t} \sqrt{\frac{\mu r}{n}},
\end{align*}
where the last inequality follows from \eqref{F norm}.
Thus we can obtain
\begin{align*}
	\fronorm{\mW_1} &\leq \sigma_{\max}(\calT)^2  \cdot\lb  \frac{8}{2^{18} \kappa^2\mu^2 r^3} \frac{1}{2^t} \sqrt{\frac{\mu r}{n}} + \frac{4}{2^{20} \kappa^4 \mu^2 r^4} \frac{1}{2^t} \sqrt{\frac{\mu r}{n}}+  \frac{8}{2^{20} \kappa^2\mu^2 r^4} \frac{1}{2^t} \sqrt{\frac{\mu r}{n}}\rb\\
	&\leq \left( \frac{8}{2^{18} } + \frac{4}{2^{20} }  + \frac{8}{2^{20} } \right)\frac{\sigma_{\max}^2(\calT)}{\kappa^2}\frac{1}{2^t}\sqrt{\frac{\mu r}{n}}\leq \frac{1}{2^{14}}\sigma_{\min}^2(\calT)\frac{1}{2^t}\sqrt{\frac{\mu r}{n}},
\end{align*} 
where the first inequality is due to \eqref{F norm  of error tensor} and \eqref{F norm}.
Consequently,
\begin{align*}
    \opnorm{\mW_1} \leq \fronorm{\mW_1}\leq \frac{1}{2^{14}}\sigma_{\min}^2(\calT).
\end{align*}
Note that the eigengap $\delta$ between the the $r$-th and $r+1$-th eigenvalues of $\calM_1^\tran (\calX^{t,\ell})\calM_1(\calX^{t,\ell})$ is bounded by
\begin{align*}
    \delta\geq \sigma_{\min}\lb\calM_1\lb\calX^{t,\ell}\rb\rb^2\geq \lb\frac{15}{16}\rb^2\sigma_{\min}^2(\calT),
\end{align*}
where the last inequality follows from Lemma~\ref{lemma: condition number of X}. Applying Lemma~\ref{rateoptimal lemma 1} and Lemma~\ref{Davis-kahan sintheta theorem} shows that
\begin{align*}
	\fronorm{ \mY_1^t{\mY_1^t}^\tran -\mY_1^{t,\ell}{\mY_1^{t,\ell}}^\tran }&\leq\sqrt{2} \frac{\fronorm{\mW_1}}{\delta - \opnorm{\mW_1} }  \leq  \frac{8\sqrt{2} }{7}\frac{1}{\sigma_{\min}^2(\calT) }  \frac{1}{2^{14}}\sigma_{\min}^2(\calT)\frac{1}{2^t} \sqrt{\frac{\mu r}{n}}  \leq \frac{1}{2^{13}} \frac{1}{2^t}\sqrt{\frac{\mu r}{n}}.
\end{align*}

\subsection{Proof of \texorpdfstring{\eqref{eq: claim rc 1}}{TEXT} in Claim \texorpdfstring{\ref{claim in R_c}}{TEXT}}
\label{first term in 5.5}
We provide a detailed proof for $i = 1$, while the proofs for $i = 2,3$ are overall similar. Let $\mD_i^{t,\ell}$ and $\mM_i^{t,\ell}$ be  two auxiliary matrices defined as 
\begin{align*}
	\mD_i^{t,\ell} :&= \mX_i^{t} \mR_i^{t} - \mX_i^{t,\ell}\mT_i^{t,\ell},\\
	\mM_i^{t,\ell} :&= \mX_i^t{\mX_i^t}^\tran - \mX_i^{t,\ell} {\mX_i^{t,\ell}}^\tran= \mD_i^{t,\ell} \lb\mX_i^t \mR_i^t\rb^\tran + \left( \mX_i^{t,\ell} \mT_i^{t,\ell} \right) \left( \mD_i^{t,\ell} \right)^\tran.
\end{align*}
From \eqref{2 inf norm of loo} and  \eqref{F norm}, one has
\begin{align}
    \label{eq: F norm of D}
	\fronorm{\mD_i^{t,\ell}} &\leq \frac{1}{2^{20} \kappa^2\mu^2 r^4} \frac{1}{2^t} \sqrt{\frac{\mu r}{n}},\\
	\label{eq two inf norm of M}
	\twoinf{\mM_i^{t,\ell}} &\leq \fronorm{\mM_i^{t,\ell}} \leq 2\fronorm{ \mD_i^{t,\ell} }\leq \frac{1}{2^{19} \kappa^2\mu^2 r^4} \frac{1}{2^t} \sqrt{\frac{\mu r}{n}}.
\end{align}
By the definition of $\calX^t$ and $\calX^{t,\ell}$, \cjc{
\begin{align*}
	\calX^t - \calX^{t,\ell} &= \lb\calT+\calE^{t-1}\rb \ttimes {\mX_i^{t}}{\mX_i^{t}}^\tran - \lb\calT+\calE^{t-1,\ell}\rb\ttimes {\mX_i^{t,\ell}}{\mX_i^{t,\ell}}^\tran\\
	& = \left( \calT+\calE^{t-1}  \right) \ttimes\mM_i^{t,\ell}  +  \left( \calT+\calE^{t-1} \right) \underset{i\neq 3}{ \times }  \mM_i^{t,\ell}  \times_3 {\mX_3^{t,\ell}}{\mX_3^{t,\ell}}^\tran  \\
	&\quad +  \left( \calT+\calE^{t-1} \right) \underset{i\neq 2}{ \times }\mM_i^{t,\ell} \times_2 {\mX_2^{t,\ell}}{\mX_2^{t,\ell}}^\tran   +  \left( \calT+\calE^{t-1} \right) \times_1 {\mX_1^{t,\ell}}{\mX_1^{t,\ell}}^\tran \underset{i\neq 1}{ \times } \mM_i^{t,\ell} \\ 
	&\quad +  \left( \calT+\calE^{t-1} \right) \times_1 \mM_1^{t,\ell} \underset{i\neq 1}{ \times } {\mX_i^{t,\ell}}{\mX_i^{t,\ell}}^\tran  + \left( \calT+\calE^{t-1} \right) \underset{i\neq 2}{ \times }{\mX_i^{t,\ell}}{\mX_i^{t,\ell}}^\tran\times_2\mM_2^{t,\ell}  \\
	&\quad + \left( \calT+\calE^{t-1} \right)\underset{i\neq 3}{ \times } {\mX_i^{t,\ell}}{\mX_i^{t,\ell}}^\tran  \times_3 \mM_3^{t,\ell}  + \left( \calE^{t-1} - \calE^{t-1,\ell} \right)  \ttimes {\mX_i^{t,\ell}}{\mX_i^{t,\ell}}^\tran \\
 &:= \sum_{i = 1}^8\calA_i.
\end{align*}}
It follows that 
\begin{align*}
	\fronorm{   \left(  \left( \calI - p^{-1} \calP_{\Omega}  \right)  \left( \calX^t-\calX^{t,\ell}\right) \right)  \underset{i\neq 1}{ \times } {\mX_{i}^{t,\ell}}^\tran    }  
	&\leq \sum_{i=1}^{8}\underbrace{\fronorm{\calM_1 \left( \left( \calI - p^{-1} \calP_{\Omega}  \right)  \left(\calA_i\right)  \right) \left( \mX_3^{t, \ell}  \otimes \mX_{2}^{t,\ell}\right)} }_{=:\eta_i}.
\end{align*}
All of  these terms but  $\eta_5$ can be bounded by the same argument as controlling the term $\gamma_{3,1}^a$, yielding
\begin{align}
	\label{eq eta}
	\eta_i \leq \frac{1}{2^{14}} \cdot \frac{1}{2^{20}\kappa^4\mu^2 r^4 } \frac{1}{2^{t+1}} \sqrt{\frac{\mu r}{n}}\sigma_{\max}(\calT),\quad i\neq 5.
\end{align}
For $\eta_5$, the same bound can be obtained, but a different strategy should be adopted.
\subsubsection{Bounding \texorpdfstring{$\eta_5$}{TEXT}}
Notice that 
\begin{align*}
	{\mX_i^{t,\ell}}{\mX_i^{t,\ell}}^\tran = \left(\mX_i^{t,\ell} \mR^{t,\ell}_i - \mU_i \right) \left(\mX_i^{t,\ell} \mR^{t,\ell}_i \right)^\tran + \mU_i \left(\mX_i^{t,\ell} \mR^{t,\ell}_i - \mU_i \right)^\tran + {\mU_i}{\mU_i}^\tran.
\end{align*}
It follows that
\begin{align*}
	\calA_{5} &=\left( \calT+\calE^{t-1} \right) \times_1 \mM_1^{t,\ell}\times_2 {\mX_2^{t,\ell}}{\mX_2^{t,\ell}}^\tran  \times_3{\mX_3^{t,\ell}}{\mX_3^{t,\ell}}^\tran\\
	& = 	\left( \calT+\calE^{t-1} \right) \times_1 \mD_1^{t,\ell}  \left( \mX_1^t \mR_1^t \right)^\tran \times_2 {\mX_2^{t,\ell}}{\mX_2^{t,\ell}}^\tran  \times_3{\mX_3^{t,\ell}}{\mX_3^{t,\ell}}^\tran \\
	&\quad+ \left( \calT+\calE^{t-1} \right) \times_1 \left( \mX_{1}^{t,\ell} \mT_1^{t,\ell} \right) \left(  \mD_1^{t,\ell}  \right)^\tran\times_2 {\mX_2^{t,\ell}}{\mX_2^{t,\ell}}^\tran \times_3{\mX_3^{t,\ell}}{\mX_3^{t,\ell}}^\tran   \\
	&= \left( \calT+\calE^{t-1} \right) \times_1 \mD_1^{t,\ell}  \left( \mX_1^t \mR_1^t \right)^\tran \times_2  \left(\mX_2^{t,\ell} \mR^{t,\ell}_2 - \mU_2 \right) \left(\mX_2^{t,\ell} \mR^{t,\ell}_2\right)^\tran   \times_3 {\mX_3^{t,\ell}}{\mX_3^{t,\ell}}^\tran  \\
	&\quad+ \left( \calT+\calE^{t-1} \right) \times_1 \mD_1^{t,\ell}  \left( \mX_1^t \mR_1^t \right)^\tran \times_2  \mU_2  \left(\mX_2^{t,\ell} \mR^{t,\ell}_2- \mU_2 \right) ^\tran\times_3 {\mX_3^{t,\ell}}{\mX_3^{t,\ell}}^\tran \\
	&\quad+  \left( \calT+\calE^{t-1} \right) \times_1 \mD_1^{t,\ell}  \left( \mX_1^t \mR_1^t \right)^\tran \times_2  {\mU_2}{\mU_2}^\tran \times_3  \left(\mX_3^{t,\ell} \mR^{t,\ell}_3 - \mU_3\right) \left(\mX_3^{t,\ell} \mR^{t,\ell}_3 \right)^\tran \\
	&\quad  + \left( \calT+\calE^{t-1} \right) \times_1 \mD_1^{t,\ell}  \left( \mX_1^t \mR_1^t \right)^\tran \times_2  {\mU_2}{\mU_2}^\tran \times_3 \mU_3  \left(\mX_3^{t,\ell} \mR^{t,\ell}_3 - \mU_3 \right) ^\tran \\
	&\quad+ \left( \calT+\calE^{t-1} \right) \times_1 \mD_1^{t,\ell}  \left( \mX_1^t \mR_1^t \right)^\tran \times_2  {\mU_2}{\mU_2}^\tran \times_3  {\mU_3}{\mU_3}^\tran\\
	&\quad + \left( \calT+\calE^{t-1} \right) \times_1 \left( \mX_{1}^{t,\ell} \mT_1^{t,\ell} \right) \left(  \mD_1^{t,\ell}  \right)^\tran \times_2 {\mX_2^{t,\ell}}{\mX_2^{t,\ell}}^\tran\times_3 {\mX_3^{t,\ell}}{\mX_3^{t,\ell}}^\tran\\
	&=:\sum_{i=1}^{6}\calB_i,
\end{align*}
which implies that $\eta_5$ can be expressed as 
\begin{align*}
	\eta_5 &\leq  \sum_{i=1}^{6}\underbrace{\fronorm{\calM_1 \left( \left( \calI - p^{-1} \calP_{\Omega}  \right)  \left(\calB_i \right)  \right) \left( \mX_3^{t, \ell}  \otimes \mX_{2}^{t,\ell}\right)}}_{=:\eta_{5,i}}.
\end{align*}

\paragraph{Bounding $\eta_{5,1},$ $\eta_{5,2}, \eta_{5,3}$ and $\eta_{5,4}$.} A simple computation yields that
\begin{align*}
	\calM_1 \left(\calB_1 \right) &= \calM_1 \left( \left( \calT+\calE^{t-1} \right) \times_1 \mD_1^{t,\ell}  \left( \mX_1^t \mR_1^t \right)^\tran \times_2  \left(\mX_2^{t,\ell} \mR^{t,\ell}_2 - \mU_2\right)  \left(\mX_2^{t,\ell} \mR^{t,\ell}_2\right)^\tran   \times_3 {\mX_3^{t,\ell}}{\mX_3^{t,\ell}}^\tran  \right)\\
	&= \mD_1^{t,\ell}  \left( \mX_1^t \mR_1^t \right)^\tran \calM_1   \left( \calT+\calE^{t-1} \right) \left( \lb{\mX_3^{t,\ell}}{\mX_3^{t,\ell}}^\tran\rb  \otimes  \lb\left(\mX_2^{t,\ell} \mR^{t,\ell}_2 - \mU_2 \right) \left(\mX_2^{t,\ell} \mR^{t,\ell}_2\right)^\tran \rb\right)^\tran\\
	&=\underbrace{ \mD_1^{t,\ell}  \left( \mX_1^t \mR_1^t \right)^\tran \calM_1   \left( \calT+\calE^{t-1} \right) \left( {\mX_3^{t,\ell} }  \otimes \lb\mX_2^{t,\ell} \mR^{t,\ell}_2\rb\right) }_{=:\mA_1}  \underbrace{\left(  \mX_3^{t,\ell} \otimes  \left(\mX_2^{t,\ell} \mR^{t,\ell}_2 - \mU_2\right)  \right) ^\tran}_{=:\mB_1 }.
\end{align*}
Thus one has
\begin{align}
	\label{eq r51}
	\eta_{5,1} &= \fronorm{\calM_1 \left( \left( \calI - p^{-1} \calP_{\Omega}  \right)  \left(\calB_1 \right)  \right) \left( \mX_3^{t, \ell}  \otimes \mX_{2}^{t,\ell}\right)} \notag \\
	&= \fronorm{ \left(  \calM_1\left( \left( \calI - p^{-1} \calP_{\Omega}  \right)  (\calJ)\right) \odot \calM_1 \left(\calB_1 \right) \right)  \left( \mX_3^{t, \ell}  \otimes \mX_{2}^{t,\ell}\right)}\notag \\
	&=\fronorm{ \left(  \calM_1\left( \left( \calI - p^{-1} \calP_{\Omega}  \right)  (\calJ)\right) \odot  \left(\mA_1 \mB_1\right) \right)  \left( \mX_3^{t, \ell}  \otimes \mX_{2}^{t,\ell}\right)} \notag \\
	&=\sqrt{ \sum_{i=1}^{n}\sum_{s=1}^{r^2} \left[ \left(  \calM_1\left( \left( \calI - p^{-1} \calP_{\Omega}  \right)  (\calJ)\right) \odot \left(\mA_1 \mB_1\right) \right)  \left( \mX_3^{t, \ell}  \otimes \mX_{2}^{t,\ell}\right) \right]_{is}^2 } \notag \\
	&=\sqrt{ \sum_{i=1}^{n}\sum_{s=1}^{r^2} \left(   \sum_{j=1}^{n^2}\left[ \calM_1\left( \left( \calI - p^{-1} \calP_{\Omega}  \right)  (\calJ)\right) \odot  \left(\mA_1 \mB_1\right) \right]_{i,j}\cdot  \left[ \mX_3^{t, \ell}  \otimes \mX_{2}^{t,\ell}  \right]_{j,s} \right)^2 } \notag \\
	&=\sqrt{ \sum_{i=1}^{n}\sum_{s=1}^{r^2} \left(   \sum_{j=1}^{n^2}  \left( 1- p^{-1}\delta_{i,j}\right)  \left(  \sum_{q=1}^{r^2}\left[\mA_1\right]_{i,q} \left[ \mB_1\right] _{q,j} \right)\cdot  \left[ \mX_3^{t, \ell}  \otimes \mX_{2}^{t,\ell}  \right]_{j,s} \right)^2 } \notag \\
	&=\sqrt{ \sum_{i=1}^{n}\sum_{s=1}^{r^2} \left(  \sum_{q=1}^{r^2}\left[\mA_1\right]_{i,q} \sum_{j=1}^{n^2}  \left( 1- p^{-1}\delta_{i,j}\right)  \left[ \mB_1\right] _{q,j} \cdot  \left[ \mX_3^{t, \ell}  \otimes \mX_{2}^{t,\ell}  \right]_{j,s} \right)^2 }\notag  \\
	&\leq \sqrt{ \sum_{i=1}^{n}\sum_{s=1}^{r^2} \left( \sum_{q=1}^{r^2}\left[\mA_1\right]_{i,q}^2 \right) \cdot \left( \sum_{q=1}^{r^2}  \left(  \sum_{j=1}^{n^2}  \left( 1- p^{-1}\delta_{i,j}\right)  \left[ \mB_1\right] _{q,j} \cdot  \left[ \mX_3^{t, \ell}  \otimes \mX_{2}^{t,\ell}  \right]_{j,s} \right)^2 \right)} \notag \\
	&\leq \sqrt{ \lb\sum_{i=1}^{n}\sum_{q=1}^{r^2}\left[\mA_1\right]_{i,q}^2\rb \cdot \max_{i\in[n]}\sum_{s=1}^{r^2} \sum_{q=1}^{r^2}  \left(  \sum_{j=1}^{n^2}  \left( 1- p^{-1}\delta_{i,j}\right)  \left[ \mB_1\right] _{q,j} \cdot  \left[ \mX_3^{t, \ell}  \otimes \mX_{2}^{t,\ell}  \right]_{j,s} \right)^2 } \notag\\ 
	&= \fronorm{\mA_1} \cdot  \sqrt{  \max_{i\in [n]}\sum_{s=1}^{r^2} \sum_{q=1}^{r^2}  \left(  \sum_{j=1}^{n^2}  \left( 1- p^{-1}\delta_{i,j}\right)  \left[ \mB_1\right] _{q,j} \cdot  \left[ \mX_3^{t, \ell}  \otimes \mX_{2}^{t,\ell}  \right]_{j,s} \right)^2 }\notag \\
	&\leq \fronorm{\mA_1}\cdot r^2 \cdot \max_{i\in [n], s,q\in [r^2]} \left|  \sum_{j=1}^{n^2}  \left( 1- p^{-1}\delta_{i,j}\right)  \left[ \mB_1\right] _{q,j} \cdot  \left[ \mX_3^{t, \ell}  \otimes \mX_{2}^{t,\ell}  \right]_{j,s}  \right|\notag \\
	&\stackrel{(a)}{\leq}  \fronorm{\mA_1}\cdot r^2  \cdot 3n^2 \cdot \infnorm{\mB_1} \cdot  \twoinf{\mX_2^{t,\ell}}\twoinf{\mX_3^{t,\ell}}\notag\\
	&\leq  \fronorm{ \mD_1^{t,\ell}   }\cdot \opnorm{\calM_1   \left( \calT+\calE^{t-1} \right) } \cdot r^2  \cdot 3n^2 \cdot \twoinf{ \mX_3^{t,\ell} }^2 \cdot \twoinf{ \mX_2^{t,\ell} \mR^{t,\ell}_2- \mU_2   } \cdot \twoinf{\mX_2^{t,\ell}} \notag\\
	&\stackrel{(b)}{\leq} \frac{1}{2^{20} \kappa^2\mu^2 r^4} \frac{1}{2^t} \sqrt{\frac{\mu r}{n}}\cdot \frac{9}{8}\sigma_{\max}(\calT)\cdot 3n^2 r^2\cdot \left( \frac{9}{8}\right)^3\left( \frac{\mu r}{n} \right)^{3/2} \cdot  \frac{1}{2^{20} \kappa^2\mu^2 r^4} \frac{1}{2^t} \sqrt{\frac{\mu r}{n}} \notag\\
	&\leq \frac{1}{6}\frac{1}{2^{14}} \cdot \frac{1}{2^{20}\kappa^4\mu^2 r^4 } \frac{1}{2^{t+1}} \sqrt{\frac{\mu r}{n}}\sigma_{\max}(\calT),
\end{align}
where $\odot$ denotes the entrywise product, step (a) uses the same argument as in \eqref{eq: sum of delta} and step (b) follows from  the inequalities~\eqref{eq: xitl} and 
\begin{align*}
\opnorm{\calM_1(\calT + \calE^{t-1})}&\leq \sigma_{\max}(\calT) + \frac{1}{2^{20}\kappa^6 \mu^2 r^4}\frac{1}{2^t}\sigma_{\max}(\calT) \leq \frac{9}{8}\sigma_{\max}(\calT).
\end{align*}
Applying the same argument as above can show that 
\begin{align}
	\label{eq r52}
	\eta_{5,2}  &\leq \frac{1}{6}\frac{1}{2^{14}} \cdot \frac{1}{2^{20}\kappa^4\mu^2 r^4 } \frac{1}{2^{t+1}} \sqrt{\frac{\mu r}{n}}\sigma_{\max}(\calT),\\
	\label{eq r53}
	\eta_{5,3}   &\leq \frac{1}{6}\frac{1}{2^{14}} \cdot \frac{1}{2^{20}\kappa^4\mu^2 r^4 } \frac{1}{2^{t+1}} \sqrt{\frac{\mu r}{n}}\sigma_{\max}(\calT),\\
	\label{eq r54}
	\eta_{5,4}  &\leq \frac{1}{6}\frac{1}{2^{14}} \cdot \frac{1}{2^{20}\kappa^4\mu^2 r^4 } \frac{1}{2^{t+1}} \sqrt{\frac{\mu r}{n}}\sigma_{\max}(\calT).
\end{align}

\paragraph{Bounding $\eta_{5,5}$.}
Notice that
\begin{align*}
	\calM_1\left(\calB_5\right) &=\calM_1 \left( \left( \calT+\calE^{t-1} \right) \times_1 \mD_1^{t,\ell}  \left( \mX_1^t \mR_1^t \right)^\tran \times_2  {\mU_2}{\mU_2}^\tran \times_3  {\mU_3}{\mU_3}^\tran\right)\\
	&= \underbrace{\mD_1^{t,\ell}  \left( \mX_1^t \mR_1^t \right)^\tran \calM_1  \left( \calT+\calE^{t-1} \right) \left(\mU_3\otimes\mU_2\right) }_{=:\mA_{5}} \underbrace{ \left(\mU_3\otimes\mU_2\right)^\tran}_{=:\mB_5},
\end{align*}
where the matrices $\mA_5$ and $\mB_5$ obey that
\begin{align*}
	\fronorm{\mA_5} &\leq \fronorm{\mD_1^{t,\ell}} \cdot \opnorm{ \calM_1\left( \calT + \calE^{t-1} \right)} \leq 
 \frac{9}{2^{23} \kappa^2\mu^2 r^4} \frac{1}{2^t} \sqrt{\frac{\mu r}{n}}\sigma_{\max}(\calT),\\
	\infnorm{\mB_5} &\leq  \twoinf{\mU_2}\cdot \twoinf{\mU_3}\leq \frac{\mu r}{n}.
\end{align*}
A simple computation yields that  
\begin{align*}
	\eta_{5,5} 
	&= \fronorm{\calM_1 \left( \left( \calI - p^{-1} \calP_{\Omega}  \right)  \left(\calB_5 \right)  \right) \left( \mX_3^{t, \ell}  \otimes \mX_{2}^{t,\ell}\right) \left(\mR_3^{t,\ell} \otimes \mR_2^{t,\ell}\right) }\\
	& \leq \fronorm{\calM_1 \left( \left( \calI - p^{-1} \calP_{\Omega}  \right)  \left(\calB_5 \right)  \right) \left( \lb\mX_3^{t, \ell}\mR_3^{t,\ell} \rb \otimes \lb\mX_{2}^{t,\ell} \mR_2^{t,\ell}\rb- \mU_3\otimes \mU_2\right)} \\
	&\quad + \fronorm{\calM_1 \left( \left( \calI - p^{-1} \calP_{\Omega}  \right)  \left(\calB_5 \right)  \right) \left(\mU_3\otimes \mU_2\right)} \\
	& \leq \underbrace{ \fronorm{\calM_1 \left( \left( \calI - p^{-1} \calP_{\Omega}  \right)  \left(\calB_5 \right)  \right) \left( \left(   \mX_3^{t,\ell} \mR_3^{t,\ell} - \mU_3  \right)  \otimes \lb\mX_{2}^{t,\ell} \mR_2^{t,\ell}\rb \right)}  }_{=:\eta_{5,5}^a }\\
	&\quad + \underbrace{\fronorm{ \calM_1 \left( \left( \calI - p^{-1} \calP_{\Omega}  \right)  \left(\calB_5 \right)  \right) \left( \mU_3\otimes\left(   \mX_2^{t,\ell} \mR_2^{t,\ell} - \mU_2 \right)\right) } }_{=:\eta_{5,5}^b} \\
	&\quad  + \underbrace{\fronorm{\calM_1 \left( \left( \calI - p^{-1} \calP_{\Omega}  \right)  \left(\calB_5 \right)  \right) \left(\mU_3\otimes \mU_2\right)} }_{=:\eta_{5,5}^c }.
\end{align*}
\begin{itemize}
	\item Controlling  $\eta_{5,5}^a$. It can be bounded as follows:
	\begin{align*}
		\eta_{5,5}^a &= \fronorm{\calM_1 \left( \left( \calI - p^{-1} \calP_{\Omega}  \right)  \left(\calB_5 \right)  \right) \left( \left(   \mX_3^{t,\ell} \mR_3^{t,\ell} - \mU_3  \right) \otimes \lb\mX_{2}^{t,\ell} \mR_2^{t,\ell} \rb\right)} \\
		&\leq \fronorm{\mA_5}\cdot r^2 \cdot \max_{i\in [n], s,q\in [r^2]} \left|  \sum_{j=1}^{n^2}  \left( 1- p^{-1}\delta_{i,j}\right)  \left[ \mB_5\right] _{q,j} \cdot  \left[ \left(   \mX_3^{t,\ell} \mR_3^{t,\ell} - \mU_3  \right)\otimes \lb\mX_{2}^{t,\ell} \mR_2^{t,\ell}\rb  \right]_{j,s}  \right| \\ 
		&\stackrel{(a)}{\leq} \fronorm{\mA_5}\cdot r^2 \cdot 3n^2 \infnorm{\mB_5}\cdot \infnorm{\left(   \mX_3^{t,\ell} \mR_3^{t,\ell} - \mU_3  \right) \otimes \lb\mX_{2}^{t,\ell} \mR_2^{t,\ell} \rb}\\
		&\leq \fronorm{\mA_5}\cdot r^2 \cdot 3n^2 \infnorm{\mB_5}\cdot \twoinf{  \mX_3^{t,\ell} \mR_3^{t,\ell} - \mU_3 } \cdot \twoinf{\mX_{2}^{t,\ell} \mR_2^{t,\ell} }\\
		&\leq \frac{9}{2^{23} \kappa^2\mu^2 r^4} \frac{1}{2^t} \sqrt{\frac{\mu r}{n}}\sigma_{\max}(\calT)\cdot 3n^2r^2\cdot \frac{\mu r}{n}\cdot \frac{1}{2^{20} \kappa^2\mu^2 r^4} \frac{1}{2^t} \sqrt{\frac{\mu r}{n}} \cdot  \frac{9}{8}\sqrt{\frac{\mu r}{n}}\\
		&\leq \frac{1}{18}\frac{1}{2^{14}} \cdot \frac{1}{2^{20}\kappa^4\mu^2 r^4 } \frac{1}{2^{t+1}} \sqrt{\frac{\mu r}{n}}\sigma_{\max}(\calT),
	\end{align*}
	where step (a) uses the same argument as in \eqref{eq: sum of delta}.
	\item Controlling  $\eta_{5,5}^b$. Similarly, one has
	\begin{align*}
		\eta_{5,5}^b&\leq \fronorm{\mA_5}\cdot r^2 \cdot \max_{i\in [n], s,q\in [r^2]} \left|  \sum_{j=1}^{n^2}  \left( 1- p^{-1}\delta_{i,j}\right)  \left[ \mB_5\right] _{q,j} \cdot  \left[\mU_3 \otimes \left(   \mX_2^{t,\ell} \mR_2^{t,\ell} - \mU_2  \right) \right]_{j,s}  \right| \\ 
		&\leq \fronorm{\mA_5}\cdot r^2 \cdot 3n^2 \infnorm{\mB_5}\cdot \infnorm{\mU_3 \otimes \left(   \mX_2^{t,\ell} \mR_2^{t,\ell} - \mU_2 \right) }\\
		&\leq \frac{1}{18}\frac{1}{2^{14}} \cdot \frac{1}{2^{20}\kappa^4\mu^2 r^4 } \frac{1}{2^{t+1}} \sqrt{\frac{\mu r}{n}}\sigma_{\max}(\calT).
	\end{align*}
	\item Controlling  $\eta_{5,5}^c$. Recall that 
	\begin{align*}
		\calB_5 &=  \left( \calT+\calE^{t-1} \right) \times_1 \mD_1^{t,\ell}  \left( \mX_1^t \mR_1^t \right)^\tran \times_2  {\mU_2}{\mU_2}^\tran \times_3  {\mU_3}{\mU_3}^\tran \\
		&= \underbrace{\left( \left( \calT+\calE^{t-1} \right) \times_1 \mD_1^{t,\ell}  \left( \mX_1^t \mR_1^t \right)^\tran \times_2  \mU_2^\tran \times_3 \mU_3^\tran \right) }_{=:\calY_{5,5}}\times_2\mU_2 \times_3 \mU_3,
	\end{align*}
	where $\calY_{5,5}\in\R^{n\times r\times r}$.
	A straightforward computation yields that
	\begin{align*}
		\eta_{5,5}^c&= \fronorm{\calM_1 \left( \left( \calI - p^{-1} \calP_{\Omega}  \right)  \left(\calB_5 \right)  \right) \left(\mU_3\otimes \mU_2\right)}\\
		&=  \fronorm{ \left( \left( \calI - p^{-1} \calP_{\Omega}  \right)  \left(\calB_5 \right)  \right) \times_2 \mU_2^\tran \times_3 \mU_3^\tran }\\
		&= \sup_{\substack{\calZ\in\Rn: \fronorm{\calZ} = 1}} \left\langle \left( \calI - p^{-1} \calP_{\Omega}  \right)  \left(\calY_{5,5}\times_2\mU_2 \times_3 \mU_3\right) , \calZ\times_2 \mU_2  \times_3 \mU_3\right\rangle \\
		&\stackrel{(a)}{\leq } C\sqrt{\frac{\mu^2 r^2\log n}{n^2p}}\cdot\fronorm{\calY_{5,5}}= C\sqrt{\frac{\mu^2 r^2\log n}{n^2p}}\cdot\fronorm{\calM_1\left(\calY_{5,5}\right)} \\
	    &\leq C\sqrt{\frac{\mu^2 r^2\log n}{n^2p}}\cdot\fronorm{\mD_1^{t,\ell}}\cdot \opnorm{\mX_1^t}\cdot \opnorm{\calM_1\left( \calT + \calE^{t-1}\right)}\\
		&\leq C\sqrt{\frac{\mu^2 r^2\log n}{n^2p}}\cdot\frac{1}{2^{20}\kappa^2\mu^2r^4}\frac{1}{2^t}\sqrt{\frac{\mu r}{n}}\cdot\frac{9}{8}\sigma_{\max}\lb\calT\rb\\
		&\leq \frac{1}{18}\frac{1}{2^{14}} \cdot \frac{1}{2^{20}\kappa^4\mu^2 r^4 } \frac{1}{2^{t+1}} \sqrt{\frac{\mu r}{n}}\sigma_{\max}(\calT),
	\end{align*}
provided that 
	$p\geq \frac{C_2 \kappa^4\mu^2 r^2\log n}{n^2},$
where  $(a)$ follows from Lemma~\ref{ZC lemma 1} .
\end{itemize}
Combining the above bounds together, we have
\begin{align}
	\label{eq r55}
	\eta_{5,5} &\leq \frac{1}{6}\frac{1}{2^{14}} \cdot \frac{1}{2^{20}\kappa^4\mu^2 r^4 } \frac{1}{2^{t+1}} \sqrt{\frac{\mu r}{n}}\sigma_{\max}(\calT).
\end{align}

\paragraph{Bounding  $\eta_{5,6}$.}
Notice that
\begin{align*}
	\calB_6 &=\left( \calT+\calE^{t-1} \right) \times_1 \left( \mX_{1}^{t,\ell} \mT_1^{t,\ell} \right) \left(  \mD_1^{t,\ell}  \right)^\tran \times_2 {\mX_2^{t,\ell}}{\mX_2^{t,\ell}}^\tran\times_3 {\mX_3^{t,\ell}}{\mX_3^{t,\ell}}^\tran\\
	&= \underbrace{ \left( \left( \calT+\calE^{t-1} \right) \times_1 \left( \mX_{1}^{t,\ell} \mT_1^{t,\ell} \right) \left(  \mD_1^{t,\ell}  \right)^\tran \times_2  {\mX_2^{t,\ell}}^\tran \times_3 {\mX_3^{t,\ell}}^\tran  \right)}_{=:\calY_{5,6}} \times_2  \mX_2^{t,\ell} \times_3 \mX_3^{t,\ell},
\end{align*}
where $\calY_{5,6}\in\R^{n\times r\times r}$. Following the same argument of bounding $\gamma_{3,1}^a$, one has
\begin{align}\label{eq r56}
    \eta_{5,6}\leq \frac{1}{6}\frac{1}{2^{14}} \cdot \frac{1}{2^{20}\kappa^4\mu^2 r^4 } \frac{1}{2^{t+1}} \sqrt{\frac{\mu r}{n}}\sigma_{\max}(\calT).
\end{align}

\paragraph{Combining $\eta_{5,i}$ together.}
Combining \eqref{eq r51}, \eqref{eq r52}, \eqref{eq r53}, \eqref{eq r54}, \eqref{eq r55} and \eqref{eq r56} together, we have
\begin{align}
	\label{eq r5}
	\eta_5 \leq \sum_{i=1}^{6} \eta_{5,i} \leq  \frac{1}{2^{14}} \cdot \frac{1}{2^{20}\kappa^4\mu^2 r^4 } \frac{1}{2^{t+1}} \sqrt{\frac{\mu r}{n}}\sigma_{\max}(\calT).
\end{align}

\subsubsection{Combining \texorpdfstring{$\eta_i$}{TEXT} together}
Combining \eqref{eq eta} and \eqref{eq r5} together yields that
\begin{align*}
	\fronorm{   \left(  \left( \calI - p^{-1} \calP_{\Omega}  \right)  \left( \calX^t-\calX^{t,\ell}\right) \right)  \underset{j\neq 1}{ \times }  {\mX_{j}^{t,\ell}}^\tran    }  &\leq \sum_{i=1}^{8}\eta_i \leq \frac{1}{2^{11}} \cdot \frac{1}{2^{20}\kappa^4\mu^2 r^4 } \frac{1}{2^{t+1}} \sqrt{\frac{\mu r}{n}}\sigma_{\max}(\calT).
\end{align*}
Thus we have completed the proof of \eqref{eq: claim rc 1} in Claim \ref{claim in R_c}.

\subsection{Proof of \texorpdfstring{\eqref{eq: claim rc 2}}{TEXT} in Claim \texorpdfstring{\ref{claim in R_c}}{TEXT}}
\label{second term in 5.5}
We prove the case $i=1$ and the others are similar. 
Notice that only the $\ell$-th row and those columns in the set $\Gamma$ of 
\begin{align*}
	\calM_1\left(    \left( p^{-1}\calP_{\Omega_{-\ell}} + \calP_{\ell} - p^{-1} \calP_{\Omega}  \right) \left(\calX^{t,\ell}-\calT\right)  \right)
\end{align*}
are non-zeros. By the definition of $\calP_{\ell, :}$ and $\calP_{-\ell, \Gamma}$ in \eqref{eq submatrix row} and \eqref{eq submatrix columns}, we have 
\begin{align*}
	& \fronorm{    \left( \left( p^{-1}\calP_{\Omega_{-\ell}} + \calP_{\ell} - p^{-1} \calP_{\Omega}  \right) \left(\calX^{t,\ell}-\calT\right) \right)\underset{j\neq 1}{ \times } {\mX_{j}^{t,\ell}}^\tran   } \\
	&= \fronorm{\calM_1\left(    \left( p^{-1}\calP_{\Omega_{-\ell}} + \calP_{\ell} - p^{-1} \calP_{\Omega}  \right) \left(\calX^{t,\ell}-\calT\right)  \right) \left(\mX_3^{t,\ell} \otimes \mX_2^{t,\ell}  \right)} \\
	&\leq  \twonorm{\ve_\ell^\tran\calM_1\left(    \left( p^{-1}\calP_{\Omega_{-\ell}} + \calP_{\ell} - p^{-1} \calP_{\Omega}  \right) \left(\calX^{t,\ell}-\calT\right)  \right) \left(\mX_3^{t,\ell} \otimes \mX_2^{t,\ell}  \right) } \\
	&\quad + \fronorm{ \calP_{-\ell, \Gamma} \left( \calM_1\left(    \left( p^{-1}\calP_{\Omega_{-\ell}} + \calP_{\ell} - p^{-1} \calP_{\Omega}  \right) \left(\calX^{t,\ell}-\calT\right)  \right)\right) \left(\mX_3^{t,\ell} \otimes \mX_2^{t,\ell}  \right)} := \phi_1 + \phi_2.
\end{align*}
\subsubsection{Bounding  \texorpdfstring{$\phi_1$}{TEXT}}
A straightforward computation yields that
\begin{align*}
	\phi_1 &=  \twonorm{\sum_{j=1}^{n^2}  \lsb \calM_1\left(    \left( p^{-1}\calP_{\Omega_{-\ell}} + \calP_{\ell} - p^{-1} \calP_{\Omega}  \right) \left(\calX^{t,\ell}-\calT\right)  \right) \rsb_{\ell,j} \lsb \mX_3^{t,\ell} \otimes \mX_2^{t,\ell}  \rsb_{j,:} } \\
	&=\twonorm{\sum_{j=1}^{n^2}  \left( 1-p^{-1}\delta_{\ell,j} \right)\lsb \calM_1 \left(\calX^{t,\ell}-\calT\right)   \rsb_{\ell,j} \lsb \mX_3^{t,\ell} \otimes \mX_2^{t,\ell}  \rsb_{j,:} } :=\twonorm{\sum_{j=1}^{n^2} \vx_j^\tran}.
\end{align*}
Notice that conditioned on $\mX_i^{t,\ell}$ for $i=2,3$, the vectors $\vx_j\in\R^{r^2\times 1}$ are independent mean-zero random vectors with
\begin{align*}
	\twonorm{\vx_j} &\leq \frac{1}{p} \infnorm{\calX^{t,\ell} - \calT}\cdot \twoinf{\mX_2^{t,\ell}}\cdot \twoinf{\mX_3^{t,\ell}},\\
	\opnorm{\sum_{j=1}^{n^2}\E{\vx_j\vx_j^\tran}} &\leq p^{-1}\sum_{j=1}^{n^2} \lsb \calM_1 \left(\calX^{t,\ell}-\calT\right)   \rsb_{\ell,j} ^2 \cdot \twonorm{ \lsb \mX_3^{t,\ell} \otimes \mX_2^{t,\ell}  \rsb_{j,:} }^2\\
	&\leq \frac{n^2}{p}\cdot \infnorm{\calX^{t,\ell} - \calT}^2\cdot  \twoinf{\mX_2^{t,\ell}}  ^2 \cdot  \twoinf{\mX_3^{t,\ell} }^2,\\
	\left|\sum_{j=1}^{n^2}\E{\vx_j^\tran\vx_j} \right| &\leq \frac{n^2}{p}\cdot \infnorm{\calX^{t,\ell} - \calT}^2\cdot  \twoinf{\mX_2^{t,\ell}}  ^2 \cdot  \twoinf{\mX_3^{t,\ell} }^2.
\end{align*}
By the matrix Bernstein inequality, with high probability, one has
\begin{align}
	\label{eq phi1}
	\phi_1 &\leq C \left( \frac{\log n}{p} + \sqrt{\frac{n^2\log n}{p}} \right)\infnorm{\calX^{t,\ell} - \calT}\cdot \twoinf{\mX_2^{t,\ell}}\cdot \twoinf{\mX_3^{t,\ell}} \notag\\
	&\leq C \left( \frac{\log n}{p} + \sqrt{\frac{n^2\log n}{p}} \right)\cdot \frac{36}{2^{20} \kappa^2 \mu^2 r^{4}}\cdot \frac{1}{2^t} \cdot \left( \frac{\mu r}{n} \right)^{3/2} \sigma_{\max}(\calT) \cdot \lb\frac{9}{8}\rb^2\frac{\mu r}{n}\notag\\
	&\leq  \frac{1}{2^{12}} \cdot \frac{1}{2^{20}\kappa^4\mu^2 r^4 } \frac{1}{2^{t+1}} \sqrt{\frac{\mu r}{n}}\sigma_{\max}(\calT)
\end{align} under the assumption
$	p\geq \frac{C_2 \kappa^4\mu^4 r^4\log n}{n^2}$.

\subsubsection{Bounding \texorpdfstring{$\phi_2$}{TEXT}}
A simple computation yields that
\begin{align*}
	\phi_2 &= \fronorm{ \calP_{-\ell, \Gamma} \left( \calM_1\left(    \left( p^{-1}\calP_{\Omega_{-\ell}} + \calP_{\ell} - p^{-1} \calP_{\Omega}  \right) \left(\calX^{t,\ell}-\calT\right)  \right)\right) \left(\mX_3^{t,\ell} \otimes \mX_2^{t,\ell}  \right)}\\
	&\leq r\opnorm{\calP_{-\ell, \Gamma} \left( \calM_1\left(    \left( p^{-1}\calP_{\Omega_{-\ell}} + \calP_{\ell} - p^{-1} \calP_{\Omega}  \right) \left(\calX^{t,\ell}-\calT\right)  \right)\right) \left(\mX_3^{t,\ell} \otimes \mX_2^{t,\ell}  \right)}\\
	&= r \opnorm{\sum_{j\in \Gamma} \va_j \vb_j^\tran }:=r \opnorm{\sum_{j\in \Gamma} \mZ_j},
\end{align*}
where $\va_j\in\R^{n\times 1}$ with $\lsb\va_j\rsb_k = \left(1-p^{-1}\delta_{k,j}\right)\lsb \calM_1\left(\calX^{t,\ell}-\calT\right)  \rsb_{k,j}$ for $k\neq \ell$ and $\lsb\va_j\rsb_{\ell} = 0$, and 
 $\vb_j = \lsb\mX_3^{t,\ell} \otimes \mX_2^{t,\ell}  \rsb_{j,:}^\tran\in\R^{r^2\times 1}$.
Notice that conditioned on $\mX_i^{t,\ell}$,  $\mZ_j$ are independent  mean-zero matrices with
\begin{align*}
	\opnorm{\mZ_j} &\leq \twonorm{\va_j}\cdot \twonorm{\vb_j} \leq \frac{\sqrt{n}}{p}\cdot \infnorm{\calX^{t,\ell} - \calT}\cdot \twoinf{\mX_2^{t,\ell}}\cdot \twoinf{\mX_3^{t,\ell}},\\
	\opnorm{\sum_{ j\in \Gamma}\E{\mZ_j\mZ_j^\tran}} & = \opnorm{\sum_{ j\in \Gamma}\E{\va_j\va_j^\tran}\twonorm{\vb_j}^2} \leq \max_{i\in [n]} \sum_{j\in \Gamma} p^{-1}\lsb \calM_1\left(\calX^{t,\ell}-\calT\right)  \rsb_{i,j}^2 \cdot \twonorm{\vb_j}^2\\
	&\leq \frac{2n}{p}\cdot \infnorm{\calX^{t,\ell} - \calT}^2 \cdot \twoinf{\mX_2^{t,\ell}}^2 \cdot \twoinf{\mX_3^{t,\ell}}^2,
\end{align*}
where we have used the fact that $|\Gamma| \leq 2n$.  Similarly, one has
\begin{align*}
	\opnorm{\sum_{ j\in \Gamma}\E{\mZ_j^\tran \mZ_j}} & = \opnorm{\sum_{ j\in \Gamma}\E{\twonorm{\va_j}^2}\vb_j\vb_j^\tran} \leq \sum_{ j\in \Gamma} \twonorm{\vb_j}^2\cdot \E{\twonorm{\va_j}^2}\\
	&\leq \sum_{ j\in \Gamma} \twonorm{\vb_j}^2 \cdot p^{-1} \sum_{i=1}^{n} \lsb \calM_1\left(\calX^{t,\ell}-\calT\right)  \rsb_{i,j}^2\\
	&\leq \frac{2n^2}{p} \infnorm{\calX^{t,\ell} - \calT}^2 \cdot \twoinf{\mX_2^{t,\ell}}^2 \cdot \twoinf{\mX_3^{t,\ell}}^2.
\end{align*}
Applying the matrix Bernstein inequality shows that, with high probability,
\begin{align}
	\label{eq phi2}
	\phi_2 & \leq r\cdot C \left( \frac{\sqrt{n}\log n}{p} +\sqrt{\frac{2n^2\log n}{p} }\right)\cdot \infnorm{\calX^{t,\ell} - \calT} \cdot \twoinf{\mX_2^{t,\ell}} \cdot \twoinf{\mX_3^{t,\ell}} \notag \\
	&\leq r\cdot C \left( \frac{\sqrt{n}\log n}{p} + \sqrt{\frac{n^2\log n}{p}} \right)\cdot \frac{36}{2^{20} \kappa^2 \mu^2 r^{4}}\cdot \frac{1}{2^t} \cdot \left( \frac{\mu r}{n} \right)^{3/2} \sigma_{\max}(\calT) \cdot \lb\frac{9}{8}\rb^2\frac{\mu r}{n}\notag\\
	&\leq  \frac{1}{2^{12}} \cdot \frac{1}{2^{20}\kappa^4\mu^2 r^4 } \frac{1}{2^{t+1}} \sqrt{\frac{\mu r}{n}}\sigma_{\max}(\calT),
\end{align}
provided that 
	$p\geq \max\left\lbrace\frac{C_1 \kappa^2 \mu^2 r^3\log n}{n^{1.5}}, \frac{C_2 \kappa^4\mu^4 r^6\log n}{n^{2}}\right\rbrace$.

Combining \eqref{eq phi1} and \eqref{eq phi2} together gives
\begin{align*}
	\fronorm{   \left(\left( p^{-1}\calP_{\Omega_{-\ell}} + \calP_{\ell} - p^{-1} \calP_{\Omega}  \right) \left(\calX^{t,\ell}-\calT\right) \right)   \underset{j\neq 1}{ \times } {\mX_{j}^{t,\ell}}^\tran   }	\leq  \frac{1}{2^{11}} \cdot \frac{1}{2^{20}\kappa^4\mu^2 r^4 } \frac{1}{2^{t+1}} \sqrt{\frac{\mu r}{n}}\sigma_{\max}(\calT),
\end{align*}
which completes the proof of \eqref{eq: claim rc 2} in Claim \ref{claim in R_c}.

\section{Auxiliary Lemmas}

\begin{lemma}[{\citep[Lemma 1]{xia2019polynomial}}] 
	\label{lemma:XY 19 lemma1}
	For any $\calX\in\R^{n\times n\times n}$ with multilinear rank $(r_1, r_2, r_3)$, one has
	\begin{align*}
		\infnorm{\calX} \leq \opnorm{\calX} \leq \fronorm{\calX} \leq \sqrt{r_1r_2r_3}\cdot \opnorm{\calX}\quad\text{and}\quad \nucnorm{\calX}\leq \min\left\lbrace \sqrt{r_1r_2}, \sqrt{r_1r_3}, \sqrt{r_2r_3}\right\rbrace\cdot \fronorm{\calX}.
	\end{align*}
\end{lemma}

\begin{lemma}[{\citep[Lemma 6]{xia2017statistically}}]
	\label{lemma: Lemma 6 in XYZ}
	For a tensor $\calX\in\Rn$ with multilinear rank  $\vr=(r_1, r_2, r_3)$, the following bound holds for $i= 1,2,3$
	\begin{align*}
		\opnorm{\calM_i(\calX)} \leq \sqrt{\frac{\prod_{j = 1}^3 r_{j}}{r_i\cdot\max_{j'\neq i}r_{j'}}}\cdot
		\opnorm{\calX} = \sqrt{\min_{j\neq i} r_j}\cdot
		\opnorm{\calX}.
	\end{align*}
\end{lemma}

\begin{lemma}[{\citep[Lemma EC.12]{cai2021nonconvex}, \citep[Lemma 13]{tong2021scaling}}]
	\label{lemma: spectral norm of E initial}
	Suppose $\Omega$ satisfies the Bernoulli observation model. Then for any fixed $\calZ\in\Rn$,
	\begin{align}
	\opnorm{\lb p^{-1}\calP_{\Omega} - \calI \rb(\calZ)} \leq C \lb p^{-1}\log^3 n \infnorm{\calZ} + \sqrt{p^{-1}\log^5 n }  \max_{i=1,2,3} \twoinf{ \calM_i^\tran \left( \calZ\right)}  \rb
	\end{align}
	holds with high probability.
\end{lemma}


\begin{lemma}
	\label{lemma: multilinear rank unchange}
	Suppose $\calZ\in\R^{n\times n\times n}$ is a tensor with multilinear rank $(r_1,r_2,r_3)$. Then for any $\va,\vb,\vc\in\R^{n\times 1}$ with $\twonorm{\va}=\twonorm{\vb}=\twonorm{\vc}=1$, the tensor  $\calZ\odot(\va\circ\vb\circ\vc)$ has multilinear rank at most $(r_1,r_2,r_3)$, where $\odot$ denotes the entrywise product.
\end{lemma}
\begin{proof}[Proof of Lemma~\ref{lemma: multilinear rank unchange}]
Since $\calZ$ is of multilinear rank $(r_1,r_2,r_3)$, it can be decomposed as $\calZ = \calH\ttimes \mZ_i$, where $\calH\in\R^{r_1\times r_2\times r_3} $ and $\mZ_i\in\R^{n\times r_i}$ for $i  = 1,2,3$. To complete the proof, it suffices to show that the matrix $\calM_i(\calZ\odot(\va\circ\vb\circ\vc))$ is of rank at most $r_i$ for  $i = 1,2,3$. By the definition of entrywise product, 
	\begin{align*}
		\calM_1(\calZ\odot(\va\circ\vb\circ\vc)) &= \calM_1(\calZ)\odot \calM_1(\va\circ\vb\circ\vc)\\
		&= \lb\mZ_1\calM_1(\calH)(\mZ_3\otimes \mZ_2)^\tran\rb \odot \lb \va(\vc\otimes\vb)^\tran \rb\\
		& = \lb\sum_{i = 1}^{r_1}\lsb\mZ_1\rsb_{:,i}\lsb\calM_1(\calH)(\mZ_3\otimes \mZ_2)^\tran\rsb_{i,:}\rb\odot \lb \va(\vc\otimes\vb)^\tran \rb\\
		& = \sum_{i = 1}^{r_1}\lb\lsb\mZ_1\rsb_{:,i}\lsb\calM_1(\calH)(\mZ_3\otimes \mZ_2)^\tran\rsb_{i,:}\rb\odot \lb \va(\vc\otimes\vb)^\tran \rb\\
		& = \sum_{ i =1}^{r_1}\lb\lsb\mZ_1\rsb_{:,i}\odot\va\rb \lb\lsb\calM_1(\calH)(\mZ_3\otimes \mZ_2)^\tran\rsb_{i,:}\odot (\vc\otimes\vb)^\tran\rb,
	\end{align*}
	which implies that this matrix has rank at most $r_1$. This is  certainly true
	 for $i=2,3$. 
\end{proof}

\begin{lemma}[Uniform bound]
	\label{lemma: uniform tensor operator norm}
	Suppose $\calJ\in\Rn$ is the full-one tensor. For any tensor $\calZ\in\R^{n\times n\times n}$ with multilinear rank $(r_1, r_2, r_3)$, the following inequality holds,
	\begin{align*}
		\opnorm{\lb \calI -p^{-1}\calP_{\Omega} \rb \left(\calZ\right) } \leq\opnorm{\lb \calI -p^{-1}\calP_{\Omega} \rb (\calJ)}\cdot \min\left\lbrace \sqrt{r_1r_2}, \sqrt{r_1r_3}, \sqrt{r_2r_3}\right\rbrace \vecnorm{\calZ}{\infty}.
	\end{align*}
\end{lemma}
\begin{proof}[Proof of Lemma~\ref{lemma: uniform tensor operator norm}]
	Let  $\va,\vb,\vc\in\R^{n\times 1}$ be  unit vectors and $\tilde{r}  = \min\left\lbrace \sqrt{r_1r_2}, \sqrt{r_1r_3}, \sqrt{r_2r_3}\right\rbrace$. We have
	\begin{align*}
		\opnorm{\lb \calI - p^{-1}\calP_{\Omega} \rb \left(\calZ\right) }  &=	\opnorm{\lb \calI - p^{-1}\calP_{\Omega} \rb (\calJ)\odot \calZ } \\
		&=\sup_{\substack{ \va,\vb,\vc}} \la\lb \calI -p^{-1}\calP_{\Omega} \rb (\calJ)\odot \calZ, \va\circ\vb\circ\vc \ra\\
		&= \sup_{\substack{ \va,\vb,\vc }} \la\lb \calI - p^{-1}\calP_{\Omega} \rb (\calJ), \calZ\odot( \va\circ\vb\circ\vc)\ra\\
		&\leq \opnorm{\lb \calI -p^{-1}\calP_{\Omega} \rb (\calJ)}\cdot \sup_{\substack{ \va,\vb,\vc }}\nucnorm{\calZ\odot(\va\circ\vb\circ\vc)  }\\
		&\stackrel{(a)}{\leq} \opnorm{\lb \calI - p^{-1}\calP_{\Omega} \rb (\calJ)}\cdot \tilde{r}\sup_{\substack{ \va,\vb,\vc }}\fronorm{ \calZ\odot(\va\circ\vb\circ\vc) }\\
		&= \opnorm{\lb \calI - p^{-1}\calP_{\Omega} \rb (\calJ)}\cdot \tilde{r}\sup_{\substack{  \va,\vb,\vc  }}\sqrt{\sum_{i_1,i_2,i_3} \calZ_{i_1,i_2,i_3}^2(a_{i_1} b_{i_2} c_{i_3})^2 }\\
		&\leq \opnorm{\lb \calI - p^{-1}\calP_{\Omega} \rb (\calJ)}\cdot \tilde{r}\sup_{\substack{  \va,\vb,\vc  }}\sqrt{\sum_{i_1,i_2,i_3} (a_{i_1} b_{i_2} c_{i_3})^2 }\cdot \vecnorm{\calZ}{\infty}\\
		&=\opnorm{\lb \calI - p^{-1}\calP_{\Omega} \rb (\calJ)}\cdot \tilde{r}\vecnorm{\calZ}{\infty},
	\end{align*}
	where (a) is due to Lemma~\ref{lemma:XY 19 lemma1} and Lemma~\ref{lemma: multilinear rank unchange}. 
 \end{proof}

\begin{lemma}[{\citep[Lemma 45]{ma2020implicit}, \citep[Lemma 16]{chen2020nonconvex}}]
	\label{lemma: Ma 2017 lemma 45}
	Let $\mM$ and $\widehat{\mM}\in\R^{n\times n}$ be  symmetric matrices with top-r eigenvalue decomposition $\mM=\mU\bm{\Lambda}\mU^{\tran}$ and $\widehat{\mM}=\widehat{\mU}\widehat{\bm{\Lambda}}\widehat{\mU}$, respectively. Assume $\sigma_r(\mM)>0$, $\sigma_{r+1}(\mM) = 0$ and $\opnorm{\mM-\widehat{\mM}}\leq \frac{1}{4}\sigma_r(\mM)$. Define
		$\mQ = \arg\min_{\mR^\tran\mR = \mI}\fronorm{\widehat{\mU}\mR-\mU}.$	 
	Then
	\begin{align*}
		\opnorm{\widehat{\mU}\mQ-\mU}\leq \frac{3}{\sigma_r(\mM)}\opnorm{\mM-\widehat{\mM}}.
	\end{align*}
\end{lemma}


\begin{lemma}
\label{ZC lemma 1}
Let $\calT\in\Rn$ be a tensor with Tucker decomposition $\calS\ttimes \mU_i$ where $\mU_i^\tran\mU_i=\mI\in\R^{r\times r}$ for $i = 1,2,3$. Suppose $\calT$ is $\mu$-incoherent, and $\Omega$ obeys the Bernoulli observation with parameter $p$. If  $p\geq C_2\mu^2r^2\log n/n^2$, then with high probability, 
\begin{align*}
\left| \la (\calI - p^{-1} \calP_{\Omega} )(\calX),~ \calZ\ra \right| \leq C\sqrt{\frac{\mu^2r^2\log n}{n^2p}}\fronorm{\calX}\fronorm{\calZ}, 
\end{align*}
holds simultaneously for all tensors $\calX, \calZ\in\R^{n\times n\times n}$ of the form
\begin{align*}
	\calX = \calG\times_i \mX_i \times_{j\neq i} \mU_i, \calZ = \calH \times_i \mZ_i \times_{j\neq i } \mU_i,
\end{align*}
where  $\calG,\calH\in\R^{n\times r\times r}$ and $\mX_i, \mZ_i\in\R^{n\times n}$ are arbitrary factors.
\end{lemma} 
\begin{remark}
This lemma is essentially a restatement  of Lemma 12 in \citep{tong2021scaling}.  Considering the case i = 1, the only  difference is that here $\calG$ (or $\calH$) and $\mX_i$ (or $\mZ_i$) are respectively of size $n\times r\times r$ and $n\times n$ instead of $r\times r\times r$ and $n\times r$. However, the proof therein is equally applicable here and thus we omit the  proof.
\end{remark}

\begin{lemma}[{\citep[Lemma 14]{tong2021scaling}}]
	\label{lemma: tong lemma 14}
	Suppose $\Omega$ satisfies the Bernoulli observation model. Then with high probability, 
	\begin{align*}
		\left\vert \left\langle \lb \calI - p^{-1}\calP_{\Omega}\rb \left(\calG\ttimes \mX_i \right) ,\calH\ttimes\mZ_i\right\rangle\right\vert\leq C\lb p^{-1}\log^3 n+\sqrt{p^{-1}n\log^5 n}\rb\tau,
	\end{align*}
	holds simultaneously for all tensors $\calG\ttimes \mX_i$ and $\calH\ttimes\mZ_i$,
	where the quantity $\tau$ obeys
	\begin{align*}
		\tau&\leq \lb\twoinf{\mX_1\calM_1(\calG)}\fronorm{\mZ_1\calM_1(\calH)} \wedge  \fronorm{\mX_1\calM_1(\calG)} \twoinf{\mZ_1\calM_1(\calH)}\rb\\
		&\quad\quad\cdot \lb\twoinf{\mX_2}\fronorm{\mZ_2}\wedge\fronorm{\mX_2}\twoinf{\mZ_2}\rb\lb\twoinf{\mX_3}\fronorm{\mZ_3}\wedge\fronorm{\mX_3}\twoinf{\mZ_3}\rb.
	\end{align*}
\end{lemma}
\begin{lemma}
	\label{lemma: loo tong lemma 14}
	Suppose $\Omega$ satisfies the Bernoulli observation model. Then with high probability, 
	\begin{align*}
		\left\vert \left\langle \lb \calI - p^{-1}\calP_{\Omega_{-\ell}} - \calP_{\ell}\rb \left(\calG\ttimes \mX_i \right) ,\calH\ttimes\mZ_i\right\rangle\right\vert\leq C\lb p^{-1}\log^3 n+\sqrt{p^{-1}n\log^5 n}\rb\tau,
	\end{align*}
	holds simultaneously for all tensors $\calG\ttimes \mX_i$ and $\calH\ttimes\mZ_i$,
	where the quantity $\tau$ obeys
	\begin{align*}
		\tau&\leq \lb\twoinf{\mX_1\calM_1(\calG)}\fronorm{\mZ_1\calM_1(\calH)} \wedge  \fronorm{\mX_1\calM_1(\calG)} \twoinf{\mZ_1\calM_1(\calH)}\rb\\
		&\quad\quad\cdot \lb\twoinf{\mX_2}\fronorm{\mZ_2}\wedge\fronorm{\mX_2}\twoinf{\mZ_2}\rb\lb\twoinf{\mX_3}\fronorm{\mZ_3}\wedge\fronorm{\mX_3}\twoinf{\mZ_3}\rb.
	\end{align*}
\end{lemma}
\begin{remark}
   The proof of Lemma \ref{lemma: loo tong lemma 14} is similar to the proof of Lemma \ref{lemma: tong lemma 14}, so we omit it.  
\end{remark}

\begin{lemma}[{\citep[Lemma 37]{ma2018implicit}}, {\citep[Lemma 6]{chen2020nonconvex}}]
	\label{lemma Ma 2017 lemma 37}
	Suppose $\mU$, $\mX_1$, $\mX_2 \in\R^{n\times r}$ are matrices such that
	\begin{align*}
		\opnorm{\mX_1-\mU}\opnorm{\mU}\leq \frac{\sigma^2_r(\mU)}{2}\quad\mbox{and}\quad\opnorm{\mX_1-\mX_2}\opnorm{\mU}\leq \frac{\sigma^2_r(\mU)}{4}.
	\end{align*} 
	Let $\mR_1$ and $\mR_2$ be orthogonal matrices such that
	\begin{align*}
		\mR_1 = \arg\min_{\mR^\tran\mR = \mI}\fronorm{\mX_1\mR-\mU}\text{ and }~\mR_2 = \arg\min_{\mR^\tran\mR = \mI}\fronorm{\mX_2\mR-\mU}.
	\end{align*}
	Then one has
	\begin{align*}
		\fronorm{\mX_1\mR_1-\mX_2\mR_2}\leq 5\frac{\sigma^2_1(\mU)}{\sigma^2_r(\mU)}\fronorm{\mX_1-\mX_2}.
	\end{align*}
\end{lemma}


\begin{lemma}[{\citep[Lemma 4.1]{wei2020guarantees}}]
\label{lemma: lemma 4.1 in Wei}
	Let $\widehat{\mX} = \widehat{\mU}\widehat{\bSigma}\widehat{\mV}^\tran$ and $\mX=\mU\bSigma\mV^\tran$  be rank $r$ matrices. Then
	\begin{align*}
    \opnorm{\widehat{\mU}\widehat{\mU}^\tran-\mU\mU^\tran}\leq \frac{\fronorm{\widehat{\mX} - \mX}}{\sigma_{\min}(\mX)}\quad\mbox{and}\quad	\opnorm{\widehat{\mV}\widehat{\mV}^\tran - \mV\mV^\tran} \leq \frac{\fronorm{\widehat{\mX} - \mX}}{\sigma_{\min}(\mX)}.
	\end{align*}
\end{lemma}

\begin{lemma}[{\citep[Lemma 1]{cai2018rate}}]
	\label{rateoptimal lemma 1}
	Suppose $\mV,\widehat{\mV}\in \R^{n\times r}$ are orthonormal matrices. Then the following relations hold,
	\begin{align*}
		&\fronorm{\sin\Theta\lb\mV,\widehat{\mV}\rb}\leq \inf_{\mR^\tran\mR = \mI}\fronorm{\widehat{\mV}-\mV\mR}\leq \sqrt{2}\fronorm{\sin\Theta\lb\mV,\widehat{\mV}\rb},\\
		&\opnorm{\sin\Theta\lb\mV,\widehat{\mV}\rb}\leq \inf_{\mR^\tran\mR = \mI}\opnorm{\widehat{\mV}-\mV\mR}\leq \sqrt{2}\opnorm{\sin\Theta\lb\mV,\widehat{\mV}\rb},\\
		&\opnorm{\sin\Theta\lb\mV,\widehat{\mV}\rb}\leq \opnorm{\widehat{\mV}{\widehat{\mV}}^{\tran}-\mV\mV^{\tran}}\leq 2\opnorm{\sin\Theta\lb\mV,\widehat{\mV}\rb},\\
		&\fronorm{\widehat{\mV}{\widehat{\mV}}^{\tran}-\mV\mV^{\tran}} = \sqrt{2}\fronorm{\sin\Theta\lb\mV,\widehat{\mV}\rb}.
	\end{align*}
\end{lemma}

\begin{lemma}[Davis-Kahan $\sin\Theta$ Theorem]
	\label{Davis-kahan sintheta theorem}
	Suppose $\mG^\natural, \bDelta\in\R^{n\times n}$ are symmetric matrices, and $\widehat{\mG} = \mG^\natural +\bDelta$. Let $\delta = \sigma_r(\mG^\natural)-\sigma_{r+1}(\mG^\natural)$ be the gap between the top $r$-th and $r+1$-th eigenvalues of $\mG^\natural$, and $\mU$, $\widehat{\mU}$ be matrices whose columns are the top $r$ orthonormal eigenvectors of $\mG^\natural$ and $\widehat{\mG}$ respectively. If $\delta>\opnorm{\bDelta}$, then we have
	\begin{align*}
\opnorm{\sin\Theta\lb\mU,\widehat{\mU}\rb}\leq\frac{\opnorm{\bDelta\mU}}{\delta-\opnorm{\bDelta}}\quad\mbox{and}\quad\fronorm{\sin\Theta\lb\mU,\widehat{\mU}\rb}\leq\frac{\fronorm{\bDelta\mU}}{\delta-\opnorm{\bDelta}}.
	\end{align*}
\end{lemma}

\begin{lemma}[{\citep[Lemma 1]{ding2020leave}}]
\label{lemma: ding lemma 1}
Suppose  $\mG^\natural, \bDelta\in\R^{n\times n}$ are symmetric matrices, and $\widehat{\mG} = \mG^\natural + \bDelta$. The eigenvalue decomposition of $\mG^\natural$ is denoted $\mG^\natural = \mU\bLambda\mU^\tran$, where $\mU\in\R^{n\times r}$ has orthonormal columns and $\bLambda = \diag(\sigma_1,\cdots, \sigma_r)$. Let $\mX\in\R^{n\times r}$ be the matrix whose columns are the top-$r$ orthonormal eigenvectors of $\widehat{\mG}$. Let the SVD of the matrix $\mH = \mX^\tran\mU$ be $\mH = \mA\bSigma\mB^\tran$, and define $\mR=\mA\mB^\tran$. If $\opnorm{\bDelta} < \frac{1}{2}\sigma_r$, then one has
\begin{align*}
\opnorm{\bLambda\mR - \mH\bLambda} &\leq \left( 2 + \frac{\sigma_1}{\sigma_r - \opnorm{\bDelta}}\right) \opnorm{\bDelta}.
\end{align*}
\end{lemma}

\includepdf[pages={-}]{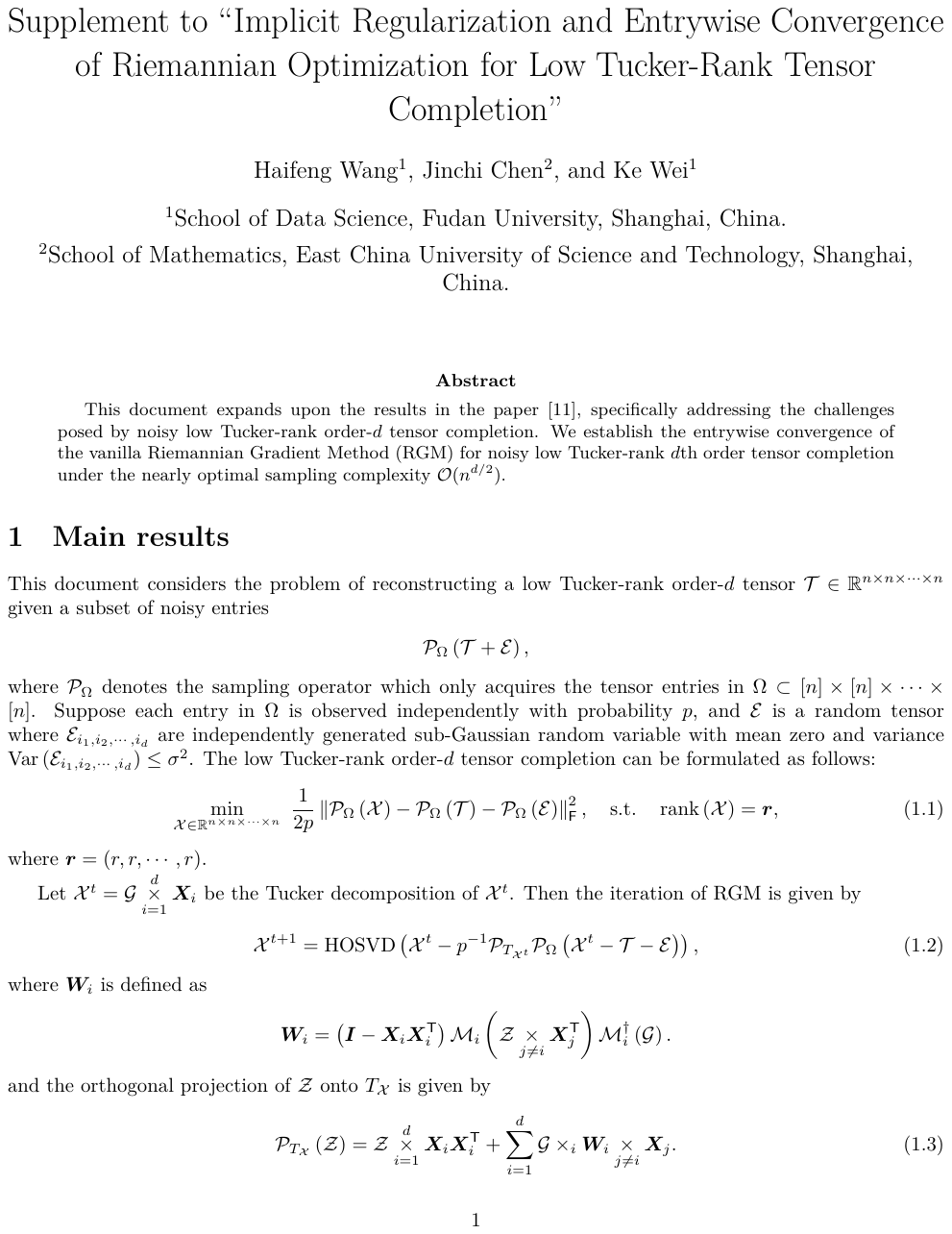}

\end{document}